\documentclass[leqno]{article}
\usepackage{fullpage}
\usepackage[T5,T1]{fontenc}
\usepackage[utf8]{inputenc}

\usepackage{lmodern}

\usepackage[sans]{dsfont}
\usepackage{amsthm, amsmath, amssymb, amsfonts, amscd}
\usepackage{stmaryrd}
\usepackage[arrow, matrix, curve]{xy}
\usepackage{mathtools} 

\usepackage{hyperref}

\usepackage{longtable}

\usepackage{todonotes}

\usepackage{subcaption}

\usepackage{ifpdf}

\usepackage{psfrag}
\usepackage{tcolorbox}

\newtcbox{\mybox}{nobeforeafter, standard jigsaw, left=0.25mm,right=0.25mm,top=0.25mm,bottom=0.25mm,boxsep=0mm,boxrule=0mm,notitle,colback=white,colframe=white,opacityframe=0.85,opacityback=0.85}
\usepackage[backend=biber,style=alphabetic,citestyle=alphabetic]{biblatex}
\usepackage{cleveref}

\DeclareFieldFormat{postnote}{#1}
\DeclareFieldFormat{multipostnote}{#1}

\title{Generic pro-$p$ Hecke algebras}
\author{Nicolas A. Schmidt}

\newcommand{\op}[1]{\operatorname{#1}}
\newcommand{\bbf}[1]{\mathds{#1}}
\newcommand{\Z}{\bbf{Z}}
\newcommand{\Q}{\bbf{Q}}
\newcommand{\R}{\bbf{R}}
\newcommand{\C}{\bbf{C}}
\newcommand{\F}{\bbf{F}}
\newcommand{\N}{\bbf{N}}

\newcommand{\axiom}[1]{\hyperref[axiom:#1]{(\textbf{#1})}}

\numberwithin{equation}{subsection}

\swapnumbers
\theoremstyle{plain}
\newtheorem*{theorem*}{Theorem}
\newtheorem{theorem}{Theorem}[subsection]
\newtheorem{lemma}[theorem]{Lemma}

\newtheorem{cor}[theorem]{Corollary}
\newtheorem{prop}[theorem]{Proposition}
\newtheorem{propdef}[theorem]{Proposition/Definition}

\theoremstyle{definition}
\newtheorem{definition}[theorem]{Definition}
\newtheorem{notation}[theorem]{Notation}
\newtheorem{terminology}[theorem]{Terminology}
\newtheorem{example}[theorem]{Example}
\newtheorem{rmk}[theorem]{Remark}

\newtheorem{convention}[theorem]{Convention}

\crefname{theorem}{theorem}{theorems}
\crefname{prop}{proposition}{propositions}
\crefname{lemma}{lemma}{lemmas}
\crefname{cor}{corollary}{corollaries}
\crefname{definition}{definition}{definitions}
\crefname{propdef}{proposition/definition}{proposition/definitions}
\crefname{rmk}{remark}{remarks}
\crefname{ex}{example}{examples}
\crefname{figure}{figure}{figures}
\crefname{terminology}{}{}

\begin{document}

\ifpdf
\DeclareGraphicsExtensions{.pdf, .jpg, .tif}
\else
\DeclareGraphicsExtensions{.eps, .jpg}
\fi

\maketitle

\begin{abstract}
   This is an extended and corrected version of the author's Diplomarbeit. A class of algebras called \textit{generic pro-$p$ Hecke algebras} is introduced, enlarging the class of generic Hecke algebras by considering certain extensions of (extended) Coxeter groups. Examples of generic pro-$p$ Hecke algebras are given by pro-$p$-Iwahori Hecke algebras and Yokonuma-Hecke algebras. The notion of an \textit{orientation} of a Coxeter group is introduced and used to define `Bernstein maps' intimately related to Bernstein's presentation and to Cherednik's cocycle. It is shown that certain relations in the Hecke algebra hold true, equivalent to Bernstein's relations in the case of Iwahori-Hecke algebras.

For a certain subclass called \textit{affine pro-$p$ Hecke algebras}, containing Iwahori-Hecke and pro-$p$-Iwahori Hecke algebras, an explicit canonical and integral basis of the center is constructed and finiteness results are proved about the center and the module-structure of the algebra over its center, recovering results of Bernstein-Zelevinsky-Lusztig and Vignéras.
\end{abstract}

\tableofcontents

\section{Introduction} 
\label{sec:Introduction}
The present article is a translated and extended version of the author's \textit{Diplomarbeit} \cite{Schmidt}. It contains various small improvements and additions as well as an entirely new proof and a generalization of the Bernstein relations, filling the gap left open by the erroneous proof in \cite{Schmidt}. Moreover, the axiomatics have been reformulated, extending the main structure theorem (see \cref{thm:structure of affine pro-$p$ Hecke algebras}) to the pro-$p$-Iwahori Hecke algebras of \textit{all} connected (possibly non-split) reductive groups.

Our initial goal in \cite{Schmidt} was to provide a self-contained account of the two articles \cite{VigGen}, \cite{VigProP} of Vignéras, both of which generalize the theory of the center of affine Hecke algebras (`Bernstein's presentation') of Bernstein-Zelevinsky and Lusztig \cite{Lusztig}, but in different directions.

The first article \cite{VigGen} develops an integral version of this theory, removing the restrictions on the ring of coefficients. Recall that affine Hecke algebras $H_q(W,S)$ are defined with respect to a base ring $R$ by generators $\{T_w\}_{w \in W}$ and relations
\begin{align}\label{eq:iwahori-matsumoto_relations_i} T_w T_{w'} & = T_{ww'} \quad\text{if}\quad \ell(w)+\ell(w') = \ell(ww') \\
   \label{eq:iwahori-matsumoto_relations_ii} T_s^2 & = q_s + (q_s - 1)T_s \quad (s \in S)
\end{align}
depending on the choice of an extended affine Weyl group $W$ (associated to some root datum $(X,\Phi,X^\vee,\Phi^\vee)$), a set $S \subseteq W$ of simple reflections (defining a length function $\ell: W \rightarrow \N$), and a family $\{q_s\}_{s \in S} \subseteq R$ of parameters subject to the constraint
\[ q_s = q_t \quad \text{if $s,t \in S$ are conjugate in $W$} \]
This general definition imposes no restrictions on the nature of the ring $R$ whatsoever. However in \cite{Lusztig}, it was assumed that the ring of coefficients be $R = \C$ and that the parameters $q_s$ are invertible\footnote{The results of \cite{Lusztig} are actually applicable to any ring $R$ as long as the $q_s$ are invertible and admit square roots.}. Traditionally, this wasn't seen as a restriction because the results of \cite{Lusztig} were usually applied in the context of \textit{complex} representations of reductive groups and the \textit{classical} Langlands program.

Let us briefly recall how affine Hecke algebras are related to reductive groups. Given a split\footnote{The splitness assumption is necessary in order to dispose of the Iwahori-Matsumoto presentation; although it was a folklore result that Iwahori-Hecke algebras of non-split groups admit an Iwahori-Matsumoto presentation with unequal parameters, there was no proof or even a precise result available until the appearance of \cite{VigProI}.} connected reductive group $G$ over a nonarchimedean local field $F$ and an Iwahori subgroup $I \leq G(F)$, a standard construction from representation theory yields the associated \textit{Iwahori-Hecke algebra} $H(G(F),I)$ over $R$. This algebra has an $R$-basis indexed by the set $I\backslash{}G(F)/I$ of double cosets with product structure given by convolution. More conceptually, the algebra $H(G(F),I)$ identifies with the endomorphism ring of the $R$-linear $G(F)$-representation $\op{ind}^{G(F)}_I \bbf{1}$ induced from the trivial representation of $I$. By Frobenius reciprocity, this induced representation also represents the functor of $I$-invariants of $G(F)$-representations, and the latter therefore lifts to a functor
\begin{equation}\label{eq:functor_of_I_invariants} \left\{ \text{$R$-linear $G(F)$-representations} \right\} \longrightarrow \left\{ \text{Right-$H(G(F),I)$-modules} \right\} \end{equation}
   relating representations of reductive groups to modules of Iwahori-Hecke algebras. Finally, affine Hecke algebras and Iwahori-Hecke algebras are related via the \textit{Iwahori-Matsumoto presentation} of $H(G(F),I)$, which defines an isomorphism
\[ H_q(W,S) \stackrel{\sim}{\longrightarrow} H(G(F),I) \]
where $W$ is the extended Weyl group of the root datum of $G$ and the parameters $q_s$ all equal the cardinality $q = p^r$ of the residue field of $F$.

Note that when $R = \C$, the parameters $q_s$ are invertible and the results of \cite{Lusztig} are applicable. On the other hand if $R$ is a field of characteristic $p$, the $q_s$ are all equal to zero. In this case there is no hope of applying the Bernstein-Zelevinsky theory as presented in \cite{Lusztig}, since the relevant constructions depend \textit{explicitly} on the invertibility of the parameters $q_s$.

In particular, Bernstein-Zelevinsky's description of the center of affine Hecke algebras had not been available to the \textit{mod $p$ Langlands program}---which aims to study representations of reductive groups in precisely this equal-characteristic situation---when it emerged in the early 2000s. The purpose of \cite{VigGen} was to remedy this fact by developing an integral version of the theory in \cite{Lusztig}. The surprising result of \cite{VigGen} was that one could completely avoid any invertibility assumptions and make the results carry over to arbitrary coefficient rings by simply replacing the Bernstein-Zelevinsky basis $\{\widetilde{\theta}_x T_w\}$ with an integral variant $\{ E_w \}$ differing from it only by explicit scalar factors.

Still, it was not clear how useful Hecke algebras would be in the study of mod $p$ representations because, in contrast to the case of characteristic zero, the functor of $I$-invariants is not exact in characteristic $p$. However, it was soon observed that a certain variant of the Iwahori subgroup enjoys a remarkable property in characteristic $p$ that almost makes up for the lack of exactness. This property goes back to the following elementary fact: a $p$-group that acts on a nonzero $\F_p$-vector space must have a nonzero fix point. It follows at once that the same holds true more generally for \textit{pro}-$p$ groups acting \textit{smoothly}, i.e. with open stabilizers, and for arbitrary coefficient rings of characteristic $p$. Thus, if one replaces an Iwahori subgroup $I$ by its maximal open normal pro-$p$ subgroup $I(1) \leq I$, the analogue
\begin{equation*} \left\{ \text{$R$-linear $G(F)$-representations} \right\} \longrightarrow \left\{ \text{Right-$H(G(F),I(1))$-modules} \right\} \end{equation*}
   of the functor \eqref{eq:functor_of_I_invariants} above sends \textit{nonzero smooth} representations to \textit{nonzero} modules (while still being not exact of course). This remarkable property has some important consequences. For example, it implies the following practical irreducibility criterion: a $G(F)$-representation $V$ generated by its $I(1)$-invariants $V^{I(1)}$ is irreducible if the $H(G,I(1))$-module $V^{I(1)}$ is simple.
   
   The subgroup $I(1) \leq I$ and the algebra $H(G(F),I(1))$ were introduced by Vignéras in the second article \cite{VigProP}, where they were named `pro-$p$-Iwahori group' and `pro-$p$-Iwahori Hecke algebra' respectively. Since the appearance of \cite{VigProP}, these `higher congruence analogues' of the Iwahori-Hecke algebras have proven to be of ever-growing importance in the mod $p$ Langlands program (see also \cite{Flicker} for an application in the classical context).

   Having removed the restrictions on the ring of coefficients in \cite{VigGen}, in \cite{VigProP} Vignéras re-developed this new integral Bernstein-Zelevinsky theory in the context of pro-$p$-Iwahori Hecke algebras (of split groups). Surprisingly, the results carried over almost verbatim. However, the methods of proof were different as \cite{VigProP} dealt with the concrete convolution Hecke algebras $H(G(F),I(1))$ and not with abstract Hecke algebras $H_q(W,S)$ defined by generators and relations. As a result, the proofs in \cite{VigProP} were less elementary as they assumed some familiarity with reductive groups. A more serious consequence was that one could not take advantage of a reduction argument (the `specialization argument', see \cref{rmk:specialization_argument}) available in the abstract setting that often allows one to reduce statements to the case of invertible parameters.
   
   For these reasons, it seemed desirable if there was a `pro-$p$ analogue' of the affine Hecke algebras. Our first contribution was to verify that such an analogue exists: the \textit{generic pro-$p$ Hecke algebras} that formed the subject of \cite{Schmidt}.   

   More precisely, generic pro-$p$ Hecke algebras are the pro-$p$ analogues of \textit{generic} Hecke algebras. The latter generalize affine Hecke algebras by allowing $(W,S)$ to be abstract `extended Coxeter groups' instead of just extended affine Weyl groups. The generic pro-$p$ Hecke algebras which are analogous to affine Hecke algebras and to which the Bernstein-Zelevinsky method applies are the \textit{affine pro-$p$ Hecke algebras} (see \cref{def:affine_pro-$p$_Hecke_algebra}).
   
   Like generic Hecke algebras, generic pro-$p$ Hecke algebras are associated to a `Coxeter-like' group $W^{(1)}$ equipped with a length function $\ell:W^{(1)} \rightarrow \N$ and a set of parameters, and are equipped with a linear basis $\{T_w\}$ indexed by $W^{(1)}$ such that relations similar to \eqref{eq:iwahori-matsumoto_relations_i}, \eqref{eq:iwahori-matsumoto_relations_ii} above (see \cref{def:generic pro-$p$ Hecke algebras} for details) hold true.
   
   However, there are two essential differences. First of all, the $W^{(1)}$ aren't extended Coxeter groups but \textit{extensions}
   \[ \begin{xy} \xymatrix{ 1 \ar[r] & T \ar[r] & W^{(1)} \ar[r] & W \ar[r] & 1 } \end{xy} \]
      of extended Coxeter groups by abelian groups (where the group $T$ is not to be confused with the basis $\{T_w\}$). In particular even for affine pro-$p$ Hecke algebras, the representation as a group of isometries of a real affine space the groups $W^{(1)}$ come equipped with is in general\footnote{The general case being $T \neq 1$; in the degenerate case $T = 1$, the notion of generic pro-$p$ Hecke algebras reduces to that of generic Hecke algebras, making the latter a special case of the former.} not faithful. In other words, the groups $W^{(1)}$ are only `geometric up to $T$', which adds an extra layer of difficulty to many statements whose analogues for affine Hecke algebras have purely geometric proofs. But, this difficulty also forces one to recognize structures that remain hidden in the classical case. Namely, of great importance is the existence of a family $(n_s)_{s \in S}$ of lifts of the simple reflections $s \in S$ to the group $W^{(1)}$ which satisfy the braid relations
      \begin{equation} \underbrace{n_s n_t n_s \ldots}_{\text{$m$ factors}} = \underbrace{n_t n_s n_t \ldots}_{\text{$m$ factors}} \quad \text{if $st$ is of order $m < \infty$} \end{equation}
         For generic Hecke algebras, the canonical choice $n_s = s$ renders this point trivial. But, even in this case it is ultimately the existence of this family (cf. the proof of \cref{thm:ex_theta_braid}) that allows the construction of the `Bernstein functions' $\theta$ on which the Bernstein-Zelevinsky theory rests. Showing the existence of such lifts $n_s \in W^{(1)}$ in examples associated to reductive groups is nontrivial (cf. \cref{lem:lifting_braid_relations}) and related to describing normalizers of maximal tori in split reductive groups, a problem that has been studied in depth by Tits \cite{TitsTores}. In fact, his `groupes de Coxeter étendu' are almost the same as our `pro-$p$ Coxeter groups' (see \cref{sub:A characterization of pro-$p$ Coxeter groups in terms of X}).
      
         The second difference is that the analogues of the quadratic relations \eqref{eq:iwahori-matsumoto_relations_ii} are more delicate and that generic pro-$p$ Hecke algebras must therefore be viewed as objects $\mathcal{H}^{(1)} = \mathcal{H}^{(1)}(a,b)$ that depend\footnote{The dependence on the group $W^{(1)}$ is surpressed in the notation.} on \textit{two} families $\{a_s\}_s$, $\{b_s\}_s$ of parameters, instead of one family $\{q_s\}_s$, and are thus actually pro-$p$ analogues of the two-parameter generic Hecke algebras $H_{a,b}(W,S)$ which are defined like $H_q(W,S)$ but with the quadratic relation $T_s^2 = q_s + (q_s - 1)T_s$ replaced by $T_s^2 = a_s + b_s T_s$.
      But whereas working with two parameters is a convenience in the classical case, in the pro-$p$ case it becomes a necessity because the parameters $b_s$ appearing in the quadratic relations
   \[ T_{n_s}^2 = a_s T_{n_s^2} + b_s T_{n_s} \]
   are longer elements of the ground ring $R$ but elements of the group ring $R[T]$ of $T$ (viewed as a subalgebra of $\mathcal{H}^{(1)}$ by identifying an element $t \in T$ with the basis element $T_t \in \mathcal{H}^{(1)}$). Thus, there is no sensible one-parameter version of the generic pro-$p$ Hecke algebras as the parameters $a_s$, $b_s$ never satisfy a simple relation like $b_s = a_s - 1$ in interesting examples. Yet, even for generic Hecke algebras it is fruitful to let $a_s$ and $b_s$ vary independently because then (and only then!) it is possible to reduce statements to the case $a_s = 1$ using the `specialization argument' mentioned before, where formulas take on a particularly simple form.
   
   With these abstract versions of the pro-$p$-Iwahori Hecke algebras at our disposal, our next goal was to redevelop the integral Bernstein-Zelevinsky theory of \cite{VigProP} using generic algebra methods as in \cite{Lusztig} and \cite{VigGen}. Recall that the method of Bernstein-Zelevinsky\footnote{The description of the center of affine Hecke algebras for `constant parameters' was obtained by Bernstein and Zelevinsky in an unpublished work. Their results were generalized by Lusztig in \cite{Lusztig}, which is the canonical reference for the theory.} rests on the decomposition
   \[ W = X\rtimes W_0 \]
   of $W$ into a semi-direct product of a finite group $W_0$ (`Weyl group') and a `large' abelian subgroup $X$ (`lattice of translations') provided by the root datum $(X,\Phi,X^\vee,\Phi^\vee)$ giving rise to $W$, with the group $W_0$ acting on $X$; because the group law on the abelian subgroup $X$ is traditionally written additively, one uses the exponential notation $\tau^x$ when viewing an $x \in X$ as an element of $W$, in order to avoid confusion. With this convention, the action of $W_0$ can be written as
   \[ w(x) = w \tau^x w^{-1} \]
   To the commutative subgroup $X$ now corresponds a commutative subalgebra $\mathcal{A} \subseteq H_q(W,S)$ via a group homomorphism\footnote{$\widetilde{\theta}$ corresponds to the map denoted by $\theta$ in \cite{Lusztig}. We use the notation $\widetilde{\theta}$ in order to be consistent with the notation for the \textit{normalized Bernstein maps} $\widetilde{\theta}_{\mathfrak{o}}$ to be introduced later that generalize $\widetilde{\theta}$ and have `unnormalized' counterparts denoted by $\theta_{\mathfrak{o}}$.}
   \[ \widetilde{\theta}: X \rightarrow H_q(W,S)^\times \]
   such that the $\widetilde{\theta}(x)$, $x \in X$ form a basis of $\mathcal{A}$ and $W_0$ acts on it via algebra automorphisms permuting the basis elements. The main result of the Bernstein-Zelevinsky method is the equality
   \[ Z(H_q(W,S)) = \mathcal{A}^{W_0} \]
   between the center of the Hecke algebra and the invariants of this commutative subalgebra under the action of the Weyl group. Proving this equality usually proceeds in two steps (cf. \cite{Lusztig}). First, one establishes that $\mathcal{A}^{W_0}$ lies in the center using the \textit{Bernstein relations}, and then one shows---using this inclusion---that equality must hold. In \cite{LusztigSingular} the last step is justified by referring to a `Nakayama argument' (without providing details). Here and in \cite{Schmidt}, we follow the mentioned outline but replace the `Nakayama argument' with a combination of an induction (\cref{thm:center}) and an explicit computation (\cref{prop:centralizer}) that shows that the subalgebra $\mathcal{A}$ equals its own centralizer (in the `\textit{split} case'; in the `\textit{non-split} case' it is a proper subalgebra of $\mathcal{A}$ in general). This step isn't difficult although somewhat convoluted (especially in the `non-split case').
   
   The essential difficulty of the Bernstein-Zelevinsky method lies in establishing the Bernstein relations. In \cite{Lusztig}, they are stated in the following form (restated here for two parameters). Given a reflection $s = s_\alpha \in W_0 \cap S$ attached to a simple root $\alpha: X \rightarrow \Z$ and an element $x \in X$, we have
   \begin{equation}\label{eq:classical_bernstein} \widetilde{\theta}(x)T_s - T_s \widetilde{\theta}(s(x)) = \begin{dcases}
      b_s \frac{\widetilde{\theta}(x) - \widetilde{\theta}(s(x))}{1-\widetilde{\theta}(-\alpha^\vee)} &\text{ if } \alpha(X) = \Z \\
      a_s^{1/2} \left(a_s^{-1/2} b_s + a_{s'}^{-1/2} b_{s'} \widetilde{\theta}(-\alpha^\vee)\right) \frac{\widetilde{\theta}(x) - \widetilde{\theta}(s(x))}{1-\widetilde{\theta}(-2\alpha^\vee)} & \text{ if } \alpha(X) = 2\Z
   \end{dcases} \end{equation}
where $\alpha^\vee \in X$ denotes the dual coroot of $\alpha$ and $s' \in S$ is any simple reflection conjugate to the affine reflection $s_{\alpha,1} = \tau^{-\alpha^\vee} s_\alpha \tau^{\alpha^\vee}$ in $W$. The homomorphism $\widetilde{\theta}$ is defined as
   \[ \widetilde{\theta}(x) = \widetilde{T}_y \widetilde{T}_z^{-1} \]
   where $y,z \in X$ are any two elements lying in the dominant cone that satisfy $x = y-z$, and the $\widetilde{T}_w$ are normalized versions of the $T_w$ determined by $\widetilde{T}_s = a_s^{-1/2} T_s$ and the analogues of the braid relations \eqref{eq:iwahori-matsumoto_relations_i}.
      
   In \cite{VigProP}, Vignéras established analogues of the Bernstein relations for pro-$p$-Iwahori Hecke algebras; her proof closely followed Lusztig's intricate computational proof \cite{Lusztig}. Shortly after her article appeared, Görtz published a simple geometric proof \cite{Goertz} of the Bernstein relations for affine Hecke algebras. When we learnt of his article, we hoped that his geometric approach might work for pro-$p$-Iwahori Hecke algebras too. His proof was based on Ram's theory of alcove walk algebras \cite{Ram}. The main input of that theory to Görtz' proof is a geometric interpretation of the elements $\widetilde{\theta}(x)$ based on identifying formal expressions in the Hecke algebra like\footnote{In order to avoid some minor subtleties arising from the group $\Omega$ of elements of length zero, we assume that $\Omega = 1$, i.e. that the extended affine Weyl group $W$ coincides with the affine Weyl group.}
   \[ \widetilde{T}_{s_1}^{\varepsilon_1} \widetilde{T}_{s_2}^{\varepsilon_2} \ldots \widetilde{T}_{s_r}^{\varepsilon_r}\quad (\varepsilon_i \in \{ \pm 1 \}) \]
   with `coloured' or `signed' galleries (i.e. `unfolded alcove walks' in the terminology of \cite{Ram}) in the Coxeter complex starting at the base alcove $C$, the above expression corresponding to the gallery
   \[ \Gamma = (C_0 = C,\ C_1 = s_1 C,\ C_2 = s_1 s_2 C,\ \dots ,\ C_r = s_1 \dots s_r C) \]
   from $C$ to $wC$, where $w = s_1 \ldots s_r$ and the colour of the arrow from $C_{i-1}$ to $C_i$ is determined by the sign $\varepsilon_i$, as in \cref{fig:coloured-gallery}.
   \begin{figure}[t]
      \centering
      \psfrag{C0}[Bc][bl][1.5]{$C_0$}
      \psfrag{Cr}[Bc][bl][1.5]{$C_5$}
      \psfrag{g1}[Bc][bl]{\mybox{$T_{s_1}$}}
      \psfrag{g2}[Bc][bl]{\mybox{$T_{s_2}$}}
      \psfrag{g3}[Bc][bl]{\mybox{$T_{s_1}$}}
      \psfrag{g4}[Bc][bl]{\mybox{$T_{s_2}^{-1}$}}
      \psfrag{g5}[Bc][bl]{\mybox{$T_{s_1}^{-1}$}}
      \includegraphics[width=0.5\linewidth]{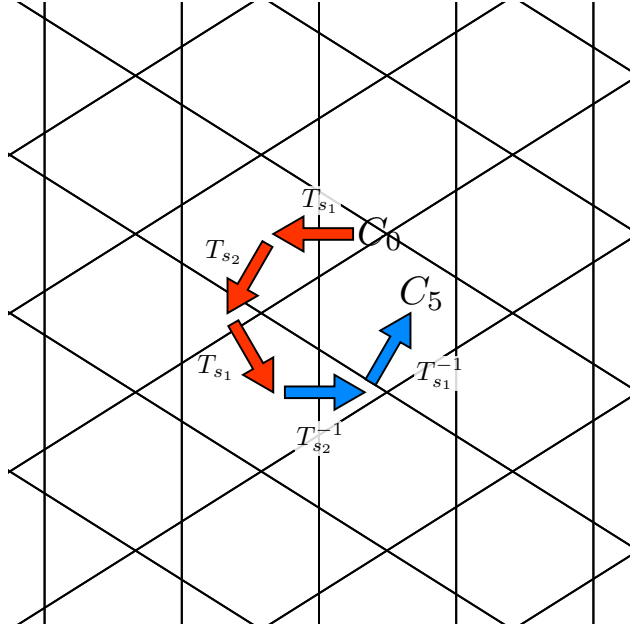}
      \caption{The coloured gallery $\Gamma = (C_0,\dots,C_5)$ corresponding to the expression $T_{s_1}T_{s_2}T_{s_1}T_{s_2}^{-1}T_{s_1}^{-1}$ in the affine Coxeter complex of type $\widetilde{A}_2$.}
      \label{fig:coloured-gallery}
   \end{figure}
   Expanding $\widetilde{T}_y$ and $\widetilde{T}_z$ in the definition of $\widetilde{\theta}(x) = \widetilde{T}_y \widetilde{T}_z^{-1}$ into a product of generators $\widetilde{T}_s$, it is easy to see that $\widetilde{\theta}(x)$ is given by \textit{some} coloured gallery in this way.
   
   The key point however is this: fixing an `orientation' (see \cref{def:or}), there is a canonical way to colour every ordinary (uncoloured) gallery starting at $C$ in such a way that \textit{any} two such coloured galleries having the \textit{same endpoint} define the same element in the Hecke algebra. For the `spherical orientation' attached to the dominant Weyl chamber (see \cref{def:spherical_orientation}), this is the content of the following theorem, quoted verbatim from \cite{Goertz} ($W$ corresponding to `$W_a$', and $C$ to `$\mathbf{a}$' in his notation):
   \begin{theorem}[{\cite[Theorem 1.1.1]{Goertz}}]\label{thm:ram} Let $w \in W_a$. For an expression
   \begin{equation}\label{eq:goertz}\tag{1.1.1}
      w = s_{i_1}\ldots s_{i_n}
   \end{equation}
   of $w$ as a product in the generators (which does not have to be reduced), consider the element
   \[ \Psi(w) := T_{s_{i_1}}^{\varepsilon_1} \ldots T_{s_{i_n}}^{\varepsilon_n} \]
   in the affine Hecke algebra, where the $\varepsilon_\nu \in \{\pm 1\}$ are determined as follows. Let $\mathbf{b}$ be an alcove far out in the anti-dominant chamber (``far out'' depends on $w$, and the result will then be independent of $\mathbf{b}$, see section 2.3 for a precise definition). For each $\nu$, consider the alcove $\mathbf{c}_\nu := s_{i_1}\ldots s_{i_{\nu - 1}}\mathbf{a}$, and denote by $H_\nu$ the affine root hyperplane containing its face of type $i_\nu$. We set
   \[ \varepsilon_\nu := \begin{cases}
      1 &\text{ if $\mathbf{c}_\nu$ is on the same side of $H_\nu$ as $\mathbf{b}$} \\
      -1 & \text{ otherwise}
   \end{cases} \]
   Then the element $\Psi(w)$ is independent of the choice of the expression (1.1.1).
\end{theorem}
 
The theorem also holds true with the $T_s$ replaced by the $\widetilde{T}_s$, and the galleries corresponding to the expressions
\[ \widetilde{T}_{s_1}\dots\widetilde{T}_{s_n}\widetilde{T}_{t_m}^{-1}\dots\widetilde{T}_{t_1}^{-1} \]
arising from expanding $\widetilde{T}_y$ and $\widetilde{T}_z$ in $\widetilde{\theta}(x) = \widetilde{T}_y \widetilde{T}_z^{-1}$ into \textit{reduced} products (i.e. $\ell(y) = n$, $\ell(z) = m$)
\[ \widetilde{T}_y = \widetilde{T}_{s_1} \dots \widetilde{T}_{s_n} \quad \text{ and }\quad \widetilde{T}_z = \widetilde{T}_{t_1}\ldots \widetilde{T}_{t_m} \]
of the generators $\widetilde{T}_s$ are easily seen to be coloured according to the method given in the theorem. Therefore, $\widetilde{\theta}(x)$ is given by \textit{any} canonically coloured gallery from $C$ to $x+C$, giving the Bernstein homomorphism $\widetilde{\theta}$ a very natural and intuitive interpretation in terms of alcove walks. This geometric interpretation fueled Görtz geometric proof of the Bernstein relations, reducing it essentially to a telescopic sum expansion of the left hand side, with each summand posessing a geometric interpretation as an alcove walk. Before we explain this in more detail, let us note some further consequences of the above theorem. These consequences played no explicit role in Görtz' original proof, but they allow us to recast it in a way that makes it adaptable to the pro-$p$ case. In order to simplify the exposition, we will discuss everything in the affine case first, making only some brief indications on how the pro-$p$ case differs, and then later discuss the pro-$p$ case more fully.

The first thing to note is that the above theorem suggests to extend the Bernstein homomorphism to a map defined on all of $W$. Further, we will see in a moment that it is useful to explicitly denote the dependence on the orientation. Let us therefore write $\widetilde{\theta}_{\mathfrak{o}}(w)$ for the element defined by a gallery from $C$ to $w(C)$ that is coloured according to the orientation $\mathfrak{o}$, and let $\mathfrak{o}$ denote the spherical orientation attached to the dominant Weyl chamber (hence $\widetilde{\theta} = \widetilde{\theta}_{\mathfrak{o}}$) in the following. The group $W$ naturally acts on orientations from the right such that the signs assigned to a gallery $\Gamma$ by $\mathfrak{o}\bullet{}w$ are the ones assigned to $w(\Gamma)$ by $\mathfrak{o}$. Granting the theorem, the definitions then immediately imply the following \textit{cocycle rule} (called \textit{product formula} in \cite{Schmidt}):
\begin{equation}\label{eq:cocycle-rule} \widetilde{\theta}_{\mathfrak{o}}(ww') = \widetilde{\theta}_{\mathfrak{o}}(w) \widetilde{\theta}_{\mathfrak{o}\bullet{}w}(w') \end{equation}
   The cocycle rule recovers the homomorphism property of the `Bernstein homomorphism' because the subgroup $X$ acts trivial on $\mathfrak{o}$ (indeed on all spherical orientations). Moreover, the cocycle rule implies the formula
   \[ \widetilde{\theta}_{\mathfrak{o}}(w)^{-1} = \widetilde{\theta}_{\mathfrak{o}\bullet{}w}(w^{-1}) \]
   Because the spherical orientation $\mathfrak{o}$ of the dominant Weyl chamber has the property that $\widetilde{\theta}_{\mathfrak{o}}(s) = \widetilde{T}_s^{-1}$ for all simple reflections $s \in W_0$, it follows from the cocycle rule that one can rewrite the second summand on the left hand side of \eqref{eq:classical_bernstein} as
\[ T_s \widetilde{\theta}(s(x)) = a_s^{1/2} \widetilde{\theta}_{\mathfrak{o}\bullet{}s}(s) \widetilde{\theta}_{\mathfrak{o}}(s(x)) = a_s^{1/2} \widetilde{\theta}_{\mathfrak{o}\bullet{}s}(s \tau^{s(x)}) = a_s^{1/2} \widetilde{\theta}_{\mathfrak{o}\bullet{}s}(\tau^x s) \]
The first summand can't be rewritten in this way, but since
\[ \widetilde{\theta}_{\mathfrak{o}\bullet{}s}(x) T_s = a_s^{1/2} \widetilde{\theta}_{\mathfrak{o}\bullet{}s}(x) \widetilde{\theta}_{\mathfrak{o}\bullet{}s}(s) = a_s^{1/2} \widetilde{\theta}_{\mathfrak{o}\bullet{}s}(\tau^x s) \]
it follows that
\[ \widetilde{\theta}(x) T_s - T_s \widetilde{\theta}(s(x)) = a_s^{1/2} \left(\widetilde{\theta}_{\mathfrak{o}}(x) - \widetilde{\theta}_{\mathfrak{o}\bullet{}s}(x)\right) \widetilde{\theta}_{\mathfrak{o}\bullet{}s}(s) \]
The proof of the Bernstein relations therefore comes down to computing the difference $\widetilde{\theta}_{\mathfrak{o}}(x) - \widetilde{\theta}_{\mathfrak{o}\bullet{}s}(x)$. To carry out this computation, one \textit{chooses} an explicit expression $\tau^x = s_1 \ldots s_r$ (not necessarily reduced) and writes this difference as a telescopic sum
\[ \widetilde{\theta}_{\mathfrak{o}}(x) - \widetilde{\theta}_{\mathfrak{o}\bullet{}s}(x) = \widetilde{T}_{s_1}^{\varepsilon_1} \ldots \widetilde{T}_{s_r}^{\varepsilon_r} - \widetilde{T}_{s_1}^{\varepsilon'_1} \ldots \widetilde{T}_{s_r}^{\varepsilon'_r} = \sum_i \widetilde{T}_{s_1}^{\varepsilon_1} \ldots \widetilde{T}_{s_{i-1}}^{\varepsilon_i} \left(\widetilde{T}_{s_i}^{\varepsilon_i} - \widetilde{T}_{s_i}^{\varepsilon'_i}\right) \widetilde{T}_{s_{i+1}}^{\varepsilon'_{i+1}} \ldots \widetilde{T}_{s_r}^{\varepsilon'_r} \]
Since the sum needs only to be taken over the indices $i$ where $\varepsilon_i \neq \varepsilon'_i$, one can use the quadratic relations in the Hecke algebra written in the symmetric form $\widetilde{T}_s - \widetilde{T}_s^{-1} = a_s^{-1/2} b_s$ to simplify each summand:
\[ \widetilde{T}_{s_1}^{\varepsilon_1} \ldots \widetilde{T}_{s_{i-1}}^{\varepsilon_i} \left(\widetilde{T}_{s_i}^{\varepsilon_i} - \widetilde{T}_{s_i}^{\varepsilon'_i}\right) \widetilde{T}_{s_{i+1}}^{\varepsilon'_{i+1}} \ldots \widetilde{T}_{s_r}^{\varepsilon'_r} = \varepsilon_i a_{s_i}^{-1/2} b_{s_i} \widetilde{T}_{s_1}^{\varepsilon_1} \ldots \widetilde{T}_{s_{i-1}}^{\varepsilon_i} \widetilde{T}_{s_{i+1}}^{\varepsilon'_{i+1}} \ldots \widetilde{T}_{s_r}^{\varepsilon'_r} \]
The crucial point of the proof now is to recognize each summand as something defined a priori, without reference to the particular chosen expression $\tau^x = s_1 \ldots s_r$. This step is very delicate in the pro-$p$ case, which is why we were unable to follow Görtz' proof and instead had to rely on Lusztig's computational approach in \cite{Schmidt}. Luckily, our proof of the Bernstein relations was not only unsatisfying but also plain \textit{wrong}\footnote{We are indebted to M. F. Vignéras for pointing this out.}. When we tried to find a new proof in the course of writing this article, it became apparent that a purely geometric proof of the Bernstein relations for pro-$p$ Hecke algebras could only exist if a miracle happened: this miracle is \cref{propdef:bernstein}.

For affine Hecke algebras no miracle beyond Görtz' theorem is needed to see that
\begin{align}\label{eq:recognize_element} \widetilde{T}_{s_1}^{\varepsilon_1} \ldots \widetilde{T}_{s_{i-1}}^{\varepsilon_{i-1}} \widetilde{T}_{s_{i+1}}^{\varepsilon'_{i+1}} \ldots \widetilde{T}_{s_r}^{\varepsilon'_r} & = \widetilde{\theta}_{\mathfrak{o}}(s_1 \ldots \widehat{s_i} \ldots s_r) \\
 \notag  & = \widetilde{\theta}_{\mathfrak{o}}(s_{H_i} \tau^x) = \widetilde{\theta}_{\mathfrak{o}}(s_{H_i}) \widetilde{\theta}_{\mathfrak{o}\bullet{}s}(x) \end{align}
 writing $s_{H_i} = (s_1 \ldots s_{i-1}) s_i (s_1 \ldots s_{i-1})^{-1}$ for the reflection at the $i$-th (affine) hyperplane $H_i$ crossed by the gallery $\Gamma = (C, s_1 C, s_1 s_2 C, \ldots{}, x + C)$, and denoting with $\widehat{s_i}$ the omission of the element $s_i$ from the sequence. One therefore arrives at the formula
\begin{equation}\label{eq:alternative_formulation_of_classical_bernstein} \widetilde{\theta}_{\mathfrak{o}}(x) - \widetilde{\theta}_{\mathfrak{o}\bullet{}s}(x) = \left(\sum_i \varepsilon_i a_{s_i}^{-1/2} b_{s_i} \widetilde{\theta}_{\mathfrak{o}}(s_{H_i})\right) \widetilde{\theta}_{\mathfrak{o}\bullet{}s}(x) \end{equation}
   If the expression $\tau^x = s_1 \ldots s_r$ is taken to be reduced, the $H_i$ appearing in the above sum are precisely the hyperplanes separating $C$ and $x+C$ at which the orientations $\mathfrak{o}$ and $\mathfrak{o}\bullet{}s$ disagree (i.e. those parallel to the hyperplane $H_\alpha$ defined by the root $\alpha$). Since $a_{s_i}$, $b_{s_i}$, and the sign $\varepsilon_i$ only depend on the hyperplane $H_i$, the whole sum is therefore purely geometric and independent of the chosen expression. The classical Bernstein relations \eqref{eq:classical_bernstein} are now easily derived from \eqref{eq:alternative_formulation_of_classical_bernstein} using the identity
\[ \widetilde{\theta}_{\mathfrak{o}}(s_H) \widetilde{\theta}_{\mathfrak{o}\bullet{}s}(x) \widetilde{\theta}_{\mathfrak{o}\bullet{}s}(s) = \widetilde{\theta}_{\mathfrak{o}}(s_H s) \widetilde{\theta}_{\mathfrak{o}}(s(x)) \]
and by recognizing $\sum_i \widetilde{\theta}_{\mathfrak{o}}(s_{H_i} s)$ as a geometric sum.

The proof sketched above is a reformulation of the proof of Görtz: his proof was direct and didn't involve formula \eqref{eq:alternative_formulation_of_classical_bernstein}. Although a general notion of orientation was defined there, the discussion in \cite{Goertz} was restricted to the spherical orientation $\mathfrak{o}$ attached to the dominant Weyl chamber and its associated Bernstein map $\widetilde{\theta}_{\mathfrak{o}}$; in particular, neither the cocycle rule nor \eqref{eq:alternative_formulation_of_classical_bernstein} appeared.

As we've already seen, the cocycle rule is important because it simplifies many computations. It also brings the connection to the work of Cherednik (see below), and forms the proper basis for the definition of the Bernstein maps in the case of extended or pro-$p$ Coxeter groups. Its discovery---a lucky byproduct of pedantic notation---was the origin of \cite{Schmidt}. Its impetus led us to consider $\{\widetilde{\theta}_{\mathfrak{o}}(w)\}_{w \in W}$ and its integral version $\{ \widehat{\theta}_{\mathfrak{o}}(w) \}_{w \in W}$ instead of the traditional Bernstein-Zelevinsky basis $\{ \widetilde{\theta}_{\mathfrak{o}}(x) T_w \}_{x \in X,\ w \in W_0}$ and its integral analogue $\{ E_w \}_{w \in W}$ defined in \cite{VigGen}. This resulted in the following integral analogue of the cocycle rule (see \cref{cor:hat mult rule}), generalizing the formula for the product $E_{w_0 x} E_{x'}$ given in \cite{VigGen}:
\begin{equation}\label{eq:integral_cocycle_rule} \widehat{\theta}_{\mathfrak{o}}(w)\widehat{\theta}_{\mathfrak{o}\bullet{}w}(w') = \overline{\bbf{X}}(w,w') \widehat{\theta}_{\mathfrak{o}}(ww') \end{equation}
   The factor $\overline{\bbf{X}}(w,w')$ that appears in this formula played an important role in establishing the integral theory, both in \cite{VigGen} and in \cite{Schmidt}. In \cite{VigGen}, it appeared\footnote{In a slightly disguised form; see \cref{rmk:gen_len_product_rule}(ii) for details.} in the crucial `lemme fondamental' (\cite{VigGen}[1.2]), which was not explicitly mentioned in \cite{Schmidt} but which we recover here in \cref{lem:gen_len_product_rule}. In \cite{Schmidt}, the factor $\overline{\bbf{X}}(w,w')$ entered through its relation to another map $\gamma$ (\cite{Schmidt}[Lemma 3.3.26]) that was used to relate the integral Bernstein map $\widehat{\theta}_{\mathfrak{o}}$ to its non-integral counterpart $\theta_{\mathfrak{o}}$ but was given no further interpretation. Here, we show that the `lemme fondamentale' and \cite{Schmidt}[Lemma 3.3.26] can be seen as exhibiting $\overline{\bbf{X}}$ as a $2$-coboundary in two different ways (see \cref{rmk:gamma_product_rule_as_coboundary_equation} for details), the latter exhibiting $\overline{\bbf{X}}$ as the coboundary of $\mathfrak{\gamma}$.
   
   Another interesting consequence of the cocycle rule---further emphasizing the importance to consider \textit{all} spherical orientations---was the realization that the basis of the center of the Hecke algebra provided by the orbit sums
   \[ z^{\mathfrak{o}}_{\gamma} = \sum_{x \in \gamma} \widetilde{\theta}_{\mathfrak{o}}(x),\quad \gamma \in W_0\backslash X \]
   under the equality
   \[ Z(H_q(W,S)) = \mathcal{A}_{\mathfrak{o}}^{W_0} \]
   is in fact \textit{canonical}, i.e. \textit{independent} of the choice of $\mathfrak{o}$. In fact, the independence of the element $z^{\mathfrak{o}}_{\gamma}$ from $\mathfrak{o}$ turns out to be equivalent to the fact that it lies in the center (see the proof of \cref{prop:invariants_lie_in_center}).

   Unfortunately, in \cite{Schmidt} we couldn't realize this geometric approach to pro-$p$ Hecke algebras to its full potential as we were unable to tranpose Görtz' proof into the context of pro-$p$ Hecke algebras. In addition, the proof of the Bernstein relations we gave was flawed. And so, although we achieved our goal of developing an abstract theory of pro-$p$ Hecke algebras and of re-deriving the integral Bernstein-Zelevinsky theory of \cite{VigProP} in this context, \cite{Schmidt} remained incomplete in a technical and a moral sense.

   Fortunately, these issues are resolved in this article. We give a new and purely geometric proof of the Bernstein relations for pro-$p$ Hecke algebras, based on formula \eqref{eq:alternative_formulation_of_classical_bernstein} derived above. First of all, there is no need to restrict to elements $x \in X$ in this formula: it remains true for any element of $W$. Second, its proof never explicitly used that $W$ is an affine Weyl group. This fact only entered indirectly through the properties of the orientations $\mathfrak{o}$ and $\mathfrak{o}\bullet{}s$ used in deducing \eqref{eq:recognize_element}. By abstracting these properties (and using the `miracle proposition' to extend to the pro-$p$ case), we can thus prove the following generalization of \eqref{eq:alternative_formulation_of_classical_bernstein}, holding for any generic pro-$p$ Hecke algebra whose parameters $a_s$ are invertible and squares (see \cref{thm:bernstein_relation})
\begin{equation}\label{eq:generalized_bernstein_relations} \widetilde{\theta}_\mathfrak{o}(w)-\widetilde{\theta}_{\mathfrak{o}'}(w) = \left(\sum_H \mathfrak{o}(1,H) \Xi_{\mathfrak{o}'}(H)\right) \widetilde{\theta}_\mathfrak{o}(w) \end{equation}
   Here, $\mathfrak{o}$, $\mathfrak{o}'$ denotes a pair of \textit{adjacent} orientations (see \cref{def:adjacency_of_or}). The only non-trivial examples of such pairs that we know of are given by $\mathfrak{o}, \mathfrak{o}\bullet{}s$ (and their $W$-translates) when $W$ is an affine Weyl group. Still, even in this case \eqref{eq:generalized_bernstein_relations} gives new relations in the Hecke algebra. Moreover, phrasing the Bernstein relations in this abstract generality makes the proof cleaner and more transparent, especially in the pro-$p$ case.

   We will now discuss the contents of this article in detail.
   After recalling the notion of Coxeter groups and some common (and maybe not so common) related geometric terminology, we introduce in \cref{sub:Basic definitions and some geometric terminology} two (successive) generalizations of the notion of Coxeter groups, \textit{extended} and \textit{pro-$p$} Coxeter groups, designed to capture the essential properties of the groups appearing in the Bruhat decomposition for Iwahori and pro-$p$-Iwahori groups.

   In \cref{sub:1-Cocycles of pro-$p$ Coxeter groups} we give a classification of $1$-cocycles of pro-$p$ Coxeter groups that is later used to construct the \textit{Bernstein maps}. \Cref{sub:Construction of generic pro-$p$ Hecke algebras} that follows is fundamental: we define generic pro-$p$ Hecke algebras and show that they behave like generic Hecke algebras (existence of a canonical linear basis, Iwahori-Matsumoto relations); we also give a first proper example of generic pro-$p$ Hecke algebras, the \textit{Yokonuma-Hecke} algebras. The Iwahori-Matsumoto presentation of generic pro-$p$ Hecke algebras is proved in the following \cref{sub:Presentations of generic pro-$p$ Hecke algebras via (generalized) braid groups}, phrased in the language of braid groups.
   
   \Cref{sub:Orientations of Coxeter groups} introduces another fundamental concept, the notion of \textit{orientations} of Coxeter groups, which abstracts and generalizes similar notions considered earlier by \cite{Goertz} and \cite{Ram}, and besides the classification of $1$-cocycles is the second ingredient in the construction of the Bernstein maps. The set $\mathcal{O}$ of all orientations of a Coxeter group $W$ is shown to be endowed with the structure of a compact Hausdorff topological space acted upon by $W$, and the group $W$ is embedded as a subset in $\mathcal{O}$ in two different ways, associating to an element $w \in W$ the orientation $\mathfrak{o}_w$ \textit{towards} $w$ and the orientation $\mathfrak{o}_w^{\op{op}}$ \textit{away} from $w$, both embeddings being exchanged under a canonical involution $\mathfrak{o} \mapsto \mathfrak{o}^{\op{op}}$ on $\mathcal{O}$. It is shown that the images of these embeddings give \textit{all} orientations when $W$ is finite, but that there must exist other orientations when $W = (W,S)$ is infinite and $\# S < \infty$, the \textit{boundary orientations}. The notion of orientation is then transferred in a natural way onto extended and pro-$p$ Coxeter groups, such that the set of orientations of an extended or pro-$p$ Coxeter group is in canonical bijection with the set of orientations of its underlying Coxeter group.
  
   The third principal protagonist, the Bernstein maps, is introduced in the following \cref{sub:Bernstein maps}. The existence theorem, \cref{thm:ex_theta_braid}, proven there should be seen as the equivalent of Görtz' theorem in our context. \Cref{sub:A $2$-coboundary appearing in Coxeter geometry} discusses a certain $2$-cocycle that is canonically associated to every Coxeter group and plays a prominent in the theory of Hecke algebras. We show that it can be written as $2$-coboundary in \textit{two different} ways, which is used to define integral and normalized Bernstein maps in \cref{sub:Integral and normalized Bernstein maps}. The intermediate \cref{sub:A characterization of pro-$p$ Coxeter groups in terms of X} contains an optional result that is not used elsewhere: we show that the $2$-cocycle $\bbf{X}$ can be used to classify pro-$p$ Coxeter groups, recovering a result of Tits \cite[3.4]{TitsTores}.

   Finally, in \cref{sub:Bernstein relations}, we prove one of the two main results of this article, the proof of the generalized Bernstein relations \eqref{eq:generalized_bernstein_relations}.

   Whereas the first part dealt with general generic pro-$p$ Hecke algebras, the second part of this article specializes to those generic pro-$p$ Hecke algebras for which a meaningful analogue of the Bernstein-Zelevinsky theory can be developed. These are the \textit{affine pro-$p$ Hecke algebras}, the generic pro-$p$ Hecke algebras whose underlying extended Coxeter group is equipped with the structure of an \textit{affine extened Coxeter group}, a notion that is introduced in \cref{sub:Affine extended Coxeter groups and affine pro-$p$ Hecke algebras} and which generalizes the class of extended affine Weyl groups to allow \textit{all} groups that appear in the Bruhat decomposition of Iwahori groups of (possibly non-split) connected reductive groups over local fields. This makes it necessary to prove some well-known facts from the theory of root data in our more general context.

   In \cref{sub:Main examples of affine pro-$p$ Hecke algebras} we show that our theory is non-empty by introducing three examples of affine pro-$p$ Hecke algebras: the \textit{affine Hecke algebras} considered in the classical Bernstein-Zelevinsky theory \cite{Lusztig}, the \text{pro-$p$ Iwahori Hecke algebras} considered in the $p$-adic and mod-$p$ Langlands programme, and the \textit{affine Yokonuma-Hecke algebras} from the theory of knot invariants.
   These examples have already appeared in the literature before (see \cite{ChlouSe}, \cite{VigProI}), and for the heavy-duty computations needed for the verification of the axioms in the case of pro-$p$-Iwahori Hecke algebra we refer to \cite{VigProI}; however, this section provides some details not found in either source, including an effective version of the existence of the lifts $(n_s)_{s \in S}$, which the reader may find helpful.
   
   \Cref{sub:Some finiteness properties of affine extended Coxeter groups} is devoted to the proof of some finiteness properties of affine extended Coxeter groups, which are the key to prove corresponding finiteness results for affine pro-$p$ Hecke algebras. These results were basically already proven in \cite[4.2.5]{Schmidt}, but the proofs were a bit ad hoc. Here, we give a more unified treatment by relating these finitness properties to the (known) fact that the weak Bruhat order is a \textit{well partial order}.

   In \cref{sub:Spherical orientations} we introduce \textit{spherical orientations} of affine extended Coxeter groups $W$ and prove that they are limits of nets of chamber orientations $\mathfrak{o}_w$, which makes them concrete examples of boundary orientations and gives a precise sense to the notion in Görtz' theorem, of the orientation `attached to an alcove infinitely deep in the anti-dominant chamber'. The most important property of the spherical orientations is that the subgroup $X \leq W$ of `translations' acts trivially on them, as the cocycle rule implies that the Bernstein map $\widetilde{\theta}_{\mathfrak{o}}$ induced an embedding of the group algebra of the stabilizer of $\mathfrak{o}$ (in $W^{(1)}$) embeds into the Hecke algebra. Thus we introduce in the following \cref{sub:Some (almost) commutative subalgebras} subalgebras $\mathcal{A}^{(1)}_{\mathfrak{o}} \subseteq \mathcal{H}^{(1)}$ for every spherical orientation, which are not far from being commutative (and are commutative for affine Yokonuma-Hecke algebras or pro-$p$-Iwahori Hecke algebras of \textit{split} groups). The main result of this section is the computation of the centralizer of these subalgebras in the Hecke algebra; in particular, we prove that the centralizer of $\mathcal{A}^{(1)}_\mathfrak{o}$ is a subalgebra of $\mathcal{A}^{(1)}_{\mathfrak{o}}$, which is an important step towards the computation of the center.
   
   In \cref{sub:The center of affine pro-$p$ Hecke algebras}, we use the Bernstein relations to show that the invariants of $\mathcal{A}^{(1)}_{\mathfrak{o}}$ under the natural action of $W^{(1)}$ are contained in the center of $\mathcal{H}^{(1)}$. Afterwards, we verify using explicit computations that this exhausts the center. The final \cref{sub:The structure of affine pro-$p$ Hecke algebras} uses the results of the previous sections to determine the structure of $\mathcal{H}^{(1)}$ in terms of its center, under very mild assumptions on the coefficient ring and the group $W^{(1)}$, verified in all the examples we consider.

   We conclude the introduction to this article with a discussion of the work of Ivan Cherednik on Hecke algebras and its connection to Bernstein maps. This connection arises through the cocycle rule.
   
   Motivated by problems in quantum physics, Cherednik has constructed various $1$-cocycles of Coxeter groups with values in (localizations of) affine Hecke algebras and their degenerate (i.e. graded) versions, viewing these cocycles as generalized `R-matrices'. By definition, \textit{R-matrices} are solutions of the \textit{Yang-Baxter equation}. This remarkable equation---connecting low-dimensional topology, representation theory, category theory and physics---was discovered independently by C. N. Yang \cite{Yang} and R. J. Baxter \cite{Baxter}, who worked on finding \textit{exact solutions} of certain physical models from quantum and statistical mechanics respectively. Its simplest and most recognizable form is
      \begin{equation}\label{eq:yang-baxter} R_{12} R_{13} R_{23} = R_{23} R_{13} R_{12} \end{equation}
      with the $R_{ij}$ being elements of some monoid (usually an algebra, although the case where the $R_{ij}$ are endomorphisms of a set is of considerable interest too; see `set-theoretical solutions of the Yang-Baxter-equation'). Assuming the existence of an action of the symmetric group $S_3$ on the monoid in which the $R_{ij}$ take values, and assuming `$W$-invariance' of the R-matrix, i.e.
      \[ {}^\sigma R_{ij} = R_{\sigma(i)\sigma(j)} \]
      for all $\sigma$ and $i,j$ for which both sides are defined, the Yang-Baxter equation \eqref{eq:yang-baxter} can be rewritten equivalently as
      \begin{equation}\label{eq:cocycle-consistency} R_s {}^sR_t {}^{st}R_s = R_t {}^tR_s {}^{ts}R_t \end{equation}
      where $s = (12)$, $t = (23)$ and indices $ij$ are identified with transpositions $(ij)$. This equation in turn is nothing but the self-consistency condition necessary for the existence of a $1$-cocycle $\sigma \mapsto R_\sigma$ that results from the braid relation $sts = tst$ in the symmetric group. This relation is almost sufficient for the existence of such a cocycle; necessary and sufficient is the above relation together with the `unitarity condition'
      \[ R_s {}^s R_s = R_t {}^t R_t = 1 \]
      resulting from $s^2 = t^2 = 1$ (cf. \cite[Prop. 4]{CherednikHalfLine}). Thus, \textit{unitary} \textit{invariant} R-matrices are identified with $1$-cocycles of the group $S_3$. Cherednik used this observation to define a general notion of `R-matrices' attached to root systems as cocycles of Weyl groups \cite[Sect. 2]{CherednikQKZ}, and has constructed examples given by the Demazure-Lusztig operators \cite[Prop. 3.5]{CherednikQKZ} and the standard intertwining operators \cite[Prop. 3.8]{CherednikQKZ} (cf. \cite[Prop. 1.2]{CherednikUnification}) familiar from the representation theory of reductive groups. The latter are directly connected to Bernstein maps, realizing them as a limit. In order to make this precise, let us recall the definition of the standard intertwiners. In the following, all algebras will be over $\C$.
      
      Given a root datum $(X,\Phi, X^\vee, \Phi^\vee)$ with basis $\Delta \subseteq \Phi$ and extended affine Weyl group $W = X\rtimes W_0$, the standard intertwiners are elements $F_w$ ($w \in W_0$) of the localization
      \[ H_{\op{gen}} := H_q(W,S) \otimes_Z \op{Frac}(Z) \]
      of the affine Hecke algebra at its center $Z$, determined by (cf. \cite[Prop. 1.2]{CherednikUnification}; also \cite[Lem. 1.13.1]{HainesKottwitzPrasad})
      \begin{align}\label{eq:intertwiners_are_r-matrices} F_{ww'} & = F_w F_{w'} \text{ if } \ell(ww') = \ell(w)+\ell(w') \\
      F_{s_\alpha} & = T_{s_\alpha} + (q_{s_\alpha}-1)(Y_\alpha - 1)^{-1},\quad \alpha \in \Delta
   \end{align}
   where we write
   \[ Y_\alpha := \widetilde{\theta}(\alpha) \]
   for easier comparison with \cite{CherednikUnification}. The $F_w$ now constitute an R-matrix in the following sense. The basis property of the Bernstein-Lusztig basis $\{\widetilde{\theta}(x)T_w\}_{x \in X,\ w \in W_0}$ implies that we have \textit{linear} isomorphisms
      \[ \C[X]\otimes H_0 \stackrel{\sim}{\longrightarrow} H_q(W,S)\quad \text{and}\quad \C(X)\otimes H_0 \stackrel{\sim}{\longrightarrow} H_{\op{gen}} \]
      Here $\C(X) = \op{Frac}(\C[X])$ and $H_0$ denotes the finite Hecke subalgebra spanned by $T_w$, $w \in W_0$. Note that the group $W_0$ acts on $\C(X)\otimes H_0$ via its canonical action on $X$. Now, if we consider $\C(X)\otimes H_0$ with its canonical (tensor) algebra structure, then the map
      \[ \phi: W_0 \longrightarrow \C(X)\otimes H_0 \]
      defined by $w \mapsto F_w$ satisfies the partial cocycle relation
      \[ \phi(ww') = \phi(w) w(\phi(w')) \quad\text{ if }\quad \ell(ww') = \ell(w)+\ell(w') \]
      i.e. defines a (non-unitary) R-matrix in the sense of Cherednik (cf. \cite[Thm. 2.3 a)]{CherednikQKZ}). This follows by easy calculations from relation \eqref{eq:intertwiners_are_r-matrices} and the intertwining property
      \[ F_w a = w(a) F_w, \quad a \in \mathcal{A} \]
      The intertwiners $F_w$ can be normalized so that one gets a proper cocycle (unitary R-matrix) instead: let (cf. \cite[2.2]{HainesKottwitzPrasad}, \cite[Prop. 5.2]{Lusztig}; also \cite[p. 146]{Opdam})
      \[ K_{s_\alpha} := q_{s_\alpha}^{-1} \frac{1 - Y_\alpha}{1 - q_{s_\alpha}^{-1} Y_\alpha} F_{s_\alpha} \]
      for a simple root $\alpha$ and extend to elements $K_w$ for all $w \in W_0$ using \eqref{eq:intertwiners_are_r-matrices} as before. It can be shown that these normalized intertwiners satisfy $K_w K_{w'} = K_{ww'}$ for \textit{all} $w,w'$ and therefore define a cocycle $\psi: W_0 \longrightarrow \C(X)\otimes H_0$ in the usual sense.

      This cocycle $\psi$ partially recovers the Bernstein map $\theta: W \longrightarrow \op{Hom}_{\op{Set}}(\mathcal{O},H_q(W,S)^\times)$ as follows (cf. \cite{Opdam}). Elements of $\C(X)\otimes H_0$ can be viewed as meromorphic function on the complex torus $\bbf{T} = X^\vee \otimes \C^\times$ with values in $H_0$. For every Weyl chamber $D$ in $V^\vee = X^\vee \otimes \R$, given as an intersection
      \[ D = \bigcap_i \{ x \in V : \alpha_i(x) > 0 \} \]
      of half-spaces, one can add a point $\xi_D$ at infinity to $\bbf{T}$, such that
      \[ \lim_{n \rightarrow \infty} t_n = \xi_D \quad \Leftrightarrow \quad \lim_{n \rightarrow \infty} \alpha_i(t) = 0 \ \forall i \]
      for every sequence $(t_n)_n$ in $\bbf{T}$. Then $\theta$ is partially recovered as the pointwise limit
      \begin{equation}\label{eq:theta_as_limit} \theta_{\mathfrak{o}_D}(w) = \lim_{t \rightarrow \xi_D} \psi(w)(t),\quad \forall w \in W_0 \end{equation}
         with respect to the natural topology on $H_0 = \bigoplus_w \C T_w$. Note that $\theta_{\mathfrak{o}_D}(w)$ lies in $H_0 \subseteq H_q(W,S)$ for all $w \in W_0$ a priori; indeed, the restriction of $\mathfrak{o}_D$ to $W_0 \subseteq W$ is nothing but the chamber orientation (\cref{def:chamber_orientations}) towards the element $w_D \in W_0$ corresponding to $D$ via $w_D(C) = D$, where $C$ denotes the fundamental Weyl chamber, and $\theta_{\mathfrak{o}_D}(w)$ identifies with the image under the Bernstein map $\theta_{\mathfrak{o}_{w_D}}: W_0 \rightarrow H_0^\times$ of the finite Hecke algebra. Thus, \cref{eq:theta_as_limit} can also be seen as recovering the cocycle $\theta:W_0 \rightarrow \op{Hom}_{\op{Set}}(\mathcal{O}(W_0),H_0^\times)$ of the finite Hecke algebra.
         
         Because of the cocycle rule, \cref{eq:theta_as_limit} needs only to be checked in the case $w = s_\alpha$, where it follows from easy computations. Indeed
      \[ \lim_{t \rightarrow \xi_D} Y_\alpha = 0 \quad\text{or}\quad \lim_{t \rightarrow \xi_D} Y_\alpha^{-1} = 0 \]
      depending on whether $D$ lies in the positive $\{ x : \alpha(x) > 0 \}$ or negative half-space $\{ x : \alpha(x) < 0 \}$ defined by $\alpha$, respectively. Moreover, from the expression defining $F_{s_\alpha}$ it is immediate that
      \[ F_{s_\alpha}(Y_\alpha = 0) = T_{s_\alpha} \]
      and
      \[ F_{s_\alpha}(Y_\alpha^{-1} = 0) = T_{s_\alpha} - (q_{s_\alpha}-1) = q_{s_\alpha} T_{s_\alpha}^{-1} \]
      where the second equality follows from the quadratic relation $T_{s_\alpha}^2 = q_{s_\alpha} T_{s_\alpha} + (q_{s_\alpha}-1)$. Hence
      \[ K_{s_\alpha}(Y_\alpha = 0) = T_{s_\alpha} \quad\text{and}\quad K_{s_\alpha}(Y_\alpha^{-1} = 0) = T_{s_\alpha}^{-1} \]
      which proves \eqref{eq:theta_as_limit} for $w = s_\alpha$, taking into account the definition of $\theta$ and $\mathfrak{o}_D$ (see \cref{def:unnormalized_bernstein,def:spherical_orientation} resp.)\footnote{Note: the expression for $F_{s_\alpha}$ can be interpreted as defining a `Yang-Baxterization' of the element $T_{s_\alpha} \in H_0$, i.e. a parametric deformation $F_{s_\alpha} = F_{s_\alpha}(Y_\alpha)$ that satisfies the Yang-Baxter equation with spectral parameter $Y_\alpha$. This deformation interpolates between $T_{s_\alpha} = F_{s_\alpha}(Y_\alpha = 0)$ and $q_{s_\alpha} T_{s_\alpha}^{-1} = F_{s_\alpha}(Y_\alpha = \infty)$.}.
      
      Thus one notices a curious fact: to construct $R$-matrices (cocycles) in the \textit{finite} Hecke algebra $H_0$, one should study intertwiners of the \textit{affine} Hecke algebra $H_q(W,S)$, which contains the former as a subalgebra. Does this pattern continue? Cherednik has shown that it does (at least for affine Hecke algebras). His \textit{double affine Hecke algebras} $\ddot{H}_q(W,S)$ contain the affine Hecke algebras $H_q(W,S)$ as subalgebras, and one can define elements $\widehat{F}_w \in \ddot{H}_q(W,S)$ for all elements $w \in W$ of the affine Weyl group (cf. \cite[Theorem 3.3]{CherednikDAHA}), satisfying
      \[ \widehat{F}_w \widehat{F}_{w'} = \widehat{F}_{ww'} \quad \text{ if } \ell(w)+\ell(w') = \ell(ww') \]
      and defining $R$-matrices (with spectral parameters) with values in $H_q(W,S)$ that recover the whole Bernstein map $\theta: W \rightarrow \op{Hom}_{\op{Set}}(\mathcal{O},H_q(W,S)^\times)$ as a limit.

\section{Generic pro-\texorpdfstring{$p$}{p} Hecke algebras and Bernstein maps} 
\label{sec:Generic pro-$p$ Hecke algebras and Bernstein maps}

\subsection{Basic definitions and some geometric terminology} 
\label{sub:Basic definitions and some geometric terminology}
We recall some standard facts and terminology from the theory of Coxeter groups (cf. \cite[Ch. IV]{Bourbaki} or \cite[II]{Brown}).

\begin{definition}\label{def:coxeter-group} A \textbf{Coxeter group} $W = (W,S)$ consists of a group $W$ and a set $S \subseteq W$ of generators of order $2$ satisfying the \textit{action condition}. That is, there exists an action
   \[ \rho: W \longrightarrow \op{Aut}_{\op{Set}}(\mathfrak{H}\times \{\pm \}) \]
   on the set $\mathfrak{H}\times \{\pm 1\}$, where
   \[ \mathfrak{H} := \{ wsw^{-1} : w \in W,\ s \in S \} \subseteq W \]
   such that a generator $s \in S$ acts as
   \[ \rho(s) (H,\varepsilon) =  \begin{cases}
      (sHs^{-1},-\varepsilon) & : H = s \\
      (sHs^{-1},\varepsilon) & : H \neq s
   \end{cases} \]
\end{definition}

\begin{rmk}\label{rmk:equivalent-definitions-of-coxeter-groups}
   There are several other equivalent definitions of the notion of a Coxeter group (see \cite[II.4]{Brown}). In particular, given a group $W$ and a set $S$ of generators of order $2$, the action condition is equivalent to both the \textit{exchange condition} $\text{(\textbf{E})}$ and the \textit{deletion condition} $\text{(\textbf{D})}$. The former states that given a reduced expression $w = s_1 \dots s_r$ and an element $s \in S$, either $\ell(sw) = \ell(w)+1$ or
   \begin{equation*}\tag{\textbf{E}} w = ss_1 \dots \widehat{s_i} \dots s_r \end{equation*}
   for some $1 \leq i \leq r$ (where $\widehat{s_i}$ denotes omission of $s_i$), the latter that if the expression $w = s_1\dots s_r$ is \textit{not} reduced, then
   \begin{equation*}\tag{\textbf{D}} w = s_1 \dots \widehat{s_i} \dots \widehat{s_j} \dots s_r \end{equation*}
   for some $1 \leq i < j \leq r$.
\end{rmk}

\begin{terminology}\label{term:basic}
   If an action as in \cref{def:coxeter-group} exists, it is uniquely determined and called the \textit{canonical action}. The set $\mathfrak{H}$ is called the set of \textit{walls} or \textit{hyperplanes}. When we want to view a hyperplane $H \in \mathfrak{H}$ as the reflection in $W$ it corresponds to, we sometimes write $s_H$ instead of $H$. Elements of $W$ are also called \textit{chambers}. A distinguished chamber is given by the neutral element $1 \in W$ and is called the \textit{fundamental chamber}. Two chambers $w,w'$ are called \textit{adjacent} if $w^{-1}w' \in S$. A \textit{gallery} from $w$ to $w'$ is a finite sequence $\Gamma = (w=w_0,\ldots,w_r = w')$ such that $w_i, w_{i+1}$ are adjacent. Galleries from the fundamental chamber to a chamber $w \in W$ correspond to expressions
   \[ w = s_1 \ldots s_r \]
   of $w$ as a product of generators $s_i \in S$, the associated gallery being
   \[ \Gamma = (1,s_1,s_1 s_2,\ldots , s_1\ldots s_r) \]
   A wall $H$ is said to \textit{separate} $w_1,w_2 \in W$ if
   \[ \rho(w_2^{-1}w_1)(w_1^{-1}Hw_1, 1) = (w_2^{-1}Hw_2, -1) \]
   Otherwise $w_1$, $w_2$ are said to \textit{lie on the same side} with respect to $H$. The number of walls separating $1$ and $w$ is finite and equal to
   \[ \ell(w) := \op{min} \{ r \in \N : \exists s_1,\ldots, s_r \in S\ \ w = s_1\ldots s_r \} \]
   which is called the \textit{length} of $w$. An arbitrary expression
   \[ w = s_1 \ldots s_r \]
   is called \textit{reduced} if $r = \ell(w)$. Given such a reduced expression, the set
   \[ \{ s_1, s_1 s_2 s_1^{-1}, \ldots{}, (s_1\ldots{}s_{r-1}) s_r (s_1\ldots{}s_{r-1})^{-1} \} \]
   is the set of hyperplanes separating $1$ and $w$. More generally, we can define for any two chambers $w,w'$ the \textit{distance} $d(w,w')$ between $w$ and $w'$ as the length of the shortest gallery from $w$ to $w'$. A gallery $\Gamma$ is called a \textit{geodesic} if its length equals the distance between its start- and endpoint. One can show that a gallery is a geodesic if and only if it does not cross a hyperplane twice. In particular the distance $d(w,w')$ equals the number of walls separating $w$ and $w'$. Moreover, the distance is $W$-invariant and so in particular $d(w,w') = d(1,w^{-1}w') = \ell(w^{-1}w')$. A wall $H = w_0 s w_0^{-1}$ divides $W$ into two equivalence classes under the relation of lying on the same side with respect to $H$, namely the \textit{positive half-space}
   \[ U^+_H = \{ w \in W : \ell(sw_0^{-1}w) = \ell(w_0^{-1}w) + (\ell(sw_0^{-1}) - \ell(w_0^{-1}w)) \} \]
   and the \textit{negative half-space}
   \[ U^-_H = \{ w \in W : \ell(sw_0^{-1}w) = \ell(w_0^{-1}w) - (\ell(sw_0^{-1}) - \ell(w_0^{-1})) \} \]
   By definition the positive half-space is the one containing the fundamental chamber. The map $(H,\varepsilon) \mapsto U^\varepsilon_H$ gives a bijection between $\mathfrak{H}\times \{\pm 1\}$ and the set of all \textit{half-spaces}. This bijection is $W$-equivariant with respect to the natural actions and allows to identify these two $W$-sets.
   
   The \textit{Bruhat order} $<$ on $W$ is the strict partial order in which $w < w'$ if and only if for some (every) reduced expression
   \[ w = s_1 \ldots s_r \]
   there exist $1 \leq i_1 < \ldots < i_m \leq r$, $m < r$ such that
   \[ w' = s_{i_1}\ldots s_{i_m} \]
   The order of the product $st \in W$ of two generators $s,t \in S$ will be denoted by $m(s,t)$ and is an element of $\{1,2,\ldots,\infty\}$.
\end{terminology}

\begin{rmk}
   The inclusion $S \subseteq \mathfrak{H}$ induces a bijection
   \[ S/_\sim \stackrel{\sim}{\longrightarrow} W\backslash \mathfrak{H} \]
   where $\sim$ is the equivalence relation given by
   \[ s \sim t \quad \Leftrightarrow \quad \exists w \in W\ \ wsw^{-1} = t \]
\end{rmk}

In the context of root systems and Iwahori-Hecke algebras one is naturally led to consider groups slightly more general than Coxeter groups. We will therefore introduce a nonstandard definition which axiomatizes extended Weyl groups.
\begin{definition}\label{def:extended_coxeter_group} An \textbf{extended Coxeter group} $W$ consists of a group $W$, subgroups $W_{\op{aff}}, \Omega \leq W$, a subset $S \subseteq W_{\op{aff}}$ and a group homomorphism $W \rightarrow \Omega$ such that
   \begin{enumerate}
      \item The sequence
         \[ \begin{xy} \xymatrix{ 1 \ar[r] & W_{\op{aff}} \ar[r] & W \ar[r] & \Omega \ar[r] & 1 } \end{xy} \]
            is exact.
      \item $(W_{\op{aff}},S)$ is a Coxeter group.
      \item The action of $\Omega$ on $W_{\op{aff}}$ by conjugation restricts to an action on $S$.
   \end{enumerate}
\end{definition}
In other words an extended Coxeter group $W$ is a semidirect product $W = W_{\op{aff}} \rtimes \Omega$ of a Coxeter group $(W_{\op{aff}},S)$ and a group $\Omega$ acting on $W_{\op{aff}}$ by automorphisms of Coxeter groups. 

\begin{notation} The action of $u \in \Omega$ on $w \in W_{\op{aff}}$ will be denoted by $u(w)$, $uwu^{-1}$ or even $u\bullet{}w$. \end{notation}
\begin{rmk}
   The conjugation action of $\Omega$ on $W_{\op{aff}}$ induces a right action on $\op{Hom}_{\op{Set}}(W_{\op{aff}},\N)$ by acting on arguments. The invariance of $S \subseteq W_{\op{aff}}$ is then equivalent to the length function $\ell: W_{\op{aff}} \rightarrow \N$ being fixed under the action of $\Omega$. We may therefore uniquely extend $\ell$ to a function $W \rightarrow \N$ denoted by the same letter and satisfying 
   \[ \ell(wu) = \ell(uw) = \ell(w),\quad w \in W_{\op{aff}}, u \in \Omega \]
\end{rmk}

\begin{rmk}\label{rmk:ex_weyl_walls}
   The group $W$ acts on the set $\mathfrak{H}$ of walls of $(W_{\op{aff}},S)$ by conjugation and we again have a bijection
   \[ S/_\sim \stackrel{\sim}{\longrightarrow} W\backslash \mathfrak{H} \]
   where $\sim$ now refers to the equivalence relation given by
   \[ s \sim t \quad \Leftrightarrow \quad \exists w \in W\ \ wsw^{-1} = t \]
   Two elements $s,t \in S$ can be conjugate in $W$ without being conjugate in $W_{\op{aff}}$. In the context of extended Coxeter groups, $\sim$ will by convention always refer the relation induced by conjugation in $W$.
\end{rmk}

\begin{rmk}\label{rmk:action_of_w_on_chambers}
   By assumption we have an action $\rho: \Omega \rightarrow \op{Aut}_{\op{Grp}}(W_{\op{aff}})$ of $\Omega$ on $W_{\op{aff}}$ by group automorphisms. On the other hand $W_{\op{aff}}$ acts on itself via left translation $\lambda: W_{\op{aff}} \rightarrow \op{Aut}_{\op{Set}}(W_{\op{aff}})$. One has
   \[ \rho_u(\lambda_w(w')) = \rho_u(ww') = \rho_u(w)\rho_u(w') = \lambda_{\rho_u(w)}(\rho_u(w')) = \lambda_{u\bullet{}w}(\rho_u(w')) \]
   for every $w' \in W_{\op{aff}}$. By the universal property of the semidirect product $\rho$ and $\lambda$ therefore combine in a unique way to an action of $W$ on the \textit{set} $W_{\op{aff}}$, which we would like to view as the set of chambers. It follows immediately that the stabilizer of the fundamental chamber $1 \in W_{\op{aff}}$ is $\Omega$. We will occasionally view elements $w \in W$ as chambers via the orbit map $W \rightarrow W_{\op{aff}}, w \mapsto w\bullet{}1$, that is $w = w'u$, $w' \in W_{\op{aff}}$, $u \in \Omega$ will be replaced by $w'$. Accordingly we will talk about walls separating two elements $w,w' \in W$ or the distance between $w$ and $w'$. This is consistent with the definitions given so far in the sense that the distance between $w,w'$ viewed as chambers is equal to $\ell(w^{-1}w')$.
\end{rmk}

\begin{rmk}\label{rmk:ex_bruhat_order}
   One can extend the Bruhat order on $W_{\op{aff}}$ uniquely to a strict partial order $<$ on $W$ which satisfies
   \[ wu < w'u' \quad \Leftrightarrow \quad w < w' \]
   for all $w,w' \in W$ and $u,u' \in \Omega$. This relation is invariant under conjugation by $\Omega$, but beware that in general
   \[ uw < u'w' \quad \not\Leftrightarrow \quad w < w' \]
\end{rmk}

\begin{rmk}
   Some caution has to be applied when dealing with length function on extended Coxeter groups. It is not true that for any $w,w' \in W$ and $u \in \Omega$
   \[ \ell(wuw') = \ell(ww') \]
   If for example $u$ permutes two distinguished generators $s \neq t$ then
   \[ \ell(sut) = \ell(s(utu^{-1})u) = \ell(ssu) = 1 \]
   whereas $\ell(st) = 2$. However, it remains true that for $s \in S$ and $w \in W$ either $\ell(sw) = \ell(w) + 1$ or $\ell(sw) = \ell(w) - 1$ and similarly either $\ell(ws) = \ell(w) + 1$ or $\ell(ws) = \ell(w) - 1$.
\end{rmk}

\begin{rmk}
   The example motivating \cref{def:extended_coxeter_group} are the \textit{extended affine Weyl groups} associated to \textit{root data}. This will be discussed later (see \cref{ex:affine_coxeter_group}), when we will introduce the stronger notion of \textit{affine extended Coxeter groups}.
\end{rmk}

We will define generic pro-$p$ Hecke algebras via a presentation à la Iwahori-Matsumoto. In this presentation the ``Weyl group'' will not be an extended Coxeter group, but a group of a more general type which naturally occurs when considering algebras of the form $\op{End}_G(\op{ind}^G_{I^{(1)}} \bbf{1})$. The following axioms are modelled on this particular case (cf. \cite[1.2]{VigProP}).
\begin{definition}\label{def:pro-$p$ Coxeter group}
   A \textbf{pro-$p$ Coxeter group} $W^{(1)}$ consists of an abelian group $T$, an extended Coxeter group $W$ and a group extension
   \[ \begin{xy} \xymatrix{ 1 \ar[r] & T \ar[r] & W^{(1)} \ar[r]^{\pi} & W \ar[r] & 1 } \end{xy} \]
      together with a family $(n_s)_{s \in S}$ of lifts $n_s \in \pi^{-1}(s)$ of the generators $s \in S$ subject to the following ``braid'' condition. If $s,t \in S$ with $m(s,t) < \infty$ then
      \begin{equation}\label{eq:braidcond} n_s n_t n_s \ldots = n_t n_s n_t \ldots \end{equation}
      where the number of factors on both sides is $m(s,t)$.
\end{definition}

\begin{convention}
   To ease the notation we will in the following always assume that the map $T \hookrightarrow W^{(1)}$ is an inclusion.
\end{convention}

\begin{notation}
   In the above situation we have a canonical action of $W^{(1)}$ on $T$ by conjugation. This action $W^{(1)}\times T \rightarrow T$ is denoted by $(w,t) \mapsto w(t)$. Since $T$ is commutative this action factors over the projection $\pi: W^{(1)} \rightarrow W$. The induced action of $W$ on $T$ will also be denoted by $(w,t) \mapsto w(t)$.   
\end{notation}

\begin{notation}
   Given a pro-$p$ Coxeter group $W^{(1)}$ as above with associated extended Coxeter group $W$ and length function $\ell: W \rightarrow \N$, we will by abuse of notation denote the composite function $\ell \circ{} \pi: W^{(1)} \rightarrow \N$ again by $\ell$ and refer to it also as ``the length function''.
\end{notation}

\begin{notation}
   We may pull back the relation $<$ on $W$ defined in \cref{rmk:ex_bruhat_order} to a strict partial order on $W^{(1)}$ again denoted by $<$ via
   \[ w < w' \quad :\Leftrightarrow\quad \pi(w) < \pi(w') \] 
\end{notation}

\begin{notation}
   Given any subset $X \subseteq W$ we will denote by $X^{(1)}$ the preimage of $X$ under $\pi$. In particular we have
   \[ \Omega^{(1)} = \{ w \in W^{(1)} : \ell(w) = 0 \} \]
\end{notation}

\subsection{1-Cocycles of pro-\texorpdfstring{$p$}{p} Coxeter groups} 
\label{sub:1-Cocycles of pro-$p$ Coxeter groups}
We recall that a $1$-cocycle of a group $G$ with values in a (possibly non-commutative) $G$-module $M$ (i.e. a \textit{group} endowed with a $G$-action by group automorphisms) is a map $\phi: G \rightarrow M$ satisfying the cocycle rule
\[ \forall g,g' \in G\ \ \phi(gg') = \phi(g) g(\phi(g')) \]
Generalizing a result of Cherednik, we will now obtain an explicit description of the set $Z^1(G,M)$ of $1$-cocycles when $G = W^{(1)}$ is a pro-$p$ Coxeter group\footnote{In a previous version of this lemma, we had assumed that $T$ acts trivial on $M$. We thank M.F. Vignéras for suggesting to remove this hypothesis.}.

\begin{lemma}\label{lem:cocycle}
   Let $M$ be a $W^{(1)}$-module. Restriction defines an injective map
   \[ Z^1(W^{(1)},M) \longrightarrow \op{Hom}_{\op{Set}}(S,M)\times Z^1(\Omega^{(1)},M) \]
   \[ \phi \mapsto ((s \mapsto \phi(n_s)),(u \mapsto \phi(u))) \]
   whose image consists of all pairs $(\sigma,\rho)$ satisfying the following properties.
   \begin{enumerate}
      \item $\sigma(s) n_s(\sigma(s)) = \rho(n_s^2)$ for all $s \in S$
      \item For all $u \in \Omega^{(1)}$, $s \in S$
         \[ \rho(u)\cdot{} u(\sigma(s)) = \sigma(u(s))\cdot{} n_{u(s)}(\rho(ut_{s,u})) \]
      where $t_{s,u} \in T$ denotes the element defined by the equation $u n_s = n_{u(s)} u t_{s,u}$
      \item For all $s,t \in S$ with $m(s,t) < \infty$, the following two products with $m(s,t)$ factors are equal
         \[ \sigma(s)\cdot{}n_s(\sigma(t))\cdot{}(n_sn_t)(\sigma(s))\cdot{}(n_sn_tn_s)(\sigma(t)) \ldots = \sigma(t)\cdot{}n_t(\sigma(s))\cdot{}(n_tn_s)(\sigma(t))\cdot{}(n_tn_sn_t)(\sigma(s)) \ldots \]

   \end{enumerate}
\end{lemma}
\begin{proof}
   The map is obviously well-defined and injective. In fact, let $\phi \in Z^1(W^{(1)},M)$ be mapped to $(\sigma,\rho)$. For any $w \in W^{(1)}$, we can find an expression
   \begin{align}\label{eq:cocycle_lemma_1} w & = n_{s_1}\ldots n_{s_r} u \end{align}
   with $s_i \in S$ and $u \in \Omega^{(1)}$. The cocycle rule for $\phi$ now implies
   \begin{align}\label{eq:cocycle_lemma_2} \phi(w) & = \sigma(s_1)\cdot{} n_{s_1}(\sigma(s_2)) \cdot{} (n_{s_1}n_{s_2})(\sigma(s_3)) \cdot{} \ldots \cdot{} (n_{s_1}\ldots n_{s_{r-1}})(\sigma(s_r)) \cdot{} (n_{s_1}\ldots n_{s_r})(\rho(u)) \end{align}
      Moreover, the cocycle rule for $\phi$ immediately implies the conditions $(i)$-$(iii)$ for the pair $(\sigma,\rho)$. We will now show that starting with any pair $(\sigma,\rho)$ satisfying $(i)$-$(iii)$, equation \eqref{eq:cocycle_lemma_2} gives rise to a well-defined cocycle $\phi: W^{(1)} \rightarrow M$. In fact, to show that \eqref{eq:cocycle_lemma_2} gives a well-defined map $\phi: W^{(1)} \rightarrow M$ of sets independent of the choice of the expression \eqref{eq:cocycle_lemma_1} it suffices to assume $(i)$, $(iii)$ and the following condition $(iv)$. It is implied by $(ii)$ by taking $u = t \in T$, observing that in this case $ut_{u,s} = s^{-1}(t)$
      \begin{equation} \tag{\text{iv}} \rho(t)\cdot{}t(\sigma(s)) = \sigma(s) n_s(\rho(s^{-1}(t))) \quad \forall s \in S,\ t \in T \end{equation}
      Now let
      \[ w = n_{\overline{s}_1}\ldots n_{\overline{s}_m} \overline{u} \]
      be another expression for $w$. We verify that
      \begin{align}\label{eq:cocycle_lemma_3} \sigma(s_1)\cdot{} n_{s_1}(\sigma(s_2)) \cdot{} \ldots \cdot{} (n_{s_1}\ldots n_{s_r})(\rho(u)) & = \sigma(\overline{s}_1)\cdot{}n_{\overline{s}_1}(\sigma(\overline{s}_2))\cdot{}\ldots\cdot{}(n_{\overline{s}_1}\ldots{}n_{\overline{s}_m})(\rho(\overline{u})) \end{align}
      It suffices to show this when $u,\overline{u} \in T$. Indeed, assume the statement is true in this case. Then, since $W = W_{\op{aff}} \rtimes \Omega$, reducing the equation
      \begin{align}\label{eq:cocycle_lemma_4} n_{s_1}\ldots n_{s_r} u & = n_{\overline{s}_1} \ldots n_{\overline{s}_m} \overline{u} \end{align}
         via $\pi:W^{(1)} \rightarrow W$ shows that $s_1\ldots s_r = \overline{s}_1\ldots \overline{s}_m$ and $\pi(u) = \pi(\overline{u})$, and therefore $u\overline{u}^{-1} \in T$. Multiplying \eqref{eq:cocycle_lemma_4} by $\overline{u}^{-1}$ and using \eqref{eq:cocycle_lemma_3} for the case $u,\overline{u} \in T$ gives
         \begin{align}\label{eq:cocycle_lemma_5} \sigma(s_1)\cdot{}\ldots\cdot{}(n_{s_1}\ldots{}n_{s_{r-1}})(\sigma(s_r))\cdot{}(n_{s_1}\ldots n_{s_r})(\rho(u\overline{u}^{-1})) & = \sigma(\overline{s}_1)\cdot{}\ldots{}\cdot{}(n_{\overline{s}_1}\ldots{}n_{\overline{s}_{m-1}})(\sigma(\overline{s}_m)) \end{align}
      The cocycle property for $\rho$ implies
      \[ \rho(u\overline{u}^{-1}) = \rho(u)\cdot{}u(\rho(\overline{u}^{-1})) = \rho(u)\cdot{}u(\overline{u}^{-1}(\rho(\overline{u})^{-1})) \]
      Therefore
      \[ (n_{s_1}\ldots n_{s_r})(\rho(u\overline{u}^{-1})) = (n_{s_1}\ldots n_{s_r})(\rho(u)) \cdot{} (n_{\overline{s}_1}\ldots{}n_{\overline{s}_m})(\rho(\overline{u})^{-1}) \]
      Multiplying \eqref{eq:cocycle_lemma_5} from the right by $(n_{\overline{s}_1}\ldots{}n_{\overline{s}_m})(\rho(\overline{u}))$ therefore gives the desired equation \eqref{eq:cocycle_lemma_3}.

      We proceed now with the proof of \eqref{eq:cocycle_lemma_3} in the case $u,\overline{u} \in T$. Since the two words $s_1\ldots s_r$ and $\overline{s}_1 \ldots \overline{s}_m$ in the generators define the same element in $W_{\op{aff}}$, by Tits' solution \cite{Tits} of the word problem for Coxeter groups we can transform $s_1 \ldots s_r$ into $\overline{s}_1 \ldots \overline{s}_r$ by applying a finite number of transformations of words in the generators $s \in S$ of the following form.
         \begin{align*} \tag{I} t_1 \ldots t_i t_{i+1} \ldots t_m & \longmapsto t_1 \ldots t_i s s t_{i+1} \ldots t_m \\
            \tag{II} t_1 \ldots t_i s s t_{i+1} \ldots t_m & \longmapsto t_1 \ldots t_i t_{i+1} \ldots t_m \\
            \tag{III} t_1 \ldots t_i \underbrace{s t s \ldots }_{m(s,t) < \infty} t_{i+1} \ldots t_m & \longmapsto t_1 \ldots t_i \underbrace{tst \ldots }_{m(s,t) < \infty} t_{i+1} \ldots t_m
         \end{align*}
      Consider the following `companion' transformations for expression of the form $n_{t_1}\ldots n_{t_m} u$ (with $t_i \in S$, $u \in T$)
         \begin{align*}
            \tag{$\text{I}^{(1)}$} n_{t_1} \ldots n_{t_i} n_{t_{i+1}} \ldots n_{t_m} u & \longmapsto n_{t_1} \ldots n_{t_i} n_s n_s n_{t_{i+1}} \ldots n_{t_m} (t_{i+1}\ldots t_m)^{-1}(n_s^{-2}) u \\
            \tag{$\text{II}^{(1)}$} n_{t_1} \ldots n_{t_i} n_s n_s n_{t_{i+1}} \ldots n_{t_m} u & \longmapsto n_{t_1} \ldots n_{t_i} n_{t_{i+1}} \ldots n_{t_m} (t_{i+1}\ldots t_m)^{-1}(n_s^2) u \\
            \tag{$\text{III}^{(1)}$} n_{t_1} \ldots n_{t_i} \underbrace{ n_s n_t n_s \ldots }_{m(s,t)} n_{t_{i+1}} \ldots n_{t_m} u & \longmapsto n_{t_1} \ldots n_{t_i} \underbrace{ n_t n_s n_t \ldots }_{m(s,t)} n_{t_{i+1}} \ldots n_{t_m} u
         \end{align*}
         Taking the sequence of transformations of type (I)-(III) which transforms $s_1\ldots s_r$ into $\overline{s}_1\ldots \overline{s}_m$ and applying the corresponding sequence of transformations of type ($\text{I}^{(1)}$)-($\text{III}^{(1)}$) to $n_{s_1}\ldots n_{s_r} u$ will give an expression of the form $n_{\overline{s}_1}\ldots n_{\overline{s}_m} t$ with $t \in T$. A simple computation shows that the transformations ($\text{I}^{(1)}$)-($\text{III}^{(1)}$) do not change the element in $W^{(1)}$ which the expression defines. Therefore
         \[ n_{\overline{s}_1}\ldots n_{\overline{s}_m} t = n_{s_1}\ldots n_{s_r} u = n_{\overline{s}_1}\ldots n_{\overline{s}_m} \overline{u} \]
         and therefore $t = \overline{u}$. To prove \eqref{eq:cocycle_lemma_3}, it is therefore enough to show that the element in $M$ defined by the right hand side of \eqref{eq:cocycle_lemma_3} corresponding to an expression $n_{s_1}\ldots n_{s_r} u$ does not change if we apply any transformation of type ($\text{I}^{(1)}$)-($\text{III}^{(1)}$). For ($\text{III}^{(1)}$) this follows immediately from property (iii). We now prove the invariance for transformations of type ($\text{I}^{(1)}$), leaving the dual case ($\text{II}^{(1)}$) to the reader. It obviously suffices to consider the case $i = 0$, i.e. the transformation
         \[ n_{s_1}\ldots n_{s_r} u \mapsto n_s n_s n_{s_1}\ldots n_{s_r} (s_1\ldots s_r)^{-1}(n_s^{-2})u \]
      and to prove that
      \[ \sigma(s_1)\cdot{}\ldots\cdot{} (n_{s_1}\ldots{}n_{s_{r-1}})(\sigma(s_r))\cdot{} (n_{s_1}\ldots{}n_{s_r})(\rho(u)) \]
      is equal to
      \[ \sigma(s)\cdot{}n_s(\sigma(s))\cdot{}(n_s n_s)(\sigma(s_1))\cdot{}\ldots\cdot{}(n_s n_s n_{s_1}\ldots{}n_{s_{r-1}})(\sigma(s_r))\cdot{}(n_s n_s n_{s_1}\ldots{}n_{s_r})(\rho((s_1\ldots s_r)^{-1}(n_s^{-2})u) \]
      But, using property (i) and the following identity (implied by the cocycle property of $\rho$)
      \[ (n_s n_s n_{s_1}\ldots{}n_{s_r})(\rho((s_1\ldots{}s_r)^{-1}(n_s^{-2})u)) = (n_s n_s n_{s_1}\ldots{}n_{s_r})(\rho((s_1\ldots s_r)^{-1}(n_s^{-2})))\cdot{} (n_{s_1}\ldots{}n_{s_r})(\rho(u)) \]
      it follows that the last expression is equivalent to
      \[ \rho(n_s^2)\cdot{}n_s^2 \left(\sigma(s_1)\cdot{}\ldots\cdot{}(n_{s_1}\ldots{}n_{s_{r-1}})(\sigma(s_r))\cdot{}(n_{s_1}\ldots n_{s_r})(\rho((s_1\ldots s_r)^{-1}(n_s^{-2})))\right)\cdot{} (n_{s_1}\ldots{}n_{s_r})(\rho(u)) \]
      Thus it suffices to show that
      \[ \sigma(s_1)\cdot{}\ldots\cdot{} (n_{s_1}\ldots{}n_{s_{r-1}})(\sigma(s_r)) \]
      is equal to
      \[ \rho(n_s^2)\cdot{}n_s^2 \left(\sigma(s_1)\cdot{}\ldots\cdot{}(n_{s_1}\ldots{}n_{s_{r-1}})(\sigma(s_r))\cdot{}(n_{s_1}\ldots n_{s_r})(\rho((s_1\ldots s_r)^{-1}(n_s^{-2})))\right) \]
      But, this follows immediately by repeated application of (iv), using that $\rho(n_s^2)\cdot{}n_s^2(\rho(n_s^{-2})) = 1$.

      Thus, we have shown the existence of a map $\phi: W^{(1)} \rightarrow M$ satisfying \eqref{eq:cocycle_lemma_2}. It remains to show that $\phi$ is a $1$-cocycle if condition (ii) is satisfied, i.e. that
      \begin{align}\label{eq:cocycle_lemma_7} \phi(ww') & = \phi(w)\cdot{}w(\phi(w')) \end{align}
      holds for all $w,w' \in W^{(1)}$. First, we consider the case when $w$ is as a product $w = n_{s_1}\ldots n_{s_r}$ in the distinguished generators. In this case, \eqref{eq:cocycle_lemma_7} follows immediately from \eqref{eq:cocycle_lemma_2}.
      Next, we treat the case $w = u \in \Omega^{(1)}$. From the identity
      \[ u(n_s) = n_{u(s)} u(t_{s,u}) \]
      it follows by induction that
      \begin{equation}\label{eq:cocycle_lemma_8} u(n_{s_1}) \ldots u(n_{s_i}) = n_{u(s_1)} \ldots n_{u(s_i)} u(t_{s_i,u} s_i^{-1}(t_{s_{i-1},u}) \ldots (s_2 \ldots s_i)^{-1}(t_{s_1,u})) \end{equation}
         Using \eqref{eq:cocycle_lemma_8}, we can now repeatedly apply property (ii) to compute $\phi(w)\cdot{}w(\phi(w'))$ for $w = u \in \Omega^{(1)}$ and $w' = n_{s_1}\ldots{}n_{s_r}u'$:
      \begin{align*} \phi(u)\cdot{}u(\phi(w')) & = \rho(u)\cdot{}u(\sigma(s_1)\cdot{}n_{s_1}(\sigma(s_2))\cdot{}\ldots \cdot{} (n_{s_1}\ldots n_{s_{r-1}})(\sigma(s_r))\cdot{} (n_{s_1}\ldots n_{s_r})(\rho(u'))) \\
               & = \rho(u)\cdot{}u(\sigma(s_1))\cdot{}u(n_{s_1})(u(\sigma(s_2)))\cdot{}\ldots \cdot{} u(n_{s_1}\ldots n_{s_r})(u(\rho(u'))) \\
               & = \sigma(u(s_1))\cdot{} n_{u(s_1)}(\rho(u t_{s_1,u}))\cdot{}u(n_{s_1})(u(\sigma(s_2)))\cdot{}\ldots \cdot{} u(n_{s_1}\ldots n_{s_r})(u(\rho(u'))) \\
               & = \sigma(u(s_1))\cdot{} n_{u(s_1)}(\sigma(u(s_2)))\cdot{}(n_{u(s_1)}n_{u(s_2)})(\rho(ut_{s_2,u}s_2^{-1}(t_{s_1,u})))\cdot{} \\
               & \cdot{}u(n_{s_1}n_{s_2})(u(\sigma(s_3)))\cdot{}\ldots\cdot{}u(n_{s_1}\ldots n_{s_r})(u(\rho(u'))) \\
               & \vdots \\
               & = \sigma(u(s_1))\cdot{}n_{u(s_1)}(\sigma(u(s_2)))\cdot{}\ldots \cdot{} (n_{u(s_1)}\ldots{}n_{u(s_{r-1})})(\sigma(u(s_r))\cdot{} \\
               & \cdot{} (n_{u(s_1)}\ldots{}n_{u(s_r)})(\rho(u t_{s_r,u} s_r^{-1}(t_{s_{r-1}})\ldots (s_2\ldots s_r)^{-1}(t_{s_1,u}))) \\
               & \cdot{} (u(n_{s_1}\ldots{}n_{s_r}))(u(\rho(u')))
      \end{align*}
      Using \eqref{eq:cocycle_lemma_8} again, we see that
      \[ u(n_{s_1}\ldots{}n_{s_r})(u(\rho(u'))) = (n_{u(s_1)}\ldots{}n_{u(s_r)} u t_{s_r,u} s_r^{-1}(t_{s_{r-1},u}) \ldots (s_2\ldots{}s_r)^{-1}(t_{s_1,u}))(\rho(u')) \]
      We can therefore apply the cocycle property of $\rho$ to finally obtain that
      \begin{equation}\label{eq:cocycle_lemma_9}\begin{aligned} \phi(u)\cdot{}u(\phi(w')) & = \sigma(u(s_1))\cdot{}n_{u(s_1)}(\sigma(u(s_2)))\cdot{}\ldots \cdot{} (n_{u(s_1)}\ldots{}n_{u(s_{r-1})})(\sigma(u(s_r))\cdot{} \\
               & \cdot{} (n_{u(s_1)}\ldots{}n_{u(s_r)})(\rho(u t_{s_r,u} s_r^{-1}(t_{s_{r-1}})\ldots (s_2\ldots s_r)^{-1}(t_{s_1,u}) u'))
            \end{aligned}
      \end{equation}
      Now
      \begin{align*} u w' & = u n_{s_1}\ldots n_{s_r}u' \\
      & = n_{u(s_1)}u t_{s_1,u} n_{s_2} \ldots n_{s_r} u' = n_{u(s_1)} u n_{s_2}\ldots n_{s_r} (s_2\ldots s_r)^{-1}(t_{s_1,u}) u' \\
      & \vdots \\
      & = n_{u(s_1)} n_{u(s_2)} \ldots n_{u(s_r)} u t_{s_r,u} s_r^{-1}(t_{s_{r-1},u}) \ldots (s_2\ldots s_r)^{-1}(t_{s_1,u}) u'\end{align*}
      and hence
      \begin{equation}\label{eq:cocycle_lemma_10} \begin{aligned} \phi(uw') & = \sigma(u(s_1))\cdot{}n_{u(s_1)}(\sigma(u(s_2))\cdot{}\ldots \cdot{} (n_{u(s_1)}\ldots n_{u(s_{r-1})})(\sigma(u(s_r))) \cdot{} \\
         & \cdot{} (n_{u(s_1)}\ldots n_{u(s_r)})(\rho(u t_{s_r,u} s_r^{-1}(t_{s_{r-1},u}) \ldots (s_2\ldots s_r)^{-1}(t_{s_1,u}) u'))  \end{aligned} \end{equation}
         Comparing \eqref{eq:cocycle_lemma_9} with \eqref{eq:cocycle_lemma_10} gives \eqref{eq:cocycle_lemma_7} for $w = u \in \Omega^{(1)}$ and $w' \in W^{(1)}$ arbitrary. The general case now follows by induction on $\ell(w)$. We have just proved the start of the induction $\ell(w) = 0$. Now let $\ell(w) = r > 0$ and write $w = n_{s_1}\ldots n_{s_r} u$. Then
         \begin{align*} \phi(w w') & = \phi(n_{s_1}n_{s_2}\ldots n_{s_r} u w') = \phi(n_{s_1}) n_{s_1}(\phi(n_{s_2}\ldots n_{s_r} u w')) \\
            & = \phi(n_{s_1}) n_{s_1}(\phi(n_{s_2}\ldots n_{s_r} u) \cdot{} (n_{s_2}\ldots n_{s_r} u)(\phi(w'))) \\
            & = \phi(n_{s_1}n_{s_2}\ldots n_{s_r} u) (n_{s_1} n_{s_2} \ldots n_{s_r} u)(\phi(w')) \\
            & = \phi(w) w(\phi(w'))\end{align*}
      where we used that $\ell(n_{s_2}\ldots n_{s_r}u) = r-1 < r$ in line 2 in order to apply the induction hypothesis.
\end{proof}

\subsection{Construction of generic pro-\texorpdfstring{$p$}{p} Hecke algebras} 
\label{sub:Construction of generic pro-$p$ Hecke algebras}
In this section we will construct the main object of this article. Throughout, $W^{(1)}$ will denote a fixed pro-$p$ Coxeter group. The notation $W,W_{\op{aff}},S,\Omega,\ell$ etc. will be conserved. We will also fix a commutative associative unital ring $R$. The monoid algebra of $T$ over $R$ will be denoted by $R[T]$. The action of $W$ on $T$ extends naturally to an action on $R[T]$ by $R$-algebra automorphisms.

\begin{theorem}\label{thm:exhecke} Let $(a_s)_{s \in S}$ and $(b_s)_{s \in S}$ be families of elements $a_s \in R$ and $b_s \in R[T]$ subject to the following condition. Given $s,t \in S$ and $w \in W^{(1)}$ such that $s\pi(w) = \pi(w)t$, the following two equalities in $R$ resp. $R[T]$ hold\footnote{Note that $n_s w n_t^{-1} w^{-1} \in T$}
   \begin{align}\label{eq:condstar}
      a_s & = a_t & (n_s w n_t^{-1} w^{-1}) w(b_t) & = b_s
   \end{align}
   Under this assumption, there exists a unique structure of an $R$-algebra on the free $R$-Module $M$ with basis $\{T_w\}_{w \in W^{(1)}}$ which is compatible with the given $R$-module structure and such that the following two conditions hold
   \begin{equation}\label{eq:heckecond1} \forall w,w' \in W^{(1)}\quad \ell(ww') = \ell(w)+\ell(w')\quad \Rightarrow\quad T_{ww'} = T_w T_w' \end{equation}
   \begin{equation}\label{eq:heckecond2} \forall s \in S\quad T_{n_s}^2 = a_s T_{n_s^2} + T_{n_s} b_s \end{equation}
\end{theorem}
Before we begin with the proof of \cref{thm:exhecke}, we make a couple of remarks. 
\begin{rmk}
   \begin{enumerate}
      \item As a consequence of the first condition, the natural embedding $R[T] \hookrightarrow M$ of $R$-modules will be a morphism of $R$-algebras because the length function vanishes on $T$. The $R$-algebra $M$ will therefore carry a canonical structure of an $(R[T],R[T])$-bimodule so that the second condition makes sense.
      \item The first condition implies the following basic commutation rule $t \in T$ and $w \in W^{(1)}$
         \begin{equation}\label{eq:commrule} T_w T_t = T_{wt} = T_{wtw^{-1} w} = T_{w(t)w}= T_{w(t)} T_w \end{equation}
         This implies more generally that for any $b \in R[T]$ we have
         \begin{equation}\label{eq:gencommrule} T_w b = w(b) T_w \end{equation}
         \item Applying relation \eqref{eq:condstar} for $w = n_s^{-1}$ and $s = t$ shows that
            \begin{equation}\label{eq:paraminv} n_s^{-1}(b_s) = b_s \end{equation}
            \item In view of \eqref{eq:gencommrule} and \eqref{eq:paraminv}, the second relation could also have been written as
         \[ T_{n_s}^2 = a_s T_{n_s^2} + b_s T_{n_s} \]
   \end{enumerate}
\end{rmk}
\begin{proof}[Proof of \cref{thm:exhecke}]
   We will closely follow the proof in the classical case (cf. \cite[Ch. IV, Exercices \textsection 2, Ex. 23]{Bourbaki}). First, we show uniqueness. It suffices to prove that for all $w,w' \in W^{(1)}$ the expansion of the product $T_w T_{w'}$ in terms of the given basis can be effectively computed in terms of the coefficient families $(a_s)_s$ and $(b_s)_s$. If $\ell(w) > 0$, we can write $w = n_s \widetilde{w}$ with $\ell(w) = 1 + \ell(\widetilde{w})$. By \eqref{eq:heckecond1}
   \[ T_w T_{w'} = T_{n_s} T_{\widetilde{w}} T_{w'} \]
   By induction it therefore suffices to compute products of the form $T_u T_w$ for $u \in \Omega^{(1)}$ and $T_{n_s} T_w$. From \eqref{eq:heckecond1} it follows immediately that $T_u T_w = T_{u w}$. We now show how to compute products of the form $T_{n_s} T_w$ by induction on $\ell(w)$. If $\ell(n_s w) = \ell(w) + 1$, again by \eqref{eq:heckecond1} we find that $T_{n_s} T_w = T_{n_s w}$. If $\ell(n_s w) = \ell(w) - 1$, we can write $w = n_s \widetilde{w}$ with $\ell(w) = \ell(\widetilde{w}) + 1$, and so
   \begin{align*} T_{n_s} T_w & = T_{n_s} T_{n_s} T_{\widetilde{w}} = (a_s T_{n_s^2} + T_{n_s} b_s) T_{\widetilde{w}} = a_s T_{n_s^2 \widetilde{w}} + T_{n_s} T_{\widetilde{w}} \widetilde{w}^{-1}(b_s)
   \end{align*}
   We now show the existence of the algebra structure in question. The construction proceeds by defining an $R$-subalgebra $\Lambda \subseteq \op{End}_R(M)$ and then showing that $\op{ev}_{T_1}: \op{End}_R(M) \rightarrow M$ induces an isomorphism $\Lambda \stackrel{\sim}{\longrightarrow} M$ of $R$-modules. By transport of structure, we obtain an $R$-algebra structure on $M$ which is then easily verified to have the required properties.
   
   First, we will construct the structure of an $(R[\Omega^{(1)}],R[\Omega^{(1)}])$-bimodule structure on $M$. Such a structure is equivalent to giving morphisms $\lambda: R[\Omega^{(1)}] \rightarrow \op{End}_R(M)$ and $\rho: R[\Omega^{(1)}]^{\op{op}} \rightarrow \op{End}_R(M)$ whose images commute. For $u \in \Omega^{(1)}$ we define $\lambda(u)$ and $\rho(u)$ on basis elements by
   \begin{align*} \lambda(u)(T_w) & := T_{uw} & \quad \rho(u)(T_w) & := T_{wu} \end{align*}
   One verifies immediately that $\lambda(u u') = \lambda(u)\lambda(u')$ and $\rho(u u') = \rho(u') \rho(u)$ and hence we get well defined morphisms $\lambda$ and $\rho$. From the definition it is immediate that the images of $\lambda$ and $\rho$ commute. With respect to this bimodule structure the following identity
   \[ T_w b = w(b) T_w \]
   holds for all $b \in R[T] \subseteq R[\Omega^{(1)}]$ and $w \in W^{(1)}$.
   
   We will now introduce for every $s \in S$ elements $\lambda_{n_s}, \rho_{n_s} \in \op{End}_R(M)$, which will a posteriori turn out the be left respectively right multiplication by $T_{n_s}$. Put
   \begin{align*} \lambda_{n_s}(T_w) & := \begin{cases}
      T_{n_s w} & : \ell(n_s w) = \ell(w) + 1 \\
      a_s T_{n_s w} + b_s T_w & : \ell(n_s w) = \ell(w) - 1
   \end{cases} \\
   \rho_{n_s}(T_w) & := \begin{cases}
      T_{w n_s} & : \ell(w n_s) = \ell(w) + 1 \\
      T_{w n_s} a_s + T_w b_s & : \ell(w n_s) = \ell(w) - 1
   \end{cases}
\end{align*}
The products $b_s T_w$, $a_s T_{n_s w}$ etc. therefore refer to the $(R[\Omega^{(1)}],R[\Omega^{(1)}])$-bimodule structure already constructed. Also note that $\lambda_{n_s}$ and $\rho_{n_s}$ are linear with respect to the right respectively left $R[\Omega^{(1)}]$-module structure.

The main part of the proof consists of showing that the elements $\lambda_{n_s}, \rho_{n_t}$ commute for all $s,t \in S$. Fix $w \in W^{(1)}$ and $s,t \in S$. We make a case distinction according to the $6$ possible constellations of $\ell(w), \ell(n_s w), \ell(w n_t)$ and $\ell(n_s w n_t)$
\begin{enumerate}
   \item $\ell(n_s w n_t) > \ell(n_s w) = \ell(w n_t) > \ell(w)$:
      \begin{align*} (\lambda_{n_s}\rho_{n_t})(T_w) = \lambda_{n_s}(T_{w n_t}) = T_{n_s w n_t} = \rho_{n_t} (T_{n_s w}) = (\rho_{n_t} \lambda_{n_s})(T_w) \end{align*}
   \item $\ell(n_s w n_t) < \ell(n_s w) = \ell(w n_t) < \ell(w)$:
      \begin{align*} (\lambda_{n_s} \rho_{n_t})(T_w) & = \lambda_{n_s}(T_{w n_t} a_t + T_w b_t) = \lambda_{n_s}(T_{w n_t}) a_t + \lambda_{n_s}(T_w) b_t \\
         & = a_s T_{n_s w n_t} a_t + b_s T_{w n_t} a_t + a_s T_{n_s w} b_t + b_s T_w b_t \\
         & = a_s \rho_{n_t}(T_{n_s w}) + b_s \rho_{n_t}(T_w) = \rho_{n_t}(a_s T_{n_s w} + b_s T_w) \\
         & = (\rho_{n_t} \lambda_{n_s})(T_w)
      \end{align*}
   \item $\ell(n_s w n_t) = \ell(w) < \ell(n_s w) = \ell(w n_t)$:
      By \cref{lem:swt}, we have $s \pi(w) = \pi(w) t$ and hence that $n_s w n_t^{-1} w^{-1} \in T$. We can therefore invoke relation \eqref{eq:condstar} to conclude that
      \begin{align*}
         (\lambda_{n_s}\rho_{n_t})(T_w) & = \lambda_{n_s}(T_{w n_t}) = a_s T_{n_s w n_t} + b_s T_{w n_t} \\
         & = a_t T_{n_s w n_t} + (n_s w n_t^{-1} w^{-1}) w(b_t) \rho_{n_t}(T_w) \\
         & = a_t T_{n_s w n_t} + \rho_{n_t}((n_s w n_t^{-1} w^{-1}) w(b_t) T_w) \\
         & = a_t T_{n_s w n_t} + \rho_{n_t}((n_s w n_t^{-1} w^{-1}) T_w b_t) \\
         & = a_t T_{n_s w n_t} + \rho_{n_t}(T_{n_s w n_t^{-1}} b_t) \\
         & = a_t T_{n_s w n_t} + \rho_{n_t}((n_s w n_t^{-1})(b_t) T_{n_s w n_t^{-1}}) \\
         & = a_t T_{n_s w n_t} + (n_s w n_t^{-1})(b_t) \rho_{n_t} (T_{n_s w n_t^{-1}}) \\
         & = a_t T_{n_s w n_t} + (n_s w n_t^{-1})(b_t) T_{n_s w} \\
         & = a_t T_{n_s w n_t} + T_{n_s w} n_t^{-1}(b_t) \\
         & \stackrel{\text{\ref{eq:paraminv}}}{=} T_{n_s w n_t} a_t + T_{n_s w} b_t \\
         & = \rho_{n_t}(T_{n_s w}) = (\rho_{n_t} \lambda_{n_s})(T_w)
      \end{align*}
   \item $\ell(n_s w n_t) = \ell(w) > \ell(n_s w) = \ell(w n_t)$: Similar to (iii).
   \item $\ell(n_s w) < \ell(w) = \ell(n_s w n_t) < \ell(w n_t)$:
      \begin{align*} (\lambda_{n_s}\rho_{n_t})(T_w) & = \lambda_{n_s}(T_{w n_t}) = a_s T_{n_s w n_t} + b_s T_{w n_t} = \rho_{n_t}(a_s T_{n_s w} + b_s T_w) \\
         & = (\rho_{n_t}\lambda_{n_s})(T_w) \end{align*}
   \item $\ell(n_s w) > \ell(w) = \ell(n_s w n_t) > \ell(w n_t)$: Similar to (v).

 Let now $\Lambda \subseteq \op{End}_R(M)$ be the $R$-subalgebra generated by $\{\lambda_{n_s}\}_{s \in S}$ and $\{\lambda_u\}_{u \in \Omega^{(1)}}$ and consider the evaluation homomorphism $\op{ev}_{T_1}: \op{End}_R(M) \rightarrow M$, $\op{ev}_{T_1}(\phi) = \phi(T_1)$. We claim that restriction to $\Lambda$ induces an isomorphism
   \[ \left.\op{ev}_{T_1}\right|: \Lambda \stackrel{\sim}{\rightarrow} M \]
   of $R$-modules. If $s \in S$ and $w \in W^{(1)}$ are such that $\ell(n_s w) = 1 + \ell(w)$, then $\lambda_{n_s}(T_w) = T_{n_s w}$ by definition. From this it follows immediately that
   \[ \op{ev}_{T_1}(\lambda_{n_{s_1}}\circ{}\ldots\circ{}\lambda_{n_{s_r}}\circ{}\lambda_u) = T_{n_{s_1}\ldots n_{s_r}u} \]
   if $w = n_{s_1}\ldots n_{s_r} u$, $u \in \Omega^{(1)}$ is a reduced expression. This proves surjectivity. To show injectivity let $\phi \in \Lambda$ be such that $\phi(T_1) = 0$. It suffices to show by induction on $\ell(w)$ that $\phi(T_w) = 0$ for all $w \in W^{(1)}$. For $\ell(w) = 0$ we have $w = u \in \Omega^{(1)}$ and hence
   \[ \phi(T_u) = \phi(\rho_u(T_1)) = \rho_u(\phi(T_1)) = 0 \]
   Here we have used the fact that $\rho_u$ commutes with all elements of $\Lambda$. If $\ell(w) > 0$, write $w = \widetilde{w} n_s$ with $\ell(w) = 1+\ell(\widetilde{w})$. Then
   \[ \phi(T_w) = \phi(\rho_{n_s} T_{\widetilde{w}}) = \rho_{n_s}(\phi(T_{\widetilde{w}})) = 0 \]
   where we have made use of the fact that $\rho_{n_s}$ commutes with the elements of $\Lambda$.

   By transport of structure, we now get on $M$ the structure of an $R$-algebra compatible with the given $R$-module structure. It remains to verify the conditions \eqref{eq:heckecond1} and \eqref{eq:heckecond2}. Assume $\ell(ww') = \ell(w)+\ell(w')$ and let $w = u n_{s_1}\ldots n_{s_r}$, $w' = n_{s_{r+1}} \ldots n_{s_{r+t}} u'$ be two reduced expressions. Then $\op{ev}_{T_1}(\lambda_u \lambda_{n_{s_1}} \ldots \lambda_{n_{s_r}}) = T_w$ and $\op{ev}_{T_1}(\lambda_{n_{s_{r+1}}}\ldots \lambda_{n_{s_{r+t}}} \lambda_u) = T_{w'}$ and hence
   \[ T_w T_{w'} = \op{ev}_{T_1}(\lambda_u \lambda_{n_{s_1}}\ldots \lambda_{n_{s_r}} \lambda_{n_{s_{r+1}}} \ldots \lambda_{n_{s_{r+t}}} \lambda_{u'}) = T_{u n_{s_1} \ldots n_{s_r} n_{s_{r+1}} \ldots n_{s_{r+t}} u' } = T_{ww'} \]
   The validity of \eqref{eq:heckecond2} is equivalent to
   \[ (\lambda_{n_s}\circ{} \lambda_{n_s})(T_1) = a_s \lambda_{n_s^2}(T_1) + (\lambda_{n_s}\circ{}\lambda_{b_s})(T_1) \]
   But
   \[ \lambda_{n_s}^2(T_1) = \lambda_{n_s}(T_{n_s}) = a_s T_{n_s^2} + b_s T_{n_s} \]
   by definition and
   \begin{align*} a_s \lambda_{n_s^2}(T_1) + (\lambda_{n_s}\circ{}\lambda_{b_s})(T_1) & = a_s T_{n_s^2} + \lambda_{n_s}(b_s) = a_s T_{n_s^2} + T_{n_s} b_s \\
      & = a_s T_{n_s^2} + n_s(b_s) T_{n_s} = a_s T_{n_s^2} + b_s T_{n_s}
   \end{align*}
\end{enumerate}
\end{proof}

\begin{lemma}\label{lem:swt}
   Let $W$ be an extended Coxeter group and $w \in W$, $s,t \in S$. If either
   \[ \ell(sw) = \ell(wt) < \ell(w) = \ell(swt) \]
   or
   \[ \ell(sw) = \ell(wt) > \ell(w) = \ell(swt) \]
   then
   \[ swt = w \]
\end{lemma}
\begin{proof} For the case of ordinary Coxeter groups we refer to \cite[Lemma 7.2]{HumReflect}. We show why the statement carries over to the case of extended Coxeter groups. Assume for concreteness that we are in the first case. Write $w = w'u$ with $w' \in W_{\op{aff}}$ and $u \in \Omega$. Then
   \[ \ell(sw') = \ell(sw) = \ell(wt) = \ell(w'u(t)u) = \ell(w'u(t)) \]
   and
   \[ \ell(w') = \ell(w) = \ell(swt) = \ell(sw'u(t)u) = \ell(sw'u(t)) \]
   According to the version of this lemma for Coxeter groups we conclude that
   \[ w' = sw'u(t) \]
   and hence $w = swt$.
\end{proof}

\begin{definition}\label{def:generic pro-$p$ Hecke algebras}The $R$-algebra constructed in \cref{thm:exhecke} is called \textbf{the generic pro-$p$ Hecke algebra for the parameters $a = (a_s)_s, b = (b_s)_s$} and is denoted by $\mathcal{H}_R^{(1)}(a,b)$.
\end{definition}

\begin{rmk}
      Because of \cref{rmk:ex_weyl_walls} and the condition \eqref{eq:condstar}, we can extend the family $(a_s)_{s \in S}$ to a family $(a_H)_{H \in \mathfrak{H}}$ by putting
   \[ a_H := a_s \]
   if $s \in S$ is an element conjugate to $H \in \mathfrak{H}$ under $W$.
\end{rmk}

\begin{rmk}\label{rmk:condstar}
   The condition \eqref{eq:condstar} of \cref{thm:exhecke} is easily seen to be equivalent to the following two conditions.
   \begin{enumerate}
      \item For any $s,t \in S$ which are conjugate under $W$ we have
   \[ a_s = a_t \]
   and for \textit{some} $w \in W^{(1)}$ with $s \pi(w) = \pi(w) t$ we have
   \[ (n_s w n_t^{-1} w^{-1}) w(b_t) = b_s \]
\item For every $s \in S$ and every $t \in T$ we have
   \[ s(t)t^{-1} b_s = b_s \]
   and for every $w \in W$ with $s w = w s$ we have
   \[ (n_s \widetilde{w} n_s^{-1} \widetilde{w}^{-1}) w(b_s) = b_s \]
   for \textit{some} lift $\widetilde{w} \in W^{(1)}$ of $w$ under $\pi: W^{(1)} \rightarrow W$.
   \end{enumerate}
\end{rmk}

\begin{example}\label{ex:yokonuma}
   The main examples of generic pro-$p$ Hecke algebras that motivated their introduction and the terminology are the double coset convolution algebras $H(G,I^{(1)})$ associated to pro-$p$-Iwahori subgroups $I^{(1)} \leq G$ of reductive groups. These will be considered in detail in the next section (\cref{ex:pro-$p$-Iwahori Hecke algebras}).
   Let us therefore consider here other important examples.
   \begin{enumerate}
      \item[(i)] Every Coxeter group $W$ can be viewed as a pro-$p$ Coxeter group with $T = \Omega = 1$ and $n_s = s$. The generic pro-$p$ Hecke algebra $\mathcal{H}^{(1)}_R((a_s)_s,(b_s)_s)$ then coincides with the classical generic Hecke algebra associated to the Coxeter group $W$ and the families $(a_s)_s, (b_s)_s \in R$ of parameters. In the notation of \cite[Ch. IV, Exercices \textsection 2, Ex. 23]{Bourbaki} we have
         \[ \mathcal{H}^{(1)}_R((a_s)_s,(b_s)_s) = E_R((b_s),(a_s)) \] 
      \item[(ii)] Given a Coxeter group $W$ and an action
         \[ W \longrightarrow \op{GL}_\Z(T) \]
         on an abelian group $T$ by group automorphisms, we get a pro-$p$ Coxeter group $W^{(1)}$ with $\Omega = 1$ by forming the semi-direct product $W^{(1)} = T\rtimes W$ and letting $n_s = s$.
         
         Generic pro-$p$ Hecke algebras of this type include Yokonuma-Hecke algebras $Y_{d,n}$ ($d,n \in \N$). These are algebras over $R = \C[u^{\pm 1},v]$ generated by elements (cf. \cite{Jacon-dAndecy})
         \[ g_1, \ldots , g_{n-1}, t_1,\ldots ,t_n \]
         subject to the relations
         \begin{alignat*}{2} g_i g_j & = g_j g_i & \quad & \text{for all $i,j = 1,\ldots ,n-1$ such that $|i-j| > 1$} \\
         g_i g_{i+1} g_i & = g_{i+1}g_i g_{i+1} & \quad & \text{for all $i = 1,\ldots ,n-2$} \\
         t_i t_j & = t_j t_i & \quad & \text{for all $i,j = 1,\ldots ,n$} \\
         g_i t_j & = t_{s_i(j)} g_i & \quad & \text{for all $i = 1,\ldots ,n-1$ and $j = 1,\ldots ,n$} \\
         t_j^d & = 1 & \quad & \text{for all $j = 1,\ldots ,n$} \\
         g_i^2 & = u^2 + ve_i g_i & \quad & \text{for all $i = 1,\ldots ,n-1$}
      \end{alignat*}
      where $s_i \in S_n$ denotes the transposition $(i\ i+1)$ and $e_i$ is given by
      \[ e_i = \frac{1}{d}\sum_{0 \leq s < d} (t_i/t_{i+1})^s \]
      In order to relate these to generic pro-$p$ Hecke algebras, let $W$ be the Coxeter group $S_n$ with the standard generators $S = \{s_1,\ldots ,s_{n-1}\}$ and let $T$ be the finite abelian group $T = (\Z/d\Z)^n$. Then we get an isomorphism
      \[ Y_{d,n} \stackrel{\sim}{\longrightarrow} \mathcal{H}^{(1)}_R((a_s)_s,(b_s)_s) \]
      of $R$-algebras by sending $g_i$ to $T_{n_{s_i}} = T_{s_i}$ and $t_j$ to the element of $T$ denoted by the same letter and given component-wise by $(t_j)_i = \overline{\delta_{ij}} \in \Z/d\Z$, if we let
      \[ a_{s_i} = u^2 \in R \quad\quad i = 1,\ldots ,n-1 \]
      and
      \[ b_{s_i} = \frac{v}{d}\sum_{s \in \Z/d\Z} (t_i/t_{i+1})^s \in R[T] \quad\quad i = 1,\ldots ,n-1 \]
   \end{enumerate}
\end{example}

\begin{rmk}\label{rmk:base_change_and_universal_coefficients}
   \begin{enumerate}
      \item Given a ring $R$ and families $a = (a_s)_{s \in S} \in R$, $b = (b_s)_{s \in S} \in R[T]$ satisfying condition \eqref{eq:condstar}, it is clear that for any ring homomorphism $\varphi: R \rightarrow R'$ the image families $\varphi(a) = (\varphi(a_s))_{s \in S} \in R'$ and $\varphi(b) = (\varphi(b_s))_{s \in S} \in R'[T]$ again satisfy condition \eqref{eq:condstar}. Moreover, the natural homomorphism of $R'$-algebras
   \begin{align*} \mathcal{H}^{(1)}_R(a,b)\otimes_R R' & \longrightarrow \mathcal{H}^{(1)}_{R'}(\varphi(a),\varphi(b)) \\
      T_w \otimes x & \longmapsto \varphi(x) T_w
   \end{align*}
   is an isomorphism, as it is a bijection on the canonical $R'$-bases on both sides.
\item Given a pro-$p$ Coxeter group $W^{(1)}$, let $\mathcal{R}(W^{(1)})$ denote the following category. Objects of $\mathcal{R}(W^{(1)})$ consists of triples $(R,a,b)$ where $R$ is a ring and $a = (a_s)_{s \in S} \in R$ and $b = (b_s)_{s \in S} \in R[T]$ are parameters satisfying condition \eqref{eq:condstar}. A morphism $f: (R,a,b) \rightarrow (R',a',b')$ is a ring homomorphism $f: R \rightarrow R'$ preserving the parameters
   \[ f(a_s) = a'_s,\quad f[T](b_s) = b'_s \quad \forall s \in S \]
   Here $f[T]: R[T] \rightarrow R'[T]$ denotes the induced ring homomorphism.
   
   If the group $T$ is finite, the category $\mathcal{R}(W^{(1)})$ has an initial object $R^{\op{univ}}$ given as follows. Consider the polynomial ring
   \[ R = \Z[\{\mathbf{a}_s, \mathbf{b}_{s,t} : s \in S,\ t \in T \}] \]
   in the formal variables $\mathbf{a}_s$ and $\mathbf{b}_{s,t}$. Let
   \[ \mathbf{b}_s := \sum_{t \in T} \mathbf{b}_{s,t}\cdot t \in R[T] \]
   The families $(\mathbf{a}_s)_{s \in S} \in R_0$ and $(\mathbf{b}_s)_{s \in S} \in R_0[T]$ do not satisfy condition \eqref{eq:condstar} in general. However, condition \eqref{eq:condstar} is equivalent to a set of relations of the form
   \[ \mathbf{a}_s = \mathbf{a}_{s'} \quad \text{and}\quad \mathbf{b}_{s,t} = \mathbf{b}_{s',t'} \]
   Letting $p: R \twoheadrightarrow R^{\op{univ}}$ denote the quotient of $R$ by the ideal generated by these relations, we obtain a well-defined object
   \[ R^{\op{univ}} = (R^{\op{univ}},a^{\op{univ}},b^{\op{univ}}) := (R^{\op{univ}},(p(\mathbf{a}_s))_{s \in S},(p[T](\mathbf{b}_s))_{s \in S}) \]
   of the category $\mathcal{R}(W^{(1)})$. It is clear from the construction that this object is initial. Moreover, by construction $R^{\op{univ}}$ is the polynomial ring over $\Z$ on a set of formal variables, that is a quotient of the set $S\amalg (S\times T)$. In particular $R^{\op{univ}}$ is noetherian if $\# S < \infty$.
\item By the above remarks, when $T$ is finite, every generic pro-$p$ Hecke algebra $\mathcal{H}^{(1)}_R(a,b)$ over a ring $R$ is naturally obtained by base change
   \[ \mathcal{H}^{(1)}_{R^{\op{univ}}}(a^{\op{univ}},b^{\op{univ}}) \otimes_{R^{\op{univ}}} R \stackrel{\sim}{\longrightarrow} \mathcal{H}^{(1)}_R(a,b) \]
   from the \textit{universal} generic pro-$p$ Hecke algebra $\mathcal{H}^{(1)}_{R^{\op{univ}}}(a^{\op{univ}},b^{\op{univ}})$ over $R^{\op{univ}}$. This allows to reduce many statements about generic pro-$p$ Hecke algebras to the `universal case'. In particular when we will study the structure of \textit{affine} pro-$p$ Hecke algebras (in which case $S$ and $T$ are finite) in \cref{sec:Affine pro-$p$ Hecke algebras}, this will allow us to reduce to the case of a noetherian coefficient ring $R$.
   \end{enumerate}
\end{rmk}

\subsection{Presentations of generic pro-\texorpdfstring{$p$}{p} Hecke algebras via braid groups} 
\label{sub:Presentations of generic pro-$p$ Hecke algebras via (generalized) braid groups}
Generic Iwahori-Hecke algebras can be described as quotients of monoid algebras of braid monoids (see \cite[4.4.1]{GeckPfeiffer}). The same holds true for generic pro-$p$ Hecke algebras if one introduces the appropriate analogue of braid monoids in the context of pro-$p$ Coxeter groups.

\begin{definition}\label{def:generalized_braid_group} Let $W^{(1)}$ be a pro-$p$ Coxeter group.
   \begin{enumerate}
      \item The \textbf{(generalized) braid monoid $\mathfrak{B}(W^{(1)})$ associated to $W^{(1)}$} is the monoid with presentation
         \[ \mathfrak{B}(W^{(1)}) = \left< \{T_w\}_{w \in W^{(1)}} : T_{ww'} = T_w T_w' \text{ if } \ell(ww') = \ell(w)+\ell(w') \right> \]
      \item The \textbf{(generalized) braid group $\mathfrak{A}(W^{(1)})$ associated to $W^{(1)}$} is the group with presentation
         \[ \mathfrak{A}(W^{(1)}) = \left< \{T_w\}_{w \in W^{(1)}} : T_{ww'} = T_w T_w' \text{ if } \ell(ww') = \ell(w)+\ell(w') \right> \]
   \end{enumerate}
\end{definition}
By \eqref{eq:heckecond1}, the canonical map $\{T_w\}_{w \in W^{(1)}} \rightarrow \mathcal{H}_R^{(1)}(a,b)$ of sets extends to a morphism
\[ \mathfrak{B}(W^{(1)}) \longrightarrow \mathcal{H}_R^{(1)}(a,b) \]
of monoids which in turn induces a morphism
\[ R[\mathfrak{B}(W^{(1)})] \longrightarrow \mathcal{H}_R^{(1)}(a,b) \]
of $R$-algebras. Let $\mathfrak{b}$ denote the two-sided ideal in $R[\mathfrak{B}(W^{(1)})]$ generated by all elements of the form $T_{n_s}^2 - a_s T_{n_s^2} - T_{n_s} b_s$ , where $s$ runs over all elements of $S$. By \eqref{eq:heckecond2}, we have an induced morphism
\[ \phi: R[\mathfrak{B}(W^{(1)})]/\mathfrak{b} \longrightarrow \mathcal{H}_R^{(1)}(a,b) \]
\begin{prop}\label{prop:first_presentation_of_hecke}
   The above map $\phi$ is an isomorphism of $R$-algebras.
\end{prop}
\begin{proof}
   The proof is standard (cf. \cite{GeckPfeiffer}). Obviously $\phi$ is surjective. It therefore suffices to show that $\phi$ has a left inverse $\psi$. Because $\mathcal{H}_R^{(1)}(a,b)$ is a free $R$-module over $\{T_w\}_{w \in W^{(1)}}$, we have a map $\psi$ which associates to any element $T_w$ of the basis the image of the generator $T_w$ of the braid monoid in the quotient $R[\mathfrak{B}(W^{(1)})]/\mathfrak{b}$. Obviously the equation $(\psi \circ{} \phi)(x) = x$ is satisfied for said images of the generators of $\mathfrak{B}(W^{(1)})$. But because of the quadratic relations, these images already generate the quotient $R[\mathfrak{B}(W^{(1)})]/\mathfrak{b}$ as an $R$-module. Hence, $\psi \circ{} \phi = \op{id}$.
\end{proof}

When the parameters $a_s$ are units in $R$, the generic pro-$p$ Hecke algebra has a second presentation in terms of $\mathfrak{A}(W^{(1)})$. In fact, in this case $T_{n_s} \in \mathcal{H}_R^{(1)}(a,b)^\times$ with inverse
\[ T_{n_s}^{-1} = a_s^{-1}(T_{n_s^{-1}} - b_s T_{n_s^{-2}}) \]
This implies that for every $w \in W^{(1)}$ we have $T_w \in \mathcal{H}_R^{(1)}(a,b)^\times$, since
\[ T_w = T_{n_{s_1}}\ldots T_{n_{s_r}} T_u \]
if $w = n_{s_1}\ldots n_{s_r}u$, $u \in \Omega^{(1)}$ is any expression with $r = \ell(w)$. Just as before, we get an induced morphism
\[ \varphi: R[\mathfrak{A}(W^{(1)})]/\mathfrak{a} \longrightarrow \mathcal{H}_R^{(1)}(a,b) \]
of $R$-algebras, where $\mathfrak{a}$ denotes the two-sided ideal generated by $T_{n_s}^2 - a_s T_{n_s^2} - T_{n_s}b_s$, $s \in S$. The same arguments as in the previous proposition show that
\begin{prop}\label{prop:second_presentation_of_hecke}
   The above morphism $\varphi$ is an isomorphism of $R$-algebras.
\end{prop}

\begin{example}\label{ex:yokonuma-pres}
   \begin{enumerate}
      \item Continuing the examples given in \cref{ex:yokonuma}, if we take $W^{(1)} = W = S_n$ to be the symmetric group on $n$ letters in the first example, the associated generalized braid group $\mathfrak{A}(W^{(1)})$ identifies canonically with the classical \textbf{Artin braid group} $B_n$ on $n$ strands. The above presentation then relates the representation theory of finite Hecke algebras associated to $S_n$ to invariants of braids and hence of links, via the construction which associates to a braid its link closure.

      \item In the second example of \cref{ex:yokonuma} the generalized braid group $\mathfrak{A}(W^{(1)})$ of $W^{(1)} = (\Z/d\Z)^n \rtimes S_n$ identifies canonically with the \textbf{$d$-modular framed braid group} $(\Z/d\Z)^n \rtimes B_n$ on $n$ strands, where $B_n$ acts on $(\Z/d\Z)^n$ by permutation. The above presentation then relates the representation theory of the Yokonuma Hecke algebra $Y_{d,n}$ to invariants of framed braids and links (see \cite{Jacon-dAndecy}). The special interest in \textit{framed} braids and links arises from the fact \cite{Kirby} that $3$-manifolds are classified up to homeomorphism (or equivalently, up to diffeomorphism) by framed links up to a certain equivalence.
   \end{enumerate}
\end{example}

\subsection{Orientations of Coxeter groups} 
\label{sub:Orientations of Coxeter groups}
The following definition is motivated by \cref{thm:ram}.

\begin{definition}\label{def:or}
   An \textbf{orientation} $\mathfrak{o}$ of a Coxeter group $(W,S)$ is a map
   \[ \mathfrak{o}: W \times S \longrightarrow \{\pm 1\} \]
   satisfying the following two properties:
\item[\textbf{(OR1)}] \label{axiom:OR1} $\mathfrak{o}(ws,s) = -\mathfrak{o}(w,s)$ for all $w \in W$, $s \in S$.
\item[\textbf{(OR2)}] \label{axiom:OR2} If $s,t \in S$ with $m(s,t) < \infty$ and $w \in W$ is arbitrary, then the sequences
         \[ (\mathfrak{o}(w,s),\mathfrak{o}(ws,t),\mathfrak{o}(wst,s),\ldots),\quad (\mathfrak{o}(w,t),\mathfrak{o}(wt,s),\mathfrak{o}(wts,t),\ldots) \]
         are either of the form
         \[ (\underbrace{+,\ldots,+}_k,\underbrace{-,\ldots,-}_{m(s,t)-k}),\quad (\underbrace{-,\ldots,-}_{m(s,t)-k},\underbrace{+,\ldots,+}_k) \]
         or
         \[ (\underbrace{-,\ldots,-}_k,\underbrace{+,\ldots,+}_{m(s,t)-k}),\quad (\underbrace{+,\ldots,+}_{m(s,t)-k},\underbrace{-,\ldots,-}_k) \]
         for some $0 \leq k \leq m(s,t)$.

         The set of all orientations of a Coxeter group $(W,S)$ \textbf{is denoted by} $\mathcal{O}(W,S)$, or simply by $\mathcal{O}$ if the underlying Coxeter group is understood.
\end{definition}
\begin{terminology}\label{term:orient}
   Viewing elements $w \in W$ of Coxeter groups as chambers according to the terminology introduced in \cref{term:basic}, the sign $\mathfrak{o}(w,s)$ should be interpreted geometrically as follows. The sequence $w, ws$ of adjacent chambers forms a gallery that crosses the hyperplane $H = wsw^{-1}$. We will say that $\mathfrak{o}(w,s)$ is the \textit{sign given to this crossing by the orientation} $\mathfrak{o}$, or that it is the \textit{sign attached to crossing} $H$ at $w$ by the orientation $\mathfrak{o}$. The axiom \axiom{OR1} therefore ensures that the sign attached to the opposite crossing $ws,w$ is opposite. 
\end{terminology}
\begin{rmk}\label{rmk:involution-on-orientations}
   \Cref{def:or} is inherently symmetric: to any orientation $\mathfrak{o}: W\times S \rightarrow \{\pm\}$ one can associate its \textit{opposite orientation} $\mathfrak{o}^{\op{op}}: W\times S \rightarrow \{\pm\}$ given by $\mathfrak{o}^{\op{op}}(w,s) = -\mathfrak{o}(w,s)$.
\end{rmk}

\begin{rmk}\label{rmk:or_geom_inter}
   \Cref{def:or} can be interpreted in terms of the (undirected) Cayley graph $\Gamma$ of $(W,S)$. Recall (cf. \cite[Def. 1.73]{AbramenkoBrown}) that $\Gamma$ is the undirected graph with $\op{Vert}(\Gamma) = W$ and $\{w_1,w_2\} \in \op{Edge}(\Gamma)$ iff $w_1^{-1}w_2 \in S$. By \axiom{OR1}, an orientation $\mathfrak{o}$ of $W$ now determines an \textit{orientation} of $\Gamma$ \textit{in the sense of graph theory}, i.e. it determines a directed graph $\Gamma_\mathfrak{o}$ whose underlying undirected graph equals $\Gamma$, if one lets
   \[ (w_1,w_2) \in \op{Edge}(\Gamma_\mathfrak{o}) \quad \Leftrightarrow \quad w_1^{-1}w_2 \in S \wedge \mathfrak{o}(w_1,w_1^{-1}w_2) = +1 \]
   In terms of $\Gamma_\mathfrak{o}$, \axiom{OR2} means that every cycle $\gamma \subseteq \Gamma$ of the form
   \[ \gamma = \{w, ws, wst, wsts, \dots, w(st)^{m(s,t)-1}\},\quad s,t \in S,\ m(s,t) < \infty \]
   is `oriented towards' some vertex $w_0 \in \gamma$, as indicated in \cref{fig:orientation-of-cayley-graph}.

   \begin{figure}[t]
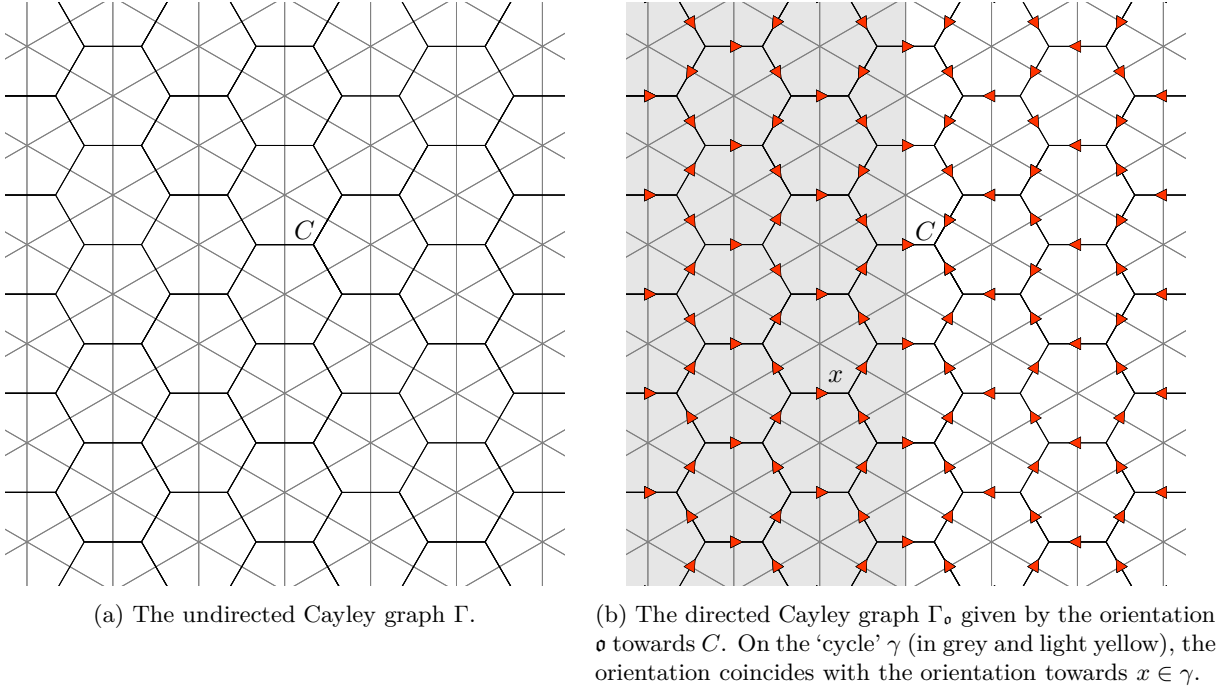

      \centering
      \begin{subfigure}[t]{0.5\textwidth}
         \centering
         \psfrag{C}[cc][bl]{$C$}
         \includegraphics[width=0.9\textwidth]{cayley-graph.eps}
         \caption{The undirected Cayley graph $\Gamma$.}
      \end{subfigure}%
      \begin{subfigure}[t]{0.5\textwidth}
         \centering
         \psfrag{C}[cc][bl]{$C$}
         \psfrag{x}[cc][bl]{$x$}
         \includegraphics[width=0.9\textwidth]{orientations-and-cayley-graph.eps}
         \caption{The directed Cayley graph $\Gamma_\mathfrak{o}$ given by the orientation $\mathfrak{o}$ towards $C$. On the `cycle' $\gamma$ (in grey and light yellow), the orientation coincides with the orientation towards $x \in \gamma$.}
      \end{subfigure}
      \caption{The Coxeter complex of the affine Coxeter group $(W,S)$ of type $\widetilde{A}_2$ and its Cayley graph $\Gamma$. The orientation $\mathfrak{o}$ of $(W,S)$ \textit{towards} the chamber $C$ (\cref{def:chamber_orientations}) determines an orientation of the Cayley graph, giving rise to the directed Cayley graph $\Gamma_\mathfrak{o}$. The condition \axiom{OR2} ensures that restriced to any `cycle' $\gamma \subseteq \mathfrak{o}$, the orientation coincides with the orientation towards a chamber $x \in \gamma$.} 
      \label{fig:orientation-of-cayley-graph}
   \end{figure}
\end{rmk}

\begin{rmk}\label{rmk:or_action}
   There exists a natural right action on the set of all orientations of $(W,S)$. Given an orientation $\mathfrak{o}$ and $w \in W$, it follows easily that the function $\mathfrak{o}\bullet{}w$ defined by
   \[ (\mathfrak{o}\bullet{}w)(w',s) := \mathfrak{o}(ww',s) \]
   is again an orientation of $W$. Moreover, it is clear that this action commutes with the involution $\mathfrak{o} \mapsto \mathfrak{o}^{\op{op}}$.
\end{rmk}

\begin{rmk}\label{rmk:set_of_orientations_is_closed}
   The set $\mathcal{O}(W,S)$ of orientations of a Coxeter group $(W,S)$ naturally carries the structure of a topological space, in fact that of a compact Hausdorff space. Namely, we can view it as a subspace of the mapping space $\{\pm\}^{W\times S}$ endowed with the compact-open topology\footnote{Which can also be viewed as the as the product space $\prod_{(w,s) \in W\times S} \{\pm\}$.}, where $\{\pm\}$ and $W\times S$ are considered discrete. By definition, a basis of the topology on $\{\pm\}^{W\times S}$ is given by
   \[ U_{\{x_i\},\{y_i\}} = \{ f \in \{\pm\}^{W\times S} : f(x_i) = y_i\quad \forall i = 1,\ldots{},n \} \]
   where $\{x_1,\ldots{},x_n\} \subseteq W\times S$ and $\{y_1,\ldots{},y_n\} \subseteq \{\pm\}$ are finite subsets.
   
   It is then easy to see that the set of orientations of $(W,S)$ forms a \textit{closed} (and hence compact) subspace of $\{\pm\}^{W\times S}$, as the conditions \axiom{OR1} and \axiom{OR2} involve only finitely many elements of $W\times S$ at a time. This is of course trivial when $W$ is finite, but this fact will be useful later when we will construct \textit{spherical orientations} of affine Coxeter groups, which are obtained as limits of orientations associated to chambers, in the sense of the definition below.
\end{rmk}

\begin{definition}\label{def:chamber_orientations}
   Given a chamber $w_0 \in W_{\op{aff}}$ let $\mathfrak{o}_{w_0}: W\times S \rightarrow \{\pm\}$ be the map defined by
   \[ \mathfrak{o}_{w_0}(w,s) := \begin{cases}
      +1 & : \ell(w_0^{-1}ws) < \ell(w_0^{-1}w) \\
      -1 & : \ell(w_0^{-1}ws) > \ell(w_0^{-1}w) \end{cases}
   \]
   Then $\mathfrak{o}_{w_0}$ is called the \textbf{orientation associated to the chamber $w_0$} or the \textbf{orientation towards the chamber $w_0$} (cf. \cref{fig:orientation-of-cayley-graph}).
\end{definition}

\begin{rmk}\label{rmk:chamber_orientations}
   The $\mathfrak{o}_{w_0}$ are indeed orientations in the sense of \cref{def:or}. In particular, the set of orientations of a Coxeter group is always non-empty. Indeed, by construction we have $\mathfrak{o}_{ww'} = \mathfrak{o}_{w'} \bullet{} w^{-1}$, so it suffices to verify that $\mathfrak{o}_1$ is an orientation. Obviously condition \axiom{OR1} holds true. An exercise in Coxeter groups \cite[Ch. IV, Exercices \textsection 1, Ex. 2]{Bourbaki} now shows that for any given $w \in W$ and $s,t \in S$ we can always find an element $w_0 \in w\left<s,t\right>$ such that
\[ \ell(w_0) = \ell(w') + \ell(w_0^{-1}w') \]
for every $w' \in w\left<s,t\right>$. So approaching $1$ is the same as moving further away from $w_0$. By \cref{rmk:or_geom_inter}, it follows that $\mathfrak{o}_1$ is an orientation.
\end{rmk}

\begin{rmk}\label{rmk:orientations_away_from}
   The orientation $\mathfrak{o}_w$ defined in \cref{def:chamber_orientations} is not the only orientation naturally attached to an element $w \in W$. One can just as well define an orientation `away from the chamber $w$', which in fact is noneother than the opposite orientation $\mathfrak{o}_w^{\op{op}}$, and so there is no need for a separate definition.
   
   Moreover when $W$ is a finite group, every orientation away from a chamber is in fact also an orientation towards another chamber, namely
   \[ \mathfrak{o}_w^{\op{op}} = \mathfrak{o}_{w_0w} \]
   if $w_0$ is the longest element of $W$. In contrast, for infinite groups orientations towards and away from chambers are disjoint in general (cf. proof of \cref{rmk:w_embeds_discretely_into_orientations}).
\end{rmk}

\begin{rmk}\label{rmk:w_embeds_discretely_into_orientations}
   For a Coxeter group $(W,S)$, the map
   \begin{align*} W & \longrightarrow \mathcal{O}(W,S) \\
      w & \longmapsto \mathfrak{o}_w
   \end{align*}
   is \textit{injective}. Moreover if $S$ is finite, then $W$ is \textit{discrete} as a subset of $\mathcal{O}(W,S)$. In fact, in this case
   \[ W \cup W^{\op{op}} = \{ \mathfrak{o}_w, \mathfrak{o}_w^{\op{op}} : w \in W \} \subseteq \mathcal{O} \]
   is discrete. In particular for an infinite Coxeter group $(W,S)$ with $\# S < \infty$, the set
   \[ \mathcal{O}_{\op{boundary}} := \overline{W\cup W^{\op{op}}} - (W\cup W^{\op{op}}) \subseteq \mathcal{O} \]
   of \textbf{boundary orientations} of $W$ is non-empty.
   \begin{proof}
      The element $w \in W$ can be recovered from the orientation $\mathfrak{o}_w$ as the unique element $w'$ satisfying
      \[ \mathfrak{o}_w(w',s) = -1 \]
      for all $s \in S$, which shows the injectiviy of the map. Moreover if $S = \{s_1,\dots,s_n\}$ is finite, this also shows that
      \[ U_{\{x_i\},\{y_i\}}\cap W = \{\mathfrak{o}_w\},\quad x_i = (w,s_i), y_i = -1 \]
      and therefore that $W$ is discrete as a subset of $\mathcal{O}(W,S)$. Furthermore, if the above neighbourhood $U_{\{x_i\},\{y_i\}}$ contains $\mathfrak{o}_{w_0}^{\op{op}}$ for some $w_0 \in W$, then we would have
      \[ \mathfrak{o}_{w_0}(w,s) = \ell(w_0^{-1}w)-\ell(w_0^{-1}ws) = 1 \]
      for all $s \in S$. This implies (cf. \cite[Ch. IV, \textsection 1, exerc. 22b]{Bourbaki}) that $w_0^{-1}w$ is a longest element of $W$; in particular, the length on $W$ is bounded. Since $S$ is finite, it follows that $W$ must be finite and so the space $\mathcal{O}$ is finite and discrete itself, and there is nothing to show. 

      Finally if $W$ is infinite and $\# S < \infty$, the set $W \cup W^{\op{op}}$ is discrete and infinite, and therefore its (compact) closure must be a proper superset.
   \end{proof}
\end{rmk}

\begin{rmk}\label{rmk:boundary_orientations}
   \begin{enumerate}
      \item The above remark gives an abstract proof (relying on Tychonoff's theorem) that for an infinite Coxeter groups $W = (W,S)$ with $\# S < \infty$ the (compact, $W$-invariant, and $\-^{\op{op}}$-invariant) set
   \[ \mathcal{O}_{\op{boundary}} = \overline{W \cup W^{\op{op}}} - (W\cup W^{\op{op}}) \subseteq \mathcal{O} \]
   is non-empty. In \cref{sub:Spherical orientations}, we will construct some concrete examples of elements $\mathfrak{o} \in \mathcal{O}_{\op{boundary}}$ for affine Coxeter groups, the \textit{spherical orientations} that lie at the heart of Bernstein-Zelevinsky method. What makes these orientations useful for the study of Hecke algebras is the fact that they have a very large stabilizer (the commutative subgroup $X \leq W$ of translations) under the action of $W$, and via the unnormalized and normalized Bernstein maps $\theta$, $\widetilde{\theta}$ (\cref{def:unnormalized_bernstein}, \cref{def:normalized_bernstein_map}) therefore give rise to embeddings (cf. \cref{prop:change_of_basis})
   \[ k[\op{Stab}_W(\mathfrak{o})] \hookrightarrow \mathcal{H}^{(1)} \]
   of the corresponding group algebra, given by $w \mapsto \theta_\mathfrak{o}(w)$ and $w \mapsto \widetilde{\theta}_{\mathfrak{o}}(w)$ respectively. Altough these are the only boundary orientations that we will be concerned with, there exist many more (infinitely many) such orientations for affine Coxeter groups\footnote{The set of boundary orientations can easily be worked out for the group of type $\widetilde{A_2}$; apart from the six spherical orientations, it contains countably many orientations all of whose stabilizers are subgroups of $X$ of rank $1$.}.

\item The set $\mathcal{O}_{\op{boundary}}$ seems to be particularly interesting in the case of hyperbolic\footnote{There are several inequivalent definitions of the term \textit{hyperbolic Coxeter groups}, see \cite[10.4]{AbramenkoBrown}.} Coxeter groups. In particular, there seems to be a rich supply of orientations having non-trivial stabilizer, although it is not clear at the moment whether the corresponding subalgebras $k[\op{Stab}_W(\mathfrak{o})] \subseteq \mathcal{H}^{(1)}$ yield any useful information about the structure of the Hecke algebras $\mathcal{H}^{(1)}$ attached to $W$. The set $\mathcal{O}_{\op{boundary}}$ also appears to be somewhat related to the \textit{Gromov boundary} $\partial (W,S)$ of $W$ (see \cite[12.4]{Davis}).
 
   To illustrate the richness of $\mathcal{O}_{\op{boundary}}$ in the hyperbolic case, let us consider the example of the group $W = \op{PGL}_2(\Z) = \op{GL}_2(\Z)/\{\pm\}$ of invertible 2x2 integer matrices modulo center. It becomes a Coxeter group via the distinguished generating set $S = \{s_1, s_2, s_3\}$ defined by
   \[ s_1 = \begin{bmatrix} & 1 \\ 1 & \end{bmatrix},\ s_2 = \begin{bmatrix} -1 & 1 \\ & 1\end{bmatrix},\ s_3 = \begin{bmatrix} -1 & \\ & 1\end{bmatrix} \]
      with orders $m(s_i,s_j) = \op{ord}(s_is_j)$ given by
   \[ m(s_1, s_2) = 3,\ m(s_1, s_3) = 2,\ m(s_2,s_3) = \infty \]
   Here
   \[ \begin{bmatrix} a & b \\ c & d\end{bmatrix} := \left(\op{GL}_2(\Z) \twoheadrightarrow \op{PGL}_2(\Z)\right)\left(\begin{pmatrix}a & b \\ c & d\end{pmatrix}\right) \]
      denotes the image of a matrix in $\op{GL}_2(\Z)$ under the projection $\op{GL}_2(\Z) \twoheadrightarrow \op{PGL}_2(\Z)$. The group $\op{PGL}_2(\Z)$ is also a hyperbolic reflection group in the sense of Vinberg (cf. \cite[Introduction]{Vinberg}), i.e. it's a discrete subgroup of the group of isometries of the hyperbolic plane $\bbf{H}^2$ generated by reflections at hyperplanes (totally geodesic codimension one submanifolds) in $\bbf{H}^2$. In fact, such a representation of $\op{PGL}_2(\Z)$ is afforded by its canonical action on the upper half-plane
   \[ \mathfrak{H} := \{ z \in \C : \Im z > 0 \} \]
   considered as a model of $\bbf{H}^2$ with metric $g(x+iy) = \frac{1}{y} dx\otimes dy$, via fractional linear transformations
   \[ \begin{bmatrix} a & b \\ c & d \end{bmatrix}\bullet{} z = \begin{dcases}
      \frac{az + b}{cz + d} &\text{ if }ad - bc = 1 \\[1em]
      \frac{a\overline{z} + b}{c\overline{z} + d} &\text{ if } ad - bc = -1
   \end{dcases} \]
   The generators $s_1,s_2,s_3$ act as the reflections at the hyperplanes
   \[ H_1 = \{ z \in \mathfrak{H} : |z| = 1 \},\quad H_2 = \{ z \in \mathfrak{H} : \Re z = \frac{1}{2} \},\quad H_3 = \{ z \in \mathfrak{H} : \Re z = 0 \} \]
   bounding the fundamental polytope
   \[ C := \{ z \in \mathfrak{H} : |z| > 1,\ 0 < \Re z < \frac{1}{2} \} \]
   To describe the boundary representations of $\op{PGL}_2(\Z)$, it is useful to extend the hyperbolic plane by its natural boundary, replacing the upper half-plane by the extended upper half-plane
   \[ \overline{\mathfrak{H}} := \mathfrak{H} \cup \bbf{P}^1(\R) = \{ z \in \C : \Im z \geq 0 \} \cup \{\infty\} \]
   considered as a closed subset of the Riemann sphere $\bbf{P}^1(\C)$. The boundary orientations $\mathfrak{o} \in \mathcal{O}_{\op{boundary}}$ we want to describe are attached to actual boundary points $x \in \bbf{P}^1(\R)$, but to certain points corresponds more than one orientation. A precise statement is that there is a $W$-equivariant correspondence
   
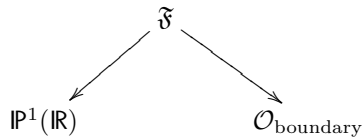
\[ \begin{xy} \xymatrix{ & \ar[ld] \mathfrak{F} \ar[rd] & \\
      \bbf{P}^1(\R) & & \mathcal{O}_{\op{boundary}}
      } \end{xy} \]
      defined as follows. Since this construction is in part completely general, let $(W,S)$ be an arbitrary Coxeter group for the moment. The set $\mathfrak{F}$ is the quotient $\mathfrak{F}/_\sim$ of the set
      \[ \Gamma := \{ (w_n)_{n \in \N} : \forall n\ w_n \in W,\ w_n^{-1} w_{n+1} \in S,\ \ell(w_0^{-1}w_n) = n \} \]
      of (semi-)infinite reduced galleries (carrying a natural $W$-action via $w\bullet{}(w_n)_{n \in \N} = (ww_n)_{n \in \N}$) by the ($W$-invariant) equivalence relation $\sim$ on $\Gamma$ characterized uniquely by requiring
      \begin{equation}\label{eq:equiv-rel-on-gamma0} w_0 = w'_0 \quad \Rightarrow \quad (w_n)_{n \in \N} \sim (w'_n)_{n \in \N}\ \Leftrightarrow\ \forall m \exists n\ n\geq m \wedge w_n = w'_n \end{equation}
      and
      \begin{equation} \forall m, m' \in \N\quad (w_n)_{n \in \N} \sim (w'_n)_{n \in \N} \Leftrightarrow (w_{n+m})_{n \in \N} \sim (w'_{n+m'})_{n \in \N} \end{equation}
         for all $(w_n)_{n \in \N}, (w'_n)_{n \in \N} \in \Gamma$. Moreover, there is a natural $W$-equivariant map
         \[ \mathfrak{F} \longrightarrow \mathcal{O}_{\op{boundary}},\quad [(w_n)_{n \in \N}] \mapsto \lim_{n \in \N} \mathfrak{o}_{w_n} \]
         where the left action on $\mathcal{O}_{\op{boundary}}$ is given in terms of the natural right action as $w\bullet{}\mathfrak{o} := \mathfrak{o}\bullet{}w^{-1}$. The equivariance then follows from the formula $\mathfrak{o}_{w'}\bullet{}w^{-1} = \mathfrak{o}_{ww'}$ (cf. \cref{rmk:chamber_orientations}). The limit $\lim_{n \in \N} \mathfrak{o}_{w_n}$ exists (a priori only as an element of $\{\pm\}^{W\times S}$, but by \cref{rmk:set_of_orientations_is_closed} also as an element of $\mathcal{O}$), as it does exist for any infinite gallery $(w_n)_{n \in \N}$ that crosses every hyperplane only finitely many times. The limit orientation $\mathfrak{o} = \lim_{n \in \N} \mathfrak{o}_{w_n}$ must lie in $\mathcal{O}_{\op{boundary}}$ because
         \[ \mathfrak{o}(w_n,w_n^{-1}w_{n+1}) = +1 \]
         for all $n$ by construction, which would be impossible if $\mathfrak{o}$ were of the form $\mathfrak{o} = \mathfrak{o}_w$ or $\mathfrak{o} = \mathfrak{o}_w^{\op{op}}$.
         
         The set $\mathfrak{F}$ can be described a little more explicitly (at the price of making the $W$-action more complicated) as follows. The embedding $\Gamma_0 \subseteq \Gamma$ of the subset of infinite reduced galleries starting in $w_0 = 1$ induces a bijection $\Gamma_0/_\sim \simeq \Gamma/_\sim$ of the quotient of $\Gamma_0$ by the equivalence relation defined by \cref{eq:equiv-rel-on-gamma0} with $\mathfrak{F}$, because given any $(w_n)_{n \in \N} \in \Gamma$, if $m \in \N$ is such that the subgallery $w_m,w_{m+1},\dots$ does not cross any of the (finitely many) hyperplanes separating $w_0$ and $1$, then $(w'_n)_{n \in \N}$ defined by
         \[ w'_n := \begin{cases}
            w''_n &\text{ if } n <= r \\
            w_{n-r+m} &\text{ if } n > r
         \end{cases} \]
         is an element of $\Gamma_0$ equivalent to $(w_n)_{n \in \N}$, for any reduced gallery $w_0 = w''_0,\dots,w''_r = 1$ from $w_0$ to $1$.
         
         Let now be $(W,S) = (\op{PGL}_2(\Z),\{s_1,s_2,s_3\})$ again, then one can make $\mathfrak{F}$ even more explicit. Indeed in this case, a complete system of representatives for $\Gamma_0/_\sim$ is given by galleries corresponding to the infinite formal words in the generators of the form
         \[ (s_2s_3)^{a_0}s_1 (s_2s_3)^{a_1} s_1 (s_2s_3)^{a_2} s_1 \dots\quad \text{ and }\quad (s_2s_3)^{a_0}s_1 \dots (s_1s_2)^{a_r} s_1 (s_2 s_3)^{\pm \infty} \]
         where $a_i \in \Z$, subject to the condition that $\forall i\ a_i \geq 0$ or $\forall i\ a_i \leq 0$, and $a_i \neq 0$ for all $i > 0$. Here, the expressions $(s_2 s_3)^{-\infty}$ and $(s_2 s_3)^{+\infty}$ are to be understood as $s_3 s_2 s_3 s_2 \dots$ and $s_2 s_3 s_2 s_3 \dots$ respectively. We can identify these expressions with formal continued fraction as
         \[ [a_0,a_1,a_2,\dots ] \quad \text{ and }\quad [a_0,\dots,a_r,\pm\infty] \]
         The map
         \[ \mathfrak{F} \longrightarrow \bbf{P}^1(\R) \]
         is then given by evaluation of formal fractions, sending $[a_0,a_1,a_2,\dots]$ to
         \[ [a_0,a_1,\dots ] := \lim_{n \rightarrow \infty} [a_0,\dots,a_n] \]
         where $[a_0,\dots,a_n] \in \Q$ is defined recursively by
         \[ [a_0,\dots,a_n,a_{n+1}] := [a_0,\dots,a_n+\frac{1}{a_{n+1}}],\quad [a_0] := \frac{1}{a_0} \]
         as usual. The $W$-equivariance of the map $\mathfrak{F} \longrightarrow \bbf{P}^1(\R)$ can most easily be verified by establishing that the value of a class $[(w_n)_{n \in \N}]$ is given by the limit $\lim_{n \rightarrow \infty} w_n \bullet{} z$, independent of the choice of a $z \in \overline{\mathfrak{H}}$. From this it also follows that the point $x \in \bbf{P}^1(\R)$ and the orientation $\mathfrak{o} \in \mathcal{O}_{\op{boundary}}$ defined by an element of $\mathfrak{F}$ satisfy
         \begin{equation}\label{eq:stabilizer-inclusion} \op{Stab}_{\op{PGL}_2(\Z)}(\mathfrak{o}) \subseteq \op{Stab}_{\op{PGL}_2(\Z)}(x) \end{equation}
            From the theory of continued fractions it follows that the map $\mathfrak{F} \rightarrow \bbf{P}^1(\R)$ is surjective and that the infinite formal continued fractions $[a_0,a_1,a_2,\dots]$ map bijectively onto $\R-\Q$, while the finite ones map many to one to $\bbf{P}^1(\Q)$, with $[a_0,\dots,a_r,-\infty], [a_0,\dots,a_r,+\infty]$ both mapping to $[a_0,\dots,a_r] \in \bbf{P}^1(\Q)$. More precisely, $\infty \in \bbf{P}^1(\Q)$ has the two preimages $[-\infty], [+\infty]$ ($r = -1$), whereas every $x \in \Q$ has four preimages because it can be expressed as two distinct continued fractions, due to the identity
         \[ [a_0,\dots,a_r,1] = [a_0,\dots,a_r+1] \]

         The orientation $\mathfrak{o} = \lim_n \mathfrak{o}_{w_n}$ defined by a $(w_n)_{n \in \N} \in \Gamma$ can be described concretely in terms of the corresponding point $x \in \bbf{P}^1(\R)$ as follows. The set $\mathfrak{H} = \{ wsw^{-1} : w \in W,\ s \in S\}$ of formal hyperplanes can be indentified as a $W$-set with the set $\{ w \bullet{} H_1, w\bullet{} H_2, w\bullet{} H_3 : w \in W \}$ of $W$-conjugates of the hyperplanes in $\bbf{H}^2 = \mathfrak{H}$ bounding the fundamental polytope. If $H = wsw^{-1}$ corresponds to a hyperplane $H \subseteq \mathfrak{H}$ such that $x \not\in \overline{H}$, then
         \begin{equation}\label{eq:orientation-attached-to-boundary-point} \mathfrak{o}(w,s) = +1 \quad \Leftrightarrow \quad \text{ $w\bullet{}C$ and $x$ lie in different connected components of $\overline{\mathfrak{H}} - \overline{H}$} \end{equation}
            The condition $x \not\in \overline{H}$ is always satisfied if $x \in \R-\Q$, since the endpoints of the hyperplanes $H \in \mathfrak{H}$ on $\bbf{P}^1(\R)$ lie in $\bbf{P}^1(\Q)$, and the orientation attached to $x$ is then uniquely and explicitly determined by \cref{eq:orientation-attached-to-boundary-point}. It follows easily that in this case the inclusion in \cref{eq:stabilizer-inclusion} is an equality. Since the stabilizer $\op{Stab}_{\op{PGL}_2(\Z)}(x)$ is non-trivial precisely when $x \in \bbf{P}^1(\Q)$ or $x$ is a quadratic irrational number (by a classical exercise), the orientation $\mathfrak{o}_x$ attached to an irrational number $x \in \R-\Q$ has a non-trivial stabilizer if and only if it is quadratic irrational.

            The two orientations attached to rational boundary points $x \in \bbf{P}^1(\Q)$ also have a non-trivial stabilizer, but the inclusion in \cref{eq:stabilizer-inclusion} is proper in this case. In fact, since
         \[ x \in \overline{H} \quad \Leftrightarrow \quad s_H = wsw^{-1} \in \op{Stab}_{\op{PGL}_2(\Z)}(x) \]
         we have that $s_2 = s_{H_2}, s_3 = s_{H_3} \in \op{Stab}_{\op{PGL}_2(\Z)}(\infty)$, but $s_2$ and $s_3$ both interchange the two orientations attached to $x = \infty$ (cf. \cref{fig:orientation-of-pgl2(z)}). The stabilizers of these orientations are instead equal to the subgroup
         \[ \left< s_2 s_3\right> = \left\{ \begin{bmatrix} 1 & z \\ & 1\end{bmatrix} : z \in \Z \right\} \]
     
   \begin{figure}
      \centering
      \psfrag{C-20}[cc][bl][0.04952192]{$C_{\text{-20}}$}
      \psfrag{C-19}[cc][bl][0.05472326]{$C_{\text{-19}}$}
      \psfrag{C-18}[cc][bl][0.060770065]{$C_{\text{-18}}$}
      \psfrag{C-17}[cc][bl][0.06785093]{$C_{\text{-17}}$}
      \psfrag{C-16}[cc][bl][0.07620831]{$C_{\text{-16}}$}
      \psfrag{C-15}[cc][bl][0.08615727]{$C_{\text{-15}}$}
      \psfrag{C-14}[cc][bl][0.09811181]{$C_{\text{-14}}$}
      \psfrag{C-13}[cc][bl][0.11262226]{$C_{\text{-13}}$}
      \psfrag{C-12}[cc][bl][0.13042839]{$C_{\text{-12}}$}
      \psfrag{C-11}[cc][bl][0.15253524]{$C_{\text{-11}}$}
      \psfrag{C-10}[cc][bl][0.18031995]{$C_{\text{-10}}$}
      \psfrag{C-9}[cc][bl][0.2156776]{$C_{\text{-9}}$}
      \psfrag{C-8}[cc][bl][0.26120445]{$C_{\text{-8}}$}
      \psfrag{C-7}[cc][bl][0.3203789]{$C_{\text{-7}}$}
      \psfrag{C-6}[cc][bl][0.39758182]{$C_{\text{-6}}$}
      \psfrag{C-5}[cc][bl][0.49748185]{$C_{\text{-5}}$}
      \psfrag{C-4}[cc][bl][0.62264264]{$C_{\text{-4}}$}
      \psfrag{C-3}[cc][bl][0.7674546]{$C_{\text{-3}}$}
      \psfrag{C-2}[cc][bl][0.9082839]{$C_{\text{-2}}$}
      \psfrag{C-1}[cc][bl][1.0]{$C_{\text{-1}}$}
      \psfrag{C0}[cc][bl][1.0]{$C_{0}$}
      \psfrag{C1}[cc][bl][0.9082839]{$C_{1}$}
      \psfrag{C2}[cc][bl][0.7674546]{$C_{2}$}
      \psfrag{C3}[cc][bl][0.62264264]{$C_{3}$}
      \psfrag{C4}[cc][bl][0.49748185]{$C_{4}$}
      \psfrag{C5}[cc][bl][0.39758182]{$C_{5}$}
      \psfrag{C6}[cc][bl][0.3203789]{$C_{6}$}
      \psfrag{C7}[cc][bl][0.26120445]{$C_{7}$}
      \psfrag{C8}[cc][bl][0.2156776]{$C_{8}$}
      \psfrag{C9}[cc][bl][0.18031995]{$C_{9}$}
      \psfrag{C10}[cc][bl][0.15253524]{$C_{10}$}
      \psfrag{C11}[cc][bl][0.13042839]{$C_{11}$}
      \psfrag{C12}[cc][bl][0.11262226]{$C_{12}$}
      \psfrag{C13}[cc][bl][0.09811181]{$C_{13}$}
      \psfrag{C14}[cc][bl][0.08615727]{$C_{14}$}
      \psfrag{C15}[cc][bl][0.07620831]{$C_{15}$}
      \psfrag{C16}[cc][bl][0.06785093]{$C_{16}$}
      \psfrag{C17}[cc][bl][0.060770065]{$C_{17}$}
      \psfrag{C18}[cc][bl][0.05472326]{$C_{18}$}
      \psfrag{C19}[cc][bl][0.04952192]{$C_{19}$}
      \psfrag{C20}[cc][bl][0.045017857]{$C_{20}$}
      \psfrag{1}[cl][bl]{\mybox{$1$}}
      \psfrag{i}[bc][bl]{\mybox{$i$}}
      \psfrag{-3/5-i4/5}[tr][bl]{$-\frac{3}{5}-i\frac{4}{5}$}
      \includegraphics[width=0.5\linewidth]{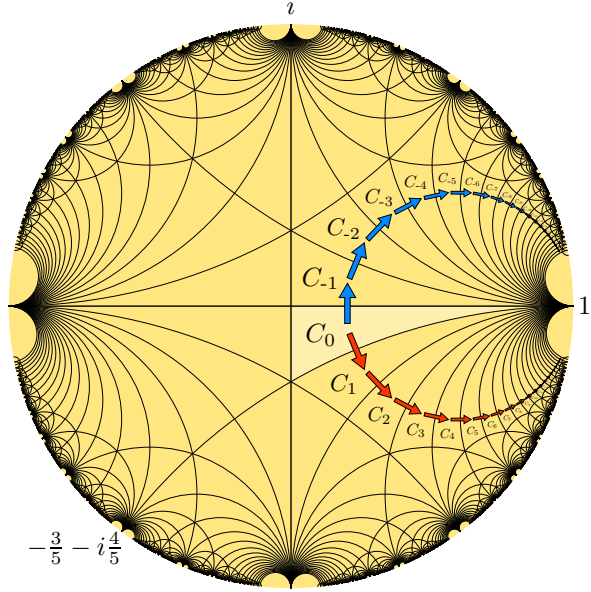}
      \caption{The canonical hyperplane arrangement realizing the Coxeter group $W = \op{PGL}_2(\Z)$ as a hyperbolic reflection group, viewed in the disk model (isometric to the upper half-plane $\mathfrak{H}$ via the Cayley transform $q = \frac{z-i}{z+i}$) of the hyperbolic plane. The two orientations of $W$ attached to the `cusp at infinity' $q = 1 \in \overline{\Delta}$ are the limits $\lim_{n \rightarrow -\infty} \mathfrak{o}_{C_n}$, $\lim_{n \rightarrow +\infty} \mathfrak{o}_{C_n}$ attached to the two semi-infinite galleries contained in the `horocycle' $(C_n)_{n \in \Z}$ and starting in the fundamental polytope $C_0$.}
      \label{fig:orientation-of-pgl2(z)}
   \end{figure}
\end{enumerate}
\end{rmk}

\begin{rmk}\label{rmk:goertz_orientations} 
   We will show later in \cref{lem:or_depend_on_half-space} that orientations $\mathfrak{o} \in \mathcal{O}$ can also be viewed as choosing for every hyperplane $H \in \mathfrak{H}$ a `positive' half-space $U_{\mathfrak{o},H}^+ \in \{U_H^+, U_H^-\}$ such that $\mathfrak{o}(w,s) = 1$ iff $ws \in U_{\mathfrak{o},H}^+$ where $H = wsw^{-1}$. In view of this, it follows easily from unwinding definitions that the union
   \[ \mathcal{O}_G := \overline{W} \cup \overline{W^{\op{op}}} = \overline{W \cup W^{\op{op}}} \subseteq \mathcal{O} \]
   identifies exactly with the \textit{root hyperplane orientations} defined in \cite[Def. 2.3.1]{Goertz}, i.e. those orientations $\mathfrak{o}$ having the property that all finite intersections
   \[ U_{\mathfrak{o},H_1}^+ \cap \dots \cap U_{\mathfrak{o},H_n}^+ \neq \emptyset \]
   of positive half-spaces with respect to $\mathfrak{o}$ are non-empty, or that all finite intersections
   \[ U_{\mathfrak{o},H_1}^- \cap \dots \cap U_{\mathfrak{o},H_n}^- \neq \emptyset \]
   of the corresponding negative half-space are non-empty. 
   The example of the infinite dihedral group (the free group on generators $s,t$; $m(s,t) = \infty$) shows that the inclusion
   \[ \mathcal{O}_G \subseteq \mathcal{O} \]
   is proper in general.
\end{rmk}

As a complement to \cref{rmk:w_embeds_discretely_into_orientations}, the following lemma shows that for \textit{finite} Coxeter groups $(W,S)$ chamber orientations \textit{do} exhaust all possibilities. This lemma is not used in our results concerning affine pro-$p$ Hecke algebras and may therefore be safely skipped.

\begin{lemma}
   Let $\mathfrak{o}$ be an orientation of $(W,S)$, and suppose that $\# W < \infty$. Then
   \[ \mathfrak{o} = \mathfrak{o}_w \]
   for some $w \in W$.
\end{lemma}
\begin{proof}
   Let us begin with a general observation. Given any (not necessarily finite) Coxeter group $(W,S)$ and an orientation $\mathfrak{o}$, we can construct a function
   \[ \phi_{\mathfrak{o}}: W \longrightarrow \Z \]
   as follows. Given a $w \in W$, let $w = s_1\ldots{}s_r$ be any expression as a product of generators, and put
   \[ \phi_{\mathfrak{o}}(w) = \sum_{i=1}^r \mathfrak{o}(s_1\ldots{}s_{i-1},s_i) \]
   In other words, $\phi_{\mathfrak{o}}(w)$ is the sum of the signs that $\mathfrak{o}$ associates to the gallery $\Gamma = (1,s_1,s_1 s_2, \ldots{}, w)$ from $1$ to $w$. We need to see that this sum is well-defined indepent of the choice of $\Gamma$.
   
   By Tits' solution of the word problem for Coxeter groups \cite[II.3C]{Brown}, any two expressions of $w$ as a product of generators are related by a sequence of transformations of the following type.
   \begin{enumerate}
      \item[(I)] $s_1\ldots{}s_i s s s_{i+1} \ldots s_r \mapsto s_1\ldots{}s_i s_{i+1} \ldots s_r$
      \item[(II)] $s_1\ldots{}s_i s_{i+1} \ldots s_r \mapsto s_1\ldots{}s_i s s s_{i+1} \ldots s_r$
      \item[(III)] $s_1 \ldots{}s_i \underbrace{sts\ldots }_{m(s,t)} s_{i+1} \ldots s_r \mapsto s_1 \ldots s_i \underbrace{tst \ldots }_{m(s,t)} s_{i+1} \ldots s_r$ if $m(s,t) < \infty$.   \end{enumerate}
   Because \axiom{OR1} guarantees the invariance under the first two transformations and \axiom{OR2} guarantees the invariance under the third, it therefore follows that $\phi_{\mathfrak{o}}(w)$ is well-defined. Moreover, it is immediate from the definitions that
   \[ \phi_{\mathfrak{o}}(ww') = \phi_{\mathfrak{o}}(w) + \phi_{\mathfrak{o}\bullet{}w}(w') \]
   which we can also write as
   \begin{equation}\label{eq:action-on-phi} \phi_{\mathfrak{o}\bullet{}w} = \phi_{\mathfrak{o}}\bullet{}w - \phi_{\mathfrak{o}}(w) \end{equation}

   For orientations $\mathfrak{o}$ of the form $\mathfrak{o} = \mathfrak{o}_w$, the function $\phi_{\mathfrak{o}}$ is easily seen to be given by
   \[ \phi_{\mathfrak{o}}(w') = \ell(w)-\ell(w^{-1}w') \]
   Conversely, if $\phi_{\mathfrak{o}}$ is of the above form, it follows that $\mathfrak{o} = \mathfrak{o}_w$, and in this case $w$ is determined as the unique element $w' \in W$ at which $\phi_{\mathfrak{o}}$ attains its global maximum.

   Let us now assume that $W$ is finite, and let $w$ be such that $\phi_{\mathfrak{o}}(w)$ is maximal. Using \eqref{eq:action-on-phi} and replacing $\mathfrak{o}$ by $\mathfrak{o}\bullet{}w$, we may assume that $w = 1$. In order to show that $\mathfrak{o} = \mathfrak{o}_1$, it suffices by the above remark to prove that
   \[ \phi_{\mathfrak{o}}(w) = -\ell(w) \]
   or equivalently, to prove that $\phi_{\mathfrak{o}}$ is monotously decreasing along geodesics, i.e. to prove that for every \textit{reduced} expression $s_1\ldots{}s_r$ the sequence
   \begin{equation}\label{eq:phi-seq} \phi_{\mathfrak{o}}(w_0),\phi_{\mathfrak{o}}(w_1),\ldots{},\phi_{\mathfrak{o}}(w_r) \quad \text{with}\quad w_i = s_1\ldots{}s_i \end{equation}
      is (strictly) decreasing (note that two consecutive elements of the above sequence differ by $\pm 1$).
      
      We prove this using induction over $r$. For $r = 1$, this follows from the fact that $\phi$ has its (a priori not unique) global maximum at $w = 1$. Let now $r \geq 2$, and assume that the claim holds for sequences of length $< r$. In particular
      \[ \phi_{\mathfrak{o}}(w_i) = -i \quad\text{for $i < r$} \]
      Suppose that we had $\phi_{\mathfrak{o}}(w_r) > \phi_{\mathfrak{o}}(w_{r-1})$, i.e. $\phi_{\mathfrak{o}}(w_r) = -(r-2)$, and put $s = s_{r-1}$, $t = s_r$. We then have the following situtation
      \[ \begin{xy} \xymatrix@R-2pc{ & s & & t &  \\
         \phi_{\mathfrak{o}}(w_{r-2}) = -(r-2) & \mid & \phi_{\mathfrak{o}}(w_{r-1}) = -(r-1) & \mid & \phi_{\mathfrak{o}}(w_r) = -(r-2) } \end{xy} \]
         By \cref{rmk:or_geom_inter}, the restriction of $\mathfrak{o}$ to the `loop' $w_{r-2} \cdot{} \left< s,t\right>$ is given by the distance to a chamber, and therefore $\phi_{\mathfrak{o}}$ attains precisely one local minimum there. Thus, this minimum is attained at $w_{r-1} = w_{r-2}s$ and for all $k \leq m(s,t)-1$ we have that
      \begin{equation}\label{eq:phi-values} \phi_{\mathfrak{o}}(w_{r-2} \underbrace{tst\ldots }_k ) = -(r-2)+k \end{equation}
         Note that $m(s,t) < \infty$ because $W$ is finite, and $m(s,t) \geq 2$ because $s = t$ would contradict the reducedness of $s_1\ldots{}s_r$. In particular, $\phi_{\mathfrak{o}}(w_{r-2}t) = -(r-3)$; because of our induction hypothesis, it follows that the expression $s_1\ldots{}s_{r-2}t$ must be reducible, yielding an immediate contradiction if $r = 2$. If $r \geq 3$, we can apply the \textit{deletion condition} (see \cref{rmk:equivalent-definitions-of-coxeter-groups}) and the reducedness of the expression $s_1\ldots{}s_{r-2}$ to conclude that
         \[ w_{r-2} t = s_1\ldots{}s_{r-2}t = s_1\ldots{}\widehat{s_j}\ldots{}s_{r-2} \]
         for some $1 \leq j \leq r-2$. This subsequence $s_1 \ldots{} \widehat{s_j} \ldots s_{r-2}$ of length $r-3$ is again reduced, and its associated sequence of values of $\phi_{\mathfrak{o}}$ is again strictly decreasing (we don't need to use the induction hypothesis for this; this already follows from the fact that $\phi_{\mathfrak{o}}(1) = 0$ and $\phi_{\mathfrak{o}}(s_1\ldots{}\widehat{s_j}\ldots{}s_{r-2}) = -(r-3)$).

         We can therefore repeat the above argument with the expression $s_1\ldots{}s_{r-2}$ replaced with $s_1\ldots{}\widehat{s_j}\ldots{}s_{r-2}$, using equation \eqref{eq:phi-values} for $k = 2$ and the induction hypothesis to conclude that $s_1\ldots{}\widehat{s_j}\ldots{}s_{r-2}s$ is reducible. We can keep iterating this argument as long as we are able to apply \eqref{eq:phi-values}, that is, applying this argument $k$ times we end up with an equation
         \[ w_{r-2} \underbrace{tst\ldots}_k = s_{j_1}\ldots{}s_{j_{r-2-k}} \]
         for some sequence $1 \leq j_1 < \ldots < j_{r-2-k} \leq r-2$, such that either $k = r-2 < m(s,t)-1$ and the product on the right hand side is empty, or $k = m(s,t)-1$. In the first case, we would have
         \[ w_{r-2} = \underbrace{\ldots tst}_{r-2} \]
         Again, using that the restriction of $\mathfrak{o}$ to the loop $\left<s,t\right>$ of length $m(s,t)$ is given by the distance to a chamber, and that the restriction of $\phi_{\mathfrak{o}}$ to this loop therefore has a unique local minimum and a unique local maximum, both of which are lying opposite to each other, it follows that the maximum must be attained at $w = 1$ (!) and that the minimum must be attained at $w = w_{r-1}$. In particular, $r-1 = m(s,t)$ which is a contradiction.

         In the second case, we would have a reduced (!) expression
         \[ w_{r-2} = s_{j_1}\ldots{}s_{j_{r-2-k}} \underbrace{\ldots tst}_{m(s,t)-1} \]
         Since $w_r = w_{r-2}st$ and $\ell(w_r) = w_{r-2}+2$ by assumption, the expression
         \[ s_{j_1}\ldots{}s_{j_{r-2-k}} \underbrace{\ldots tst}_{m(s,t)-1} st \]
         would be reduced. But already the subexpression
         \[ \underbrace{\ldots tst}_{m(s,t)-1} st \]
         is reducible, yielding a contradiction.
\end{proof}

We will now extend the notion of an orientation to extended and pro-$p$ Coxeter groups. The extension from extended to pro-$p$ Coxeter groups is trivial, but the extension from Coxeter to extended Coxeter groups is a bit subtle because of the action of $\Omega$.

\begin{definition}\label{def:orientation_of_extended_Coxeter_group}
   Let $W$ be an extended Coxeter group and $\mathfrak{o}$ be an orientation of $W_{\op{aff}}$. Then the map
   \[ \widetilde{\mathfrak{o}}: W \times S \longrightarrow \{\pm 1\} \]
   given by $\widetilde{\mathfrak{o}}(wu,s) := \mathfrak{o}(w,u(s))$, $w \in W_{\op{aff}}$, $u \in \Omega$ is called the \textbf{orientation of $W$ associated to} $\mathfrak{o}$.

   A map $\mathfrak{o}: W \times S \longrightarrow \{\pm 1\}$ is called an \textbf{orientation} if it is associated to an orientation of $W_{\op{aff}}$ in the above sense, and the set of all such orientations is denoted by $\mathcal{O}(W)$, or simply by $\mathcal{O}$ if $W$ is understood.
\end{definition}

\begin{rmk}
   There exists a natural right action of $\Omega$ on the set $\mathcal{O}(W_{\op{aff}},S)$ of all orientations of $W_{\op{aff}}$. Given an orientation $\mathfrak{o} \in \mathcal{O}(W_{\op{aff}},S)$ and $u \in \Omega$
   \[ (\mathfrak{o}\bullet{}u)(w,s) := \mathfrak{o}(uwu^{-1},u(s)) \]
   again defines an orientation of $W_{\op{aff}}$. On the other hand by \cref{rmk:or_action}, we also have a natural right action of $W_{\op{aff}}$ on $\mathcal{O}(W_{\op{aff}},S)$. From the definitions it follows immediately that
   \[ (\mathfrak{o}\bullet{}u)\bullet{}w = (\mathfrak{o}\bullet{}(uwu^{-1}))\bullet{} u \]
   and hence by the universal property of the semidirect product $\Omega \ltimes W_{\op{aff}}$ the two actions give rise to an action of $W$ on $\mathcal{O}(W_{\op{aff}},S)$.
\end{rmk}

\begin{rmk}\label{rmk:or_cox_vs_ext_cox}
   There exists a natural intrinsic right action of an extended Coxeter $W$ on the set $\mathcal{O}(W)$ of its orientations. If $\widetilde{\mathfrak{o}}$ is an orientation of $W$ associated to an orientation $\mathfrak{o}$ of $W_{\op{aff}}$, then for any $w \in W$ the map $\widetilde{\mathfrak{o}}\bullet{}w$ defined by
   \[ (\widetilde{\mathfrak{o}}\bullet{}w)(w',s) := \widetilde{\mathfrak{o}}(ww',s) \]
   is again an orientation. In fact, if we write $w = w_0 u$ and $w' = w_0' u'$ with $w_0, w_0' \in W_{\op{aff}}$ and $u,u' \in \Omega$ then
   \begin{align*} \widetilde{\mathfrak{o}}(ww',s) & = \widetilde{\mathfrak{o}}(w_0 uw_0'u^{-1} uu', s) = \mathfrak{o}(w_0 u w_0' u^{-1}, (uu')(s)) \\
      & = (\mathfrak{o}\bullet{}w_0)(uw_0'u^{-1}, u(u'(s))) = ((\mathfrak{o}\bullet{}w_0)\bullet{}u)(w_0',u'(s)) \\
      & = (\mathfrak{o}\bullet{}w)(w_0',u'(s)) \end{align*}
   Hence, $\widetilde{\mathfrak{o}}\bullet{}w$ is associated to $\mathfrak{o}\bullet{}w$. This computation also shows that the natural bijective map
   \[ \mathcal{O}(W_{\op{aff}},S) \stackrel{\sim}{\longrightarrow} \mathcal{O}(W) \]
   is $W$-equivariant with respect to the two actions described.
\end{rmk}

\begin{rmk}\label{rmk:homeomorphism_of_orientations}
   The set $\mathcal{O}(W)$ of orientations of an extended Coxeter group $W$ also carries a natural topology, namely the subspace topology induced by the space $\{\pm\}^{W\times S}$ and its compact-open topology. The above bijection then is actually a homeomorphism.
   This follows immediately from the fact that the extension map
   \begin{align*} \{\pm\}^{W_{\op{aff}}\times S} & \hookrightarrow \{\pm\}^{W\times S} \\
      f & \longmapsto ((wu,s) \mapsto f(w,u(s)))
   \end{align*}
   is a homeomorphism onto the subspace
   \[ \{ f \in \{\pm\}^{W\times S} : f(wu,s) = f(w,u(s)) \quad \forall w \in W,\ u \in \Omega,\ s \in S\} \]
   Since this subspace is closed, it follows that also the set of orientations of $(W,S)$ is a closed subspace of $\{\pm\}^{W\times S}$.
\end{rmk}

\begin{definition}
   Let $W^{(1)}$ be a pro-$p$ Coxeter group and $\mathfrak{o}$ be an orientation of the underlying extended Coxeter group $W$. The map $\widetilde{\mathfrak{o}}: W^{(1)}\times S \longrightarrow \{\pm 1\}$ defined by
   \[ \widetilde{\mathfrak{o}}(w,s) := \mathfrak{o}(\pi(w),s) \]
   is called the \textbf{orientation of $W^{(1)}$ associated to $\mathfrak{o}$}.

   An \textbf{orientation} of $W^{(1)}$ is a map $W^{(1)} \times S \longrightarrow \{\pm 1\}$ associated to an orientation of $W$ in the above sense, and the set of all such orientations is denoted by $\mathcal{O}(W^{(1)})$, or simply by $\mathcal{O}$ if $W^{(1)}$ is understood.
\end{definition}

\begin{rmk}\label{rmk:or_pro_p}
   There exists a natural right action of $W^{(1)}$ on the set of all orientations of $W^{(1)}$ again by the formula $(\mathfrak{o}\bullet{}w)(w',s) := \mathfrak{o}(ww',s)$. There also exists an action of $W^{(1)}$ on the set of all orientations of $W$ and $W_{\op{aff}}$ respectively via pulling back the $W$-actions along $\pi: W^{(1)} \rightarrow W$. The natural bijection
   \[ \mathcal{O}(W) \stackrel{\sim}{\longrightarrow} \mathcal{O}(W^{(1)}) \]
   is then equivariant with respect to these $W^{(1)}$-actions.
   
   By \cref{rmk:or_cox_vs_ext_cox}, we may therefore identify $\mathcal{O}(W^{(1)})$ and $\mathcal{O}(W_{\op{aff}},S)$ as $W^{(1)}$-sets, and may consider the former as a topological space through identification with the latter.
\end{rmk}

\subsection{Bernstein maps} 
\label{sub:Bernstein maps}
In this section, we will introduce the first of three related families of functions $\theta_\mathfrak{o},\widehat{\theta}_{\mathfrak{o}},\widetilde{\theta}_{\mathfrak{o}}$ which we loosely refer to as ``Bernstein maps'', as they are related to Bernstein's presentation of Iwahori-Hecke algebras. We fix a pro-$p$ Coxeter group $W^{(1)}$ throughout and denote by $\mathcal{O} = \mathcal{O}(W^{(1)})$ the set of orientations of $W^{(1)}$.

The following theorem is essentially the transposition of (\cite[Thm 1.1.1]{Goertz}) into our context. We first phrase it in terms of the braid group $\mathfrak{A}(W^{(1)})$ (see \cref{def:generalized_braid_group}).

\begin{theorem}\label{thm:ex_theta_braid}
   There exists a unique map
   \[ \theta: W^{(1)} \longrightarrow \op{Hom}_{\op{Set}}(\mathcal{O},\mathfrak{A}(W^{(1)})), \quad w \mapsto (\mathfrak{o} \mapsto \theta_\mathfrak{o}(w)) \]
   satisfying the cocycle rule
   \[ \theta_\mathfrak{o}(ww') = \theta_\mathfrak{o}(w)\theta_{\mathfrak{o}\bullet{}w}(w') \quad \forall w,w' \in W^{(1)} \]
   such that for $s \in S, \mathfrak{o} \in \mathcal{O}$
   \[ \theta_\mathfrak{o}(n_s) = T_{n_s^{\varepsilon}}^{\varepsilon} \quad \text{where}\quad \varepsilon = \mathfrak{o}(1,s) \in \{\pm 1\} \]
   and for $u \in \Omega^{(1)}, \mathfrak{o} \in \mathcal{O}$
   \[ \theta_\mathfrak{o}(u) = T_u \]
\end{theorem}
\begin{proof}
   We apply \cref{lem:cocycle} to the $W^{(1)}$-module $M = \op{Hom}_{\op{Set}}(\mathcal{O},\mathfrak{A}(W^{(1)}))$ and the pair $(\sigma,\rho)$, where
   \[ \sigma(s) = \left(\mathfrak{o} \mapsto T_{n_s^{\varepsilon}}^{\varepsilon}\right),\quad \varepsilon = \mathfrak{o}(1,s) \]
   and $\rho$ is the `trivial' cocycle
   \[ \rho(u) = \left(\mathfrak{o} \mapsto T_u\right) \]
   Here, the monoid structure on $M$ is given by pointwise multiplication and the left $W^{(1)}$-action is induced by the right action on $\mathcal{O}$ of \cref{rmk:or_pro_p}. It then only remains to verify conditions (i)-(iii) of \cref{lem:cocycle}. Bearing in mind the defining property \axiom{OR1} of an orientation, condition (i) amounts to showing that for all $s \in S$ and $\mathfrak{o} \in \mathcal{O}$
   \[ T_{n_s^{\varepsilon}}^{\varepsilon} T_{n_s^{-\varepsilon}}^{-\varepsilon} = T_{n_s^2} \]
   where $\varepsilon = \mathfrak{o}(1,s)$. First of all, note that $T_{n_s^2}$ commutes with $T_{n_s}$ since
   \[ T_{n_s} = T_{n_s^2 n_s n_s^{-2}} \stackrel{\text{(!)}}{=} T_{n_s^2} T_{n_s} T_{n_s^2}^{-1} \]
   where we used that $n_s^2 \in T \subseteq \Omega^{(1)}$. Therefore $T_{n_s}$ commutes also with
   \[ T_{n_s^{-1}} = T_{n_s^{-2} n_s} = T_{n_s^2}^{-1} T_{n_s} \]
   Given $\varepsilon \in \{\pm 1\}$, we have
   \[ T_{n_s^{\varepsilon}} = T_{n_s^{2\varepsilon} n_s^{-\varepsilon}} = T_{n_s^2}^{\varepsilon} T_{n_s^{-\varepsilon}} \]
   and hence
   \[ T_{n_s^{\varepsilon}} T_{n_s^{-\varepsilon}}^{-1} = T_{n_s^2}^{\varepsilon} \]
   Since $T_{n_s}$ and $T_{n_s^{-1}}$ commute, we can raise the last equation to the power $\varepsilon$ to get
   \[ T_{n_s^{\varepsilon}}^{\varepsilon} T_{n_s^{-\varepsilon}}^{-\varepsilon} = T_{n_s^2} \]
   We now turn to the verification of condition (ii). Unwinding the definitions and observing that the values of $\rho$ lie in the invariants $M^{W^{(1)}}$, we see that condition (ii) amounts to showing that for $\mathfrak{o} \in \mathcal{O}$, $s \in S$ and $u \in \Omega^{(1)}$ we have
   \[ T_u T_{n_s^{\varepsilon}}^{\varepsilon} = T_{n_{u(s)}^{\varepsilon}}^{\varepsilon} T_{u t_{s,u}} \]
   where we abbreviated $\varepsilon = \mathfrak{o}(1,u(s))$. When $\varepsilon = 1$, this reduces immediately to the defining equation $u n_s = n_{u(s)} u t_{s,u}$ of $t_{s,u}$. When $\varepsilon = -1$, we first compute
   \[ T_{n_{u(s)}^{-1}} T_u = T_{n_{u(s)}^{-1} u} = T_{u t_{s,u} n_s^{-1}} = T_{u t_{s,u}} T_{n_s^{-1}} \]
   Rearranging then gives the desired equation. Finally, let us verify condition (iii). Given $s,t \in S$ with $m(s,t) < \infty$, we have to show that for every orientation $\mathfrak{o}$ we have
   \begin{equation}\label{eq:thm_bernstein1} T_{n_s^{\varepsilon(1)}}^{\varepsilon(1)} T_{n_t^{\varepsilon(2)}}^{\varepsilon(2)} T_{n_s^{\varepsilon(3)}}^{\varepsilon(3)} \ldots = T_{n_t^{\varepsilon'(1)}}^{\varepsilon'(1)} T_{n_s^{\varepsilon'(2)}}^{\varepsilon'(2)} T_{n_t^{\varepsilon'(3)}}^{\varepsilon'(3)} \ldots \end{equation}
   where
   \[ \varepsilon(1) = \mathfrak{o}(1,s),\quad \varepsilon(2) = \mathfrak{o}(s,t),\quad \varepsilon(3) = \mathfrak{o}(st,s), \quad \ldots \]
   and
   \[ \varepsilon'(1) = \mathfrak{o}(1,t),\quad \varepsilon'(2) = \mathfrak{o}(t,s),\quad \varepsilon'(3) = \mathfrak{o}(ts,t),\quad \ldots \]
   are precisely the sign sequences appearing in condition \axiom{OR2} for $w = 1$ in \cref{def:or}. By condition \axiom{OR2}, these sequences are in one of two forms. Without loss of generality we may assume that they are in the first form, i.e.
   \[ (\varepsilon(1),\varepsilon(2),\ldots) = (\underbrace{+,\ldots,+}_k,\underbrace{-,\ldots,-}_{m(s,t)-k}) \]
   and
   \[ (\varepsilon'(1),\varepsilon'(2),\ldots) = (\underbrace{-,\ldots,-}_{m(s,t)-k},\underbrace{+,\ldots,+}_k) \]
   Writing
   \[ s_1 = n_s,\ s_2 = n_t,\ s_3 = n_s,\ \ldots \quad s'_1 = n_t,\ s'_2 = n_s,\ s'_3 = n_t,\ \ldots \]
   Equation \eqref{eq:thm_bernstein1} is thus of the form
   \[ T_{s_1} \ldots T_{s_k} T_{s_{k+1}^{-1}}^{-1} \ldots T_{s_{m(s,t)}^{-1}}^{-1} = T_{{s'}_1^{-1}}^{-1} \ldots T_{{s'}_{m(s,t)-k}^{-1}}^{-1} T_{{s'}_{m(s,t)-k+1}}\ldots T_{{s'}_{m(s,t)}} \]
   Rearranging the last equation slightly, we see that it is equivalent to
   \[ T_{{s'}_{m(s,t)-k}^{-1}} \ldots T_{{s'}_1^{-1}} T_{s_1} \ldots T_{s_k} = T_{{s'}_{m(s,t)-k+1}} \ldots T_{{s'}_{m(s,t)}} T_{s_{m(s,t)}^{-1}} \ldots T_{s_{k+1}^{-1}} \]
   Both sides of this equation are words $T_{w_1} \ldots T_{w_{m(s,t)}}$ of length $m(s,t)$ in the distinguished generators $T_w$ of $\mathfrak{A}(W^{(1)})$. Moreover, the words $w_1 \ldots w_{m(s,t)}$ in the elements of $W^{(1)}$ corresponding to them define reduced expressions, since under $W^{(1)} \twoheadrightarrow W$ they project to alternating words of length $m(s,t)$ in $s$ and $t$. Therefore, we can simplify both sides of the above equation to get
   \[ T_{{s'}_{m(s,t)-k}^{-1} \ldots {s'}_1^{-1} s_1 \ldots s_k} = T_{{s'}_{m(s,t)-k+1}\ldots {s'}_{m(s,t)} s_{m(s,t)}^{-1} \ldots s_{k+1}^{-1}} \]
   The validity of this equation now follows from the equation
   \[ {s'}_{m(s,t)-k}^{-1} \ldots {s'}_1^{-1} s_1 \ldots s_k = {s'}_{m(s,t)-k+1}\ldots {s'}_{m(s,t)} s_{m(s,t)}^{-1} \ldots s_{k+1}^{-1} \]
   in $W^{(1)}$, which by backtransforming is seen to be equivalent to the braid relation \eqref{eq:braidcond}
   \[ s_1 \ldots s_{m(s,t)} = {s'}_1 \ldots {s'}_{m(s,t)} \]
   which holds by assumption.
\end{proof}

\begin{definition}\label{def:unnormalized_bernstein}
   The map $\theta$ defined in the previous theorem is called the (unnormalized) \textbf{Bernstein map}. Given a generic pro-$p$ Hecke algebra $\mathcal{H}^{(1)} = \mathcal{H}^{(1)}(a,b)$ associated to parameters $a = (a_s)_{s \in S}$ and $b = (b_s)_{s \in S}$ with $a_s$ invertible, we have (by \cref{prop:second_presentation_of_hecke}) a morphism of monoids
   \[ \mathfrak{A}(W^{(1)}) \longrightarrow R[\mathfrak{A}(W^{(1)})] \longrightarrow \mathcal{H}^{(1)} \]
   Using this map we can push $\theta$ forward to obtain a map
   \[ W^{(1)} \longrightarrow \op{Hom}_{\op{Set}}(\mathcal{O},\mathcal{H}^{(1)}) \]
   still satisfying the $1$-cocycle rule. This map will also be denoted by $\theta$ and referred to as the (unnormalized) Bernstein map.
\end{definition}

Let now $\mathcal{H}^{(1)}$ be a generic pro-$p$ Hecke-algebra with invertible parameters $a_s$ as above. Fixing an orientation $\mathfrak{o} \in \mathcal{O}$, we thus have a family $\{ \theta_\mathfrak{o}(w) \}_{w \in W^{(1)}}$ of elements in $\mathcal{H}^{(1)}$.

The crucial point is that this family forms an $R$-basis of $\mathcal{H}^{(1)}$. This is the content of the next proposition, which shows that in fact the change of basis matrix between $\{ \theta_\mathfrak{o} \}_{w \in W^{(1)}}$ and $\{ T_w \}_{w \in W^{(1)}}$ is `upper triangular'. We will see later (equation \eqref{eq:theta_coincides_with_T_on_dominant}) that for a certain orientation $\mathfrak{o}$ the restriction of $\theta_\mathfrak{o}$ to the subgroup $X \leq W$ of translations recovers the map $\theta$ of Lusztig \cite{Lusztig}. This motivates the terminology `Bernstein map'.

\begin{prop}\label{prop:change_of_basis}
   In $\mathcal{H}^{(1)}$ one has an expansion of the form
   \[ \theta_\mathfrak{o}(w) = c_{w,w}T_w + \sum_{w' < w} c_{w,w'} T_{w'} \]
   with $c_{w,w} \in R^\times$ and $c_{w,w'} \in R$ which are zero for almost all $w'$. In particular $\{\theta_\mathfrak{o}(w)\}_{w \in W^{(1)}}$ is an $R$-basis of $\mathcal{H}^{(1)}$.
\end{prop}
\begin{proof}
   The first claim follows by taking an expression $w = n_{s_1}\ldots n_{s_r} u$ with $\ell(w) = r$ and expanding
   \[ \theta_\mathfrak{o}(w) = T^{\varepsilon_1}_{n_{s_1}^{\varepsilon_1}} \ldots T^{\varepsilon_r}_{n_{s_r}^{\varepsilon_r}} T_u \]
   using
   \begin{equation}\label{eq:t_inverse} T_{n_s^{-1}}^{-1} = a_s^{-1}(T_{n_s} - b_s) \end{equation}
      and the commutation rule \eqref{eq:gencommrule}. Here one also uses that $T_{n_{s_1}}\ldots T_{n_{s_r}} = T_{n_{s_1}\ldots n_{s_r}}$ and that for every $w \in W^{(1)}$ one either has
      \[ T_{n_s} T_w = T_{n_s w} \]
      or
      \[ T_{n_s} T_w = a_s T_{n_s w} + b_s T_w \]
      according to whether $\ell(n_s w) = 1+\ell(w)$ or $\ell(n_s w) = \ell(w)-1$.
      The second claim is a formal consequence of the first and the irreflexivity and transitivity of the relation $<$.
\end{proof}

\subsection{A \texorpdfstring{$2$}{2}-coboundary \texorpdfstring{$\bbf{X}$}{X} appearing in Coxeter geometry} 
\label{sub:A $2$-coboundary appearing in Coxeter geometry}
The purpose of this section is to pave the way for introducing an integral $\widehat{\theta}$ and a normalized version $\widetilde{\theta}$ of the Bernstein map $\theta$ defined in the previous section.

The map $\theta$ has the `defect' that it is only defined when the parameters $a_s$ are invertible. In view of the study mod $p$ representations of pro-$p$-Iwahori Hecke algebras (where $a_s = 0$), it is important to have an integral version which is defined for \textit{all} parameters. Such variants of the classical Bernstein-Lusztig basis have been first introduced by Vignéras \cite{VigProP}, \cite{VigGen}. The construction of $\widehat{\theta}$ is based on the following relation
\[ a_s T_{n_s^{-1}}^{-1} = T_{n_s} - b_s \]
which is an immediate consequence of the quadratic relations. It suggests to formally multiply
\[ \theta_\mathfrak{o}(w) = T_{n_{s_1}^{\varepsilon_1}}^{\varepsilon_1} \ldots T_{n_{s_r}^{\varepsilon_r}}^{\varepsilon_r} T_u \]
by the product
\[ \overline{\gamma}_{\mathfrak{o}}(w) = \prod_{i : \varepsilon_i = -1} a_{s_i} \]
to get an integral expression in the generators $T_w$. However, a priori the factor $\overline{\gamma}_{\mathfrak{o}}(w)$ and therefore $\widehat{\theta}_\mathfrak{o}(w)$ depends on the chosen expression $w = n_{s_1}\ldots n_{s_r} u$ for $w$ as a product in the distinguished generators. The first goal of this section is therefore to establish the independence of $\overline{\gamma}_\mathfrak{o}(w)$ from the chosen expression for $w$.
As this is a purely combinatorial question, it will be useful to work with formal products of hyperplanes instead of products of the parameters $a_s$, and to replace $\overline{\gamma}$ by a purely combinatorially defined map $\gamma$.

The second goal of this section is to determine the multiplicative properties of $\gamma$, as these determine the multiplicative properties of $\widehat{\theta}$ and the usefulness of $\theta$ wholly depends on the fact that it satisfies the cocycle rule. We will achieve this by identifying the coboundary of $\gamma$ (viewed as a map $w \mapsto (\mathfrak{o} \mapsto \gamma_\mathfrak{o}(w))$ in one parameter) with another explicitly and combinatorially defined map $\bbf{X}$.

We will then give a second characterization of $\bbf{X}$ as a coboundary of a `generalized length function' $\sqrt{\bbf{L}}$, which is needed in order to introduce and prove the multiplicative properties of a normalized variant $\widetilde{\theta}$ of $\theta$. This normalized version is closely related to the classical Bernstein-Lusztig basis of the Iwahori-Hecke algebra (see \cref{ex:affine Hecke algebras}).

Since everything in this section only involves the combinatorics of extended Coxeter groups, we need only to fix an extended Coxeter group $W = (W,W_{\op{aff}},\Omega, S)$.

Let us start by defining the `coboundary' mentioned in the title of this section.

\begin{definition}\label{def:coboundary}
   Given $w,w' \in W$ let
   \[ \bbf{X}(w,w') := \prod_{H} \mathbf{a}_H \in \N[\mathfrak{H}] \]
   where $\N[\mathfrak{H}]$ denotes the free abelian monoid on the set $\mathfrak{H}$ of hyperplanes and the product is taken over all hyperplanes $H \in \mathfrak{H}$ which both separate $1$ from $w$ and $w$ from $ww'$.
   
   In other words, $\bbf{X}(w,w')$ is the product over all hyperplanes which are crossed twice by any gallery that is the concatenation of a minimal gallery from $1$ to $w$ and a minimal gallery from $w$ to $ww'$. In particular we have the following observation, which we record separately.
\end{definition}

\begin{rmk}\label{rmk:coboundary}
   For all $w,w' \in W$
   \[ \ell(ww') = \ell(w)+\ell(w') \quad\Rightarrow\quad \bbf{X}(w,w') = 1 \]
\end{rmk}

\begin{rmk}\label{rmk:coboundary_vanishes}
   From the definition of $\bbf{X}$, it also follows directly that
   \[ \bbf{X}(w,w') = 1 \]
   whenever $w \in \Omega$ or $w' \in \Omega$.
\end{rmk}

Next, we will show that the sign attached by an orientation to crossing a hyperplane $H$ at a chamber $w \in W_{\op{aff}}$ does not depend upon the chamber itself but only upon which half-space with respect to $H$ this chamber lies in.

\begin{lemma}\label{lem:or_depend_on_half-space}
   If $w,\widetilde{w} \in W_{\op{aff}}$ and $s,\widetilde{s} \in S$ are such that
   \[ wsw^{-1} = \widetilde{w}\widetilde{s}\widetilde{w}^{-1} \quad \text{and} \quad \ell(sw^{-1}\widetilde{w}) = 1+\ell(w^{-1}\widetilde{w}) \]
   that is, if $w,ws$ and $\widetilde{w},\widetilde{w}\widetilde{s}$ are separated by the same wall $H = wsw^{-1} = \widetilde{w}\widetilde{s}\widetilde{w}^{-1}$ and $w,\widetilde{w}$ lie on the same side with respect to $H$, then
   \[ \mathfrak{o}(w,s) = \mathfrak{o}(\widetilde{w},\widetilde{s}) \]
\end{lemma}
\begin{proof}
   After replacing $\mathfrak{o}$ by $\mathfrak{o}\bullet{}w$, we may assume that $w = 1$. Then $s\widetilde{w} = \widetilde{w}\widetilde{s}$ and $\ell(s\widetilde{w}) = \ell(\widetilde{w}\widetilde{s}) = \ell(\widetilde{w}) + 1$. Therefore, if we take any reduced expression $\widetilde{w} = s_1\ldots s_r$, then
   \[ s s_1\ldots s_r = s_1\ldots s_r \widetilde{s} \]
   will be two \textit{reduced} expressions of the same element in $W_{\op{aff}}$ and $\mathfrak{o}(1,s), \mathfrak{o}(\widetilde{w},\widetilde{s})$ are the signs which appear in these galleries when crossing the wall $H$. It therefore suffices to show that for any two reduced expressions of the same element in $W_{\op{aff}}$ and any hyperplane $H$ the signs which appear when crossing $H$ are the same for both expressions. By Tits' solution of the word problem \cite[II.3C]{Brown}, two such reduced expressions can be transformed into each other by a finite sequence of transformations of type $(\text{III})$ (cf. proof of \cref{thm:ex_theta_braid})
   \[ t_1 \ldots t_i \underbrace{s t s \ldots }_{m(s,t) < \infty} t_{i+1} \ldots t_m \longmapsto t_1 \ldots t_i \underbrace{tst \ldots }_{m(s,t) < \infty} t_{i+1} \ldots t_m \]
   If $H$ is crossed before or after the part where these two galleries differ, the signs are equal for trivial reasons. It therefore suffices to show that for $s,t \in S$ with $m(s,t) < \infty$ the signs of all the walls crossed by the two galleries corresponding to the reduced expressions
   \[ \underbrace{sts\ldots}_{m(s,t)} = \underbrace{tst\ldots }_{m(s,t)} \]
   are equal. But by \cref{rmk:or_geom_inter}, the signs are determined by the distance to some reference chamber in $\left<s,t\right>$. In particular the sign $\mathfrak{o}(w,s)$ only depends on which half-space with respect to $H = wsw^{-1}$ the fundamental chamber lies in.
\end{proof}

\begin{notation}\label{not:orientations}
   Thanks to the previous lemma, we may extend any orientation $\mathfrak{o}$ canonically to a map
   \[ \mathfrak{o}: W\times \mathfrak{H} \rightarrow \{\pm \} \]
   by letting
   \[ \mathfrak{o}(w,H) := \mathfrak{o}(w w_0,s),\quad w \in W,\ H \in \mathfrak{H} \]
   where $w_0 \in W$ and $s \in S$ are such that
   \[ w_0 s w_0^{-1} = H \]
   and $1, w_0$ lie in the same half-space with respect to $H$. It follows quite easily that this does indeed give rise to a well-defined map $W\times \mathfrak{H} \rightarrow \{\pm\}$ that extends $\mathfrak{o}$. In the terminology of \cref{term:basic} and \cref{term:orient}, the sign $\mathfrak{o}(w,H)$ has the geometric interpretation as being the sign that is attached to crossing the hyperplane $wHw^{-1}$ at any chamber that lies in the same half-space with respect to $wHw^{-1}$ as $w$. In particular $\mathfrak{o}(1,H)$ is the sign attached by $\mathfrak{o}$ to crossing $H$ at any chamber that lies in the same half-space with respect to $H$ as the fundamental chamber.
\end{notation}

\begin{cor}
   Given an orientation $\mathfrak{o}$ of $W$, there exists a unique map from $W$ into the free commutative monoid $\N[\mathfrak{H}]$ with generators $\mathbf{a}_H$ corresponding to the hyperplanes $H \in \mathfrak{H}$ such that if $w = s_1 \ldots s_r u$, $s_i \in S$, $u \in \Omega$ is a reduced expression for $w$, then $\gamma_\mathfrak{o}(w)$ equals the product of the hyperplanes crossed in the negative direction by the gallery corresponding to this reduced expression. In other words
   \begin{equation}\label{eq:gamma_def} \gamma_\mathfrak{o}(w) = \prod_{i\ :\ \varepsilon_i = -1} \mathbf{a}_{H_i} \end{equation}
   where $\varepsilon_i = \mathfrak{o}(s_1\ldots s_{i-1},s_i)$ and $H_i = (s_1\ldots s_{i-1}) s_i (s_1 \ldots s_{i-1})^{-1}$.
\end{cor}
\begin{proof}
   We need to verify the independence of the right-hand side of equation \eqref{eq:gamma_def} from the choice of the reduced expression. Since $s_1 \ldots s_r$ is a reduced expression of $wu^{-1} \in W_{\op{aff}}$, the walls $H_i$ appearing are pairwise distinct and are equal to the walls separating $1$ and $w$. On the other hand, by the previous lemma the sign $\varepsilon_i$ only depends on which half-space with respect to $H_i$ the fundamental chamber lies in. Therefore, the $H_i$ with $\varepsilon_i = -1$ only depend on $w$ and $\mathfrak{o}$.
\end{proof}

\begin{rmk}\label{rmk:gamma_product_rule_as_coboundary_equation}
   As promised, we will now explicitly determine the `coboundary' of the map $\gamma$ defined above. More precisely, let us view $\gamma$ as a map $\gamma: W \longrightarrow M$ taking values in the $W$-module $M = \op{Hom}_{\op{Set}}(\mathcal{O},\Z[\mathfrak{H}])$. The structure of an abelian group on $M$ is `pointwise', and $\Z[\mathfrak{H}] \supseteq \N[\mathfrak{H}]$ denotes the free commutative \textit{group} on $\mathfrak{H}$. The $W$-action on $M$ is induced by the canonical right action on $\mathcal{O}$ and the canonical left action on $\mathfrak{H}$, i.e.
   \[ (w\bullet{}\phi)(\mathfrak{o}) = w\bullet{}\phi(\mathfrak{o}\bullet{}w) \quad \forall w \in W,\ \phi \in M,\ \mathfrak{o} \in \mathcal{O} \]
Finally, let us view $\bbf{X}$ as a map
\[ \bbf{X}: W\times W \longrightarrow M,\quad (w,w') \mapsto \left(\mathfrak{o} \mapsto \bbf{X}(w,w')\right) \]
The statement of the next lemma is then equivalent to the coboundary equation
\[ d \gamma = \bbf{X} \]
of the inhomogeneous standard cochain complex on $M$.
\end{rmk}

\begin{lemma}\label{lem:gamma_product_rule}
   For all $w,w' \in W$, one has
   \[ \gamma_\mathfrak{o}(w) w\left(\gamma_{\mathfrak{o}\bullet{}w}(w')\right) = \bbf{X}(w,w') \gamma_\mathfrak{o}(ww') \]
\end{lemma}
\begin{proof}
   Write $w = w_0 u$ and $w' = w_0' u'$ with $w_0,w_0' \in W_{\op{aff}}$ and $u,u' \in \Omega$. Then by definition
   \[ \gamma_\mathfrak{o}(w) = \gamma_\mathfrak{o}(w_0) \quad \gamma_{\mathfrak{o}\bullet{}w}(w') = \gamma_{\mathfrak{o}\bullet{}w}(w_0') \quad \gamma_\mathfrak{o}(ww') = \gamma_\mathfrak{o}(w_0 u(w_0)') \]
   and
   \[ \bbf{X}(w,w') = \bbf{X}(w_0, w_0 u(w_0')) \]
   Moreover, it follows from the definitions that
   \[ \gamma_{(\mathfrak{o}\bullet{}w_0)\bullet{}u}(w_0') = u^{-1}\left(\gamma_{\mathfrak{o}\bullet{}w_0}(u(w_0'))\right) \]
   It therefore suffices to prove the formula for $w,w' \in W_{\op{aff}}$. Taking reduced expressions $w = s_1\ldots s_r$ and $w' = s_{r+1} \ldots s_{r+m}$, one has
   \[ \gamma_\mathfrak{o}(w)w\left(\gamma_{\mathfrak{o}\bullet{}w}(w')\right) = \prod_H \mathbf{a}_H \]
   where the product extends over all walls $H$ which are crossed with a negative sign by the gallery corresponding to the possibly nonreduced expression $s_1\ldots s_{r+m}$. A wall $H$ will be crossed by this gallery if and only if it separates $1$ from $w$ or $w$ from $ww'$. A wall $H$ is crossed twice iff it separates both $1$ from $w$ and $w$ from $ww'$, otherwise it is crossed only once. The walls that are crossed once are exactly the walls that separate $1$ from $ww'$ and they are crossed with the same sign as in a minimal gallery from $1$ to $ww'$. The walls that are crossed twice are crossed once with a positive and once with a negative sign. It therefore follows immediately that
   \[ \gamma_\mathfrak{o}(w)w\left(\gamma_{\mathfrak{o}\bullet{}w}(w')\right) = \bbf{X}(w,w') \gamma_\mathfrak{o}(ww') \]
\end{proof}

The length $\ell(w)$ of an element $w \in W$ is given by the \textit{number} of walls separating $1$ and $w$. Replacing numbers by formal products of walls we get the notion of the \textit{generalized length} $\bbf{L}(w)$ of an element, which leads to another characterization of $\bbf{X}$ as a coboundary.

\begin{definition}\label{def:generalized_length}
   The \textbf{generalized length} $\bbf{L}(w)$ of $w \in W$ is the element of $\N[\mathfrak{H}]$ given by
   \[ \bbf{L}(w) := \prod_H \mathbf{a}_{H} \]
   where the product is taken over all $H \in \mathfrak{H}$ separating $1$ and $w$.
\end{definition}

\begin{lemma}\label{lem:gen_len_product_rule}
   For all $w,w' \in W$ we have
   \[ \bbf{L}(w) w\left(\bbf{L}(w')\right) = \bbf{X}(w,w')^2 \bbf{L}(ww') \]
\end{lemma}
\begin{proof}
   This follows from the same arguments given in the proof of \cref{lem:gamma_product_rule}. The only difference is that here every wall that is crossed twice also appears twice.
\end{proof}

\begin{rmk}\label{rmk:gen_len_product_rule}
   \begin{enumerate}
      \item The length $\ell(w)$ of an element $w \in W$ and its generalized length $\bbf{L}(w)$ are related by the `cardinality morphism'
         \[ \#: (\N[\mathfrak{H}],\cdot ) \longrightarrow (\N,+),\quad \mathbf{a}_H \mapsto 1 \]
      by
      \[ \ell(w) = \# \bbf{L}(w) \]
      The lemma above therefore gives the formula
      \[ \# \bbf{X}(w,w') = \ell(w)+\ell(w')-\ell(ww') \]
      which reproves and generalizes \cref{rmk:coboundary}. The lemma also shows that
      \[ \ell(ww') = \ell(w)+\ell(w') \quad\Rightarrow\quad \bbf{L}(ww') = \bbf{L}(w) w(\bbf{L}(w')) \]

      \item The above lemma says that $\bbf{X}$ is the coboundary of the formal square root $\sqrt{\bbf{L}}$ of $\bbf{L}$. More precisely, letting $\Z[\sqrt{\mathfrak{H}}]$ denote the free abelian group on the symbols $\sqrt{\mathbf{a}_H}$, $H \in \mathfrak{H}$ we can view $\Z[\mathfrak{H}]$ as a subset of $\Z[\sqrt{\mathfrak{H}}]$ via the embedding given by $\mathbf{a}_H \mapsto \left(\sqrt{\mathbf{a}_H}\right)^2$.
         Pushing $\bbf{L}: W \rightarrow \Z[\sqrt{\mathfrak{H}}]$ forward via this embedding, it has a unique square root $\sqrt{\bbf{L}}: W \rightarrow \Z[\sqrt{\mathfrak{H}}]$. Viewing $\bbf{X}$ as a map $W\times W \rightarrow \Z[\sqrt{\mathfrak{H}}]$, the formula of the above lemma is equivalent to the coboundary equation
   \[ d\sqrt{\bbf{L}} = \bbf{X} \]
   of the inhomogeneous cochain complex on $\Z[\sqrt{\mathfrak{H}}]$.

\item The construction of the integral Bernstein-Lusztig basis in \cite{VigGen} heavily depends on the `lemme fondamental' \cite[1.2]{VigGen}. There it is proven that a certain expression $q_{wv}q^{-1}_w q_v$ ($w,v \in W$) which is a product of formal parameters is a square of an element $c_{w,v}$. This relates to the previous lemma as follows. Consider the orbit map $\mathfrak{H} \rightarrow W\backslash \mathfrak{H}$ of the canonical action of $W$ on $\mathfrak{H}$. Pushing $\bbf{L}$ and $\bbf{X}$ forward along the induced map $\Z[\mathfrak{H}] \rightarrow \Z[W\backslash \mathfrak{H}]$, we get maps $\overline{\bbf{L}}$ and $\overline{\bbf{X}}$ with values in $\Z[W\backslash \mathfrak{H}]$. The formula proven in the above lemma then simplifies to
   \begin{equation}\label{eq:simple_formula_for_coboundary} \overline{\bbf{L}}(w)\overline{\bbf{L}}(w') \overline{\bbf{L}}(ww')^{-1} = \overline{\bbf{X}}(w,w')^2 \end{equation}
      Identifying the formal parameter (`poid générique') $q_s$ ($s \in S$) of \cite{VigGen} with the generator $\mathbf{a}_{[s]} \in \Z[W\backslash \mathfrak{H}]$ corresponding to the class $[s] \in W\backslash \mathfrak{H}$, the element $q_w$ ($w \in W$) defined in loc. cit. identifies with $\overline{\bbf{L}}(w)$. In this notation the above formula reads
   \[ q_w q_{w'} q_{ww'}^{-1} = \overline{\bbf{X}}(w,w')^2 \]
   In particular we find that
   \[ q_{wv}q_w^{-1}q_v = q_{wv} q_w^{-1} q_v^{-1} q_v^2 = \overline{\bbf{X}}(w,v)^{-2} q_v^2 \]
   and therefore that the element $c_{w,v}$ defined in \cite{VigGen} is given by
   \[ c_{w,v} = \overline{\bbf{X}}(w,v)^{-1} q_v \]
   This element is more explicity given as the product
   \[ c_{w,v} = \prod_{H} \mathbf{a}_{[H]} \]
   where the product runs over all hyperplanes $H$ which separate $1$ from $v$ \textit{but don't} separate $1$ from $w$.
   \end{enumerate}
\end{rmk}

\subsection{A characterization of pro-\texorpdfstring{$p$}{p} Coxeter groups in terms of \texorpdfstring{$\bbf{X}$}{X}} 
\label{sub:A characterization of pro-$p$ Coxeter groups in terms of X}
Throughout this section, we will fix an extended Coxeter group $W = (W,W_{\op{aff}},\Omega,S)$. Our goal here is to characterize some (all, if $W = W_{\op{aff}}$) pro-$p$ Coxeter groups $W^{(1)}$ whose underlying extended Coxeter group equals $W$, using the `$2$-coboundary' $\bbf{X}$ of the previous section. Even though this result will not be used in the rest of the text, we choose to present it because we think it is of independent interest.

By definition, a pro-$p$ Coxeter group $W^{(1)}$ is given by a group extension
\[ \begin{xy} \xymatrix{ 1 \ar[r] & T \ar[r] & W^{(1)} \ar[r] & W \ar[r] & 1 } \end{xy} \]
   of $W$ by an \textit{abelian} group $T$, together with a choice of lifts $(n_s)_{s \in S}$ of the distinguished generators which satisfy the braid relations. In the case $W = W_{\op{aff}}$, such groups have been studied by Tits \cite{TitsTores} under the name\footnote{We apologize for not following the terminology of \cite{TitsTores}, because in our contexts we have to consider extensions of groups which are themselves (split) extensions of Coxeter groups.} of `extended Coxeter groups'. Among the many interesting results obtained in \cite{TitsTores} is a characterization (\cite[3.4 Proposition]{TitsTores}) of such extensions in terms of data related to $W$ and $T$, and the construction and explicit description of a `universal' extension $V$. Implicit in this (see especially \cite[3.4 Proposition]{TitsTores}) is that the $2$-cocycle
   \[ \phi: W\times W \rightarrow T,\quad \phi(w,w') = n(w)n(w')n(ww')^{-1} \]
associated to the extension and the \textit{canonical} set-theoretic section
\[ n: W \rightarrow W^{(1)} \]
determined by $n(s) = n_s$ and
\[ \ell(ww') = \ell(w)+\ell(w') \quad \Rightarrow \quad n(ww') = n(w)n(w') \]
can be explicitly computed. However in \cite{TitsTores} an explicit expression for this $2$-cocycle was not given. We shall therefore explicitly compute these cocycles in terms of $\bbf{X}$, and deduce the existence of a universal extension (without reference to \cite{TitsTores}) whose corresponding $2$-cocycle identifies with $\bbf{X}$.

Let us begin with a definition.
\begin{definition}
   The category $\mathcal{W}^{(1)}_{/W}$ is the category whose objects consist of extensions
   \[ \begin{xy} \xymatrix{ 1 \ar[r] & T \ar[r] & G \ar[r] & W \ar[r] & 1 } \end{xy} \]
   of $W$ by an abelian group, together with a set-theoretic section $n: W \rightarrow G$ of the map $G \rightarrow W$ satisfying
   \[ \ell(ww') = \ell(w)+\ell(w') \quad \Rightarrow \quad n(ww') = n(w)n(w') \quad \forall w,w' \in W \]
   A morphism $f: (G,T,n) \rightarrow (G',T',n')$ is given by a morphism $f: G \rightarrow G'$ which makes the diagram
   \[ \begin{xy} \xymatrix{ T \ar[d]^f \ar[r] & G \ar[d]^f \ar[r] & W \ar[d]^{\op{id}} \\ T' \ar[r] & G' \ar[r] & W } \end{xy} \]
   commute and which satisfies $f\circ{} n = n'$.
\end{definition}
Thus essentially the objects are pro-$p$ Coxeter groups whose underlying extended Coxeter group equals $W$. However, not all pro-$p$ Coxeter groups give rise to objects of this category. More precisely, an object $(W^{(1)},T,n)$ of $\mathcal{W}^{(1)}$ corresponds to a pro-$p$ Coxeter group $W^{(1)}$ together with a section \textit{of groups}
\[ \widetilde{n}: \Omega \rightarrow \pi^{-1}(\Omega) \]
of the restriction of $\pi: W^{(1)} \rightarrow W$ to $\Omega$, such that we have the relation
\[ \widetilde{n}(u)n_s\widetilde{n}(u)^{-1} = n_{u(s)} \quad \forall u \in \Omega,\ s \in S \]
The map $n: W \rightarrow W^{(1)}$ is then uniquely determined by $\widetilde{n}$ and $(n_s)_{s \in S}$ by requiring
\[ n(u) = \widetilde{n}(u),\ n(s) = n_s \quad \forall u \in \Omega,\ s \in S \]
Note that given a pro-$p$ Coxeter group $W^{(1)}$ such a section $\widetilde{n}$ might not exist, and even if it does the relation $\widetilde{n}(u)n_s \widetilde{n}(u)^{-1} = n_{u(s)}$ might not be fulfilled. However, when $W = W_{\op{aff}}$, the set of pro-$p$ Coxeter groups with underlying extended Coxeter group $W$ and the set of objects $\mathcal{W}^{(1)}_{/W}$ are canonically identified.

\begin{lemma}\label{lem:cocycle_of_group_extension}
   Given an object $(G,T,n)$ of $\mathcal{W}^{(1)}_{/W}$, the $2$-cocycle $\phi: W\times W \rightarrow T$ determined by the section $n$ via
   \[ \phi(w,w') = n(w)n(w')n(ww')^{-1} \]
   satisfies
   \[ \phi(w,w') = h(\bbf{X}(w,w')) \]
   Here
   \[ h: \Z[\mathfrak{H}] \longrightarrow T \]
   denotes the unique $W$-equivariant homomorphism of abelian groups satisfying
   \[ h(s) = n(s)^2 \quad \forall s \in S \]
\end{lemma}
\begin{proof}
   First, note that $h$ is obviously unique if it exists since we have
   \[ h(wsw^{-1}) = w n(s)^2 w^{-1} \quad \forall w \in W,\ s \in S \]
   by assumption. Therefore, such a map exists if and only if for $w \in W$ and $s,t \in S$ we have
   \[ wsw^{-1} = t \quad \Rightarrow \quad w n(s)^2 w^{-1} = n(t)^2 \]
   But replacing $w$ by $ws$ if necessary, we may assume $\ell(ws) = \ell(w) + 1$. Hence, we also have $\ell(tw) = \ell(w) + 1$ and therefore
   \[ n(w)n(s) = n(ws) = n(tw) = n(t) n(w) \]
   implying
   \[ wn(s)^2 w^{-1} = (n(w)n(s)n(w)^{-1})^2 = n(t)^2 \]
   Since both $\phi$ and $h\circ{}\bbf{X}$ fulfill the $2$-cocycle relation
   \[ \phi(w_1,w_2) \phi(w_1w_2, w_3) = w_1(\phi(w_2,w_3)) \phi(w_1,w_2w_3) \]
   and moreover both of these maps vanish whenever one of their arguments lies in $\Omega$, which for $\phi$ follows from the relation $n(ww') = n(w)n(w')$ for $w,w'$ satisfying $\ell(ww') = \ell(w)+\ell(w')$ and for $h\circ{}\bbf{X}$ follows from \cref{rmk:coboundary_vanishes}, in order to show that $\phi = h\circ{}\bbf{X}$ it hence suffices to prove that
   \[ \phi(s,w) = h(\bbf{X}(s,w)) \quad \forall s \in S,\ w \in W \]
   Since both maps vanish on pairs $(w,w')$ satisfying $\ell(ww') = \ell(w)+\ell(w')$ (cf. \cref{rmk:coboundary}), it suffices to treat the case $\ell(sw) = \ell(w) - 1$. Take a reduced expression
   \[ sw = s_1 \ldots s_r u,\quad s_i \in S,\ u \in \Omega,\ r = \ell(sw) \]
   of $sw$, then
   \[ w = s s_1 \ldots s_r u \]
   is a reduced expression of $w$. Hence
   \[ n(w) = n(s) n(s_1) \ldots n(s_r) n(u) \]
   and
   \[ n(sw) = n(s_1) \ldots n(s_r) n(u) \]
   and therefore
   \[ \phi(s,w) = n(s)^2 = h(\bbf{X}(s,w)) \]
   because $\bbf{X}(s,w) = s$ is the unique hyperplane crossed twice by the gallery $(s,s,s_1,\ldots, s_r)$.
\end{proof}

\begin{definition}
   $\mathcal{T}_W$ is the category whose objects are given by pairs $(T,h)$ consisting of an abelian group $W$ endowed with a $\Z$-linear $W$-action and a $W$-equivariant map
   \[ h: \mathfrak{H} \longrightarrow T \]
   (identified with its linear extension $h: \Z[\mathfrak{H}] \rightarrow T$) and whose morphisms $f: (T,h) \rightarrow (T',h')$ are given by $W$-equivariant group homomorphisms $f: T \rightarrow T'$ satisfying $f\circ{} h = h'$.
\end{definition}

With the above definition, we have the following immediate corollary of the above lemma.
\begin{cor}
   The functor
   \[ \mathcal{W}^{(1)}_{/W} \longrightarrow \mathcal{T}_W \]
   given on morphisms in the obvious way and on objects by
   \[ (G,T,n) \longmapsto (T,h),\quad h(s) = n(s)^2 \]
   is an equivalence of categories, with quasi-inverse associating to an object $(T,h)$ the object $(T\times W, T, \iota)$, where $T\times W$ is the group with group law
   \[ (t,w)\cdot{} (t',w') = (tw(t') h(\bbf{X}(w,w')), ww') \]
   and $\iota: W \rightarrow T\times W$ is given by $\iota(w) = (1,w)$.
\end{cor}

The following corollary essentially recovers Tits description of the group $V$ (cf. \cite[2.5 Théorème]{TitsTores}).

\begin{cor}
   The category $\mathcal{W}^{(1)}_{/W}$ has an initial object $V$ given by
   \[ V = (V,T,n) = (\Z[\mathfrak{H}]\times W,\Z[\mathfrak{H}],\iota) \]
   where $V = \Z[\mathfrak{H}]\times W$ is endowed with a group law via
   \[ (t,w)\cdot{}(t',w') = (tw(t') \bbf{X}(w,w'), ww') \]
   and $\iota: W \rightarrow V$ is given by $\iota(w) = (1,w)$.
\end{cor}
\begin{proof} Immediate from the above corollary, since the pair $(T,h)$ with $T = \Z[\mathfrak{H}]$ and $h = \op{id}$ obviously forms an initial object of the category $\mathcal{T}_W$.
\end{proof}

\subsection{Integral and normalized Bernstein maps} 
\label{sub:Integral and normalized Bernstein maps}
We now apply the results of \cref{sub:A $2$-coboundary appearing in Coxeter geometry} to the construction of an integral and a normalized version of the Bernstein map.
Throughout this section we fix a pro-$p$ Coxeter group $W^{(1)}$, a coefficient ring $R$, and a generic pro-$p$ Hecke algebra $\mathcal{H}^{(1)} = \mathcal{H}^{(1)}(a,b)$ with arbitrary parameters.

Let us begin by constructing the integral Bernstein map.

\begin{theorem}\label{thm:existence_widehat_theta}
   For every orientation $\mathfrak{o}$ of $W^{(1)}$, there exists a unique map
   \[ \widehat{\theta}_\mathfrak{o}: W^{(1)} \longrightarrow \mathcal{H}^{(1)} \]
   such that if $w = n_{s_1}\ldots n_{s_r} u$ with $u \in \Omega^{(1)}$ and $\ell(w) = r$, then
   \[ \widehat{\theta}_\mathfrak{o}(w) = T_1 \ldots T_r T_u \]
   where
   \[ T_i := \begin{cases}
      T_{n_{s_i}} & : \varepsilon_i = +1 \\
      T_{n_{s_i}} - b_{s_i} & : \varepsilon_i = -1
   \end{cases} \]
   and $\varepsilon_i = \mathfrak{o}(s_1\ldots s_{i-1},s_i)$. Moreover, whenever the $a_s$ are units in $R$ we have the equality
   \begin{equation}\label{eq:relation_theta_thetahat} \widehat{\theta}_\mathfrak{o}(w) = \overline{\gamma}_\mathfrak{o}(\pi(w)) \theta_\mathfrak{o}(w) \end{equation}
   where $\overline{\gamma}_\mathfrak{o}: W \rightarrow R$ is the composition of $\gamma_\mathfrak{o}$ with the specialization map $\N[\mathfrak{H}] \rightarrow R$ sending $\mathbf{a}_H$ to $a_H$.
\end{theorem}
\begin{proof}
   Because of the relation $\eqref{eq:t_inverse}$, we have
   \[ a_s T_{n_s^{-1}}^{-1} = T_{n_s} - b_s \]
   whenever $a_s \in R^\times$. The second claim therefore follows immediately from the definitions provided the existence of $\widehat{\theta}_\mathfrak{o}$. We are therefore left to show that the expression $T_1 \ldots T_r T_u$ does not depend on the choice of the expression $w = n_{s_1}\ldots n_{s_r} u$. If this indepence result is true for the generic pro-$p$ Hecke algebra $\mathcal{H}^{(1)}$ over $R$, then it is also true for the generic pro-$p$ Hecke algebra $\mathcal{H}^{(1)}((\phi(a_s))_s, (\phi(b_s))_s) \simeq \mathcal{H}^{(1)} \otimes_R R'$ over $R'$ for every $\phi: R \rightarrow R'$. We may therefore replace the parameters $a_s \in R$ by indeterminates $\mathbf{a}_s$ which satisfy $\mathbf{a}_s = \mathbf{a}_t$ whenever $s,t$ are conjugate via $W$ and replace $R$ by the polynomial ring $R[\mathbf{a}_s]$ and prove the claim for $\mathcal{H}^{(1)}((\mathbf{a}_s)_s, (b_s)_{s \in S})$. Because $\mathcal{H}^{(1)}$ is a free $R[\mathbf{a}_s]$-module, the localization map
   \[ \mathcal{H}^{(1)} \longrightarrow \mathcal{H}^{(1)}\otimes_{R[\mathbf{a}_s]} R[\mathbf{a}_s,\mathbf{a}_s^{-1}] \]
   is injective. It therefore suffices to prove the indepence in the localization, that is it suffices to prove it in the case the $a_s$ are invertible. In this case we may use \eqref{eq:relation_theta_thetahat} as a definition of $\widehat{\theta}_\mathfrak{o}$. From the definition of $\theta_\mathfrak{o}$ it follows immediately that the map defined this way satisfies $\widehat{\theta}_\mathfrak{o}(w) = T_1 \ldots T_r T_u$ for every reduced expression $w = n_{s_1}\ldots n_{s_r} u$.
\end{proof}

\begin{definition}\label{def:integral_bernstein_map}
   The map 
   \[ \widehat{\theta}: W^{(1)} \longrightarrow \op{Hom}_{\op{Set}}(\mathcal{O},\mathcal{H}^{(1)}),\quad w \longmapsto (\mathfrak{o} \mapsto \widehat{\theta}_\mathfrak{o}(w)) \]
   defined in the above theorem is called the \textbf{integral Bernstein map}.
\end{definition}

\begin{rmk}\label{rmk:specialization_argument}
   The above technique of establishing a certain identity for generic pro-$p$ Hecke algebras with arbitrary parameters by reducing it to the case where the $a_s$ are invertible is the main advantage of considering Hecke algebras with two formal parameters over considering only one-parameter Hecke algebras or Hecke algebras with fixed parameters.
   
   We will use this argument over and over again, and will therefore often refer to it simply as the `\textbf{specialization argument}'.
\end{rmk}

\begin{notation}\label{not:overline_bbf_X}
   In \cref{rmk:gen_len_product_rule} we considered the composition $\overline{\bbf{X}}$ of the `$2$-coboundary' $\bbf{X}: W \times W \rightarrow \N[\mathfrak{H}]$ with the quotient map $\N[\mathfrak{H}] \rightarrow \N[W\backslash\mathfrak{H}]$. Let us, by abuse of notation, write $\overline{\bbf{X}}$ to also denote the composition of $\overline{\bbf{X}}$ with the evaluation map $\N[W\backslash \mathfrak{H}] \rightarrow R$ sending $\mathbf{a}_{[s]}$ ($s \in S$) to $a_s$. Let us further denote by $\overline{\bbf{X}}$ the composition of $\overline{\bbf{X}}: W \times W \rightarrow R$ with $\pi \times \pi: W^{(1)} \times W^{(1)} \rightarrow W \times W$. With these conventions, we have the following corollary of the previous theorem and of \cref{lem:gamma_product_rule}.
\end{notation}

\begin{cor}\label{cor:hat mult rule}
   For every $w,w' \in W^{(1)}$
   \[ \widehat{\theta}_\mathfrak{o}(w)\widehat{\theta}_{\mathfrak{o}\bullet{}w}(w') = \overline{\bbf{X}}(w,w') \widehat{\theta}_\mathfrak{o}(ww') \]
\end{cor}
\begin{proof}
   By the specialization argument, it suffices to prove this when the $a_s$ are invertible. In this case, the claim follows by combining the identity $\widehat{\theta}_\mathfrak{o}(w) = \overline{\gamma}_\mathfrak{o}(\pi(w)) \theta_\mathfrak{o}(w)$ with the cocycle property of $\theta$ and the equation
   \[ \overline{\gamma}_\mathfrak{o}(\pi(w)) \overline{\gamma}_{\mathfrak{o}\bullet{}w}(\pi(w')) = \overline{\bbf{X}}(w,w') \overline{\gamma}_\mathfrak{o}(\pi(ww')) \]
   following immediately from \cref{lem:gamma_product_rule}.
\end{proof}

\begin{rmk}\label{rmk:hat comm rules}
   Let us record a few relations that will be useful later. First of all, we have that for any $u \in \Omega^{(1)}$ and any orientation $\mathfrak{o} \in \mathcal{O}$
   \[ \widehat{\theta}_\mathfrak{o}(u) = T_u \]
   by construction. This together with \cref{rmk:coboundary_vanishes} and the formula proven in the previous corollary shows that
   \[ \widehat{\theta}_\mathfrak{o}(w) T_u = \widehat{\theta}_\mathfrak{o}(w u) \]
   and that
   \[ T_u \widehat{\theta}_{\mathfrak{o}\bullet{}u}(w) = \widehat{\theta}_\mathfrak{o}(uw) \]
   for any $w \in W^{(1)}$. In particular, we get that
   \[ T_u \widehat{\theta}_\mathfrak{o}(w) T_u^{-1} = \widehat{\theta}_{\mathfrak{o}\bullet{}u^{-1}}(uwu^{-1}) \]
   Moreover, since the group $T$ acts trivial on orientations by definition, for $u = t \in T$ these relations simplify to
   \[ \widehat{\theta}_\mathfrak{o}(w) T_t = \widehat{\theta}_\mathfrak{o}(wt) \]
   and
   \[ T_t \widehat{\theta}_\mathfrak{o}(w) = \widehat{\theta}_\mathfrak{o}(tw) \]
   respectively. Using the conjugation action $w(t) = wtw^{-1}$ of $W^{(1)}$ on $T$, these relations combine to give
   \[ \widehat{\theta}_\mathfrak{o}(w) T_t = T_{w(t)} \widehat{\theta}_\mathfrak{o}(w) \]
   and more generally
   \begin{equation}\label{eq:widehat_t_comm_rule} \widehat{\theta}_\mathfrak{o}(w) b = w(b) \widehat{\theta}_\mathfrak{o}(w) \end{equation}
   for any $b \in R[T] \subseteq \mathcal{H}^{(1)}$.
\end{rmk}

\begin{cor}\label{cor:hat change of basis}
   \[ \widehat{\theta}_\mathfrak{o}(w) = T_w + \sum_{w' < w} c_{w,w'} T_{w'} \]
   for some $c_{w,w'} \in R$, almost all of them being zero. In particular, $(\widehat{\theta}_\mathfrak{o}(w))_{w \in W^{(1)}}$ is an $R$-basis of $\mathcal{H}^{(1)}$.
\end{cor}
\begin{proof}
   The proof is the same as for \cref{prop:change_of_basis}.
\end{proof}

\begin{rmk}\label{rmk:comm subalgebras}
   Consider an orientation $\mathfrak{o}$ and a submonoid $U \leq \op{Stab}_{W^{(1)}}(\mathfrak{o})$. By \cref{cor:hat change of basis}, the $R$-submodule $\mathcal{A}^{(1)}_\mathfrak{o}(U)$ of $\mathcal{H}^{(1)}$ spanned by $\widehat{\theta}_\mathfrak{o}(x)$, $x \in U$ is in fact a free $R$-module on $\{\widehat{\theta}_\mathfrak{o}(x)\}_{x \in U}$. By \cref{cor:hat mult rule}, this submodule $\mathcal{A}^{(1)}_\mathfrak{o}(U) \subseteq \mathcal{H}^{(1)}_\mathfrak{o}$ is also an $R$-subalgebra. When the $a_s$ are units in $R$, the $\theta_\mathfrak{o}(x)$, $x \in U$ provide a different basis of $\mathcal{A}^{(1)}_\mathfrak{o}(U)$ inducing an isomorphism of the monoid algebra $R[U]$ with $\mathcal{A}^{(1)}_\mathfrak{o}(U)$. In particular, $\mathcal{A}^{(1)}_\mathfrak{o}(U)$ is commutative if $U$ is commutative.
   From the specialization argument it follows that this last statement is true even if the $a_s$ are not invertible. In fact, this also follows directly from the product formula (\cref{cor:hat mult rule}) and the fact that
   \[ \overline{\bbf{X}}(w,w') = \overline{\bbf{X}}(w',w) \]
   whenever $ww' = w'w$, which itself follows immediately from formula \eqref{eq:simple_formula_for_coboundary}.
\end{rmk}

The statement of \cref{cor:hat mult rule} says informally that $d\widehat{\theta} = \overline{\bbf{X}}$. The fact that $d\sqrt{\bbf{L}} = \bbf{X}$ suggests that we can restore the cocycle property of $\widehat{\theta}$ by formally twisting it with $\sqrt{\bbf{L}}^{-1}$. This is made precise in the following definition.

\begin{definition}\label{def:normalized_bernstein_map}
   Assume that the parameters $a_s \in R$ are units and squares in $R$. Recall that $a_s$ only depends on the class $[s] \in W\backslash \mathfrak{H}$. For every such class $[s]$ choose a square root $\sqrt{a_{[s]}} \in R^\times$ of $a_{[s]} = a_s$ and let
   \[ \widetilde{\theta}_\mathfrak{o}(w) := \overline{\sqrt{\bbf{L}}}(w)^{-1} \widehat{\theta}_\mathfrak{o}(w) \quad \forall \mathfrak{o} \in \mathcal{O},\ w \in W^{(1)} \]
   Here $\overline{\sqrt{\bbf{L}}}$ denotes the composition of maps
   \[ W^{(1)} \stackrel{\pi}{\longrightarrow} W \stackrel{\sqrt{\bbf{L}}}{\longrightarrow} \Z[\sqrt{\mathfrak{H}}] \longrightarrow R \]
   where $\sqrt{\bbf{L}}: W \rightarrow \Z[\sqrt{\mathfrak{H}}]$ is the formal square root defined in \cref{rmk:gen_len_product_rule} of the generalized length function $\bbf{L}$ and
   \[ \Z[\sqrt{\mathfrak{H}}] \longrightarrow R \]
   is the morphism of monoids sending a formal square $\sqrt{\mathbf{a}_H}$ to $\sqrt{a_{[H]}}$.
   
   The map
   \[ \widetilde{\theta}: W^{(1)} \longrightarrow \op{Hom}_{\op{Set}}(\mathcal{O},\mathcal{H}^{(1)}),\quad w \longmapsto (\mathfrak{o} \mapsto \widetilde{\theta}_\mathfrak{o}(w)) \]
   is called the \textbf{normalized Bernstein map} (with respect to the chosen square roots $\sqrt{a_s}$).
\end{definition}

In the situation of the above definition, we have the following immediate corollary of \cref{cor:hat mult rule} and \cref{rmk:gen_len_product_rule}.
\begin{cor}\label{cor:normalized_bernstein_mult_rule}
   For all $w,w' \in W^{(1)}$ and $\mathfrak{o} \in \mathcal{O}$
   \[ \widetilde{\theta}_\mathfrak{o}(ww') = \widetilde{\theta}_\mathfrak{o}(w) \widetilde{\theta}_{\mathfrak{o}\bullet{}w}(w') \]
\end{cor}

\begin{rmk}\label{rmk:usefulness_of_normalized_bernstein}
   For our purposes the main reason for introducing the normalized Bernstein map lies in the fact that it get transformed into the integral Bernstein map under a certain isomorphism of Hecke algebras. More precisely, in the situation of the above definition we have an isomorphism
   \[ \varphi: \mathcal{H}^{(1)}(a_s,b_s) \stackrel{\sim}{\longrightarrow} \mathcal{H}^{(1)}(1,\sqrt{a_s}^{-1}b_s) \]
   of $R$-modules determined by $T_w \mapsto \overline{\sqrt{\bbf{L}}}(w) T_w$. Note that $\mathcal{H}^{(1)}(1,\sqrt{a_s}^{-1}b_s)$ is well-defined as the parameters again satisfy the conditions of \cref{thm:exhecke}. This isomorphism is also an isomorphism of $R$-algebras, which follows easily by combining the presentation of $\mathcal{H}^{(1)}(a_s,b_s)$ given in \cref{sub:Presentations of generic pro-$p$ Hecke algebras via (generalized) braid groups} with \cref{rmk:coboundary,rmk:gen_len_product_rule} and verifying the following quadratic relation
   \[ (\sqrt{a}_s T_{n_s})^2 = a_s T_{n_s^2} + (\sqrt{a}_s T_{n_s})b_s \]
   in the Hecke algebra $\mathcal{H}^{(1)}(1,\sqrt{a_s}^{-1}b_s)$.
   
   The normalized Bernstein map $\widetilde{\theta}$ of $\mathcal{H}^{(1)}(a_s,b_s)$ and the integral Bernstein map $\widehat{\theta}$ of $\mathcal{H}^{(1)}(1,\sqrt{a_s}^{-1}b_s)$ are now related as follows
   \begin{equation}\label{eq:rel_integral_and_normalized} \widehat{\theta}_\mathfrak{o}(w) = \varphi(\widetilde{\theta}_\mathfrak{o}(w)) \quad \forall \mathfrak{o} \in \mathcal{O},\ w \in W^{(1)} \end{equation}
\end{rmk}

\subsection{Bernstein relations} 
\label{sub:Bernstein relations}
In this section we again fix a generic pro-$p$ Hecke algebra $\mathcal{H}^{(1)} = \mathcal{H}^{(1)}(a_s,b_s)$ and assume that the parameters $a_s \in R$ are units and squares in $R$, and that a choice of square roots $\sqrt{a_s}$ and consequently of a normalized Bernstein map $\widetilde{\theta}$ has been made according to the previous section.

The goal of this section is to compute the difference
\begin{equation}\label{eq:or_diff} \widetilde{\theta}_\mathfrak{o}(w) - \widetilde{\theta}_{\mathfrak{o}'}(w) \end{equation}
as a sum over certain hyperplanes, for two orientations $\mathfrak{o},\mathfrak{o}' \in \mathcal{O}$ that are `adjacent'. This computation will be crucial in \cref{sec:Affine pro-$p$ Hecke algebras}, where we will use it to show that certain elements $z_\mathfrak{o}(\gamma)$ of an affine pro-$p$ Hecke algebra lie in the center. In the classical case ($W^{(1)} = W$) this computation is essentially equivalent to Bernstein's relations for the Iwahori-Hecke algebra.

We remind the reader that (see \cref{term:basic})
\[ \mathfrak{H} = \{ wsw^{-1} : s \in S,\ w \in W_{\op{aff}} \} = \{ wsw^{-1} : s \in S,\ w \in W \} \subseteq W_{\op{aff}} \]
denotes the set of \textit{hyperplanes} of the underlying Coxeter group $W_{\op{aff}}$ of $W^{(1)}$.

The next proposition introduces some canonical elements in the generic pro-$p$ Hecke algebra, which will appear in the sum expansion of expression \eqref{eq:or_diff}.

\begin{propdef}\label{propdef:bernstein}
   For any hyperplane $H \in \mathfrak{H}$ and any orientation $\mathfrak{o} \in \mathcal{O}$, there exists a unique element $\Xi_\mathfrak{o}(H) \in \mathcal{H}^{(1)}$, such that if $s \in S$, $w \in W^{(1)}$ with
   \[ \pi(wn_s w^{-1}) = H \]
   then
   \[ \Xi_\mathfrak{o}(H) = \sqrt{a_s}^{-1} w(b_s) \cdot{} \widetilde{\theta}_\mathfrak{o}(w n_s^{-1} w^{-1}) = \sqrt{a_s}^{-1} \widetilde{\theta}_\mathfrak{o}(w n_s^{-1} w^{-1})\cdot{} w(b_s) \]
\end{propdef}
\begin{proof}
   Applying the isomorphism $\varphi$ of \cref{rmk:usefulness_of_normalized_bernstein}, we may assume that $a_s = 1$ for all $s \in S$ and that $\widetilde{\theta} = \widehat{\theta}$. Moreover, we observe that $w(b_s)$ and $\widehat{\theta}_\mathfrak{o}(wn_s^{-1} w^{-1})$ commute with each other. Indeed, by applying the commutation relation \eqref{eq:widehat_t_comm_rule} this is easily reduced to show the basic identity
   \[ n_s^{-1}(b_s) = b_s \]
   which was already seen to be true in \eqref{eq:paraminv}. Therefore, it only remains to show that the expression
   \[ w(b_s) \cdot{} \widehat{\theta}_\mathfrak{o}(w n_s^{-1} w^{-1}) = \widehat{\theta}_\mathfrak{o}(w n_s^{-1} w^{-1})\cdot{} w(b_s) \]
   only depends on the element
   \[ \pi(w n_s w^{-1}) = H \in \mathfrak{H} \]
   and not on the choice of $w \in W^{(1)}$ and $s \in S$.
   So let $w_1,w_2 \in W^{(1)}$ and $s,t \in S$ with
   \[ \pi(w_1 n_s w_1^{-1}) = \pi(w_2 n_t w_2^{-1}) \]
   By the above equation, we may apply condition \eqref{eq:condstar} of \cref{thm:exhecke} on the existence of generic pro-$p$ Hecke algebras to $w = w_1^{-1}w_2$ (in the notation of said theorem). Condition \eqref{eq:condstar} then states that
   \[ (n_s w n_t^{-1} w^{-1}) \cdot{} w(b_t) = b_s \]
   as an equality in $R[T]$. Acting on both sides with $w_1$, we get the formula
   \[ w_1(b_s) = (w_1 n_s w_1^{-1} w_2 n_t^{-1} w_2^{-1}) \cdot{} w_2(b_t) \]
   Bearing in mind that $w_1 n_s w_1^{-1} w_2 n_t^{-1} w_2^{-1} \in T$, we can use the relations proved in \cref{rmk:hat comm rules} to compute
   \begin{align*} \widehat{\theta}_\mathfrak{o}(w_1 n_s^{-1} w_1^{-1})\cdot{} w_1(b_s) & = \widehat{\theta}_\mathfrak{o}(w_1 n_s^{-1} w_1^{-1})\cdot{}(w_1 n_s w_1^{-1} w_2 n_t^{-1} w_2^{-1}) \cdot{} w_2(b_t) \\
      & = \widehat{\theta}_\mathfrak{o}(w_1 n_s^{-1} w_1^{-1} w_1 n_s w_1^{-1} w_2 n_t^{-1} w_2^{-1}) \cdot{} w_2(b_t) \\
      & = \widehat{\theta}_\mathfrak{o}(w_2 n_t^{-1} w_2^{-1}) \cdot{} w_2(b_t)
   \end{align*}
\end{proof}

The classical Bernstein relations compute the difference \eqref{eq:or_diff} when $\mathfrak{o}' = \mathfrak{o}\bullet{}s_\alpha$ for a `simple root' $\alpha$ of a root system and the orientation $\mathfrak{o}$ is `spherical' (cf. \cref{def:spherical_orientation}). The following definition allows us to state the Bernstein relations in a more general context.

Recall that by \cref{lem:or_depend_on_half-space}, an orientation $\mathfrak{o}$ of a Coxeter group is given by defining for every hyperplane $H \in \mathfrak{H}$ a notion of \textit{positive/negative crossing} for passing from one half-space (with respect to $H$) into the other. It therefore makes sense to say that two orientation agree (or disagree) at a hyperplane $H$ if the signs attached by the orientations to passing from one half-space with respect to $H$ into the other are equal (or unequal).

\begin{definition}\label{def:adjacency_of_or}
   Two orientations $\mathfrak{o},\mathfrak{o}' \in \mathcal{O}$ of $W$ are said to be \textbf{adjacent} if for every wall $H \in \mathfrak{H}$ at which $\mathfrak{o}$ and $\mathfrak{o}'$ disagree, we have
   \[ \mathfrak{o}\bullet{}s_H = \mathfrak{o}' \]
\end{definition}
Note that the notion of adjacency is symmetric in $\mathfrak{o}$ and $\mathfrak{o}'$. We are now ready to give the `Bernstein relation'.
\begin{theorem}\label{thm:bernstein_relation}
   Let $w \in W^{(1)}$ and $\mathfrak{o}, \mathfrak{o}' \in \mathcal{O}$ be adjacent. Then
   \begin{equation}\label{eq:bernstein_relation} \widetilde{\theta}_\mathfrak{o}(w)-\widetilde{\theta}_{\mathfrak{o}'}(w) = \left(\sum_H \mathfrak{o}(1,H) \Xi_{\mathfrak{o}'}(H)\right) \widetilde{\theta}_\mathfrak{o}(w) \end{equation}
      where the sum is taken over all hyperplanes $H \in \mathfrak{H}$ which separate $1$ and $w$, and at which $\mathfrak{o}$ and $\mathfrak{o}'$ disagree.
\end{theorem}
\begin{proof}
   We may again invoke \cref{rmk:usefulness_of_normalized_bernstein} to reduce to the case $a_s = 1$ and $\widetilde{\theta} = \widehat{\theta} = \theta$. Now take \textit{any} (not necessarily reduced) expression
   \[ w = n_{s_1}\ldots n_{s_r} u, \quad s_i \in S,\ u \in \Omega^{(1)} \]
   Using this expression, the cocycle rule and the definition of the Bernstein map together give the following explicit expressions
   \[ \widehat{\theta}_\mathfrak{o}(w) = T_{n_{s_1}^{\varepsilon_1}}^{\varepsilon_1} \ldots T_{n_{s_r}^{\varepsilon_r}}^{\varepsilon_r} T_u,\quad \widehat{\theta}_{\mathfrak{o}'}(w) = T_{n_{s_1}^{\varepsilon_1'}}^{\varepsilon_1'}\ldots T_{n_{s_r}^{\varepsilon_r'}}^{\varepsilon_r'} T_u \]
   where
   \[ \varepsilon_i = \mathfrak{o}(s_1\ldots s_{i-1},s_i),\quad \varepsilon_i' = \mathfrak{o}'(s_1\ldots s_{i-1},s_i) \]
   We expand the difference $\widehat{\theta}_\mathfrak{o}(w) - \widehat{\theta}_{\mathfrak{o}'}(w)$ now as a telescopic sum
   \begin{align*}
      \widehat{\theta}_\mathfrak{o}(w) - \widehat{\theta}_{\mathfrak{o}'}(w) & = \sum_{i = 1}^r T_{n_{s_1}^{\varepsilon'_1}}^{\varepsilon'_1} \ldots T_{n_{s_{i-1}}^{\varepsilon'_{i-1}}}^{\varepsilon'_{i-1}} \left(T_{n_{s_i}^{\varepsilon_i}}^{\varepsilon_i}-T_{n_{s_i}^{\varepsilon'_i}}^{\varepsilon'_i}\right) T_{n_{s_{i+1}}^{\varepsilon_{i+1}}}^{\varepsilon_{i+1}} \ldots T_{n_{s_r}^{\varepsilon_r}}^{\varepsilon_r} T_u
   \end{align*}
   In this sum the $i$-th summand vanishes unless $\varepsilon_i \neq \varepsilon'_i$, so let us fix an index $i$ where $\varepsilon_i \neq \varepsilon'_i$. Observing that
   \[ T_{n_s^{\varepsilon}}^{\varepsilon} - T_{n_s^{-\varepsilon}}^{-\varepsilon} = \varepsilon b_s \quad \forall s \in S,\ \varepsilon \in \{\pm \} \]
   and using the commutation rule (cf. \eqref{eq:widehat_t_comm_rule})
   \[ \widehat{\theta}_\mathfrak{o}(w) b = w(b) \widehat{\theta}_\mathfrak{o}(w) \quad \forall w \in W^{(1)},\ b \in R[T] \]
   we see that the $i$-th summand can be rewritten as
   \[ \varepsilon_i \widetilde{w}(b_{s_i}) T_{n_{s_1}^{\varepsilon'_1}}^{\varepsilon'_1} \ldots T_{n_{s_{i-1}}^{\varepsilon'_{i-1}}}^{\varepsilon'_{i-1}} T_{n_{s_{i+1}}^{\varepsilon_{i+1}}}^{\varepsilon_{i+1}} \ldots T_{n_{s_r}^{\varepsilon_r}}^{\varepsilon_r} T_u \]
   where we have put
   \[ \widetilde{w} := n_{s_1}\ldots n_{s_{i-1}} \]
   Since $\mathfrak{o}$ and $\mathfrak{o}'$ disagree at
   \[ s_H = H := \pi(\widetilde{w} n_{s_i} \widetilde{w}^{-1}) = (s_1\ldots s_{i-1}) s_i (s_1\ldots s_{i-1})^{-1} \]
   and $\mathfrak{o}$ and $\mathfrak{o}'$ are adjacent, we have
   \[ \mathfrak{o}' = \mathfrak{o}\bullet{}s_H \]
   In particular, for $j > i$ we have
   \begin{align*} \varepsilon_j & = \mathfrak{o}(s_1\ldots s_{j-1},s_j) \\
      & = \mathfrak{o}(s_H s_1 \ldots \widehat{s_i} \ldots s_{j-1},s_j) \\
      & = (\mathfrak{o}\bullet{}s_H)(s_1\ldots \widehat{s_i} \ldots s_{j-1},s_j) \\
      & = \mathfrak{o}'(s_1\ldots \widehat{s_i} \ldots s_{j-1},s_j)
   \end{align*}
   This implies that
   \begin{align*} T_{n_{s_1}^{\varepsilon'_1}}^{\varepsilon'_1} \ldots T_{n_{s_{i-1}}^{\varepsilon'_{i-1}}}^{\varepsilon'_{i-1}} T_{n_{s_{i+1}}^{\varepsilon_{i+1}}}^{\varepsilon_{i+1}} \ldots T_{n_{s_r}^{\varepsilon_r}}^{\varepsilon_r} T_u & \stackrel{\text{(!)}}{=} \widehat{\theta}_{\mathfrak{o}'}(n_{s_1}\ldots \widehat{n_{s_i}} \ldots n_{s_r} u) \\
      & = \widehat{\theta}_{\mathfrak{o}'}(\widetilde{w} n_{s_i}^{-1} \widetilde{w}^{-1} w) \\
      & = \widehat{\theta}_{\mathfrak{o}'}(\widetilde{w} n_{s_i}^{-1} \widetilde{w}^{-1}) \widehat{\theta}_{\mathfrak{o}'\bullet{}H}(w) \\
      & = \widehat{\theta}_{\mathfrak{o}'}(\widetilde{w} n_{s_i}^{-1} \widetilde{w}^{-1}) \widehat{\theta}_{\mathfrak{o}}(w)
   \end{align*}
   Recalling \cref{propdef:bernstein}, we see that
   \[ \widetilde{w}(b_{s_i}) \widehat{\theta}_{\mathfrak{o}'}(\widetilde{w} n_{s_i}^{-1} \widetilde{w}^{-1}) = \Xi_{\mathfrak{o}'}(H) \]
   and therefore
   \[ \widehat{\theta}_{\mathfrak{o}}(w) - \widehat{\theta}_{\mathfrak{o}'}(w) = \sum_{\substack{i \in \{1,\ldots,r\} \\ \varepsilon_i \neq \varepsilon'_i}} \varepsilon_i \Xi_{\mathfrak{o}'}(H_i) \widehat{\theta}_{\mathfrak{o}}(w) \]
   where
   \[ H_i := (s_1\ldots s_{i-1}) s_i (s_1\ldots s_{i-1})^{-1} \]
   is the hyperplane crossed by the gallery $(s_1,\ldots,s_r)$ in the $i$-th step. Until now we have not assumed this gallery, i.e. the expression
   \[ w = n_{s_1} \ldots n_{s_r} u \]
   to be reduced. Assume now that this is the case. Then the hyperplanes crossed by the gallery $(s_1,\ldots,s_r)$ are exactly the hyperplanes separating $1$ and $w$. Moreover, in this case we have
   \[ \varepsilon_i = \mathfrak{o}(1,H_i) \]
   and hence the theorem follows.
\end{proof}

As already mentioned, the `Bernstein relation' proven above will be used to show that certain elements of affine pro-$p$ Hecke algebras lie in the center. This application of the Bernstein relation will involve showing that
\[ \widehat{\theta}_\mathfrak{o}(x) - \widehat{\theta}_{\mathfrak{o}\bullet{}s_\alpha}(x) = -\left( \widehat{\theta}_\mathfrak{o}(s_\alpha(x)) - \widehat{\theta}_{\mathfrak{o}\bullet{}s_\alpha}(s_\alpha(x))\right) \]
for $x$ an element of a certain subgroup $X^{(1)} \subseteq W^{(1)}$ and $s_\alpha \in W$ a reflection associated to a simple root $\alpha$. This will follow from the above theorem and the following elementary property of the elements $\Xi_\mathfrak{o}(H)$.

\begin{lemma}\label{lem:property_of_bernstein}
   Let $H \in \mathfrak{H}$, $\mathfrak{o}$ an orientation and $x \in C_{W^{(1)}}(T)$ an element of the centralizer of $T$ in $W^{(1)}$. Then we have that
   \[ \Xi_{\mathfrak{o}}(H) \cdot{} \widetilde{\theta}_{\mathfrak{o}\bullet{}s_H}(s_H(x)x^{-1}) = \Xi_{\mathfrak{o}}(\pi(x)H\pi(x)^{-1}) \]
   where $s_H(x)$ denotes the induced action of $W$ on $C_{W^{(1)}}(T)$ by conjugation.
\end{lemma}
\begin{proof} Letting $w \in W^{(1)}$ and $s \in S$ be such that
   \[ H = \pi(w n_s w^{-1}) \]
   we have by definition that
   \begin{align*}
      \Xi_{\mathfrak{o}}(H) & = \sqrt{a_s}^{-1} w(b_s) \widetilde{\theta}_\mathfrak{o}(w n_s^{-1} w^{-1})
   \end{align*}
   Since $w n_s^{-1} w^{-1}$ acts both on $\mathfrak{o}$ and on $s_H(x)$ via $s_H$, we have that
   \begin{align*}
      \widetilde{\theta}_\mathfrak{o}(w n_s^{-1} w^{-1}) \widetilde{\theta}_{\mathfrak{o}\bullet{}s_H}(s_H(x)x^{-1}) & = \widetilde{\theta}_{\mathfrak{o}}(w n_s^{-1} w^{-1} s_H(x) x^{-1}) \stackrel{\text{(!)}}{=} \widetilde{\theta}_\mathfrak{o}(x w n_s^{-1} w^{-1} x^{-1})
   \end{align*}
   Therefore
   \begin{align*} \Xi_{\mathfrak{o}}(H) \cdot{} \widetilde{\theta}_{\mathfrak{o}\bullet{}s_H}(s_H(x)x^{-1}) & = \sqrt{a_s}^{-1} w(b_s) \widetilde{\theta}_\mathfrak{o}(x w n_s^{-1} w^{-1} x^{-1}) \\
      & = \sqrt{a_s}^{-1} (xw)(b_s) \widetilde{\theta}_\mathfrak{o}(xw n_s^{-1} w^{-1} x^{-1}) \\
      & = \Xi_\mathfrak{o}(\pi(x)H\pi(x)^{-1})
   \end{align*}
   where we have used the fact that $x$ acts trivially on $T$ on the second line, and the definition of $\Xi$ on the third line.
\end{proof}


\section{Affine pro-\texorpdfstring{$p$}{p} Hecke algebras} 
\label{sec:Affine pro-$p$ Hecke algebras}

In this final section we want to apply the general theory developed in the preceding sections to the study of a special class of generic pro-$p$ Hecke algebras, the `\textit{affine} pro-$p$ Hecke algebras'. We will give a description of the center of these algebras and prove that there they are module-finite over their center in \cref{sub:The structure of affine pro-$p$ Hecke algebras}, recovering classical results of Bernstein-Zelevinsky in the case of $W = W^{(1)}$. 

In order to obtain these results, we need to assume that the group $W$ is of a special form. Basically we need $W$ to be a semi-direct product $W = X\rtimes W_0$ of a finitely generated commutative group $X$ and a finite reflection group $W_0$. Moreover, we need to assume that there exists a representation of $W$ as a group of isometries preserving a locally finite affine hyperplane arrangement which is compatible with the abstract decompositions $W = W_{\op{aff}} \rtimes \Omega$ and $W = X\rtimes W_0$. Finally we need to assume that $X$ is `large enough' with respect to this representation. This will be made precise in the next section.

\subsection{Affine extended Coxeter groups and affine pro-\texorpdfstring{$p$}{p} Hecke algebras} 
\label{sub:Affine extended Coxeter groups and affine pro-$p$ Hecke algebras}

Before we give the definition of an \textit{affine} extended Coxeter group, let us introduce some notations and recall some basic facts from the theory of affine reflection groups (see for instance \cite[Ch. V, \textsection 1]{Bourbaki}).

Given a finite dimensional euclidean vector space $V$ and a hyperplane\footnote{Unless specified otherwise, hyperplane means affine hyperplane.} $H \leq V$, there exists a unique element $s_H \in \op{Aut}_{\op{Euclid}}(V)$ of the group of euclidean motions such that $s_H \neq \op{id}$ and $s_H$ operates on $H$ as the identity. This element $s_H$ is called the \textit{reflection} with respect to $H$. Given a set $\mathfrak{H}$ of hyperplanes in $V$, we let
\[ W(\mathfrak{H}) = \left< s_H : H \in \mathfrak{H}\right> \leq \op{Aut}_{\op{Euclid}}(V) \]
denote the group generated by the reflections with respect to the hyperplanes in $\mathfrak{H}$. If $\alpha \in V^\vee$ is a non-zero functional and $k \in \R$, we write
\[ H_{\alpha,k} := \{ x \in V : \alpha(x) + k = 0 \} \]
and $s_{\alpha,k} := s_{H_{\alpha,k}}$, $s_{\alpha} := s_{\alpha,0}$.

A point $x \in V$ is called \textit{special} with respect to $\mathfrak{H}$ if for every $H \in \mathfrak{H}$ there exists a hyperplane $H' \in \mathfrak{H}$ parallel to $H$ with $x \in H'$. A set $\mathfrak{H}$ of hyperplanes in $V$ is called \textit{locally finite} if for every $x \in X$ there exists a neighbourhood $U$ of $x$ such that $\{ H \in \mathfrak{H} : H\cap U \neq \emptyset \}$ is finite.

Assume that a locally finite set $\mathfrak{H}$ of hyperplanes on $V$ is given. The elements of the set
\[ \mathfrak{C} := \pi_0(V - \bigcup_{H \in \mathfrak{H}} H) \]
of connected components of the complement of all hyperplanes are called \textit{chambers}\footnote{It is common to use the term \textit{alcove} instead of chamber if the hyperplanes $H \in \mathfrak{H}$ aren't all linear, but we will not make this distinction.}. A hyperplane $H \in \mathfrak{H}$ is called a \textit{wall} of a chamber $C$ if $H \cap \overline{C}$ has non-empty interior as a subset of $H$, or equivalently, if the affine span of $H\cap \overline{C}$ equals $H$. We let
\[ S(C) := \{ H \in \mathfrak{H} : \text{$H$ wall of $C$} \} \]
denote the set of all walls of $C$. If the group $W(\mathfrak{H})$ leaves the set $\mathfrak{H}$ invariant, then it follows\footnote{In \cite{Bourbaki} is assumed that the group $W(\mathfrak{H})$ acts properly discontinuously, and the local finiteness is deduced as a consequence. However, it is enough to only assume that $\mathfrak{H}$ is locally finite and $W(\mathfrak{H})$ preserves $\mathfrak{H}$, as these assumptions already imply that $W(\mathfrak{H})$ acts properly discontinuously.} that for every chamber $C$ the pair $(W(\mathfrak{H}),\{ s_H : H \in S(C)\})$ is a Coxeter group (cf. \cite[Ch. V, \textsection 3.2, Théorème 1]{Bourbaki}) and that $S(C)$ is finite (cf. \cite[Ch. V, \textsection 3.6, Théorème 3]{Bourbaki}).

We will now give the definition of `affine' extended Coxeter groups.

\begin{definition}\label{def:affine_coxeter_group}
   An \textbf{affine extended Coxeter group} $W$ consists of a group $W$ together with a homomorphism
   \[ \rho: W \longrightarrow \op{Aut}_{\op{aff}}(V) \]
   of $W$ into the group of affine automorphisms of a finite-dimensional real vector space $V$, a locally finite set $\mathfrak{H}$ of (affine) hyperplanes in $V$, a chamber $C_0 \in \pi_0(V-\bigcup_{H \in \mathfrak{H}} H)$ and for every $H \in \mathfrak{H}$ an element $\widetilde{s}_H \in W$ such that the following hold.

   \begingroup 
   \renewcommand\arraystretch{2}
   \renewcommand{\tabcolsep}{0pt}
   \begin{longtable}[l]{l @{\extracolsep{1cm}} p{13cm}}
      (\textbf{ACI}) \label{axiom:ACI} & $W$ leaves $\mathfrak{H}$ invariant, i.e. $\rho(w)(H) \in \mathfrak{H}$ for all $w \in W$ and $H \in \mathfrak{H}$. \\
      (\textbf{ACII}) \label{axiom:ACII} & For every $H \in \mathfrak{H}$, $\rho(\widetilde{s}_H)$ is a \textit{reflection} fixing $H$. \\
      (\textbf{ACIII}) \label{axiom:ACIII} & Letting $\rho_0$ denote the composition of $\rho$ with the projection
         \[ \op{Aut}_{\op{aff}}(V) = V \rtimes \op{GL}(V) \longrightarrow \op{GL}(V) \]
         onto the linear part, the group
         \[ W_0 := \rho_0(W) \]
         is \textit{finite}. \\
      (\textbf{ACIV}) \label{axiom:ACIV} & $0 \in V$ is a special point of $\mathfrak{H}$. \\
      (\textbf{ACV}) \label{axiom:ACV} & The subgroup $\rho(W) \cap V$ of translations in $\rho(W)$ generates the quotient $V/L$ as an $\R$-vector space, where
      \[ L = \bigcap_{H \in \mathfrak{H},\ 0 \in H} H \] \\
      (\textbf{ACVI}) \label{axiom:ACVI} & For every $H \in \mathfrak{H}$ and $w \in W$ we have
         \[ w\widetilde{s}_H w^{-1} = \widetilde{s}_{w(H)} \]
         where we abbreviate $w(H) = \rho(w)(H)$. \\
      (\textbf{ACVII}) \label{axiom:ACVII} & For every pair $H_1, H_2 \in S(C_0)$ of walls of $C_0$ such that $\rho(\widetilde{s}_{H_1}\widetilde{s}_{H_2})$ is of finite order $m_{1,2}$, we have the relation
         \[ (\widetilde{s}_{H_1} \widetilde{s}_{H_2})^{m_{1,2}} = 1 \]
         in $W$. \\
      (\textbf{ACVIII}) \label{axiom:ACVIII} & The group $W_0$ is generated by the images of the $\widetilde{s}_H$, $H \in \mathfrak{H}$ under the natural map $W \rightarrow W_0$. \\
      (\textbf{ACIX}) \label{axiom:ACIX} & $0 \in \overline{C_0}$ \\
      (\textbf{ACX}) \label{axiom:ACX} & Let
         \[ X = \rho^{-1}(V) \leq W \]
         denote the subgroup of all elements of $W$ which are mapped to a translation under $\rho$. Then $X$ is finitely generated and commutative.
   \end{longtable}
   \endgroup
\end{definition}

Note that given the remaining axioms, \axiom{ACIV} and \axiom{ACIX} are always satisfied up to a translation and only serve to fix notation. The rationale behind the above definition of an \textit{affine extended Coxeter group} is to have a set of axioms which are easy to verify in examples. However, as it stands the definition does not even mention extended Coxeter groups. Our first task will therefore be to `unpack' this definition.

\begin{lemma}\label{lem:unpacking_affine_extended_Coxeter_groups}
   Let $W = (W,V,\rho,\mathfrak{H},C_0,(\widetilde{s}_H)_H)$ be an affine extended Coxeter group. Let
   \[ W_{\op{aff}} := \left<\widetilde{s}_H : H \in \mathfrak{H}\right>,\quad S := \{ \widetilde{s}_H : H \in S(C_0) \} \]
   and
   \[ \Omega := \op{Stab}_W(C_0) \]
   Then the following holds.
   \begin{enumerate}
      \item[(i)] There exists a positive definite scalar product on $V$ invariant with respect to $W_0$, i.e. such that $W$ acts by euclidean motions.
      \item[(ii)] $(W_{\op{aff}},S)$ is a Coxeter group and for any choice of an invariant scalar product, $\rho$ induces an isomorphism
         \[ (W_{\op{aff}},S) \stackrel{\sim}{\longrightarrow} (W(\mathfrak{H}),\{ s_H : H \in S(C_0) \}) \]
         of Coxeter groups, where $s_H$ denotes the orthogonal reflection with respect to $H$ and $W(\mathfrak{H})$ denotes the group generated by $s_H$ for $H \in \mathfrak{H}$. In particular, $W(\mathfrak{H})$ and the $s_H$ do not depend on the choice of the scalar product.
      \item[(iii)] $(W,W_{\op{aff}},S,\Omega)$ is an extended Coxeter group.
      \item[(iv)] The group $W_0$ is equal to the special subgroup of $(W(\mathfrak{H}), \{ s_H : H \in S(C_0)\})$ generated by the $s_H$ with $0 \in H$. In particular, $(W_0, \{ s_H : H \in S(C_0),\ 0 \in H \})$ is a Coxeter group. Moreover, the subspace $L \leq V$ of \axiom{ACV} is given by
         \[ L = \bigcap_{H \in \mathfrak{H},\ 0 \in H} H = V^{W_0} \]
      \item[(v)] Let
         \[ \Phi := \{ \alpha \in V^\vee : \forall k \in \R\ \ H_{\alpha,k} \in \mathfrak{H} \Leftrightarrow k \in \Z \} \]
         Then $(\R \Phi, \Phi)$ is a reduced root system and
         \[ \mathfrak{H} = \{ H_{\alpha,k} : \alpha \in \Phi,\ k \in \Z \} \]
         Moreover, $V \stackrel{\sim}{\longrightarrow} V^{\vee \vee}$ induces an isomorphism
         \[ V/L \stackrel{\sim}{\longrightarrow} (\R \Phi)^\vee \]
      \item[(vi)]
      The map
      \[ W_0 \longrightarrow \op{GL}(V/L)\simeq \op{GL}((\R\Phi)^\vee) \]
      induced by $\rho_0: W_0 \rightarrow \op{GL}(V)$ is injective and identifies $W_0$ with the Weyl group $W(\Phi^\vee)$ of the dual root system $((\R \Phi)^\vee, \Phi^\vee)$. Moreover, this is an identification of Coxeter groups if we endow $W(\Phi^\vee)$ with the generating set $\{ s^\vee_\alpha : \alpha \in \Delta \}$ corresponding to the basis
      \[ \Delta = \{ \alpha \in \Phi : H_\alpha \in S(C_0),\ \left.\alpha\right|_{C_0} > 0 \} \]
      The basis $\Delta$ corresponds to the positive root system $\Phi^+ \subseteq \Phi$ given by
      \[ \Phi^+ = \{ \alpha \in \Phi: \alpha(x) > 0 \} \]
      where $x \in C_0$ is arbitrary.

   \item[(vii)] The exact sequence
      \[ \begin{xy} \xymatrix{ 0 \ar[r] & X \ar[r] & W \ar[r]^{\rho_0} & W_0 \ar[r] & 1 } \end{xy} \]
      splits via the map $W_0 \rightarrow W$ given by the composition
      \[ W_0 \subseteq W(\mathfrak{H}) \stackrel{\rho^{-1}}{\longrightarrow} W_{\op{aff}} \subseteq W \]
      Viewing $W_0$ as a subgroup of $W_{\op{aff}}$ via this splitting, $W_0$ equals the special subgroup of $(W_{\op{aff}},S)$ generated by
      \[ S_0 := \{ s_\alpha : \alpha \in \Delta \} = \{ \widetilde{s}_H : H \in S(C_0),\ 0 \in H \} \subseteq S \]
      where $s_\alpha := \widetilde{s}_{H_{\alpha,0}}$.
   \end{enumerate}
\end{lemma}
\begin{proof}
   Point (i) follows immediately from the finiteness of $W_0$, since given any positive definite scalar product $B: V \times V \rightarrow \R$, the expression
   \[ (x,y) := \sum_{w \in W_0} B(w(x),w(y)),\quad x,y \in V \]
   defines a $W_0$-invariant positive definite scalar product. To prove (ii) we may assume a $W_0$-invariant scalar product has been fixed. In this case we may invoke \cite[Ch. V, \textsection 3.2, Théorème 1]{Bourbaki} to conclude that $(W(\mathfrak{H}),\{ s_H : H \in S(C_0)\})$ is a Coxeter group. Since $\rho(\widetilde{s}_H)$ is a reflection fixing $H$ by \axiom{ACII} and $W$ acts by euclidean motions with respect to the chosen scalar product, we must have $\rho(\widetilde{s}_H) = s_H$ for every $H \in \mathfrak{H}$. Since $W_{\op{aff}}$ is generated by the $\widetilde{s}_H$, this shows that we have a well-defined group homomorphism
   \[ \rho: W_{\op{aff}} \longrightarrow W(\mathfrak{H}) \]
   that moreover maps $S$ into $S(C_0)$. Since $W(\mathfrak{H},S(C_0))$ is a Coxeter group, by one of the various characterizations (\cite[Ch. IV, \textsection 1.3, Définition 3]{Bourbaki}) of Coxeter groups, $W(\mathfrak{H})$ has a presentation
   \[ W(\mathfrak{H}) = \left< s_H,\ H \in S(C_0)\ |\ (s_{H} s_{H'})^m = 1\text{ if $m = \op{ord}(s_{H}s_{H'}) < \infty$} \right> \]
   and hence by property \axiom{ACVII} there exists a unique homomorphism
   \[ \varphi: W(\mathfrak{H}) \longrightarrow W_{\op{aff}} \]
   of groups with $\varphi(s_H) = \widetilde{s}_H$ for every $H \in S(C_0)$. Since $\rho \circ{} \varphi = \op{id}$, it follows that $\varphi$ is injective. We claim that $\varphi$ is also surjective, or equivalently that $S$ generates $W_{\op{aff}}$. From the theory of affine reflection groups it follows (cf. \cite[Ch. V, \textsection 3.2, Corollaire]{Bourbaki}) that for every $H \in \mathfrak{H}$ there exists an element $w \in W(\mathfrak{H})$ and a wall $H' \in S(C_0)$ such that $w s_{H'} w^{-1} = s_H$ or equivalently $w(H') = H$. Writing
   \[ w = s_{H_1}\ldots s_{H_r},\quad H_i \in S(C_0) \]
   and putting
   \[ \widetilde{w} := \widetilde{s}_{H_1}\ldots \widetilde{s}_{H_r} \in \left<S\right> \subseteq W_{\op{aff}} \]
   we have $\rho(\widetilde{w}) = w$ and hence by \axiom{ACVI}
   \[ \widetilde{w}\widetilde{s}_{H'} \widetilde{w}^{-1} = \widetilde{s}_{\rho(\widetilde{w})(H')} = \widetilde{s}_H \]
   lies in the subgroup of $W_{\op{aff}}$ generated by $S$. Since $W_{\op{aff}}$ is generated by the $\widetilde{s}_H$, it follows that $\left<S\right> = W_{\op{aff}}$ and hence that $\varphi$ is an isomorphism of groups. Since $\rho \circ{} \varphi = \op{id}$, also $\rho$ must be an isomorphism of groups. Moreover, as $\rho$ preserves the distinguished sets of generators, it is also an isomorphism of Coxeter groups.

   Now to prove (iii) we need only need to verify that $\Omega$ preserves the subset $S \subseteq W_{\op{aff}}$ under conjugation and that every element $w \in W$ can be written as a product $w = w' u$ with $w' \in W_{\op{aff}}$ and $u \in \Omega$. But the invariance of $S$ follows immediately from \axiom{ACVI} and the fact $\Omega$ permutes the walls of $C_0$ (as it preserves $C_0$ and therefore also $\overline{C_0}$ setwise). Because $W(\mathfrak{H})$ acts transitive on the set $\pi_0(V-\bigcup_{H \in \mathfrak{H}} H)$ of chambers (see \cite[Ch. V, \textsection 3.2, Théorème 1]{Bourbaki}), we can find $w'' \in W(\mathfrak{H})$ with $\rho(w)(C_0) = w''(C_0)$. Since $\rho(W_{\op{aff}}) = W(\mathfrak{H})$, we can find $w' \in W_{\op{aff}}$ with $\rho(w') = w''$. It follows that $u := w'^{-1}w \in \Omega$.

   Next, we show that (iv) holds. Observe that by \axiom{ACVIII}, the group $W_0$ is generated by the set of linear parts of the $s_H$ with $H \in \mathfrak{H}$. By \axiom{ACIV}, the point $0 \in V$ is special and hence the aforementioned set coincides with $\{s_H : H \in \mathfrak{H},\ 0 \in H \}$. In particular, $W_0 \subseteq W(\mathfrak{H})$ and the formula $L = V^{W_0}$ hold. Let $F \subseteq V$ be the unique facet of $(V,\mathfrak{H})$ containing $0$. By \axiom{ACIX}, $F$ is a face of $C_0$. From \cite[Ch. V, \textsection 3.3, Proposition 1]{Bourbaki} it therefore follows that $W_0$ must be contained in the subgroup of $W(\mathfrak{H})$ generated by the $s_H$ with $H \in S(C_0)$ and $F \subseteq H$. So we have the inclusion
   \[ W_0 = \left< s_H : H \in \mathfrak{H},\ 0 \in H \right> \subseteq \left< s_H : H \in S(C_0),\ 0 \in H \right> \]
   and hence equality holds.

   Claim (v) follows from \axiom{ACIV}, \axiom{ACV} and a slight modification of the arguments in \cite[Ch. VI, \textsection 2.5, Proposition 8]{Bourbaki}. Fix an invariant positive definite scalar product $(\--,\--)$ on $V$. Given $H \in \mathfrak{H}$ with $0 \in H$, let $\alpha \in V^\vee$ be any element with $\op{ker}(\alpha) = H$. Consider
   \[ \Lambda_\alpha := \{ k \in \R : H_{\alpha,k} \in \mathfrak{H} \} \]
   Then $k \mapsto H_{\alpha,k}$ gives a bijection between $\Lambda_\alpha$ and the $H' \in \mathfrak{H}$ parallel to $H$. Then $\Lambda$ must contain a positive element, for we have $0 \in \Lambda_\alpha$ and by \axiom{ACV} there exists an element $w \in W$ such that $\rho(w)$ equals the translation by a vector $v \in V$ with $\alpha(v) \neq 0$. Replacing $w$ by $w^{-1}$ if necessary, we may assume that $\alpha(v) < 0$. Since $W$ preserves $\mathfrak{H}$, it follows that
   \[ \rho(w)(H_{\alpha,0}) = v+H_{\alpha,0} = H_{\alpha,-\alpha(v)} \in \mathfrak{H} \]
   and hence $-\alpha(v) \in \Lambda$. Let now $\delta > 0$ be the smallest positive element of $\Lambda_\alpha$. This element exists because $\mathfrak{H}$ is locally finite. We claim that $\Lambda_\alpha = \Z \delta$. To see this, first note that given any two parallel hyperplanes $H'$ and $H''$ the product $s_{H''}s_{H'}$ of the associated orthogonal reflections equals the translation by $2t$, where $t$ is the unique vector orthogonal to $H'$ with $H'' = t+H'$. Let now $H' = H_{\alpha,k}$, $H'' = H_{\alpha,\ell}$ with $k, \ell \in \Lambda_\alpha$ and let $n \in V$ be the unique vector orthogonal to $H$ satisfying $\alpha(n) = 1$. Then $H_{\alpha,k} = -k n + H_\alpha$ and $H_{\alpha,\ell} = -\ell n + H_\alpha$. Since $W(\mathfrak{H})$ leaves $\mathfrak{H}$ invariant, it follows that
   \[ (s_{H''}s_{H'})(H') = 2(k-\ell)n + H' = H_{\alpha,2\ell - k} \]
   must again be a member of $\mathfrak{H}$, i.e. $2\ell - k \in \Lambda_\alpha$. Taking $\ell = 0$ it follows that $\Lambda_\alpha$ is stable under inversion. Taking $\ell = \delta$ it follows that $\Lambda_\alpha$ is stable under translation by $\pm 2 \delta$. Every element $k \in \Lambda_\alpha$ can therefore be written in the form $k = x+n\delta$ with $n \in \Z$ and $0 \leq x < 2\delta$. If $x \leq \delta$, it follows that $x = \delta$ by minimality of $\delta$. It $\delta < x \leq 2\delta$, it follows by the above that $0 \leq 2\delta - x < \delta$ lies in $\Delta_\alpha$ and hence $2\delta - x = \delta$ by minimality. In both cases it follows that $k \in \Z \delta$.

   From the above discussion it is now clear that given $H \in \mathfrak{H}$ with $0 \in H$ there exists $\alpha \in V^\vee$ uniquely determined up to $\pm$ such that
   \[ \{ H' \in \mathfrak{H} : \text{$H'$ parallel to $H$}\} = \{ H_{\alpha,k} : k \in \Z \} \]
   Then 
   \[ \Phi = \{ \alpha \in V^\vee : \forall k \in \R\ \ H_{\alpha,k} \in \mathfrak{H} \Leftrightarrow k \in \Z \} \]
   is just the set of these $\alpha$. Obviously $(\R \Phi, \Phi)$ is reduced if it is a root system, so it suffices to verify the root system axioms $\op{(RSI)}$-$\op{(RSIII)}$ (see \cite[Ch. VI, \textsection 1.1]{Bourbaki}). There are only a finite number of $H \in \mathfrak{H}$ with $0 \in H$ by the local finitness of $\mathfrak{H}$ and hence it follows readily that $\Phi$ is finite. Moreover, $0 \not\in \Phi$ by construction, and hence (RSI) is verified.
   
   Now we prove (RSII). First, we remark that $\R \Phi$ equals the image of the dual of the projection $V \twoheadrightarrow V/L$. This is equivalent to the claim that $V \stackrel{\sim}{\longrightarrow} V^{\vee \vee}$ induces an isomorphim $(V/L) \stackrel{\sim}{\rightarrow} (\R \Phi)^\vee$ and follows from
   \[ L = \bigcap_{H \in \mathfrak{H},\ 0 \in H} H = \bigcap_{\alpha \in \Phi} \op{ker}(\alpha) \]
   Given $\alpha \in \Phi$, the associated reflection $s_\alpha \in O(\R \Phi)$ is given by the restriction
   \[ s_\alpha = \left. s_H^\vee\right|_{\R \Phi} \]
   of the transpose of the orthogonal reflection $s_H \in O(V)$ with respect to $H = \op{ker}(\alpha)$. This holds since both are elements of $O(\R \Phi)$ having as fix-point set the hyperplane
   \[ \alpha^{\perp} = \R \Phi \cap \{ \omega \in V^\vee : \omega(v) = 0 \} \]
   where $v \in V$ is any vector $\neq 0$ orthogonal to $H$. Since $s_H$ leaves $\mathfrak{H}$ invariant, it follows that $s_\alpha$ leaves $\Phi$ invariant; thus $s_{\alpha,\alpha^\vee} = s_\alpha$ for $\alpha^\vee := 2\frac{(\alpha,\cdot{})}{(\alpha,\alpha)} \in (\R \Phi)^{\vee\vee}$ leaves $\Phi$ invariant, and (RSII) is verified.
   Lastly to prove (RSIII), let $\alpha,\beta \in \Phi$ be given. Identifying $(\R \Phi)^\vee$ with the subspace $L^\perp \leq V$, the dual root $\alpha^\vee$ is the unique element of $V$ orthogonal to $H_\alpha$ satisfying $\alpha(\alpha^\vee) = 2$. In particular letting $H' = H_{\alpha,0}$ and $H'' = H_{\alpha,1} = -\frac{1}{2}\alpha^\vee + H'$ we have that
   \[ (s_{H''}s_{H'})(H_{\beta,0}) = -\alpha^\vee + H_{\beta,0} = H_{\beta,\beta(\alpha^\vee)} \in \mathfrak{H} \]
   and hence $\beta(\alpha^\vee) \in \Z$ since $\beta \in \Phi$.

   Next, we prove (vi) keeping the choice of an invariant scalar product on $V$. The injectivity of the map $W_0 \rightarrow \op{GL}(V/L)$ follows from the fact $W_0$ is finite and hence acts by semi-simple transformations on $V$. Indeed since $L = V^{W_0}$, any $w \in W_0$ lying in the kernel of $W_0 \rightarrow \op{GL}(V/L)$ acts trivially on $L$ and $V/L$ and hence must act trivially on $V$ by semi-simplicity. In the proof of (iv) we have already seen that $W_0$ is generated by the $s_H$ with $H \in \mathfrak{H}$ and $0 \in H$. By (v) we know that $H$ is of the form $H = \op{ker}(\alpha)$ with $\alpha$. Moreover, we have already seen that the image of $s_H$ under $W_0 \rightarrow \op{GL}(V/L)\simeq \op{GL}((\R \Phi)^\vee)$ equals the transpose $s_\alpha^\vee$ of the reflection associated to $\alpha$. This shows that the image of $W_0 \hookrightarrow \op{GL}((\R \Phi)^\vee)$ is given by $W(\Phi)^\vee = W(\Phi^\vee)$. Moreover it's clear by the previous remarks that under $W_0 \stackrel{\sim}{\longrightarrow} W(\Phi^\vee)$ the generating set $\{s_H : H \in S(C_0) \}$ corresponds to $\{ s_\alpha^\vee : \alpha \in \Delta \}$.
   Now to see that $\Delta$ is a basis of the root system $\Phi$, let $D_0 \in \pi_0(V-\bigcup_{\alpha \in \Phi} \op{ker}(\alpha))$ be the unique chamber of the spherical arrangement containing $C_0$. The image $\pi(D_0)$ of $D_0$ under $\pi: V \twoheadrightarrow V/L \simeq (\R \Phi)^\vee$ then is a chamber of the linear arrangement on $(\R \Phi)^\vee$ induced by $\Phi$. Moreover, for $\alpha \in \Phi$ the hyperplane $\op{ker}(\alpha)$ is a wall of $C_0$ if and only if the hyperplane in $(\R \Phi)^\vee$ associated to $\alpha$ is a wall of $\pi(D_0)$, and $\alpha$ is positive on $C_0$ if and only if $\alpha$ is positive on $\pi(D_0)$. By the theory of root systems, it then follows that $\Delta$ is a basis of $\Phi$, in fact $\Delta$ is the basis of $\Phi$ associated to the dual chamber $\pi(D_0)^\vee \subseteq \R \Phi$ (see \cite[Ch. VI, \textsection 1.5, Rémarque 5]{Bourbaki}). Moreover, it is obvious that $\Phi^+$ consists of the roots which take positive values on $\pi(D_0)$. It hence follows (see \cite[Ch. VI, \textsection 1.6]{Bourbaki}) that $\Phi$ coincides with the set of positive roots associated to $\pi(D_0)^\vee$. Since we have $\Delta \subseteq \Phi^+$, it follows that $\Delta$ is the root basis associated to $\Phi^+$.
   
   Finally (vii) follows immediately from (iv)-(vi) and the fact that for $\alpha \in \Phi$ we have $\rho^{-1}(s_\alpha) = \widetilde{s}_{H_{\alpha,0}}$.
\end{proof}

\begin{example}\label{ex:affine_coxeter_group}
   \begin{enumerate}
      \item Let $(X,\Phi,X^\vee,\Phi^\vee)$ be a root datum (in the sense of \cite[Exposé XXI]{SGAIII}) and $\Delta \subseteq \Phi^\vee$ a root basis. In particular, $\Phi \subseteq X$ and $\Phi^\vee \subseteq X^\vee$ are finite subsets that are in bijection via a \textit{given} pair of inverse bijections
         \[ \Phi \stackrel{\sim}{\leftrightarrow} \Phi^\vee \]
         \textit{both} denoted by $\alpha \mapsto \alpha^\vee$, and $X,X^\vee$ are free abelian groups of finite rank in duality via a \textit{given} pairing
      \[ \left<\cdot{},\cdot{}\right>: X^\vee \times X \longrightarrow \Z \]
   Let $W_0 := W(\Phi)$ be the finite Weyl group, i.e. the subgroup of $\op{GL}_\Z(X)$ generated by the reflections $s_\alpha$, $\alpha \in \Phi$ given by
   \[ s_\alpha(x) = x - \left<\alpha^\vee,x\right>\alpha \]
   Let $W := X \rtimes W_0$ be the \textit{extended affine Weyl group}. Let us now see that $W$ carries a canonical structure of an affine extended Coxeter group in the sense of \cref{def:affine_coxeter_group}, and therefore also a canonical structure of an extended Coxeter group via \cref{lem:unpacking_affine_extended_Coxeter_groups}.

   We let $V := X\otimes_\Z \R$ and let $\rho: W \longrightarrow \op{GL}_{\op{aff}}(V)$ be the inclusion
   \[ W = X\rtimes W_0 \subseteq V \rtimes \op{GL}(V) \simeq \op{GL}_{\op{aff}}(V) \]
   This action leaves invariant the collection $\mathfrak{H}$ of hyperplanes given by $H_{\alpha,k}$, $\alpha \in \Phi^\vee$, $k \in \Z$ where
   \[ H_{\alpha,k} = \{ x \in V : \left<\alpha,x\right> + k = 0 \} \]
   Since $\rho$ is injective, the choice of the $\widetilde{s}_H$ is unique in this case. Moreover, it is clear that for any choice of a chamber $C_0$ with $0 \in \overline{C_0}$, the axioms \axiom{ACI}-\axiom{ACX} are satisfied, in particular if we let $C_0$ be chamber corresponding to $\Delta$ determined by the conditions
   \[ 0 \in \overline{C_0}\quad \text{and}\quad C_0 \subseteq \{ x \in V : \left<\alpha,x\right> > 0\quad \forall \alpha \in \Delta \} \]
   Moreover, the groups $W_0$ and $X$ of \cref{def:affine_coxeter_group} coincide with the groups denoted by the same letters here. The root system $\Phi$ and the basis $\Delta$ constructed in the lemma above coincide with $\Phi^\vee$ and $\Delta$ respectively. The structure $W = (W,W_{\op{aff}},S,\Omega)$ of an extended Coxeter group induced on $W$ by the above lemma can be made more explicit as follows. Let $Q := \Z \Phi \leq X$ be the root lattice. Then elementary arguments (see \cite[Ch. VI, \textsection 1.2, Proposition 1]{Bourbaki}) show that the affine Weyl group $W_{\op{aff}} \leq W$ is the semi-direct product $W_{\op{aff}} = Q \rtimes W_0$. Hence, there is an isomorphism
   \[ \Omega \simeq X/Q \]
   By definition, the generating set $S$ of $W_{\op{aff}}$ consists of the reflections $s_H$ for all walls $H$ of $C_0$. Using the theory of root systems it can be seen that the walls of $C_0$ either of the form $H = H_{\alpha,0}$ with $\alpha \in \Delta$ or $H = H_{-\alpha,1}$ with $\alpha$ a \textit{highest coroot}, i.e. a maximal element of $\Phi^\vee$ with respect to the partial order
   \[ \alpha \leq \beta \quad \Leftrightarrow\quad \left<\alpha,x\right> \leq \left<\beta,x\right> \ \ \forall x \in C_0 \]
   Hence
   \[ S = \{ s_\alpha : \alpha \in \Delta \} \cup \{ s_{-\alpha,1} : \text{$\alpha \in \Phi^\vee$ maximal} \} \]
   where (by slight abuse of notation) $s_\alpha$ and $s_{\alpha,k}$ for $\alpha \in \Phi^\vee$, $k \in \Z$ denote the elements of $W_{\op{aff}}$ given by
   \[ s_\alpha(x) = x - \left<\alpha,x\right> \alpha^\vee\quad \text{and} \quad s_{\alpha,k}(x) = x - (\left<\alpha,x\right> + k)\alpha^\vee \]
\item We specialize the above example now to the root datum of the group $\op{GL}_n$. In this case we have
   \[ X = X^\vee = \Z^n \]
   with the pairing between $X$ and $X^\vee$ being the canonical one. Moreover
   \[ \Phi = \Phi^\vee = \{ e_i - e_j : 1 \leq i,j \leq n,\ i \neq j\} \]
   and the correspondence $\alpha \leftrightarrow \alpha^\vee$ between roots and coroots is the identity. The finite Weyl group $W_0$ identifies with the symmetric group $S_n$ on $n$ letters. The choice of the (co-)root basis
   \[ \Delta = \{ e_2 - e_1, \ldots{}, e_n - e_{n-1} \} \]
   makes $W_0 = S_n$ into a Coxeter group with generators $s_1,\ldots{},s_{n-1}$, where
   \[ s_i = s_{e_{i+1}-e_i} = (i\ i+1) \]
   is the transposition permuting the $i$-th and $i+1$-th coordinate. The chamber determined by $\Delta$ is given by
   \[ C_0 = \{ x \in \R^n : x_1 < \ldots < x_n < x_1 + 1 \} \]
   The root sublattice $Q = \Z \Phi \leq \Z^n$ is the kernel of the `augmentation map'
   \[ \Z^n \longrightarrow \Z,\quad e_i \mapsto 1 \]
   hence the group $\Omega \simeq X/Q$ (which as a subgroup of $W$ depends on the choice of $C_0$!) is canonically isomorphic to $\Z$, with canonical generator $u$ given by
   \[ u = \tau^{e_n} (n\ n-1\ \ldots\ 1) \]
   The highest (co-)root is unique and given by $\alpha = e_n - e_1$. Hence, the generating set $S$ of $W_{\op{aff}}$ is given by
   \[ S = \{ s_1,\ldots{},s_{n-1}, s_{-\alpha,1} \} \]
   with
   \[ s_{-\alpha,1} = \tau^{e_n - e_1} (1\ n) \]
   Writing $s_0 = s_{-\alpha,1}$ and viewing $\{0,1,\ldots{},n-1\}$ as the group $\Z/n\Z$, the action of $\Omega$ on $S$ is determined by
   \begin{equation}\label{eq:Omega_acting_on_S} u s_i u^{-1} = s_{i-1} \end{equation}
   \end{enumerate}
\end{example}

We are now in the position to define the principal object of study of this article, the class of \textit{affine} generic pro-$p$ Hecke algebras (or simply \textit{affine pro-$p$ Hecke algebras}) as those algebras whose underlying extended Coxeter group $W$ arises as in the above lemma from an affine extended Coxeter group. Since the description of the structure of these algebras will depend on the decomposition $W = X\rtimes W_0$, it makes sense to make the affine extended Coxeter group part of the datum.

\begin{definition}\label{def:affine_pro-$p$_Hecke_algebra}
   An \textbf{affine pro-$p$ Hecke algebra} $\mathcal{H}^{(1)}$ over a ring $R$ consists of a generic pro-$p$ Hecke algebra $\mathcal{H}^{(1)}$ over $R$ and an affine extended Coxeter group $W$ such that the extended Coxeter group underlying the pro-$p$ Coxeter group $W^{(1)}$ associated with $\mathcal{H}^{(1)}$ coincides with the extended Coxeter group associated to $W$ by \cref{lem:unpacking_affine_extended_Coxeter_groups}.
  \end{definition}

  \begin{terminology}
     Following tradition and to prevent confusion with the chambers $C \in \pi_0(V-\bigcup_{H \in \mathfrak{H}} H)$, the connected components of the complement of the finite linear hyperplane arrangement $\{ H \in \mathfrak{H} : 0 \in H \}$ will be called \textit{Weyl chambers}. They will usually denoted by the letter 'D', while 'C' will be used to denote the chambers of the affine hyperplane arrangement $\mathfrak{H}$.
   \end{terminology}

   The main goal of this article will be to describe the center of affine pro-$p$ Hecke algebras using the Bernstein maps introduced in the previous section. As in the classical work of Benstein and Lusztig, this involves constructing big (almost) commutative subalgebras of $\mathcal{H}^{(1)}$. In view of \cref{rmk:comm subalgebras}, this amounts to constructing orientations with big stabilizers, which we will do later in \cref{sub:Spherical orientations}.

\subsection{Main examples of affine pro-\texorpdfstring{$p$}{p} Hecke algebras} 
\label{sub:Main examples of affine pro-$p$ Hecke algebras}
In this section we want to consider the main examples of affine pro-$p$ Hecke algebras, the classic \textit{affine Hecke algebras} and two `new' examples, the \textit{pro-$p$-Iwahori Hecke} algebras and the \textit{affine Yokonuma-Hecke} algebras. We point out that the last two examples slightly overlap.

\begin{example}[Affine Hecke algebras and Iwahori-Hecke algebras]\label{ex:affine Hecke algebras}
   The \textit{affine Hecke algebras} are (cf. \cite[4.1]{Macdonald}) the generic pro-$p$ Hecke algebras for the \textit{extended affine Weyl groups}, i.e. for pro-$p$ Coxeter groups $W^{(1)}$ of the form $W^{(1)} = W$, $T = 1$, $W = X \rtimes W_0$ for a root datum $(X,\Phi,X^\vee,\Phi^\vee)$ with chosen basis $\Delta \subseteq \Phi^\vee$ as in \cref{ex:affine_coxeter_group}. As was explained there, the group $W$ carries a canonical structure of an affine extended Coxeter group in the sense of \cref{def:affine_coxeter_group}, hence these algebras are affine pro-$p$ Hecke algebras in the sense of \cref{def:affine_pro-$p$_Hecke_algebra}.

   Affine Hecke algebras play an important role in various different but related subjects, including the representation theory of reductive groups over local fields, the theory of orthogonal polynomials \cite{Macdonald}, the theory of knot invariants, and in physics in the study of certain exactly solvable systems (see \cite{Martin}). Historically, affine Hecke algebras made their debut in the first of the subjects mentioned, namely in the 1965 paper of Iwahori and Matsumoto \cite{IwahoriMatsumoto} that elucidated the structure of double coset algebras $H_R(G,I)$ (cf. \cref{ex:pro-$p$-Iwahori Hecke algebras}) attached to pairs $(G,I)$, where $G = \mathbf{G}(F)$ is the group of rational points of a split, connected, semisimple reductive group (Chevalley group) $\mathbf{G}$ over a nonarchimedean local field $F$, and $I \leq G$ is a certain open compact subgroup nowadays referred to as `Iwahori subgroup'.
   
   One of the main results (Propositions 3.5, 3.7 and 3.8) of \cite{IwahoriMatsumoto} was the description of a presentation of $H_R(G,I)$ in terms of the extended affine Weyl group $W = X \rtimes W_0$ of the root datum corresponding to $\mathbf{G}$, i.e. an isomorphism of $H_R(G,I)$ with an affine Hecke algebra. More precisely, they showed that $H_R(G,I)$ is isomorphic to the $R$-algebra generated by symbols $T_w$, $w \in W$ subject to the relations
   \begin{align*}
      \tag{1} T_w T_{w'} & = T_{ww'} & & \text{if $\ell(w)+\ell(w') = \ell(ww')$} \\
      \tag{2} T_s^2 & = q + (q-1) T_s && s \in S
   \end{align*}
   where $q$ denotes the cardinality of the residue field of $F$. Hence, $H_R(G,I)$ identifies with the generic pro-$p$ Hecke algebra $\mathcal{H}^{(1)}(a_s,b_s)$ for $W^{(1)} = W$ and constant parameters $a_s = q$, $b_s = q-1$ by \cref{prop:first_presentation_of_hecke}, which is an affine Hecke algebra.
   
   The algebras of the form $H_R(G,I)$ are commonly referred to as \textit{Iwahori-Hecke algebras}. Sometimes the terms `affine Hecke algebra' and `Iwahori-Hecke algebra' are used synonymous, but here we will distinguish between the two. The notion of Iwahori subgroup is defined in great generality for any connected reductive group $\mathbf{G}$ over a local field \cite[3.7]{TitsCorvallis}, and one can consider the corresponding algebras $H_R(G,I)$. These more general Iwahori-Hecke algebras have a similar presentation in terms of a certain group $W = X \rtimes W_0$ which admits the structure of an affine extended Coxeter group but where the constant coefficients $q$ and $q-1$ are replaced by coefficients $q_s$ and $q_s - 1$ that can depend on $s$ (cf. \cref{lem:pro-$p$-Iwahori Hecke algebras are generic pro-$p$ Hecke algebras}).
   
   This was first proved by Vignéras \cite[Proposition 4.1, 4.4]{VigProI}, although it has long been a part of mathematical folkore that `Iwahori-Hecke algebras for non-split groups are affine Hecke algebras for unequal parameters'. The latter is in fact not true. The algebra $H_R(G,I)$ is isomorphic to the generic pro-$p$ Hecke algebra $\mathcal{H}^{(1)}(q_s,q_s-1)$ associated to the affine extended Coxeter group $W$ and hence is an affine pro-$p$ Hecke algebra, but it is not always an affine Hecke algebra (in our sense) as the group $W = X\rtimes W_0$ does not necessarily arise from a root datum. In fact, $X$ is a finitely generated abelian group with nontrivial torsion part in general. However, when the group $\mathbf{G}$ is split, this subtlety disappears and the corresponding Iwahori-Hecke algebras are affine Hecke algebras with constant coefficients $a_s = q$, $b_s = q-1$ for the extended affine Weyl group corresponding to the root datum of $\mathbf{G}$.

   The most important structural results concerning affine Hecke algebras in general are the `Bernstein relations', the `Bernstein presentation' and the computation of the center in terms of invariants of certain commutative subalgebras. These results were obtained by Bernstein and Zelevinsky in an unpublished work for the special case of constant parameters $a_s = q, b_s = q-1$. Lusztig later published a generalized version of these results in \cite{Lusztig}, where he took the parameters to be of the form $a_s = q_s$, $b_s = q_s - 1$ with $q_s = v^{2n_s}$ for some integers $n_s$ and an invertible formal variable $v \in R = \C[v,v^{-1}]$. Lusztig obtained these results using a group homomorphism $\theta$ from the group $X$ of translations into the group of units of the affine Hecke algebra. We will see below (in \eqref{eq:theta_coincides_with_T_on_dominant}) that this map coincides with the restriction of our map $\widetilde{\theta}_\mathfrak{o}$ (see \cref{def:normalized_bernstein_map}) to $X \leq W$, where $\mathfrak{o} = \mathfrak{o}_D$ denotes the spherical orientation (see \cref{def:spherical_orientation}) corresponding to the dominant Weyl chamber $D$. These results of Bernstein, Zelevinsky and Lusztig were further generalized by Vignéras in \cite{VigGen} to allow for parameters of the form $a_s = q_s$, $b_s = q_s-1$ with $q_s$ not necessarily invertible or admitting a square root.
   
   The Bernstein relations and the description of the center in all of the above cases are recovered here in \cref{thm:bernstein_relation,thm:structure of affine pro-$p$ Hecke algebras}. Note that the results of \cref{thm:structure of affine pro-$p$ Hecke algebras} hold unconditionally in these cases since $T = 1$ (cf. \cref{rmk:finiteness_of_orbits}). For the readers convenience we will quote the construction of the Bernstein-Zelevinsky subalgebra and the description of the center from \cref{thm:structure of affine pro-$p$ Hecke algebras} for our special case. For every spherical orientation $\mathfrak{o} = \mathfrak{o}_D$ of $W$, associated to a Weyl chamber $D$ (see \cref{def:spherical_orientation}), the integral Bernstein map $\widehat{\theta}_\mathfrak{o}: W \rightarrow \mathcal{H}^{(1)}$ (see \cref{def:integral_bernstein_map}) gives rise to a commutative subalgebra
   \[ \mathcal{A}_\mathfrak{o} := \bigoplus_{x \in X} R \widehat{\theta}_\mathfrak{o}(x) \subseteq \mathcal{H}^{(1)} \]
   whose multiplicative structure is determined by the product rule (\cref{cor:hat mult rule})
   \[ \widehat{\theta}_\mathfrak{o}(x)\widehat{\theta}_\mathfrak{o}(y) = \overline{\bbf{X}}(x,y) \widehat{\theta}_\mathfrak{o}(x+y) \]
   If the parameters $a_s \in R$ are units and squares, then we can also consider the normalized Bernstein map $\widetilde{\theta}_\mathfrak{o}: W \rightarrow \mathcal{H}^{(1)}$ whose restriction to $X$ is determined by the fact that it is multiplicative and satisfies the following relation (see \cref{def:normalized_bernstein_map} for details)
   \begin{equation}\label{eq:theta_coincides_with_T_on_dominant} \widetilde{\theta}_\mathfrak{o}(x) = \overline{\sqrt{\bbf{L}}}(x)^{-1} T_x \quad \forall x \in X\cap \overline{D} \end{equation}
      These properties together imply that $\left.\widetilde{\theta}_\mathfrak{o}\right|_X$ coincides with the map denoted by $\theta$ by Lusztig \cite{Lusztig}, which appears in the classical Bernstein-Lusztig basis $\{\theta_x T_w\}_{x \in X, w \in W_0}$. Moreover, $\widetilde{\theta}_\mathfrak{o}$ is related to the integral Bernstein map via $\widetilde{\theta}_\mathfrak{o}(w) = \overline{\sqrt{\bbf{L}}}(L)^{-1}(w) \widehat{\theta}_\mathfrak{o}(w)$ (see \cref{thm:existence_widehat_theta}). It follows from $\mathcal{A}_\mathfrak{o}$ can also be expressed as
   \[ \mathcal{A}_\mathfrak{o} = \bigoplus_{x \in X} R \widetilde{\theta}_\mathfrak{o}(x) \]
   and that $\widetilde{\theta}_\mathfrak{o}$ induces an isomorphism of the group algebra $R[X]$ with $\mathcal{A}_\mathfrak{o}$. In any case, the group $W_0$ acts on $\mathcal{A}_\mathfrak{o}$ by permuting the basis elements $w(\widehat{\theta}_\mathfrak{o}(x)) = \widehat{\theta}_\mathfrak{o}(w(x))$ and the center of $\mathcal{H}^{(1)}$ is given by the $W_0$-invariants
   \[ Z(\mathcal{H}^{(1)}) = \mathcal{A}_\mathfrak{o}^{W_0} = \bigoplus_{\gamma \in W_0\backslash X} R z_\gamma \]
   with
   \[ z_\gamma = \sum_{x \in \gamma} \widehat{\theta}_\mathfrak{o}(x) \]
   independent of the orientation (Weyl chamber) chosen.
\end{example}

\begin{example}[The affine Hecke algebra of $\op{GL}_n$]\label{ex:affine_hecke_algebra_of_gl_n}
   We now specialize our discussion to the case where the root datum defining $W$ is the root datum of $\op{GL}_n$. The affine Hecke algebra of $W$ with parameters $a_s, b_s$ will be denoted by $H_n^{\op{aff}}(a_s,b_s)$ or simply by $H_n^{\op{aff}}$. The generalized braid groups $\mathcal{A}(W)$, $\mathcal{A}(W_0)$ (see \cref{def:generalized_braid_group}) associated to $W = X\rtimes W_0$ and $W_0$ will be denoted by $\mathfrak{A}_n^{\op{aff}}$ and $\mathfrak{A}_n$ respectively. From \cref{ex:affine_coxeter_group}(ii) we recall that 
   \[ W = \Z^n \rtimes S_n = W_{\op{aff}} \rtimes \Omega,\quad W_{\op{aff}} = \left<S\right>,\quad S = \{s_0, \ldots{}, s_n\},\quad \Omega = \left<u\right> \]
   where
   \[ s_0 = \tau^{e_n - e_1} (1\ n),\quad s_i = (i\ i+1) \text{ for } i > 0\quad \text{and}\quad u = \tau^{e_n} (n\ n-1\ \ldots 1) \]
   with $u$ acting on $S$ as
   \[ u s_i u^{-1} = s_{i-1} \]
   In particular, all the generators $s \in S$ are conjugate under $W$ and condition \eqref{eq:condstar} on the parameters $a_s$, $b_s$ is equivalent to
   \[ a_s = a_t,\ b_s = b_t \quad \forall s,t \in S \]
   Hence, we can write $H_n^{\op{aff}}(a,b) = H_n^{\op{aff}}(a_s,b_s)$ with parameters $a,b \in R$ subject to no further constraint.
   The `affine braid group' $\mathfrak{A}_n^{\op{aff}}$ can be interpreted (\cite{Nguyen},\cite{Lek}) topologically as a group of braids as follows. Consider the real affine hyperplane arrangement
   \[ \mathfrak{H} = \{ H_\alpha : \alpha \in \Phi_{\op{af}} \},\quad H_{\alpha} = \{ x \in A : \alpha(x) = 0 \},\quad \Phi_{\op{af}} = \{ \alpha + k : \alpha \in \Phi,\ k \in \Z \} \subset \op{Hom}(A,\R) \]
   in $A = \R^n$ induced by the root datum
   \[ (X,\Phi,X^\vee,\Phi^\vee) = (\Z^n, \{ e_i - e_j : i \neq j \}, \Z^n, \{e_i - e_j : i \neq j\}) \]
   of $\op{GL}_n$ (cf. \cref{ex:affine_coxeter_group}). The complement $X := A - \bigcup_{H \in \mathfrak{H}} H$ is disconnected, the connected components being in bijection with the infinite group $W_{\op{aff}}$, but the complement $Y := A_\C - \bigcup_{H \in \mathfrak{H}} H_\C$ of the complexified arrangement is connected. The fundamental groupoid $\pi_1(Y)$ of $Y$ can be described as follows (see \cite{Nguyen},\cite{Lek}). For any two points $x,y \in X \subseteq Y$ let $\mathcal{P}_{y,x}$ be the subspace of the space of all paths $\gamma:[0,1] \rightarrow Y$ consisting of those $\gamma$ which satisfy
   \begin{enumerate}
      \item $\gamma(0) = x, \gamma(1) = y$
      \item $\forall t \in [0,1],\ \alpha \in \Phi_{\op{af}}\quad \Re(\alpha_\C(\gamma(t))) = 0 \ \ \Rightarrow\ \ \alpha(x)\alpha(y) < 0$
      \item $\forall t \in [0,1],\ \alpha \in \Phi_{\op{af}}\quad \alpha(x)\alpha(y) < 0\ \ \Rightarrow\ \ \Im(\alpha_\C(\gamma(t)))\cdot{}(\alpha(x) - \alpha(y)) \geq 0$
   \end{enumerate}
   In words the second condition says that the real part of $\gamma$ should only cross these hyperplanes $H \in \mathfrak{H}$ which separate $x$ and $y$, while the third condition means that for every hyperplane $H_\alpha \in \mathfrak{H}$, $\alpha \in \Phi_{\op{af}}$ separating $x$ and $y$, the path $\alpha_\C \circ{} \gamma: [0,1] \rightarrow \C$ should wind around the origin counter-clockwise and should stay completely in either the upper or lower half-plane. It is easy to see that $\mathcal{P}_{y,x}$ is contractible, hence giving rise to a well-defined homotopy-class
   \[ \gamma_{y,x} \in \op{Hom}_{\pi_1(Y)}(x,y) \]
It is even easier to see that
\[ w(\mathcal{P}_{y,x}) = \mathcal{P}_{w(y),w(x)} \]
and hence
\[ w(\gamma_{y,x}) = \gamma_{w(y),w(x)} \]
for $w \in W$. Moreover, for any three points $x,y,z \in X$ it holds true that
\[ \mathcal{P}_{z,y} \circ{} \mathcal{P}_{y,x} \subseteq \mathcal{P}_{z,x} \]
and therefore that
\[ \gamma_{z,y} \circ{} \gamma_{y,x} = \gamma_{z,x} \]
if the set of hyperplanes separating $x$ and $y$ is disjoint from the set of hyperplanes separating $y$ and $z$, i.e. if
\[ d(C_x,C_y) + d(C_y, C_Z) = d(C_x,C_z) \]
where $C_p$ denotes the connected component of $X$ (chamber) containing $p$ and $d(C,C')$ denotes the distance between two chambers. One can now show (\cite{Nguyen},\cite{Lek}) that the full subgroupoid of $\pi_1(Y)$ corresponding to $X \subseteq Y$ is described algebraically as the free groupoid on symbols $\gamma_{y,x}$ subject to the relation
\[ d(C_x,C_y) + d(C_y,C_z) = d(C_x, C_z) \quad \Rightarrow \quad \gamma_{z,y} \circ{} \gamma_{y,x} = \gamma_{z,x} \]
From this one deduces a description of the fundamental group of the quotient space $W\backslash Y$, where the action of $W$ is naturally extended to $A_\C$. Indeed, $W$ acts properly discontinuously and without fix points on $Y$, therefore $Y \rightarrow W\backslash Y$ is a covering map with Galois group $W$. Fixing a base point $x_0 \in X$ and letting
\[ T_w := p_\ast(\gamma_{x_0,w(x_0)}^{-1}) \in \pi_1(W\backslash Y, p(x_0)),\quad w \in W \]
it follows easily from the above that
\[ T_wT_{w'} = T_{ww'} \quad\text{if}\quad \ell(w)+\ell(w') = \ell(ww') \]
and that the elements $T_w$ together with the above relation define a presentation of $\pi_1(W\backslash Y, p(x_0))$ and hence an isomorphism of this group with $\mathfrak{A}_n^{\op{aff}}$. The interpretation of $\pi_1(W\backslash Y,p(x_0))$ as a group of `affine braids' arises by viewing $W\backslash Y$ as the iterated quotient
\[ W\backslash Y \simeq S_n\backslash (\Z^n \backslash Y) \simeq S_n \backslash ((\C^\times)^n - \Delta) \]
where $\Delta = \bigcup_{i \neq j} \{ z_i = z_j \}$ is the diagonal and $\Z^n \backslash Y$ is identified with $(\C^\times)^n - \Delta$ via $z \mapsto \op{exp}(2\pi i z)$. A loop $\gamma$ in $S_n \backslash ((\C^\times)^n - \Delta)$ around $p(x_0)$ can be identified with the braid
   \[ \bigcup_{i = 1}^n \{ (t,\widehat{\gamma}(t)_i) : t \in [0,1] \} \subseteq [0,1] \times \C^\times \]
   where $\widehat{\gamma}$ denotes any lift of $\gamma$ to a path in $(\C^\times)^n - \Delta$. Under this bijection, composition of paths corresponds to `stacking' of braids (rescaling the $t$-coordinate by $\frac{1}{2}$), and the inverse of a braid is given by its reflection along the $t=\frac{1}{2}$-plane. This is illustrated in \cref{fig:braid-group-as-group-of-braids} (for the base point $x_0 = \left(0,\frac{1}{n},\dots,\frac{n-1}{n}\right)$ and $n = 3$), where the `missing' central line $[0,1] \times \{0\} \subseteq [0,1]\times \C$ has been enlarged to a flagpole for better visibility. \Cref{fig:affine-braid-group-generators} depicts the braids corresponding to some representatives of the generators $T_i = T_{s_i}$, $X_1 = T_{-e_1}^{-1}$ of the group $\mathfrak{A}_3^{\op{aff}}$ appearing in \cref{lem:presentation_of_affine_braid_group} below.

   \begin{figure}[t]
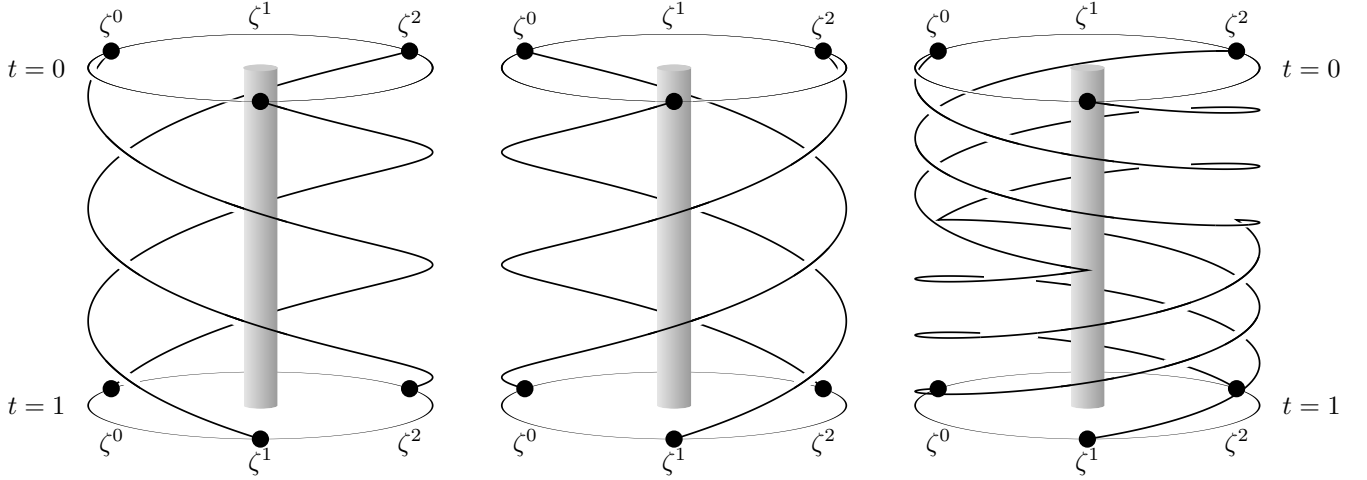

      \centering
      \begin{subfigure}[t]{0.333\textwidth}
         \centering
         \psfrag{zeta0}[cc][bl]{$\zeta^0$}
         \psfrag{zeta1}[cc][bl]{$\zeta^1$}
         \psfrag{zeta2}[cc][bl]{$\zeta^2$}
         \psfrag{t0}[cc][bl]{$t=0$}
         \psfrag{t1}[cc][bl]{$t=1$}
         \includegraphics[width=\textwidth]{affine-braid-group4.eps}
         \caption{The braid $T$ given by the path $\widehat{\gamma}(t) = (e^{2\pi i t}, \zeta e^{2\pi i t}, \zeta^2 e^{2\pi i t})$.}
      \end{subfigure}%
      \begin{subfigure}[t]{0.333\textwidth}
         \centering
         \psfrag{zeta0}[cc][bl]{$\zeta^0$}
         \psfrag{zeta1}[cc][bl]{$\zeta^1$}
         \psfrag{zeta2}[cc][bl]{$\zeta^2$}
         \includegraphics[width=\textwidth]{affine-braid-group5.eps}
         \caption{The inverse $T^{-1}$ of the braid $T$.}
      \end{subfigure}%
      \begin{subfigure}[t]{0.333\textwidth}
         \centering
         \psfrag{zeta0}[cc][bl]{$\zeta^0$}
         \psfrag{zeta1}[cc][bl]{$\zeta^1$}
         \psfrag{zeta2}[cc][bl]{$\zeta^2$}
         \psfrag{t0}[cc][bl]{$t=0$}
         \psfrag{t1}[cc][bl]{$t=1$}
         \includegraphics[width=\textwidth]{affine-braid-group6.eps}
         \caption{The `composite' $T T^{-1}$.}
      \end{subfigure}
      \caption{Loops in $S_n\backslash((\C^\times)^n - \Delta)$ based at $\left[(1,\zeta,\zeta^2,\dots,\zeta^{n-1})\right]$ ($\zeta := \exp{\frac{2\pi i}{n}}$) can be identified with braids in $\left[0,1\right] \times \C^\times$ with endpoints $\{(0,\zeta^i) : i = 0,\dots,n-1\}$ and $\{(1,\zeta^i) : i = 0,\dots,n-1\}$, by lifting a loop $\gamma$ to a path $\widehat{\gamma}$ in $(\C^\times)^n - \Delta$ and associating to it the braid $\bigcup_{i = 1}^n \{ (t, \widehat{\gamma}(t)_i) : t \in \left[0,1\right]\}$.}
      \label{fig:braid-group-as-group-of-braids}
   \end{figure}

   \begin{figure}[t]
      \centering

      \begin{subfigure}[t]{0.3\textwidth}
         \centering
         \psfrag{zeta0}[cc][bl]{$\zeta^0$}
         \psfrag{zeta1}[cc][bl]{$\zeta^1$}
         \psfrag{zeta2}[cc][bl]{$\zeta^2$}
         \includegraphics[width=\textwidth]{affine-braid-group1.eps}
         \caption{$T_1 = T_{s_1}$}
      \end{subfigure}%
      \begin{subfigure}[t]{0.3\textwidth}
         \centering
         \psfrag{zeta0}[cc][bl]{$\zeta^0$}
         \psfrag{zeta1}[cc][bl]{$\zeta^1$}
         \psfrag{zeta2}[cc][bl]{$\zeta^2$}
         \includegraphics[width=\textwidth]{affine-braid-group2.eps}
         \caption{$T_2 = T_{s_2}$}
      \end{subfigure}%
      \begin{subfigure}[t]{0.3\textwidth}
         \centering
         \psfrag{zeta0}[cc][bl]{$\zeta^0$}
         \psfrag{zeta1}[cc][bl]{$\zeta^1$}
         \psfrag{zeta2}[cc][bl]{$\zeta^2$}
         \includegraphics[width=\textwidth]{affine-braid-group3.eps}
         \caption{$X_1 = T_{-e_1}^{-1}$}
      \end{subfigure}
      \caption{The generators of $\widetilde{\mathfrak{A}}_3^{\op{aff}}$ viewed as braids in $\left[0,1\right]\times \C^\times$.}
      \label{fig:affine-braid-group-generators}
   \end{figure}

The classical Artin braid group $\mathfrak{A}_n$ can be interpreted similarly either as the fundamental group of $S_n\backslash (\C^n - \Delta)$ or as a group braids (without a flagpole). From the topological picture it is therefore clear that there should be a canonical map
\[ \mathfrak{A}_n^{\op{aff}} \longrightarrow \mathfrak{A}_n \]
induced by the inclusion $S_n\backslash ((\C^{\times})^n - \Delta) \subseteq S_n\backslash (\C^n - \Delta)$, corresponding to `removing the flagpole' on the level of braids. However, this map is \emph{not} simply given by $T_w \mapsto T_{p(w)}$, where $p$ denotes the canonical projection $W = X\rtimes W_0 \rightarrow W_0$. To describe it we need another presentation of the group $\mathfrak{A}(W)$.

\begin{lemma}\label{lem:presentation_of_affine_braid_group}
   Let $\widetilde{\mathfrak{A}}_n^{\op{aff}}$ be the group generated by elements
   \[ T_1,\ldots{}T_{n-1}, X_1 \]
   subject to the relations
   \begin{alignat}{2}
      \tag{1} T_i T_j & = T_j T_i & \quad & \text{for all $i,j = 1,\ldots{},n-1$ such that $|i-j| > 1$} \\
      \tag{2} T_i T_{i+1} T_i & = T_{i+1} T_i T_{i+1} & \quad & \text{for all $i = 1,\ldots{},n-2$} \\
      \tag{3} X_1 T_1 X_1 T_1 & = T_1 X_1 T_1 X_1 & & \\
      \tag{4} X_1 T_i & = T_i X_1 & \quad & \text{for all $i = 2,\ldots{},n-1$}
   \end{alignat}
   Then there are inverse isomorphisms $\Phi: \widetilde{\mathfrak{A}}_n^{\op{aff}} \rightarrow \mathfrak{A}_n^{\op{aff}}$, $\Psi: \mathfrak{A}_n^{\op{aff}} \rightarrow \widetilde{\mathfrak{A}}_n^{\op{aff}}$ of groups determined by
   \begin{align*}
      \Phi(T_i) & = T_{s_i} \quad i = 1,\ldots{},n-1 \\
      \Phi(X_1) & = T_{-e_1}^{-1}
   \end{align*}
   and
   \begin{align*}
      \Psi(T_{s_i}) & = T_i \quad i = 1,\ldots{},n-1 \\
      \Psi(T_{s_0}) & = \Psi(T_u) T_1 \Psi(T_u)^{-1} \\
      \Psi(T_{u}) & = T_{n-1} \ldots {}T_1 X_1
   \end{align*}
\end{lemma}
\begin{proof}
   We only give some brief indications as the proof consists mostly of straightforward computations. First of all, for every extended Coxeter group $W$ the decomposition $W = W_{\op{aff}} \ltimes \Omega$ induces an isomorphism
   \[ \mathfrak{A}(W_{\op{aff}} \ltimes \Omega) \simeq \mathfrak{A}(W_{\op{aff}}) \ltimes \Omega \]
   where the action of $\Omega$ on $\mathfrak{A}(W_{\op{aff}})$ is determined by $u(T_w) = T_{u(w)}$. Moreover, one sees easily that there is an isomorphism
   \[ \mathfrak{A}(W_{\op{aff}}) \simeq \left< \{ T_s\}_{s \in S} : \underbrace{T_s T_t T_s \ldots }_{\text{$m$ factors}} = \underbrace{T_t T_s T_t \ldots }_{\text{$m$ factors}},\quad s,t \in S,\ \op{ord}(st) = m < \infty \right> \]
   To see that $\Psi$ is well-defined it is therefore enough to check that
   \begin{align*} \Psi(T_{s_i}) \Psi(T_{s_j}) & = \Psi(T_{s_j}) \Psi(T_{s_i}),\quad i,j = 1,\ldots{},n-1,\ |i-j| > 1 \\
      \Psi(T_{s_i}) \Psi(T_{s_{i+1}}) \Phi(T_{s_i}) & = \Psi(T_{s_{i+1}}) \Psi(T_{s_i}) \Psi(T_{s_{i+1}}),\quad i = 1,\ldots{},n-2 \\
      \Psi(T_{u^{-1}}) \Psi(T_{s_i}) \Psi(T_{u^{-1}})^{-1} & = \Psi(T_{s_{i-1}}),\quad i \in \{0,\ldots{},n-1\} = \Z/n\Z
   \end{align*}
   The first two relations are immediate and the last one follows from a lengthy computation. By definition, the well-definedness of $\Phi$ amounts to checking relations (1)-(4). Again relations (1) and (2) are immediate, while (3) and (4) follow from (1), (2) and
   \[ T_{u^{-1}} T_{s_i} T_{u^{-1}}^{-1} = T_{s_{i-1}} \]
   Finally, more straightforward and lengthy computations show that $\Phi$ and $\Psi$ are inverse to each other.
\end{proof}

In terms of this description, the map $\mathfrak{A}_n^{\op{aff}} \longrightarrow \mathfrak{A}_n$ is then given by
\[ \widetilde{\mathfrak{A}}_n^{\op{aff}} \longrightarrow \mathfrak{A}_n,\quad T_i \longmapsto T_{s_i},\ X_1 \mapsto 1 \]
Writing $H_n = H_n(a,b)$ for the generic pro-$p$ Hecke algebra of $W_0$ with constant parameters $a,b$, the above morphism of groups induces a morphism of algebras
\[ \pi: H_n^{\op{aff}} \longrightarrow H_n \]
by \cref{prop:second_presentation_of_hecke} (one easily checks that the quadratic relations are preserved). Explicitly, this map is the identity on $H_n$ (viewing it as a subalgebra of $H_n^{\op{aff}}$) and sends the generator $T_{s_0}$ to the elements
\[ T_{s_{n-1}} \ldots T_{s_1} T_{s_1} T_{s_{n-1}}^{-1} \ldots T_{s_1}^{-1} \]
The map $\pi$ is very important because it gives a description of the center of $H_n$ in terms of the center of $H_n^{\op{aff}}$. Namely, it turns out that $\pi$ maps the center of $H_n^{\op{aff}}$ surjectively onto the center of $H_n$. This should be contrasted with the fact that
\[ Z(H_n^{\op{aff}})\cap H_0 = R \]
More explicitly, the center of $H_n$ is the algebra of symmetric polynomials in the pairwise commutative \textit{Jucys-Murphy elements} $J_1,\ldots{},J_n$ given recursively by
\[ J_1 := 1,\quad J_{i+1} = T_{s_i} J_i T_{s_i} \]
The following lifts of the $J_i$ under $\pi$ are also called Jucys-Murphy elements
\[ J^{\op{aff}}_1 := X_1,\quad J_{i+1}^{\op{aff}} = T_{s_i} J_i^{\op{aff}} T_{s_i} \]
The elements $J_1^{\op{aff}},\ldots{},J_n^{\op{aff}}$ also commute pairwise. In fact, they are nothing else but the images of the standard basis vectors $e_i \in \Z^n$ under the unnormalized Bernstein map.

\begin{lemma}\label{lem:jucys-murphys_equals_bernstein-zelevinsky}
   Let $\mathfrak{o} = \mathfrak{o}_D$ be the spherical orientation (see \cref{def:spherical_orientation}) of $W$ associated to the dominant Weyl chamber
   \[ D = \{ x \in \R^n : x_1 < \ldots < x_n \} \]
   Then
   \[ J_i^{\op{aff}} = \theta_\mathfrak{o}(e_i),\quad i = 1,\ldots{},n \]
   as an equality in $H_n^{\op{aff}}$ (in fact already in $\mathfrak{A}_n^{\op{aff}}$).
\end{lemma}
\begin{proof}
   By induction. For $i = 1$ the statement follows immediately from the definitions. Indeed, $-e_1 \in \overline{D}$ and therefore
   \[ \theta_\mathfrak{o}(e_1) = \theta_\mathfrak{o}(-e_1)^{-1} = T_{-e_1}^{-1} = X_1 = J_1^{\op{aff}} \]
   For the induction step we need to prove that
   \[ T_{s_i} \theta_\mathfrak{o}(e_i) T_{s_i} = \theta_\mathfrak{o}(e_{i+1}) \]
   But this is shown in \cite[3.2.4]{Macdonald}, where the notation $Y^x$ is used instead of $\theta_\mathfrak{o}(x)$.
\end{proof}
\end{example}

\begin{example}[pro-$p$-Iwahori Hecke algebras]\label{ex:pro-$p$-Iwahori Hecke algebras} Let $F$ be a nonarchimedean local field, i.e. a field endowed with a nontrivial discrete valuation $\nu_F: F \rightarrow \Z \cup \{+\infty\}$ whose residue field $k$ is a finite field with cardinality $q$ a power of some prime $p$. Let $\mathbf{G}$ be a connected reductive group over $F$, $G = \mathbf{G}(F)$ the group of rational points, $I \leq G$ an Iwahori subgroup in the sense of \cite[3.7]{TitsCorvallis} and $I(1) \leq I$ its pro-$p$ radical. Recall that the pro-$p$ radical of a profinite group containing an open pro-$p$ subgroup is by definition (see \cite[3.6]{VigHenn}) its largest open normal pro-$p$ subgroup. Finally, let $R$ be a commutative ring.

     To this data one associates an $R$-algebra, the pro-$p$-Iwahori Hecke algebra, as follows. Let
     \[ \mathcal{H}^{(1)} := H_R(G,I(1)) = \op{End}_G(\op{ind}^G_{I(1)} \bbf{1}_R) \]
     be the ring of endomorphisms of the $G$-representation induced from the trivial representation of $I(1)$. This $R$-algebra is canonically identified with the convolution algebra $R[I(1)\backslash G/I(1)]$ of $I(1)$-double cosets, where the product of the basis elements $T_t, T_{t'}$ corresponding to double cosets $t,t' \in I(1)\backslash G/I(1)$ is given by
     \[ T_t  T_{t'} = \sum_{t''} m(t,t'; t'') T_{t''} \]
     Here the sum runs over all double cosets and $m(t,t'; t'')$ denotes the number of $I^{(1)}$-left cosets of $t \cap gt'^{-1}$ for $g \in t''$ arbitrary.
     
     Vignéras \cite{VigProI} has shown that the set $I(1)\backslash G/I(1)$ is in bijection with a certain group $W^{(1)}$ (which can be given the structure of a pro-$p$ Coxeter group) and that the basis elements $T_w$ of $\mathcal{H}^{(1)}$ satisfy relations of Iwahori-Matsumoto type
     \begin{align*} T_w T_{w'} & = T_{ww'}, \quad \text{ if } \ell(w)+\ell(w') = \ell(ww') \\
        T_s^2 & = a_s T_{s^2} + b_s T_s, \quad \text{ if } \ell(s) = 1
      \end{align*}
      Given a suitable structure of a pro-$p$ Coxeter group on $W^{(1)}$ and an affine extended Coxeter group on $W$, the above presentation implies that $\mathcal{H}^{(1)}$ is an affine pro-$p$ Hecke algebra in the sense of \cref{def:affine_pro-$p$_Hecke_algebra}. Our goal now is to explicitly construct these structures.

      Let $C$ be the chamber of the building of $\mathbf{G}$ which corresponds to the Iwahori subgroup $I = I_C$ and let $\mathbf{S} \leq \mathbf{G}$ denote the maximal split torus corresponding to an apartment containing $C$. Let $\mathbf{Z} \leq \mathbf{N} \leq \mathbf{G}$ respectively denote the centralizer and normalizer of $\mathbf{S}$ in $\mathbf{G}$, and let $Z = \mathbf{Z}(F)$ and $N = \mathbf{N}(F)$ denote their groups of rational points. Let
     \[ Z_0 := Z \cap I,\quad Z_0(1) := Z \cap I(1),\quad Z_k := Z_0/Z_0(1) \]
     The groups $Z_0$ and $Z_0(1)$ are normal in $N$ (cf. \cite[3.7]{VigProI}) and thus we may form the quotient groups
     \[ W := N/Z_0,\quad W^{(1)} := N/Z_0(1) \]
     The inclusion $N \subseteq G$ induces a bijection (\cite[Proposition 3.35]{VigProI})
     \[ W^{(1)} \stackrel{\sim}{\longrightarrow} I^{(1)}\backslash G/I^{(1)},\quad [n] \longmapsto I^{(1)} n I^{(1)} \]
     and therefore $\mathcal{H}^{(1)}$ has a canonical basis $T_w$ indexed by elements $w \in W^{(1)}$. Moreover, we have an exact sequence
     \begin{equation}\label{eq:pro-p-Coxeter} \begin{xy} \xymatrix{ 1 \ar[r] & Z_k \ar[r] & W^{(1)} \ar[r] & W \ar[r] & 1 } \end{xy} \end{equation}
        with $Z_k$ finite abelian (in fact $Z_k$ identifies with the rational points of a torus over $k$, cf. \cite[3.7]{VigProI}).
     
        Let us now see how $W$ can be given the structure of an affine extended Coxeter group. The theory of buildings associates to the triple $(\mathbf{G},\mathbf{S},F)$ an apartment $A = A(\mathbf{G},\mathbf{S},F)$ (see \cite[1.2]{TitsCorvallis}), which is an affine space over the vector space $V = X_\ast(\mathbf{S})\otimes \R$ endowed with a homomorphism
     \[ \nu: N \longrightarrow \op{Aut}_{\op{aff}}(A) \]
     into the group of affine automorphisms of $A$, such that we have a commutative diagram
     \begin{equation}\label{eq:comm_diag} \begin{xy} \xymatrix{ 1 \ar[r] & Z \ar[d]^\nu \ar[r] & N \ar[d]^\nu \ar[r] & W_0 \ar[d] \ar[r] & 1 \\ 1 \ar[r] & V \ar[r] & \op{Aut}_{\op{aff}}(A) \ar[r] & \op{GL}(V) \ar[r] & 1 } \end{xy} \end{equation}
     Here the rightmost vertical map is the canonical (faithful) representation of the finite Weyl group $W_0$ as a reflection group in $V$, and the leftmost vertical map is uniquely determined by the condition
     \[ \chi(\nu(z)) = -\nu_F(\chi(z)) \quad \forall z \in Z,\ \chi \in X^\ast(\mathbf{Z}) \]
      This condition implies that $\nu(Z) \leq V$ is a discrete subgroup of rank $\op{dim}(V)$, i.e. a lattice in $V$. As $I$ is compact so is $Z_0 = I \cap Z$, and hence $Z_0 \leq \op{ker}(\nu)$ since $\nu$ is continuous. Therefore, $\nu$ factors to a map
     \[ \nu: W \longrightarrow \op{Aut}_{\op{aff}}(A) \]
     which after an appropriate choice of an origin in $A$ and hence an identification $A \simeq V$ will define a map $\rho: W \rightarrow \op{Aut}_{\op{aff}}(V)$. Regardless of this choice the above diagram shows that the subgroup $\rho_0(W)$ defined in \axiom{ACIII} is equal to the image of $W_0$ in $\op{GL}(V)$, and hence \axiom{ACIII} is verified. Moreover, the injectivity of $W_0 \rightarrow \op{GL}(V)$ and the commutativity of the above diagram imply that
     \[ \rho(W) \cap V = \nu(Z) \]
     regardless of the choice of an origin. In particular \axiom{ACV} holds, since $\nu(Z)$ is a lattice in $V$. Now in order to choose an origin, we first need to define the locally finite set $\mathfrak{H}$ of hyperplanes.
     
     For this we need to recall a few more facts from the theory of buildings. Let
     \[ \Phi := \Phi(\mathbf{G},\mathbf{S}) \subseteq X^\ast(\mathbf{S}) \subseteq V^\vee \]
     denote the root system\footnote{Note that in general, $\Phi$ is not reduced and so in particular it will \textit{not} be the reduced root system $\Phi$ attached to the affine extended Coxeter group $W$ by \cref{lem:unpacking_affine_extended_Coxeter_groups}.} of the pair $(\mathbf{G},\mathbf{S})$.
     
     Let us for a root $a \in \Phi$ denote by $\mathbf{U}_a \leq \mathbf{G}$ the root subgroup corresponding to $a$, let $U_a := \mathbf{U}_a(F)$ and $U_a^\ast := U_a - \{1\}$. For every $u \in U^\ast_a$, the set $U_{-a} u U_{-a} \cap N$ consists of a single element, denoted $m(u)$ in \cite[1.4]{TitsCorvallis}. The linear part of the image $\nu(m(u)) \in \op{Aut}_{\op{aff}}(A)$ is the reflection $s_a \in \op{GL}(V)$ associated to $a$, which implies that $\nu(m(u))$ is an affine reflection. Let $\alpha(a,u)$ denote the affine function whose linear part is $a$ and whose vanishing set is the hyperplane fixed by $\nu(m(u))$. For any affine function $\alpha: A \rightarrow \R$ with linear part $a$, let
     \[ X_\alpha := \{ u \in U_a^\ast : \alpha(a,u) \geq \alpha \} \cup \{1\} \]
     It is a major result of the theory of buildings that $X_\alpha$ is in fact a subgroup of $U_a$. However, from the definition it is immediately clear that the $X_\alpha$ with $\alpha$ running over all affine functions with linear part $a$ form an exhaustive and separated filtration of $U_a$. Moreover, any two elements $m(u), m(u')$ with $u,u' \in U_a^\ast$ differ only by an element of $Z$. Since $\nu(Z) \leq V$ is a discrete subgroup, it follows that the filtration $\{X_\alpha\}_\alpha$ is locally constant and that
     \[ \mathfrak{H} := \{ \{ \alpha(a,u) = 0 \} : a \in \Phi,\ u \in U_a^\ast\} \]
     is a locally finite set of hyperplanes. The set $\mathfrak{H}$ is left invariant under the action of $W$, thus verifying \axiom{ACI}, which follows from
     \[ nU_a^\ast n^{-1} = U_{n(a)}^\ast,\quad n \in N,\ a \in \Phi \]
     and
     \begin{equation}\label{eq:m_equivariance} nm(u)n^{-1} = m(nun^{-1}), \quad n \in N,\ a \in \Phi,\ u \in U_a^\ast \end{equation}
     The last equation also shows that
     \[ n^{-1} X_\alpha n = X_{\alpha \circ \nu(n)} \]
     Let $W(\mathfrak{H})$ denote the subgroup of $\op{Aut}_{\op{aff}}(A)$ generated by all $\nu(m(u))$ with $a \in \Phi'$, $u \in U_a^\ast$. Then $W(\mathfrak{H})$ leaves $\mathfrak{H}$ invariant as $\mathfrak{H}$ is already invariant under $W$. Fixing a $W_0$-invariant positive definite scalar product on $V$, the group $W(\mathfrak{H})$ becomes the affine reflection group generated by the orthogonal reflections $s_H$ with $H \in \mathfrak{H}$. By \cite[Ch. V, \textsection 3.10, Proposition 10]{Bourbaki}, there exists a special point $p \in A$. As $W(\mathfrak{H})$ maps special points to special points and acts transitively on the set of chambers, we may assume that $p$ lies in the closure of the chamber $C$ which corresponds to $I$.
     
     Identifying $A$ and $V$ via $p$, we will assume from now on that $A = V$ and $p = 0$ and we will put $\rho := \nu$. Letting $C_0 := C$, we fulfill \axiom{ACIV} and \axiom{ACIX}. This gives $W(\mathfrak{H})$ the structure of a Coxeter group (cf. \cref{sub:Affine extended Coxeter groups and affine pro-$p$ Hecke algebras}) by letting the set of distinguished generators be the set $S(C_0)$ of reflections with respect to the walls of $C_0$.
     
   We now want to construct lifts $\widetilde{s}_H \in W$ of the reflections $s_H \in W(\mathfrak{H})$, $H \in \mathfrak{H}$ satisfying \axiom{ACVI} and \axiom{ACVII}. Note that \axiom{ACII} and \axiom{ACVIII} are then satisfied automatically. Consider the subgroup $N_{\op{aff}} \leq N$ generated by $Z_0$ and all $m(u)$, $u \in U_a^\ast$, $a \in \Phi$. From relation \eqref{eq:m_equivariance} it follows that it is a normal subgroup of $N$. Moreover, by construction the map $\nu$ restricts to a surjection
     \[ N_{\op{aff}} \twoheadrightarrow W(\mathfrak{H}) \]
     which we claim has kernel $Z_0$ (cf. \cite[3.9]{VigProI}). Admitting this claim, it follows that the subgroup
     \[ W_{\op{aff}} = N_{\op{aff}}/Z_0 \subseteq N/Z_0 = W \]
     maps isomorphically onto $W(\mathfrak{H})$ under $\nu$ and hence that \axiom{ACVI} and \axiom{ACVII} are fulfilled by letting $\widetilde{s}_H$ be the unique preimage in $W_{\op{aff}}$ of $s_H$ under $\nu$.
     
     Let us now show that $\op{ker}(\nu)\cap N_{\op{aff}} = Z_0$. We already saw that $Z_0$ is contained in the kernel. The reverse inclusion follows from the following characterization of the Iwahori subgroup given by Haines and Rapoport (Def. 1, Prop. 3 and Lemma 17 in \cite{HaRa})
     \[ I_C = \op{Fix}_G(\overline{C}) \cap G_{\op{aff}} \]
     Here $\overline{C}$ denotes the chamber in the reduced building of $\mathbf{G}$ corresponding to $C$ and $\op{Fix}_G(\overline{C})$ denotes the subgroup of all elements of $G$ fixing $\overline{C}$ pointwise. Note that $\op{Fix}_G(C) \subseteq \op{Fix}_G(\overline{C})$. Moreover, $G_{\op{aff}}$ denotes the subgroup of $G$ generated by all parahoric subgroups, or equivalently, the subgroup generated by $Z_0$ and the root subgroups $U_a$, $a \in \Phi$. It follows that $N_{\op{aff}} \subseteq G_{\op{aff}}$ and therefore that
     \[ \op{ker}(\nu)\cap N_{\op{aff}} \subseteq \op{Fix}_G(\overline{C}) \cap G_{\op{aff}} = I \]
     Since $\op{ker}(\nu) \subseteq Z$, this implies
     \[ \op{ker}(\nu)\cap N_{\op{aff}} \subseteq Z \cap I =  Z_0 \]
     We have therefore now verified conditions \axiom{ACI}--\axiom{ACIX}. It remains to show that \axiom{ACX} holds, i.e. that the subgroup
     \[ X = \rho^{-1}(V) \stackrel{\text{(!)}}{=} Z/Z_0 \leq W \]
     is finitely generated and commutative. But this is shown in \cite[Theorem 1.0.1]{HaRo}, where are $Z$,$Z_0$ are denoted by $M(F)$, $M(F)_1$. To apply this theorem one has to note that $Z_0$ is the unique parahoric subgroup of $Z$ (see the discussion before \cite[Proposition 3.15]{VigProI}).

     We have therefore given $W = N/Z_0$ the structure of an affine extended Coxeter group. By \cref{lem:unpacking_affine_extended_Coxeter_groups}, this induces on $W$ the structure $W = (W,W_{\op{aff}},S,\Omega)$ of an extended Coxeter group. The exact sequence \eqref{eq:pro-p-Coxeter} therefore makes $W^{(1)} = N/Z_0(1)$ into a pro-$p$ Coxeter group, provided we specify lifts of the generators $\widetilde{s}_H \in W_{\op{aff}}$, $H \in S(C_0)$ that satisfy the braid relations. This is the content of the next lemma. In order to state it, we need to recall a few more things from \cite{TitsCorvallis}.
     
     Let $\{\overline{X}_\alpha\}_{\alpha}$ denote the family of quotients of the descending filtration $\{X_\alpha\}_{\alpha}$, i.e.
     \[ \overline{X}_\alpha = X_\alpha/X_{\alpha+\varepsilon} \]
     for $\varepsilon > 0$ sufficiently small. If $a \in \Phi$ with $2a \in \Phi$, then the inclusion $U_{2a} \subseteq U_a$ induces an inclusion $\overline{X}_{2\alpha} \subseteq \overline{X}_\alpha$ for every $\alpha$ with $\alpha_0 = a$. The set $\Phi_{\op{af}}$ of affine roots is then defined to be (cf. \cite[1.6]{TitsCorvallis})
     \[ \Phi_{\op{af}} = \{ \alpha : \alpha_0 \in \Phi,\ \overline{X}_{2\alpha} \neq \overline{X}_{\alpha} \} \]
     where $\overline{X}_{2\alpha} = \{1\}$ by convention if $2\alpha_0 \not\in \Phi$. Note that if $\overline{X}_\alpha \neq 1$ but $\alpha \not\in \Phi_{\op{af}}$, then necessarily $2\alpha \in \Phi_{\op{af}}$. Hence, every $H \in \mathfrak{H}$ is of the form $H = \{\alpha = 0\}$ for some $\alpha \in \Phi_{\op{af}}$.

\begin{lemma}\label{lem:lifting_braid_relations}
   In the situation of the above example, the following holds. Given a wall $H \in S(C_0)$ let $\alpha \in \Phi_{\op{af}}$ denote the unique affine root with
   \[ H = \{ \alpha = 0 \},\quad C_0 \subseteq \{ \alpha > 0 \} \quad \text{and} \quad \frac{1}{2}\alpha \not\in \Phi_{\op{af}} \]
   and put $n_H = m(u)$ for some arbitrary $u \in X_{\alpha}$ with nonzero image under $X_\alpha \twoheadrightarrow \overline{X}_\alpha$. Then for all $H,H' \in S(C_0)$ with $\op{ord}(s_H s_{H'}) < \infty$ we have the relation
   \[ n_H n_{H'} n_H \ldots \equiv n_{H'} n_H n_{H'} \ldots \mod{Z_0}(1) \]
   in $N$, where the number of factors on both sides equals $\op{ord}(s_H s_{H'})$.
\end{lemma}
\begin{proof}
   If $H,H'$ are parallel, then either $H = H'$ or $\op{ord}(s_H s_{H'}) = \infty$, in which case there is nothing to prove. So we may assume that $H,H'$ are not parallel, and hence that the intersection $H \cap H'$ contains a non-empty face of the fundamental chamber $C_0$. To every face $F$ of $C_0$ is associated a subgroup (parahoric) $K_F \leq G$ as follows (see also \cite[3.7]{VigProI}). Every face $F$ of $C_0$ corresponds to a face in the apartment $A^{\natural}$ corresponding to $\mathbf{S}$ in the reduced building of $\mathbf{G}$. To every \textit{nonempty} bounded subset $\Omega \subseteq A^{\natural}$ is attached (\cite[4.6.26]{BTII} and \cite[5.1.9]{BTII}) a smooth affine group scheme $\mathfrak{G}_\Omega^0$ over the ring of integers of the local field $F$ ($\mathcal{O}^{\natural}$ in the notation of \cite{BTII}) with generic fiber $G$. In the notation of \cite{BTII} the parahoric $K_F$ corresponding to $F$ is then defined to be (see \cite[5.2.6]{BTII} and the remark before \cite[5.2.9]{BTII})
   \[ K_F = \mathfrak{G}_F^0(\mathcal{O}) \cap G(K^{\natural}) = \mathfrak{G}_F^0(\mathcal{O}^{\natural}) \]
   From \cite[5.2.4]{BTII} it follows that the group $K_F$ is also characterized as the subgroup generated by $Z_0$ and the $X_\alpha$ for all $\alpha$ with $\alpha_0 \in \Phi$ and $F \subseteq \{ \alpha > 0 \}$.

   For $F = C_0$ one has $K_F = I$ and for any two faces $F,F'$ of $C_0$ (see \cite[Corollary 3.21]{VigProI})
   \[ F \subseteq \overline{F'} \quad \Rightarrow\quad K_{F} \supseteq K_{F'}\ \ \text{and}\ \  K_{F}(1) \subseteq K_{F'}(1) \]
   Here $K_{F}(1)$ denotes the pro-$p$ radical of $K_F$. In particular
   \begin{equation}\label{eq:all_contained_in_I(1)} K_F(1) \subseteq I(1) \end{equation}
   for all faces $F$ of $C_0$. Let now $F \neq \emptyset$ be a face of $C_0$ contained in $H \cap H'$. The subset
   \[ \Phi_F := \{ \alpha_0 : \alpha \in \Phi_{\op{af}},\ F \subseteq \{ \alpha = 0 \} \} \subseteq \Phi \]
   is a sub root system of $\Phi$. Moreover, elementary arguments show that $\alpha_0, \alpha'_0 \in \Phi_F$ are part of a basis of $\Phi_F$. Here it is used that $\frac{1}{2}\alpha, \frac{1}{2}\alpha' \not\in \Phi_{\op{af}}$.

   Let $\overline{G}_F = \mathfrak{G}_F^0 \times_{\op{Spec}(\mathcal{O}_F)} \op{Spec}(k)$ be the reduction of the group scheme $\mathfrak{G}_F^0$ and let $\overline{G}_F^{\op{red}}$ denote the quotient of $\overline{G}_F$ by its unipotent radical. Identifying $\overline{G}_F^{\op{red}}$ with the unique Levi subgroup of $\overline{G}_F$ containing the reduction $\overline{S}$ of the canonical model of $S$ over $\mathcal{O}_F$, the group $\overline{G}_F^{\op{red}}$ coincides with the group denoted by the same symbol in \cite[3.5]{TitsCorvallis}. The canonical map
   \[ K_F = \mathfrak{G}_F^0(\mathcal{O}_F) \longrightarrow \overline{G}_F^{\op{red}}(k) \]
   is surjective and its kernel is equal to the pro-$p$-radical $K_F(1)$ by \cite[3.7]{VigHenn}. The group $\overline{G}_F^{\op{red}}$ is a connected reductive group over $k$ and its root system with respect to the maximal split subtorus $\overline{S}$, as a subset of $X^\ast(\overline{S}) = X^\ast(S)$, is equal to $\Phi_F$ (see \cite[3.5.1]{TitsCorvallis}). Moreover, for any $\alpha \in \Phi_{\op{af}}$ with $F \subseteq \{ \alpha = 0 \}$ we have
   \begin{equation}\label{eq:x_a_contained} X_\alpha \subseteq K_F \end{equation}
   and (see \cite[3.5.1]{TitsCorvallis})
   \begin{equation}\label{eq:x_bar_is_u_bar} \overline{X}_\alpha = \overline{U}_{\alpha_0}(k) \end{equation}
   as an equality of subgroups of $\overline{G}_F^{\op{red}}(k) = K_F/K_F(1)$. Here $\overline{U}_{\alpha_0}$ denotes the root subgroup of $\overline{G}_F^{\op{red}}$ corresponding to $\alpha_0 \in \Phi_F$.

   Let now $\alpha,\alpha' \in \Phi_{\op{af}}$ and $u \in X_{\alpha}$, $u' \in X_{\alpha'}$ with $n_H = m(u)$, $n_{H'} = m(u')$ be as in the statement of this lemma. Denote by $\overline{u}, \overline{u'}$ the images of $u,u'$ under $K_F \twoheadrightarrow \overline{G}_F^{\op{red}}(k)$. By \eqref{eq:x_bar_is_u_bar} and the choice of $u,u'$, the elements $\overline{u},\overline{u'}$ are not reduced to the neutral element. Applying \eqref{eq:x_bar_is_u_bar} to the reductions of the elements appearing in the decomposition of $m(u)$ and $m(u')$ respectively, it follows that
   \begin{equation}\label{eq:m_compatible_with_reduction} \overline{m(u)} = m(\overline{u}),\quad \overline{m(u')} = m(\overline{u'}) \end{equation}
   by uniqueness, where $\overline{m(u)}$, $\overline{m(u')}$ denote the images of $m(u),m(u') \in K_F$ under $K_F \twoheadrightarrow \overline{G}_F^{\op{red}}(k)$ and $m(\overline{u})$, $m(\overline{u'})$ are associated to $\overline{u},\overline{u'}$ in the same way as $m(u),m(u')$ are associated to $u,u'$. In fact, $m(\overline{u}),m(\overline{u'})$ are the elements canonically associated to the elements $\overline{u},\overline{u'}$ and the root datum $(Z(\overline{S})(k),(\overline{U}_a(k))_{a \in \Phi_F})$ (in the sense of \cite[6.1.1]{BTI}) by \cite[6.1.2 (2)]{BTI}. Applying Proposition \cite[6.1.8]{BTI} to the root datum given by restricting $(Z(\overline{S})(k), (\overline{U}_a(k))_{a \in \Phi_F})$ to the rank two sub root system $(\Z \alpha_0 + \Z \alpha'_0) \cap \Phi_F$, it follows that
   \begin{equation}\label{eq:braid_relation_in_finite_quotient} m(\overline{u}) m(\overline{u'}) m(\overline{u}) \ldots = m(\overline{u'}) m(\overline{u}) m(\overline{u'}) \ldots \end{equation}
      where the number of factors on both sides equals the order of $s_{\alpha_0} s_{\alpha'_0} \in W_0(\Phi_F) \subseteq W_0(\Phi)$. Let $x$ be an arbitrary point of the face $F$. As $s_H$ and $s_{H'}$ both lie in the stabilizer $\op{Aut}_{\op{aff}}(V)_x$ and the map $\op{Aut}_{\op{aff}}(V) \rightarrow \op{GL}(V)$ restricts to an injection $\op{Aut}_{\op{aff}}(V)_x \hookrightarrow \op{GL}(V)$, the order of the homomorphic image $s_{\alpha_0} s_{\alpha'_0}$ of $s_H s_H' \in \op{Aut}_{\op{aff}}(V)$ is equal to the order of $s_H s_H'$. Hence, it follows from \eqref{eq:braid_relation_in_finite_quotient} and \eqref{eq:m_compatible_with_reduction} that $ab^{-1} \in K_F(1)$ where
   \[ a = m(u)m(u')m(u) \ldots \quad b = m(u')m(u)m(u') \ldots \]
   and the number of factors on the right hand side of each equation equals the order of $s_H s_H'$. On the other hand we have $ab^{-1} \in Z$ since $Z$ equals the kernel of the composition $N \rightarrow W(\mathfrak{H}) \rightarrow W_0$, and $a,b$ are mapped to the same element under $N \rightarrow W(\mathfrak{H})$. Hence
   \[ ab^{-1} \in Z \cap I^{(1)} = Z_0(1) \]
   and the claim follows.
\end{proof}

Now, for a generator $s = \widetilde{s}_H$, $H \in S(C_0)$ of $W_{\op{aff}}$ let $n_s \in W^{(1)} = N/Z_0(1)$ be the class of an element $n_H \in N$ as chosen according to the lemma. Then the lemma states that $W^{(1)}$ together with the choice of these lifts becomes a pro-$p$ Coxeter group in the sense of \cref{def:pro-$p$ Coxeter group}.

We can now finally state the relation between pro-$p$-Iwahori Hecke algebras and generic pro-$p$ Hecke algebras.

\begin{lemma}\label{lem:pro-$p$-Iwahori Hecke algebras are generic pro-$p$ Hecke algebras}
   Given a generator $s = \widetilde{s}_H \in W_{\op{aff}}$, $H \in S(C_0)$, let $\alpha \in \Phi_{\op{af}}$ be the unique affine root with $H = \{ \alpha = 0\}$ and $\frac{1}{2}\alpha \not\in \Phi_{\op{af}}$. Then the following holds.
   \begin{enumerate}
      \item We have
         \[ q_s := \# I n_s I /I = \# I(1) n_s I(1)/I(1) = \# \overline{X}_{\alpha} = q^{d(v)} \]
   where $d(v) \in \N$ denotes the integer associated to the vertex $v$ of the local Dynkin diagram $\Delta(\Phi_{\op{af}})$ (see \cite[1.8]{TitsCorvallis}) corresponding to $H$.
\item Let $F$ be any face of $C_0$ contained in $H$, let $\overline{G}_{F,s}$ be the subgroup of $\overline{G}^{\op{red}}_F$ generated by $\overline{X}_\alpha$ and $\overline{X}_{-\alpha}$ (cf. proof of \cref{lem:lifting_braid_relations}) and let
   \[ Z_{k,s} := \overline{G}_{F,s} \cap Z_k \leq Z_k \]
   Then $Z_{k,s}$ is independent of the choice of $F$.
\item Let
   \[ c_s := \sum_{t \in Z_{k,s}} c_s(t)t \in R[Z_k] \]
   with
   \[ c_s(t) := \# \left( n_s \overline{X}_\alpha n_s \cap \overline{X}_\alpha n_s t \overline{X}_\alpha \right) \]
   where the intersection is taken inside $\overline{G}^{\op{red}}_F(k)$ (for any $F$ as in (ii)), and $n_s$ denotes (by abuse of notation) the image of the element $n_H = m(u) \in X_\alpha$ under $X_\alpha \twoheadrightarrow \overline{X}_\alpha \subseteq \overline{G}^{\op{red}}_F(k)$.

   Then the families $(q_s)_s$, $(c_s)_s$ fulfill condition \eqref{eq:condstar} with respect to the pro-$p$ Coxeter group $W^{(1)}$ defined above, and the $R$-linear isomorphism
   \[ \mathcal{H}^{(1)}_R(q_s,c_s) \stackrel{\sim}{\longrightarrow} H_R(G,I^{(1)}),\quad T_w \longmapsto T_w \]
   is a morphism of $R$-algebras. In addition, if the order of $Z_{k,s}$ equals $q_s - 1$ (e.g. if $\mathbf{G}$ is split), the coefficients $c_s(t)$ are equal to $1$ for all $t$.
   \end{enumerate}
\end{lemma}
\begin{proof}
   ad (i): For the first equality and second equality we refer to \cite[Corollary 3.30]{VigProI}, recalling that (see \eqref{eq:x_bar_is_u_bar}) $\overline{X}_\alpha$ is naturally identified with the group $U_{\alpha_0}(k)$ (denoted $U_{\alpha,\mathfrak{F},k}$ in \cite{VigProI}). The last equality follows directly from the definition of the integer $d(v)$ as the sum of the dimensions of $k$-vector spaces (cf. \cite[1.6]{TitsCorvallis}) $\overline{X}_\alpha / \overline{X}_{2\alpha}$ and $X_{2\alpha}$.

   ad (ii): The independence of $Z_{k,s}$ from the choice of $F$ is implicit in the proof of \cite[Proposition 4.4]{VigProI}. It can also be seen as follows (cf. \cite[Proposition 3.26]{VigProI}). Given two (nonempty) faces $F,F'$ of $C_0$ with $F' \subseteq \overline{F}$, we have an inclusion $K_{F} \subseteq K_{F'}$. The image of $K_F$ in $K_{F',k}$ under the natural map $K_{F'} \twoheadrightarrow K_{F',k} = K_{F'}/K_{F'}(1)$ is equal to the subgroup $M_F$ generated by $Z_k$ and the groups $\overline{U}_a(k)$, $a \in \Phi_F \subseteq \Phi_{F'}$. Moreover, $M_F$ appears as a Levi subgroup of a parabolic subgroup $Q_F = M_F \ltimes U_F$, such that the inverse image of the unipotent radical $U_F$ under $K_F \rightarrow K_{F',k}$ equals $K_F(1)$. Hence, we have an induced injective map
   \[ K_{F,k} \stackrel{\sim}{\longrightarrow} M_F \subseteq K_{F',k} \]
   which is the identity on $Z_k$ and the $\overline{U}_a(k)$, $a \in \Phi_F$. In particular, the subgroup $\overline{G}_{F',s}$ of $\overline{G}^{\op{red}}_{F'} = K_{F',k}$ generated by $\overline{X}_\alpha = \overline{U}_{\alpha_0}(k)$ and $\overline{X}_{-\alpha} = \overline{U}_{-\alpha_0}(k)$ equals the image of $\overline{G}_{F,s}$ under the embedding $K_{F,k} \hookrightarrow K_{F',k}$. As this embedding is the identity on $Z_k$, it follows that
   \[ \overline{G}_{F,s} \cap Z_k = \overline{G}_{F',s} \cap Z_k \]

   ad (iii): As Vignéras has observed \cite[Theorem 4.7]{VigProI}, the condition \eqref{eq:condstar} is not only sufficient but also necessary for the existence of an algebra structure on the free $R$-module over $W^{(1)}$ satisfying \eqref{eq:heckecond1} and \eqref{eq:heckecond2}. As the latter two conditions are satisfied for $H_R(G,I(1))$ by \cite[Proposition 4.1]{VigProI} and \cite[Proposition 4.4]{VigProI}, the first claim follows. For the proof of the second claim we refer to \cite[Proposition 4.4]{VigProI}.
\end{proof}

We have therefore now recognized $H_R(G,I(1))$ as an affine pro-$p$ Hecke algebra. Since in this case the abelian group $T = Z_k$ underlying $W^{(1)}$ is finite, all the structure results of \cref{thm:structure of affine pro-$p$ Hecke algebras} hold unconditionally for $H_R(G,I(1))$ (cf. \cref{rmk:finiteness_of_orbits}). In particular the center of $H_R(G,I(1))$ is finitely generated as an $R$-algebra, and $H_R(G,I(1))$ is module-finite over its center.
\end{example}

\begin{example}[Affine Yokonuma-Hecke algebras]\label{ex:affine-yokonuma-hecke-algebra}
   The Yokonuma-Hecke algebras $Y_{d,n}$ of \cref{ex:yokonuma} have a natural variant, the \textit{affine Yokonuma-Hecke algebras} $Y_{d,n}^{\op{aff}}$. According to \cite[Introduction]{ChlouSe} these algebras have first been introduced by Juyumaya and Lambropoulou \cite{JuLa} under the name of `\textit{$d$-th framization of the Iwahori-Hecke algebra of $B$-type}'. Later they were studied by Chlouveraki and Poulain d'Andecy \cite{ChlouPou} under the name `\textit{affine Yokonuma-Hecke algebra}'. The different terminologies reflect the two different ways in which $Y_{d,n}^{\op{aff}}$ can be seen as modifications of other algebras. This is visualized in the following commutative diagram (defined down below)
   \begin{equation}\label{eq:framization_and_affinization} \begin{xy} \xymatrix{ Y_{d,n}^{\op{aff}} \ar[d] \ar[r] & H_n^{\op{aff}} \ar[d]^\pi \\ Y_{d,n} \ar[r] & H_n } \end{xy} \end{equation}
   were the left column is the `framization' of the right column, and the upper row is the `affinization' of the lower one.
   
   Chlouveraki and Sécherre have recognized \cite{ChlouSe} the algebra $Y_{d,n}^{\op{aff}}$ as (in our terminology) generic pro-$p$ Hecke algebras for the split pro-$p$ Coxeter group $W^{(1)} = T \rtimes W$, $T = (\Z/d\Z)^n$, $W = \Z^n \rtimes S_n$. In fact, we will see in a moment that they are affine pro-$p$ Hecke algebras in the sense of \cref{def:affine_pro-$p$_Hecke_algebra}.
   
   Let us first recall the definition (cf. \cite[3.1]{ChlouPou}) of the affine Yokonuma-Hecke algebras. For integers $d,n \geq 1$, the algebra $Y_{d,n}^{\op{aff}}$ is the algebra over $R = \C[u^{\pm 1},v]$ generated by elements
   \[ g_1,\ldots{},g_{n-1}, t_1,\ldots{},t_n, X_1, X_1^{-1} \]
   subject to the relations
   \begin{alignat}{2}
      \tag{1} g_i g_j & = g_j g_i & \quad & \text{for all $i,j = 1,\ldots{},n-1$ such that $|i-j| > 1$} \\
      \tag{2} g_i g_{i+1} g_i & = g_{i+1} g_i g_{i+1} & \quad & \text{for all $i = 1,\ldots{},n-2$} \\
      \tag{3} t_i t_j & = t_j t_i & \quad & \text{for all $i,j = 1,\ldots{},n$} \\
      \tag{4} g_i t_j & = t_{s_i(j)} g_i & \quad & \text{for all $i = 1,\ldots{},n-1$ and $j = 1,\ldots{},n$} \\
      \tag{5} t_j^d & = 1 & \quad & \text{for all $j = 1,\ldots{},n$} \\
      \tag{6} X_1 X_1^{-1} & = X_1^{-1} X_1 = 1 & & \\
      \tag{7} X_1 g_1 X_1 g_1 & = g_1 X_1 g_1 X_1 & & \\
      \tag{8} X_1 g_i & = g_i X_1 & \quad & \text{for all $i = 2,\ldots{},n-1$} \\
      \tag{9} X_1 t_j & = t_j X_1 & \quad & \text{for all $j = 1,\ldots{},n$} \\
      \tag{10} g_i^2 & = u^2 + ve_i g_i & \quad & \text{for all $i = 1,\ldots{},n-1$}
   \end{alignat}
   where as in \cref{ex:yokonuma} we let
   \[ e_i = \frac{1}{d}\sum_{0 \leq s < d} (t_i/t_{i+1})^s \]
   Note that this definition of the Yokonuma-Hecke algebra slightly differs from the one given in \cite{ChlouPou}, as we are considering $Y_{d,n}^{\op{aff}}$ as an algebra over the ring $\C[u^{\pm},v]$ in two formal variables. The algebra of \cite{ChlouPou} is obtained by specializing $Y_{d,n}^{\op{aff}}$ along the ring homomorphism $\C[u^{\pm 1},v] \rightarrow \C[q^{\pm 1}]$ sending $u \mapsto 1$ and $v \mapsto q-q^{-1}$.

   Let us now recognize $Y_{d,n}^{\op{aff}}$ as an affine pro-$p$ Hecke algebra. More precisely, let us show that $Y_{d,n}^{\op{aff}}$ is isomorphic to a generic pro-$p$ Hecke algebra $\mathcal{H}^{(1)}$ for the pro-$p$ Coxeter group $W^{(1)} = (\Z/d\Z)^n \rtimes W$, where $W = \Z^n \rtimes S_n$ is the affine extended Coxeter group of \cref{ex:affine_coxeter_group}(ii), (cf. also \cref{ex:affine_hecke_algebra_of_gl_n}) acting on $(\Z/d\Z)^n$ by permuting the coordinates via the projection $W \twoheadrightarrow W_0 = S_n$. By \cref{prop:second_presentation_of_hecke}, we have an isomorphism $\mathcal{H}^{(1)} \simeq R[\mathfrak{A}(W^{(1)})]/I$, where $I$ is the ideal generated by the elements $T_{n_s}^s - a_s T_{n_s^2} - b_s T_{n_s}$. It therefore suffices to see that $Y_{d,n}^{\op{aff}}$ is a quotient of $R[\mathfrak{A}(W^{(1)})]$ by the same ideal $I$.
   
   For this we need the following `framed version' of \cref{lem:presentation_of_affine_braid_group}, providing two descriptions of the \textbf{$d$-modular framed affine braid group}.

\begin{lemma}\label{lem:affine-pro-p-braid}
   Let $\widetilde{\mathfrak{A}}_{d,n}^{\op{aff},(1)}$ denote the group generated by elements
   \[ g_1,\ldots{},g_{n-1},t_1,\ldots,{}t_n,X_1 \]
   subject to the relations (1)-(9) above, and let $\mathfrak{A}_{d,n}^{\op{aff},(1)} = \mathfrak{A}(W^{(1)})$ with $W^{(1)}$ as above. Then there are inverse isomorphisms $\Phi: \widetilde{\mathfrak{A}}_{d,n}^{\op{aff},(1)} \rightarrow \mathfrak{A}_{d,n}^{\op{aff},(1)}$ and $\Psi: \mathfrak{A}_{d,n}^{\op{aff},(1)} \rightarrow \widetilde{\mathfrak{A}}_{d,n}^{\op{aff},(1)}$ determined by
   \begin{align*}
      \Phi(g_i) & = T_{s_i} \quad i = 1,\ldots{},n-1 \\
      \Phi(t_i) & = T_{t_i} \quad i = 1,\ldots{},n \\
      \Phi(X_1) & = T_{-e_1}^{-1}
   \end{align*}
   and
   \begin{align*}
      \Psi(T_{s_i}) & = g_i \quad i = 1,\ldots{},n-1 \\
      \Psi(T_{s_0}) & = \Psi(T_u) g_1 \Psi(T_u)^{-1} \\
      \Psi(T_u) & = g_{n-1} \ldots g_1 X_1 \\
      \Psi(T_{t_i}) & = t_i \quad i = 1,\ldots{},n
   \end{align*}
   where $t_1,\ldots{},t_n$ denote the canonical generators of $(\Z/d\Z)^n$.
\end{lemma}
\begin{proof} Follows immediately from $\mathfrak{A}(W^{(1)}) = T \rtimes \mathfrak{A}(W)$, $\widetilde{\mathfrak{A}}^{\op{aff},(1)}_{d,n} = T \rtimes \widetilde{\mathfrak{A}}_n^{\op{aff}}$ and \cref{lem:presentation_of_affine_braid_group} (where the $T_i$ have to be replaced by $g_i$).
\end{proof}

From the above lemma, it follows readily that the affine Yokonuma-Hecke algebra $Y_{d,n}^{\op{aff}}$ is the quotient of the group algebra $R[\mathfrak{A}_{d,n}^{\op{aff},(1)}]$ by the ideal generated by the relations in (10). This doesn't yet prove that $Y_{d,n}^{\op{aff}}$ is isomorphic to the generic pro-$p$ Hecke algebra $\mathcal{H}^{(1)}$ (with the obvious parameters $a_s$, $b_s$), since we still need to show that the latter exists. Moreover, carefully comparing the relations in (10) with the generators of the ideal $I$ realizing the isomorphism $\mathcal{H}^{(1)} \simeq R[\mathfrak{A}(W^{(1)})]/I$ of \cref{prop:second_presentation_of_hecke}, one notes that $I$ is generated by one extra relation not appearing in (10). However, as we will see now, this extra relation is redundant.

\begin{theorem}\label{thm:affine_yokonuma_are_pro-$p$}
   Let $W^{(1)} = T \rtimes W$, $T = (\Z/d\Z)^n$, $n_s = s$, $W = \Z^n \rtimes S_n$ be the split pro-$p$ Coxeter group constructed above. For $i = 0,\ldots{},n-1$ put
   \[ a_{s_i} := u^2 \in R,\quad b_{s_i} := \frac{v}{d} \sum_{s \in \Z/d\Z} (t_i/t_{i+1})^s \in R[T] \]
   where $t_0 := t_n$ by convention and the group $T$ is written multiplicatively. Then the following holds.
   \begin{enumerate}
      \item The parameter families $(a_s)_{s \in S}$, $(b_s)_{s \in S}$ defined above satisfy condition \eqref{eq:condstar} of \cref{thm:exhecke}, and hence the generic pro-$p$ Hecke algebra $\mathcal{H}^{(1)}_{d,n} := \mathcal{H}^{(1)}(a_s,b_s)$ for these parameters exists.
      \item There is an isomorphism of $R$-algebras
         \[ Y_{d,n}^{\op{aff}} \stackrel{\sim}{\longrightarrow} \mathcal{H}^{(1)}_{d,n} \]
         determined by
         \begin{align*}
            g_i & \mapsto T_{s_i} \\
            t_i & \mapsto T_{t_i} \\
            X_1 & \mapsto T_{-e_1}^{-1}
         \end{align*}
         where the element $T_{-e_1}^{-1} \in \mathcal{H}^{(1)}_{d,n}$ is well-defined since the parameters $a_s$ are invertible in $R$ (cf. \cref{sub:Presentations of generic pro-$p$ Hecke algebras via (generalized) braid groups}).
      \item Via the structure of an affine extended Coxeter group on $W$ from \cref{ex:affine_coxeter_group}, $\mathcal{H}^{(1)}_{d,n}$ (and hence $Y_{d,n}^{\op{aff}}$) becomes an affine pro-$p$ Hecke algebra. Moreover, as the group $T$ is finite, all results of \cref{thm:structure of affine pro-$p$ Hecke algebras} apply without restriction (cf. \cref{rmk:finiteness_of_orbits}). In particular $\mathcal{H}^{(1)}_{d,n}$ is finite as a module over its center, and the latter is given by the invariants
         \[ Z(\mathcal{H}^{(1)}_{d,n}) = \left(\mathcal{A}^{(1)}_\mathfrak{o}\right)^W \]
         of the subalgebra
         \[ \mathcal{A}^{(1)}_\mathfrak{o} = \bigoplus_{x \in X^{(1)}} R \widehat{\theta}_\mathfrak{o}(x) \subseteq \mathcal{H}^{(1)}_{d,n} \]
         where $\mathfrak{o}$ is any spherical orientation of $W$. Since the parameters $a_s$ are units in $R$, the unnormalized Bernstein map $\theta_\mathfrak{o}$ exists and provides an isomorphism
         \[ R[X^{(1)}] \stackrel{\sim}{\longrightarrow} \mathcal{A}^{(1)}_\mathfrak{o},\quad x \mapsto \widehat{\theta}_\mathfrak{o}(x) \quad (x \in X^{(1)}) \]
         Since the group $X^{(1)} = T \times \Z^n$ is commutative, the algebra $\mathcal{A}^{(1)}_\mathfrak{o}$ is also commutative and hence by \cref{thm:structure of affine pro-$p$ Hecke algebras} is equal to its centralizer in $\mathcal{H}^{(1)}_{d,n}$.
   \end{enumerate}
\end{theorem}
\begin{proof}
   We begin by showing (i) using the equivalent reformulation of condition \eqref{eq:condstar} given in \cref{rmk:condstar}. From \cref{ex:affine_coxeter_group}(ii) recall that
   \[ us_iu^{-1} = s_{i-1} \]
   for all $i \in \{0,\ldots{},n-1\} = \Z/n\Z$. In particular all elements of $S$ are conjugate. The first of the two conditions of \cref{rmk:condstar} therefore follows from the equation
   \begin{equation}\label{eq:affine-yoko-1} u(b_{s_i}) = b_{s_{i-1}} \end{equation}
      Recall here that $u = \tau^{e_n} s_{n-1} \ldots s_1$ and hence $u$ acts on $T = (\Z/d\Z)^n$ via the cycle $s_{n-1} \ldots s_1 = (n-1\ \ldots 1)$. Now in order to see that the second condition of \cref{rmk:condstar} holds true, first note that by a general result on reflection groups (see \cite[Ch. V, \textsection 3.3, Proposition 2 (I)]{Bourbaki}), it follows that for every $s \in S$
   \[ \{ w \in W : wsw^{-1} = s \} = \{ 1, s \} \cdot{} \{ v \in \Omega: vsv^{-1} = s \} = \{ 1,s \} \cdot{} \{ u^{nk} : k \in \Z \} \]
   Here
   \[ u^n = \tau^{e_1+\ldots{}+e_n} \in X \]
   Taking $\widetilde{w} = w$ to be the canonical lift for the split pro-$p$ group $W^{(1)} = T \rtimes W$ for every $w$ as above, condition (ii) of \cref{rmk:condstar} follows then from
   \[ s(b_s) = b_s,\quad s \in S \]
   and
   \[ s(t)t^{-1} b_s = b_s,\quad s \in S,\ t \in T \]
   The latter two equations follow by a simple computation, for instance
   \[ s_i(b_{s_i}) = \frac{v}{d} \sum_{s \in \Z/d\Z} s_i (t_i/t_{i+1})^s = \frac{v}{d} \sum_{s \in \Z/d\Z} (t_{i+1}/t_i)^s = \frac{v}{d} \sum_{s \in \Z/d\Z} (t_i/t_{i+1})^{-s} = b_{s_i} \]
   and writing $t = t_1^{k_1} \ldots t_n^{k_n}$ we have
   \[ s_i(t)t^{-1} b_{s_i} = t_i^{k_{i+1}-k_i} t_{i+1}^{k_i - k_{i+1}} b_{s_i} = (t_i/t_{i+1})^{k_{i+1}-k_i} b_{s_i} = b_{s_i} \]
   Now claim (ii) is an almost immediate consequence of (i), \cref{lem:affine-pro-p-braid} and \cref{prop:second_presentation_of_hecke}, since $Y_{d,n}^{\op{aff}}$ and $\mathcal{H}^{(1)}_{d,n}$ both are quotients of the group algebra $R[\mathfrak{A}^{(1)}_{d,n}]$ by ideals $I$ and $I'$ respectively, generated by the elements
   \[ T_{n_{s_i}}^2 - a_{s_i}T_{n_{s_i}^2} - b_{s_i} T_{n_{s_i}} = T_{s_i}^2 - a_{s_i} - b_{s_i} T_{s_i} \]
   However, for $I'$ the index $i$ ranges from $0$ to $n$ whereas for $I$ it only ranges from $1$ to $n$. But by equation \eqref{eq:affine-yoko-1} we have
   \[ T_u \left(T_{s_1}^2 - a_{s_1} - b_{s_1} T_{s_1}\right) T_{u^{-1}} = T_{s_0}^2 - a_{s_0} - b_{s_0} T_{s_0} \]
   and therefore $I = I'$. In fact, this argument shows that
   \[ I = I' = (T_{s_1}^2 - a_{s_1} - b_{s_1} T_{s_1}) \]
   Finally for (iii), there is nothing to prove.
\end{proof}

\begin{rmk}
   By definition, the algebras $Y_{d,n}^{\op{aff}}$ and $\mathcal{H}^{(1)}_{d,n}$ are algebras over the ring $R = \C[u^{\pm 1},v]$. However, the definition of $Y_{d,n}^{\op{aff}}$ and the verification of condition \eqref{eq:condstar} did not make use of the invertibility of $u$, i.e. both algebras can already be defined over $\C[u,v]$. In contrast, the above isomorphism between $Y_{d,n}^{\op{aff}}$ and $\mathcal{H}^{(1)}_{d,n}$ \textit{does} make explicit use of the invertibility of $u$. This poses the question whether both algebras are isomorphic over $\C[u,v]$.
\end{rmk}

\begin{rmk}
   In the beginning of \cref{sub:Main examples of affine pro-$p$ Hecke algebras} we remarked that the examples of pro-$p$-Iwahori Hecke algebras and affine Yokonuma-Hecke algebras overlap. Let us now make this more precise. It is not hard to see that whenever $d$ is of the form
   \[ d = q-1, \quad \text{$q = p^r$ a prime-power} \]
   the pro-$p$ Coxeter group $W^{(1)} = T\rtimes W$ considered above (together with the structure of an affine extended Coxeter group on $W$ !) can be identified with the pro-$p$ Coxeter $W^{(1)} = N/Z_0(1)$ associated to the reductive group $\op{GL}_n$, the diagonal subtorus in $\op{GL}_n$ and the Iwahori subgroup
   \[ I = \begin{pmatrix}
      \mathcal{O}_F^\times & \mathfrak{p}_F & \ldots & \mathfrak{p}_F \\
      \mathcal{O}_F & \ddots & & \vdots \\
      \vdots & \ddots & \ddots & \mathfrak{p}_F \\
      \mathcal{O}_F & \ldots & \mathcal{O}_F & \mathcal{O}_F^\times
   \end{pmatrix} \leq \op{GL}_n(F) \]
   by \cref{ex:pro-$p$-Iwahori Hecke algebras}, where $F$ denotes any nonarchimedean local field with residue field $k = \mathcal{O}_F/\mathfrak{p}_F$ of cardinality $q$. Explicitly, the choice of a uniformizing element $\pi \in \mathcal{O}_F$ provides a splitting of the exact sequence
   \[ \begin{xy} \xymatrix{ 1 \ar[r] & Z_k \ar[r] & W^{(1)} \ar[r] & W \ar[r] & 1 } \end{xy} \]
      by identifying an element $w = \tau^x \sigma \in W = \Z^n \rtimes S_n$ with the class of the monomial matrix $(\pi^{-x_i} \delta_{\sigma(i),j})_{i,j}$. Moreover, the choice of a primitive $d$-th root of unity in $k$ provides an isomorphism of the group $Z_k = (k^\times)^n$ with $(\Z/d\Z)^n$. Unwinding the definition of the groups $Z_{k,s}$ in \cref{lem:pro-$p$-Iwahori Hecke algebras are generic pro-$p$ Hecke algebras}, one sees that for $s = s_i$, $i \in \{1,\ldots{},n-1\}$
      \[ Z_{k,s} = \{ (t_i/t_{i+1})^j : j \in \Z/d\Z \} \]
      as subgroups of $Z_k = T$, where we recall that the $t_i$ denote the standard generators of $T = (\Z/d\Z)^n$ and that this group is written multiplicatively. From this and \cref{lem:pro-$p$-Iwahori Hecke algebras are generic pro-$p$ Hecke algebras}, it then follows immediately that we have an isomorphism
      \[ Y_{d,n}^{\op{aff}} \otimes_R \C \simeq \mathcal{H}^{(1)}_{d,n} \otimes_R \C \simeq H_\C(\op{GL}_n(F),I(1)) \]
      where the base change $\-- \otimes_R \C$ is with respect to the homomorphism $R = \C[u^{\pm 1},v] \rightarrow \C$ sending $u$ to $\sqrt{q}$ and $v$ to $q-1$.
\end{rmk}

We now come back to the commutative diagram \eqref{eq:framization_and_affinization}
\begin{equation*} \begin{xy} \xymatrix{ Y_{d,n}^{\op{aff}} \ar[d] \ar[r] & H_n^{\op{aff}} \ar[d]^\pi \\ Y_{d,n} \ar[r] & H_n } \end{xy} \end{equation*}
   that was mentioned in the beginning, and will define all the maps involved. The right vertical arrow is the quotient map $\pi$ constructed in \cref{ex:affine_hecke_algebra_of_gl_n}. We recall that $\pi$ was induced by the map
   \begin{align*} \mathfrak{A}_n^{\op{aff}} \simeq \widetilde{\mathfrak{A}}_n^{\op{aff}} & \longrightarrow \mathfrak{A}_n \\
      T_i & \longmapsto T_{s_i},\quad i = 1,\ldots{},n-1 \\
      X_1 & \longmapsto 1
   \end{align*}
   between braid groups. In fact, the whole diagram \eqref{eq:framization_and_affinization} is induced by a diagram
   \begin{equation}\label{eq:framization_and_affinization_of_braid_groups} \begin{xy} \xymatrix{ \mathfrak{A}_n^{\op{aff},(1)} \ar[d] \ar[r] & \mathfrak{A}_n^{\op{aff}} \ar[d]^\pi \\
   \mathfrak{A}_n^{(1)} \ar[r] &  \mathfrak{A}_n } \end{xy} \end{equation}
   of braid groups, where $\mathfrak{A}_n^{(1)} = \mathfrak{A}(T \rtimes W_0)$ denotes the $d$-modular framed braid group. The horizontal arrows are given by the projection onto the second factor, with respect to the isomorphisms (cf. \cref{ex:yokonuma-pres})
\[ \mathfrak{A}_n^{\op{aff},(1)} \simeq T \rtimes \mathfrak{A}_n^{\op{aff}},\quad \mathfrak{A}_n^{(1)} \simeq T \rtimes \mathfrak{A}_n \]
Finally the left vertical arrow is given by
\[ \mathfrak{A}_n^{\op{aff},(1)} \simeq T \rtimes \mathfrak{A}_n^{\op{aff}} \stackrel{\op{id}\times \pi}{\longrightarrow} T \rtimes \mathfrak{A}_n \simeq \mathfrak{A}_n^{(1)} \]
It is easy to see that \eqref{eq:framization_and_affinization_of_braid_groups} respects the quadratic relations and hence induces a diagram between Hecke algebras.
\end{example}

\subsection{Some finiteness properties of affine extended Coxeter groups} 
\label{sub:Some finiteness properties of affine extended Coxeter groups}
In this section, we will collect some properties of affine extended Coxeter groups and their associated affine hyperplane arrangements that will be needed in the structure theorem of affine pro-$p$ Hecke algebras. In particular, we will prove some finiteness properties which will directly imply corresponding finitness properties for affine pro-$p$ Hecke algebras.

Throughout this section, we will fix an affine extended Coxeter group $W$. The notations introduced in \cref{def:affine_coxeter_group} and \cref{lem:unpacking_affine_extended_Coxeter_groups} will be used freely.

\begin{rmk}\label{rmk:abstract_and_concrete_coxeter_geometry}
   Let us begin by relating the abstract geometric terminology introduced in \cref{term:basic} to the concrete geometry of the hyperplane arrangement $(V,\mathfrak{H})$. It is a basic result (see. \cite[Ch. V, \textsection 3.2, Théorème 1]{Bourbaki}) of the theory of affine reflection groups that $W(\mathfrak{H})$ acts simply transitively on the set of chambers $\pi_0(V-\bigcup_{H \in \mathfrak{H}} H)$. Since $\rho$ induces an isomorphism $W_{\op{aff}} \stackrel{\sim}{\longrightarrow} W(\mathfrak{H})$, also $W_{\op{aff}}$ acts simply transitively on the set of chambers. Via the map $w \mapsto w\bullet{}C_0$ we can therefore identify the set of `abstract chambers' (in the sense of \cref{term:basic}) with the chambers in $V$. Moreover, under this identification the `abstract orbit map' $W \longrightarrow W_{\op{aff}}$ of \cref{term:basic} coincides with the actual orbit map given by $w \mapsto w\bullet{}C_0$.
   The identification of abstract and concrete chambers also extends to hyperplanes such that the notion of `separation' is preserved. More precisely, the map
   \begin{align*} \mathfrak{H} & \longrightarrow \{ w s w^{-1} : w \in W_{\op{aff}},\ s \in S \} \\
      H & \longmapsto \widetilde{s}_H
   \end{align*}
   is a bijection, and for $H \in \mathfrak{H}$ and $w,w' \in W_{\op{aff}}$ it holds true that $H$ separates $w(C_0)$ from $w'(C_0)$ if and only if the abstract hyperplane $\widetilde{s}_H$ separates the abstract chambers $w,w'$ in the sense of \cref{term:basic}. The bijectivity follows easily from the fact that $\rho$ gives an isomorphism $W_{\op{aff}} \stackrel{\sim}{\longrightarrow} W(\mathfrak{H})$ that satisfies $\rho(\widetilde{s}_H) = s_H$ and maps the set $S \subseteq W_{\op{aff}}$ bijectively onto the set of reflections with respect to the walls of the fundamental chamber $C_0$. That the notion of `separation' is preserved follows from the fact that the set of abstract hyperplanes separating $1,w$ and the set of concrete hyperplanes separating $C_0, w\bullet{}C_0$ respectively can both be read off from the choice of a reduced expression $w = s_1\ldots s_r$. 
   We may therefore identify concrete and abstract hyperplanes without harm and write
   \[ H = \widetilde{s}_H \]
   Using the formal notation $s_H = H$ of \cref{term:basic}, we therefore have
   \[ \widetilde{s}_H = s_H \in W_{\op{aff}} \]
   and the compatibility $\rho(\widetilde{s}_H) = s_H$ can be written as
   \[ \rho(s_H) = s_H \in W(\mathfrak{H}) \]
   Whenever it matters, it will be either be stated explicitly or it will be clear from the context whether we view $s_H$ as an element of $W_{\op{aff}}$ or of $W(\mathfrak{H})$, so that no confusion will arise.
\end{rmk}

   As we just saw, the abstract geometry of an affine extended Coxeter group $W$ is faithfully reflected (no pun intended) in the geometry of the affine hyperplane arrangement $(V,\mathfrak{H})$. Using the extra structure available on $(V,\mathfrak{H})$, this dictionary between abstract and concrete geometry makes somes questions concerning $W$ very transparent.
   
   Consider for instance the following basic problem of Coxeter geometry. Given chambers $C,C'$ and $C''$, when does
   \[ d(C,C'') = d(C,C') + d(C',C'') \]
   hold true? This problem can be made more transparent with the help of the following `vector-valued' distance.

\begin{definition}\label{def:vector-valued_distance}
   Given chambers $C,C' \in \pi_0(V-\bigcup_{H \in \mathfrak{H}} H)$ the element
   \[ \vec{d}(C,C') \in \Z^{\Phi^+} \]
   defined component-wise via
   \[ \vec{d}(C,C')_\alpha = \pi_0(-\alpha)(C') - \pi_0(-\alpha)(C) \in \Z \]
   is called the \textbf{vector-valued distance} between $C$ and $C'$. Here $\pi_0(-\alpha)$ denotes the map induced on connected components by
   \[ -\alpha: V-\bigcup_{H \in \mathfrak{H}}H \longrightarrow \R-\Z \]
   and the difference $\pi_0(-\alpha)(C') - \pi_0(-\alpha)(C)$ is to be understood in the sense of affine spaces over $\Z$, with $\pi_0(\R-\Z)$ carrying the obvious affine structure. In other words
   \[ \vec{d}(C,C')_\alpha = k'-k \]
   if $k, k' \in \Z$ are such that
   \[ -\alpha(C) \subseteq \left]k,k+1\right[ \quad\text{and}\quad -\alpha(C') \subseteq \left]k',k'+1\right[ \]
\end{definition}

\begin{rmk}\label{rmk:vector-valued_vs_plain_old_distance}
   From the definition it is quite obvious that $|\vec{d}(C,C')_\alpha|$ equals the number of hyperplanes of the form $H_{\alpha,k}$, $k \in \Z$ separating $C$ from $C'$. In particular the vector-valued and the normal distance are related by the formula
   \begin{equation}\label{eq:vector-valued_and_plain_old_distance} d(C,C') = |\vec{d}(C,C')| := \sum_{\alpha \in \Phi^+} |\vec{d}(C,C')_\alpha| \end{equation}
   which justifies the terminology. In particular, using \cref{rmk:abstract_and_concrete_coxeter_geometry}, the length $\ell$ on $W$ can be expressed in terms of $\vec{d}$ as
   \[ \ell(w) = d(C_0,w(C_0)) = \sum_{\alpha \in \Phi^+} |\vec{d}(C_0,w(C_0))_\alpha| \]
   where $C_0$ denotes the fundamental chamber and $w(C_0) = \rho(w)(C_0)$ the action of $W$ via $\rho: W \rightarrow \op{Aut}_{\op{aff}}(V)$. By definition, an element $x \in X$ acts by translation by $\rho(x) \in V$ on $V$. It is therefore easy to see that 
   \[ \vec{d}(C_0,\rho(x)(C_0))_\alpha = -\alpha(\rho(x)) \]
   leading to the more useful formula
   \begin{equation}\label{eq:length_of_translation} \ell(x) = \sum_{\alpha \in \Phi^+} |\alpha(\rho(x))| ,\quad x \in X\end{equation}
\end{rmk}

\begin{rmk}\label{rmk:vector-valued_distance}
   Let us now return to the problem posed above, to determine when three chambers $C,C',C''$ fulfill the relation
   \[ d(C,C'') = d(C,C') + d(C',C'') \]
   and let us see how the vector-valued distance helps in making this problem more transparent. Immediately from the definition it follows that
   \[ \vec{d}(C,C'') = \vec{d}(C,C') + \vec{d}(C',C'') \]
   By equation \eqref{eq:vector-valued_and_plain_old_distance}, we are therefore reduced to determine for which $x,y \in \Z^{\Phi^+}$ we have
   \[ |x+y| = |x| + |y| \]
   where $|x| = \sum_\alpha |x_\alpha|$. Let $\preceq$ denote relation on $\Z^{\Phi^+}$ defined by
   \begin{align}\label{eq:well_partial_order}
      x \preceq y & :\Leftrightarrow |y| = |x|+|y-x| \\
      & \Leftrightarrow x_\alpha (y_\alpha - x_\alpha) \geq 0 \quad \forall \alpha \notag \\
      & \Leftrightarrow x_\alpha = 0 \ \vee \ (x_\alpha y_\alpha > 0\ \wedge \ |x_\alpha| \leq |y_\alpha|) \quad \forall \alpha \notag
   \end{align}
   It is easy to see that $\preceq$ is a partial order. Moreover, the above problem can now be phrased equivalently in terms of $\preceq$ as
   \begin{equation}\label{eq:reducedness_and_weak_Bruhat_order} d(C,C'') = d(C,C')+d(C',C'') \quad \Leftrightarrow \quad \vec{d}(C,C') \preceq \vec{d}(C,C'') \end{equation}
\end{rmk}

\begin{rmk}\label{rmk:weak_Bruhat_order}
   In particular, fixing a chamber $C$, the relation
   \begin{align*} C' \preceq_C C'' \quad & :\Leftrightarrow \quad d(C,C'') = d(C,C') + d(C',C'') \\
      & \Leftrightarrow\quad \vec{d}(C,C') \preceq \vec{d}(C,C'')
   \end{align*}
   defines a partial order on chambers. For $C = C_0$ this is just the \textit{weak Bruhat order}, i.e. the partial order induced on $W_{\op{aff}}$ via
   \begin{align*} w' \preceq w'' \quad & :\Leftrightarrow\quad w'(C_0) \preceq_{C_0} w''(C_0) \\
      & \Leftrightarrow d(C_0,w''(C_0)) = d(C_0,w'(C_0)) + d(w'(C_0),w''(C_0)) \\
      & \Leftrightarrow \quad \ell(w'') = \ell(w') + \ell((w')^{-1}w'')
   \end{align*}
   is the weak Bruhat order.
   
   It is known that the weak Bruhat order on an affine Coxeter group is a \textit{well partial order}, in fact the affine Coxeter groups are characterized among the infinite Coxeter group as those for which this property holds (see \cite{Hultman}). In the next lemma we will prove that $\preceq$ defines a well partial order on $\Z^{\Phi^+}$, thus recovering the first statement about the weak Bruhat order, as the proof is not difficult and moreover the result is crucial for the structure theory of affine pro-$p$ Hecke algebras. In fact, the well partial order property guarantees that $\mathcal{H}^{(1)}$ is finitely generated as a left module over a certain subalgebra $\mathcal{A}^{(1)}_\mathfrak{o} \subseteq \mathcal{H}^{(1)}$ (see the proof of \cref{thm:structure of affine pro-$p$ Hecke algebras}), which is an important step in showing that $\mathcal{H}^{(1)}$ is finitely generated as a module over its center.
\end{rmk}

Let us recall the notion of a well partial order (cf. \cite{Kruskal}).

\begin{definition}\label{def:well partial order}
   A partial order $\leq$ on a set $X$ is said to be a \textbf{well partial order} if for every nonempty subset $\Lambda \subseteq X$ the set $\op{min}(\Lambda)$ of minimal elements of $\Lambda$ is \textit{nonempty} and \textit{finite}.
\end{definition}

Obviously this generalizes the notion of a well ordering from total orders to partial orders, hence the name. Let us now show that $\preceq$ defines a well partial order on $\Z^{\Phi^+}$.

\begin{lemma}\label{lem:weak Bruhat is wpo}
   $(\Z^{\Phi^+},\preceq)$ is a well partial order.
\end{lemma}
\begin{proof}
   Let $\Lambda \subseteq X$ be a nonempty subset and assume that $\op{min}(\Lambda)$ was infinite. We would then find a sequence $(\lambda_n)_{n \in \N}$ of pairwise distinct elements $\lambda_n \in \op{min}(\Lambda)$, which would necessarily be also pairwise incomparable. Choose a numbering $\Phi^+ = \{\alpha_1,\ldots,\alpha_m\}$ of the positive roots, and look at the sequence $(\lambda_n(\alpha_1))_{n \in \N}$ of `first coordinates'.
   
   There are two possibilities, either this sequence is finite or infinite. In the first case we may (after possibly replacing $(\lambda_n)_{n \in \N}$ by a subsequence) assume that the sequence $(\gamma_n(\alpha_1))_{n \in \N}$ is constant. In the second case we can assume (again replacing $(\lambda_n)_{n \in \N}$ by a subsequence if necessary) that the sequence $(\gamma_n(\alpha_1))_{n \in \N}$ is strictly increasing or decreasing with respect to the usual total order on $\Z$, i.e. strictly increasing with respect to the well partial order (!) $x \preceq y :\Leftrightarrow x(y-x) \geq 0$ on $\Z$.
   
   Repeating this procedure with $\alpha_2, \alpha_3, \dots , \alpha_m$, we may therefore assume that for every $\alpha \in \Phi^+$ the sequence $(\gamma_n(\alpha))_{n \in \N}$ is either constant or strictly increasing with respect to the well partial order $\preceq$ on $\Z$. In particular, since the order $(\Z^{\Phi^+},\preceq)$ is just the power of the order $(\Z,\preceq)$, we would have $\lambda_1 \preceq \lambda_2$, contradicting the fact that the $\lambda_n$ are pairwise incomparable.
\end{proof}

\begin{cor}\label{cor:weak Bruhat is wpo}
   For every chamber $C$, the relation $\preceq_C$ on the set of chambers defined in \cref{rmk:weak_Bruhat_order} is a well partial order.
\end{cor}

\begin{rmk}\label{rmk:products_of_wpos_are_wpos}
   Obviously the above proof holds verbatim with $(\Z,\preceq)$ replaced by any well partial order, and the argument recovers the basic fact that finite products of well partial orders are again well partial orders (cf. \cite{Kruskal}).
\end{rmk}

As already mentioned, the well partial order property of $\preceq$ is crucial for proving the finiteness of $\mathcal{H}^{(1)}$ as a left module over $\mathcal{A}^{(1)}_\mathfrak{o}$. But it is also crucial for proving yet another finiteness property, namely it ensures that $\mathcal{A}^{(1)}_\mathfrak{o}$ is finitely generated as an algebra (see \cref{thm:structure of affine pro-$p$ Hecke algebras}). This rests on the finiteness property of the submonoids $X_D \leq X$ defined below, which we will prove in the next lemma.

\begin{definition}\label{def:X_D}
   Given a Weyl chamber $D \in \pi_0\left(V-\bigcup_{\alpha \in \Phi}H_\alpha\right)$, we let
   \[ X_D := \{ x \in X\ :\ \rho(x) \in \overline{D} \} \]
   be the submonoid of $X$ consisting of all elements which acts by translation by an element of the closure of $D \subseteq V$ under $\rho: W \rightarrow \op{Aut}_{\op{aff}}(V)$.
\end{definition}

\begin{lemma}\label{lem:finiteness_of_X_D}
   $X_D$ is finitely generated as a monoid.
\end{lemma}
\begin{proof}
   Consider the evaluation map
   \begin{align*} \nu: V & \longrightarrow \R^{\Phi^+} \\
      v & \longmapsto (\alpha \mapsto \alpha(v))
   \end{align*}
   Since the action of $X$ preserves the set $\mathfrak{H}$ of affine hyperplanes and by definition of $\Phi$ we have for every $\alpha \in \Phi$
   \[ \forall k \in \R\ \ H_{\alpha,k} \in \mathfrak{H} \Leftrightarrow k \in \Z \]
   it follows from $\rho(x)(H_{\alpha,0}) = H_{\alpha,-\alpha(\rho(x))} \in \mathfrak{H}$ that $\alpha(\rho(x)) \in \Z$ for every $x \in X$. Hence, $\nu(\rho(x))$ lies in $\Z^{\Phi^+}$ and we may consider the image
   \[ \Xi_D := \nu(\rho(X_D)) \subseteq \Z^{\Phi^+} \]
   As $\nu$ is a group homomorphism, $\Xi_D \subseteq \Z^{\Phi^+}$ is a submonoid. Moreover, the partial order $\preceq$ restriced to $\Xi_D$ is compatible with the monoid structure in the sense that
   \begin{equation}\label{eq:compatible} a \preceq a+b \quad \forall a,b \in \Xi_D \end{equation}
   To see this, let $\varepsilon_{D,\alpha}$ denote the sign of $\alpha(q)$ for $q \in D$ arbitrary. We then have the following equivalence for an element $v \in V$
   \[ v \in \overline{D} \quad \Leftrightarrow\quad \forall \alpha \in \Phi^+\ \ \varepsilon_{D,\alpha} \alpha(v) \geq 0 \]
   The implication `$\Rightarrow$' is obvious and the reverse implication follows by choosing a point $q \in D$ and noting that $v$ lies in the closure of the half-open line segement
   \[ \{ (1-\lambda) q + \lambda v\ :\ 0 \leq \lambda < 1 \} \subseteq D \]
   This implies that $\Xi_D$ is characterized as
   \begin{equation}\label{eq:characterizing_xi_D} \Xi_D = \{ a \in \nu(\rho(X)) \ :\ \forall \alpha \in \Phi^+\ \ \varepsilon_{D,\alpha} a(\alpha) \geq 0 \} \end{equation}
   In particular, for $a,b \in \Xi_D$ and every $\alpha \in \Phi^+$ we have $a(\alpha)b(\alpha) \geq 0$ and hence $a \preceq a+b$ by definition of $\preceq$.
 
   Let us now call an element $a \in \Xi_D$ \textit{irreducible} if $a \neq 0$ and $a$ cannot be written as a sum $a = b+c$ with $b,c \in \Xi_D$ and $b,c \neq 0$. Since $\preceq$ is a well partial order, it is in particular a well-founded relation. This implies that every element $a \in \Xi_D$ can be written as a (possibly empty) sum of irreducible elements. Indeed, if this was not the case, we repeatedly expand $a$ as a sum
   \[ a = a_1 + b_1 = a_2 + b_2 + b_1 = a_3 + b_3 + b_2 + b_1 = \ldots \]
   with $a_i, b_i \neq 0$ and it would follow from property \eqref{eq:compatible} that we would have an infinite strictly descending chain
   \[ \ldots \prec a_2 \prec a_1 \prec a \]
   contradicting the fact that $\preceq$ is well-founded.
   
   Hence, every element $a \in \Xi_D$ can be written as a sum of irreducible elements. Because of property \eqref{eq:compatible}, every element of $\Xi_D - \{0\}$ minimal with respect to $\preceq$ is irreducible. But the converse also holds. Indeed, by \eqref{eq:characterizing_xi_D} and the fact that $\nu(\rho(X)) \subseteq \Z^{\Phi^+}$ is a subgroup, it follows that for $a,b \in \Xi_D$ we have the implication
   \[ a \preceq b \quad \Rightarrow \quad b-a \in \Xi_D \]
   Namely if $b-a$ would not lie in $\Xi_D$, it would follow from \eqref{eq:characterizing_xi_D} that $\varepsilon_{D,\alpha} (b-a)(\alpha) < 0$ for some $\alpha \in \Phi^+$. If $a(\alpha) = 0$, this would imply that $\varepsilon_{D,\alpha} b(\alpha) < 0$ and hence $b \not\in \Xi_D$ by \eqref{eq:characterizing_xi_D} again. If $a(\alpha) \neq 0$, it would follow from $a \preceq b$ that $\op{sgn}(a(\alpha)) = \op{sgn}((b-a)(\alpha))$ and hence $\varepsilon_{D,\alpha} a(\alpha) < 0$ implying $a \not\in \Xi_D$ by \eqref{eq:characterizing_xi_D}.

   Therefore, the irreducible elements are precisely the minimal elements of $\Xi_D - \{0\}$ with respect to $\preceq$. Since $\preceq$ is a well partial order, this set is finite. Hence, there exist finitely many elements $x_1,\ldots{},x_r \in X_D$ such that every element $x \in X$ can be written as
   \[ x = \sum_{i = 1}^r n_i x_i + y \]
   with $n_i \in \Z_{\geq 0}$ and
   \[ y \in \op{ker}(\nu\circ{}\rho) \]
   By \axiom{ACX} $X$, is a finitely generated abelian group. Hence, the subgroup $\op{ker}(\nu\circ{}\rho) \leq X$ is also finitely generated as a group, say by $y_1,\ldots{},y_s$. Hence
   \[ \{ x_1, \ldots{}, x_r, y_1, -y_1, \ldots{}, y_s, -y_s \} \]
   forms a set of generators of $X_D$ as a monoid. 
\end{proof}

For later reference we need to record another property of $X_D$.

\begin{lemma}\label{lem:X_D_generates}
   The submonoid $X_D \leq X$ generates $X$ as a group, i.e. every element $x \in X$ can be written as
   \[ x = y-z \]
   with $y,z \in X_D$.
\end{lemma}
\begin{proof}
   It suffices to show that the subset (using the notation of the proof of the previous lemma)
   \[ \{ x \in X\ :\ \rho(x) \in D \} = \{ x \in X\ :\ \forall \alpha \in \Phi^+\ \ \varepsilon_{D,\alpha} \alpha(\rho(x)) > 0 \} \subseteq X_D \]
   is non-empty, since if $y$ denotes an element of this set, then for $n \in N$ sufficiently large we have
   \[ x + ny \in X_D \]
   and hence
   \[ x = (x+ny) - ny \]
   with $x+ny, ny \in X_D$. Let $p: V \twoheadrightarrow V/L$ denote the projection, where
   \[ L = \bigcap_{\alpha \in \Phi^+} H_\alpha \]
   is the common kernel of the $\alpha \in \Phi^+$. Denoting $\overline{\alpha}: V/L \rightarrow \R$ the functional induced by $\alpha \in \Phi^+$, we have
   \[ \{ x \in X\ :\ \forall \alpha \in \Phi^+\ \varepsilon_{\alpha,D} \alpha(\rho(x)) \} = \{ x\in X\ :\ \forall \alpha \in \Phi^+\ \varepsilon_{D,\alpha} \overline{\alpha}(p(\rho(x))) > 0 \} \]
   Thus it suffices to show that the subgroup $p(\rho(X)) \leq V/L$ has non-empty intersection. But since by \axiom{ACV} this subgroup generates $V/L$ as an $\R$-vector space, it contains a basis and hence a full sublattice of $V/L$. Since
   \[ \{ v \in V/L\ :\ \forall \alpha \in \Phi^+\ \ \varepsilon_{D,\alpha} \overline{\alpha}(v) > 0 \} \subseteq V/L \]
   is a non-empty open cone, it has non-empty intersection with every full sublattice of $V/L$.
\end{proof}

\subsection{Spherical orientations} 
\label{sub:Spherical orientations}
In this section we fix an affine extended Coxeter group $W$.

Our goal (in view of \cref{rmk:comm subalgebras}) is to construct for every Weyl chamber $D$ an orientation $\mathfrak{o}_D$ satisfying
\[ \mathfrak{o}_D \bullet{} x = \mathfrak{o}_D \quad \forall x \in X \]
and
\[ \mathfrak{o}_D \bullet{} w = \mathfrak{o}_{w^{-1}(D)} \quad \forall w \in W_0 \]
The construction of $\mathfrak{o}_D$ can be seen as a variant of the orientations $\mathfrak{o}_{w_0}$ defined in \cref{def:chamber_orientations}. Instead of `orienting towards' a chamber $w \in W_{\op{aff}}$ of the affine chamber complex corresponding to $W_{\op{aff}}$, we orient towards the chamber induced by $D$ in the `spherical chamber complex at infinity'. In fact, we will show that $\mathfrak{o}_D$ is the limit
\[ \mathfrak{o}_D = \op{lim}_{w_0} \mathfrak{o}_{w_0} \]
in the sense of \textit{nets}, where the limit is taken over the directed set of chambers endowed with the \textit{dominance order} induced by $D$ (defined below).

Let us now define these orientations.

\begin{definition}\label{def:spherical_orientation}
   Given a Weyl chamber $D \in \pi_0\left(V-\bigcup_{\alpha \in \Phi}H_\alpha\right)$, the associated \textbf{spherical orientation} $\mathfrak{o}_D$ of $W$ is the map
   \[ \mathfrak{o}_D: W\times S \longrightarrow \{\pm\} \]
   defined as follows. Given $w \in W$ and $s \in S$, let $(\alpha,k) \in \Phi\times\Z$ be the unique pair such that $\alpha$ is \textit{$D$-positive}, i.e.
   \[ D \subseteq \{ v \in V\ :\ \alpha(v) > 0 \} \]
   and such that $H_{\alpha,k}$ is the hyperplane separating $w(C_0)$ and $ws(C_0)$. Then let
   \[ \mathfrak{o}_D(w,s) := \op{sgn} \left(\pi_0(\alpha)(ws(C_0)) - \pi_0(\alpha)(w(C_0))\right) \]
   where
   \[ \pi_0(\alpha): \pi_0\left(V-\bigcup_{H \in \mathfrak{H}}H\right) \longrightarrow \pi_0(\R-\Z) \]
   is the map induced on connected components by the restriction of $\alpha$, and the difference is to be understood with respect to the structure on $\pi_0(\R-\Z)$ of an affine space over $\Z$.
\end{definition}

\begin{rmk}\label{rmk:X stabilizes sph or}
   From the defining formula of $\mathfrak{o}_D$ it follows immediately that
   \[ \mathfrak{o}_D \bullet{}w = \mathfrak{o}_{\rho_0(w)^{-1}(D)} \quad \forall w \in X \]
   In particular
   \[ \mathfrak{o}_D \bullet{}x = \mathfrak{o}_D \quad \forall x \in X \]
   and
   \[ \mathfrak{o}_D \bullet{} w = \mathfrak{o}_{w^{-1}(D)} \quad \forall w \in W_0 \]
   However, we still have to show that $\mathfrak{o}_D$ actually is an orientation.
\end{rmk}

\begin{rmk}\label{rmk:intuition_of_spherical_orientations}
   In \cref{lem:or_depend_on_half-space} we have seen that orientations are given by singling out for every hyperplane $H \in \mathfrak{H}$ one of the two half-spaces bounded by $H$ as positive, such that $\mathfrak{o}(w,s) = +$ iff $ws$ lies in the positive half-space bounded by $H = wsw^{-1}$, where the notions of hyperplane and half-space are to be understood in the sense of abstract Coxeter geometry. Unwinding the above definition, one sees that under the dictionary between the abstract geometry of $W$ and the concrete geometry of the hyperplane arrangement $(V,\mathfrak{H})$, the orientation $\mathfrak{o}_D$ is given by letting
   \[ U^+_H = \{ v \in V\ :\ \alpha(v)+k > 0 \} \]
   be the positive half-space bounded by $H = H_{\alpha,k}$ if $\alpha$ is $D$-positive.
\end{rmk}

\begin{definition}\label{def:dominance_order}
   Given a Weyl chamber $D \in \pi_0\left(V-\bigcup_{\alpha \in \Phi} H_\alpha\right)$ the \textbf{dominance order} $\preccurlyeq_D$ \textbf{associated to $D$} is the partial order on the set of chambers given by
   \[ C \preccurlyeq_D C' \quad :\Leftrightarrow\quad \pi_0(\alpha)(C) \leq \pi_0(\alpha)(C') \quad \forall \alpha \text{ $D$-positive} \]
   where $\pi_0(\R-\Z)$ is endowed with the total order $\leq$ induced from $\R$.
\end{definition}

\begin{rmk}
   Obviously $\preccurlyeq$ is a partial order. Moreover, any two chambers $C,C'$ are dominated $C,C' \preccurlyeq_D C''$ by a third, thus making the set of chambers endowed with $\preccurlyeq_D$ into a directed set.
   
   Indeed, for a $D$-positive root $\alpha$ let
   \[ r_\alpha := \op{max}(\op{sup} \pi_0(\alpha)(C),\op{sup} \pi_0(\alpha)(C')) \in \Z \]
   Then any chamber $C''$ contained in
   \[ U := \{ v \in V : \alpha(v) > r_\alpha \quad \forall \alpha \text{ $D$-positive}\} \]
   satisfies $C,C' \preccurlyeq_D C''$. It's easy to see that such a chamber always exists. Since
   \[ D = \{ v \in V : \alpha(v) > 0 \quad \forall \alpha \text{ $D$-positive} \} \neq \emptyset \]
   it follows that $U$ must also be non-empty, hence it (as an open non-empty subset in $V$) must meet some chamber $C''$, which then must already be contained in $U$.
\end{rmk}

Let us now show that `spherical orientations' are indeed orientations.

\begin{prop}\label{prop:spherical_orientations_as_limits}
   The map $\mathfrak{o}_D$, considered as an element of the mapping space $\{\pm\}^{W\times S}$ with its compact-open topology (cf. \cref{rmk:set_of_orientations_is_closed,rmk:homeomorphism_of_orientations}), is the limit
   \[ \mathfrak{o}_D = \op{lim}_C \mathfrak{o}_C \]
   in the sense of nets, where the limit is taken over the directed set of chambers endowed with the dominance order $\preccurlyeq_D$, and where $\mathfrak{o}_C = \mathfrak{o}_w$ denotes the orientation towards the `chamber' $w$ in the sense of \cref{def:chamber_orientations} and $w \in W_{\op{aff}}$ is the unique abstract chamber corresponding to $C$ via $w(C_0) = C$.
   
   In particular, $\mathfrak{o}_D$ lies in the closure of
   \[ \{\mathfrak{o}_w : w \in W_{\op{aff}}\} \subseteq \{\pm\}^{W\times S} \]
   and hence by \cref{rmk:homeomorphism_of_orientations} it also lies in the subset of orientations.
\end{prop}
\begin{proof}
   To show that
   \[ \mathfrak{o}_D = \op{lim}_C \mathfrak{o}_C \]
   means concretely to show that for every $w \in W$ and $s \in S$ we have
   \[ \mathfrak{o}_D(w,s) = \mathfrak{o}_C(w,s) \]
   for $C$ sufficiently large with respect to $\preccurlyeq_D$. Recall that we have $\mathfrak{o}_C(w,s) = +$ iff $ws(C_0)$ is closer to $C$ than $w(C_0)$, i.e. iff the hyperplane $H$ separating $w(C_0)$ and $ws(C_0)$ also separates $w(C_0)$ from $C$, i.e. if $C$ and $ws(C_0)$ lie in the same half-space with respect to $H$. Let $H = H_{\alpha,k}$ with $\alpha$ $D$-positive. Then on the other hand we have $\mathfrak{o}_D(w,s) = +$ iff $\pi_0(\alpha)(w(C_0)) < \pi_0(\alpha)(ws(C_0))$, i.e. if $ws(C_0)$ lies in the positive half-space $U^+_H$ determined by $\mathfrak{o}_D$. Therefore, $\mathfrak{o}_D(w,s) = \mathfrak{o}_C(w,s)$ iff $C$ lies in the positive half-space $U^+_H$. Moreover, if $C,C'$ are chambers with $C \subseteq U^+_H$ and $C \preccurlyeq_D C'$ then $C'$ also lies in $U^+_H$. Letting $C$ denote an arbitrary chamber contained in $U^+_H$, we therefore have
   \[ \mathfrak{o}_D(w,s) = \mathfrak{o}_{C'}(w,s) \]
   for every chamber $C'$ with $C \preccurlyeq_D C'$.
\end{proof}

\subsection{Some (almost) commutative subalgebras} 
\label{sub:Some (almost) commutative subalgebras}
In this section, we let $\mathfrak{o}$ denote an arbitrary spherical orientation (see \cref{def:spherical_orientation}) of $W$. In \cref{rmk:comm subalgebras}, we saw that every submonoid $U \leq \op{Stab}_{W^{(1)}}(\mathfrak{o})$ gives rise to a subalgebra $\mathcal{A}^{(1)}_\mathfrak{o}(U) \subseteq \mathcal{H}^{(1)}$ that has a canonical $R$-basis $\{\widehat{\theta}_\mathfrak{o}(x)\}_{x \in U}$ indexed by the elements of $U$. By \cref{rmk:X stabilizes sph or}, we may take $U = X^{(1)}$.

\begin{definition}
   \[ \mathcal{A}^{(1)}_\mathfrak{o} := \mathcal{A}^{(1)}_\mathfrak{o}(X^{(1)}) = \bigoplus_{x \in X^{(1)}} R\widehat{\theta}_\mathfrak{o}(x) \]
\end{definition}

As a first step towards the computation of the center of $\mathcal{H}^{(1)}$ in \cref{thm:center}, we will now determine the centralizer of the subalgebra $\mathcal{A}^{(1)}_\mathfrak{o}$ of $\mathcal{H}^{(1)}$. Here and in \cref{thm:center}, we will make use of the following auxiliary notion.

\begin{definition}\label{def:support}
   Given an element
   \[ z = \sum_{w \in W^{(1)}} c_w \widehat{\theta}_\mathfrak{o}(w) \in \mathcal{H}^{(1)},\quad c_w \in R \]
   and an orientation $\mathfrak{o}$ of $W^{(1)}$, the set
   \[ \op{supp}_\mathfrak{o}(z) := \{ w \in W^{(1)} : c_w \neq 0 \} \]
   is called the \textbf{support} of $z$ (with respect to $\mathfrak{o}$).
\end{definition}

\begin{prop}\label{prop:centralizer}
   The centralizer $C_{\mathcal{H}^{(1)}}(\mathcal{A}^{(1)}_\mathfrak{o})$ of the $R$-subalgebra
   \[ \mathcal{A}^{(1)}_\mathfrak{o} \subseteq \mathcal{H}^{(1)} \]
   is given by the $X^{(1)}$-invariants
   \[ C_{\mathcal{H}^{(1)}}(\mathcal{A}^{(1)}_\mathfrak{o}) = \left(\mathcal{A}^{(1)}_{\mathfrak{o}}\right)^{X^{(1)}} \]
   with respect to the $R$-linear $X^{(1)}$-action on $\mathcal{A}^{(1)}$ determined by
   \[ x(\widehat{\theta}_\mathfrak{o}(y)) = \widehat{\theta}_\mathfrak{o}(xyx^{-1}) \]
   In particular $\op{Z}(\mathcal{H}^{(1)}) \subseteq \left(\mathcal{A}^{(1)}_\mathfrak{o}\right)^{X^{(1)}}$.
\end{prop}
\begin{proof}
   First, we show that $C_{\mathcal{H}^{(1)}}(\mathcal{A}^{(1)}_\mathfrak{o}) \subseteq \mathcal{A}^{(1)}_\mathfrak{o}$. For this, consider an arbitrary element $z$ of the centralizer of $\mathcal{A}^{(1)}_\mathfrak{o}$ in $\mathcal{H}^{(1)}$. Write
   \[ z = \sum_{w \in W^{(1)}} c_w \widehat{\theta}_\mathfrak{o}(w),\quad c_w \in R \]
   We need to show that $\op{supp}_\mathfrak{o}(z) \subseteq X^{(1)}$. Assume this is not the case and choose $w \in \op{supp}_\mathfrak{o}(z) - X^{(1)}$ with $\ell(w)$ maximal. Fix an element $x \in X^{(1)}$ such that $\pi(x) \in \Xi$, where $\Xi \subseteq X$ is the set associated to $w$ by \cref{lem:centralizer_prop_lem1} below. Consider now the elements $\widehat{\theta}_\mathfrak{o}(x)z$ and $z\widehat{\theta}_\mathfrak{o}(x)$. Using the product formula (\cref{cor:hat mult rule}) and the fact that $\mathfrak{o}$ is invariant under $X$, we see that on the one hand we have
   \begin{align*} \widehat{\theta}_\mathfrak{o}(x) z & = \sum_{w' \in W^{(1)}} c_{w'} \widehat{\theta}_\mathfrak{o}(x) \widehat{\theta}_\mathfrak{o}(w') = \sum_{w' \in W^{(1)}} c_{w'} \overline{\bbf{X}}(x,w') \widehat{\theta}_\mathfrak{o}(\tau^x w') \end{align*}
   On the other hand we have (again using the product formula) 
   \begin{align*}
      z \widehat{\theta}_\mathfrak{o}(x) & = \sum_{w' \in W^{(1)}} c_{w'} \widehat{\theta}_\mathfrak{o}(w') (\widehat{\theta}_{\mathfrak{o}\bullet{}w'}(x) + \widehat{\theta}_\mathfrak{o}(x) - \widehat{\theta}_{\mathfrak{o}\bullet{}w'}(x)) \\
      & = \sum_{w' \in W^{(1)}} c_{w'} \overline{\bbf{X}}(w',x) \widehat{\theta}_\mathfrak{o}(w' \tau^x) + \sum_{w' \in W^{(1)}} c_{w'} \widehat{\theta}_\mathfrak{o}(w') (\widehat{\theta}_\mathfrak{o}(x) - \widehat{\theta}_{\mathfrak{o}\bullet{}w'}(x)) 
   \end{align*}
   By the change of basis formula (\cref{cor:hat change of basis}), the expansions of the two elements $\widehat{\theta}_\mathfrak{o}(x)$ and $\widehat{\theta}_{\mathfrak{o}\bullet{}w'}(x)$ in the Iwahori-Matsumoto basis $\{T_{w''}\}_{w'' \in W^{(1)}}$ have the same leading term $T_x$ with respect to the Bruhat order on $W^{(1)}$. Therefore, $\widehat{\theta}_\mathfrak{o}(x) - \widehat{\theta}_{\mathfrak{o}\bullet{}w'}(x)$ is an $R$-linear combination of terms $T_{w''}$ with $w'' < \tau^x$ and hence $\ell(w'') < \ell(x)$. It follows that in the expansion of $\widehat{\theta}_\mathfrak{o}(w') (\widehat{\theta}_\mathfrak{o}(x) - \widehat{\theta}_{\mathfrak{o}\bullet{}w'}(x))$ in the Iwahori-Matsumoto basis only terms $T_{w''}$ with
   \[ \ell(w'') < \ell(w')+\ell(x) \leq \ell(w)+\ell(x) = \ell(\tau^x w) \]
   appear. Using \cref{cor:hat change of basis} again, it follows that the same is true for the expansion of this expression in the basis $\{\widehat{\theta}_\mathfrak{o}(w'')\}_{w'' \in W^{(1)}}$. In particular, the coefficient of $\widehat{\theta}_\mathfrak{o}(\tau^x w)$ vanishes. Comparing the coefficients of $\widehat{\theta}_\mathfrak{o}(\tau^x w)$ on both sides of the equation $\widehat{\theta}_\mathfrak{o}(x) z = z \widehat{\theta}_\mathfrak{o}(x)$, we see that there exists $w' \in W^{(1)}$ such that $\tau^x w = w' \tau^x$ and
\[ c_{w} \overline{\bbf{X}}(x,w) = c_{w'} \overline{\bbf{X}}(w',x) \]
Since $\pi(x) \in \Xi$ we have $\ell(\tau^x w) = \ell(x) + \ell(w)$ by definition of $\Xi$ and hence $\overline{\bbf{X}}(x,w) = 1$ by \cref{rmk:coboundary}. Since $c_w \neq 0$ by assumption, it follows from the above equation that $c_{w'} \neq 0$ and hence $w' \in \op{supp}_\mathfrak{o}(z)$. Moreover, we have
   \[ w' = \tau^x w \tau^{-x} = \tau^{x - w(x)} w \]
   By \cref{lem:centralizer_prop_lem1} below, we can assume that $x$ has been chosen such that $\ell(w(x) - x) > 2\ell(w)$. But then
   \[ \ell(w') = \ell(\tau^{x - w(x)}w) \geq \ell(\tau^{x - w(x)}) - \ell(w) > \ell(w) \]
   But this is a contradiction to the choice of $w$, and hence we have shown that
   \[ C_{\mathcal{H}^{(1)}}(\mathcal{A}^{(1)}_\mathfrak{o}) \subseteq \mathcal{A}^{(1)}_\mathfrak{o} \]
   Now in order to show that
   \[ C_{\mathcal{H}^{(1)}}(\mathcal{A}^{(1)}_\mathfrak{o}) \subseteq \left(\mathcal{A}^{(1)}_\mathfrak{o}\right)^{X^{(1)}} \]
   we have to show that the coefficients of $z$ satisfy
   \[ c_x = c_{yxy^{-1}} \quad \forall x,y \in X^{(1)} \]
   By \cref{lem:centralizer_prop_lem2} below, it suffices to show this for $y \in X^{(1)}$ satisfying $\ell(xy) = \ell(x)+\ell(y)$. From
   \[ \widehat{\theta}_\mathfrak{o}(y) z = z \widehat{\theta}_\mathfrak{o}(y) \]
   and the product formula it follows immediately that
   \[ \overline{\bbf{X}}(y,x) c_x = \overline{\bbf{X}}(yxy^{-1},y) c_{yxy^{-1}} \]
   Since the image of $X^{(1)}$ under $\pi: W^{(1)} \rightarrow W$ is commutative, we have
   \[ \overline{\bbf{X}}(yxy^{-1},y) = \overline{\bbf{X}}(x,y) \]
   by definition of $\overline{\bbf{X}}$. Moreover, from $\ell(xy) = \ell(x)+\ell(y)$ and \cref{rmk:coboundary} it follows that 
   \[ \overline{\bbf{X}}(x,y) = \overline{\bbf{X}}(y,x) = 1 \]
   Therefore
   \[ c_x = c_{yxy^{-1}} \]
   Thus it only remains to show the reverse inclusion
   \[ \left(\mathcal{A}^{(1)}_\mathfrak{o}\right)^{X^{(1)}} \subseteq C_{\mathcal{H}^{(1)}}(\mathcal{A}^{(1)}_\mathfrak{o}) \]
   So let
   \[ z = \sum_{x \in X^{(1)}} c_x \widehat{\theta}_\mathfrak{o}(x) \]
   be an element of the invariants, i.e.
   \[ c_x = c_{yxy^{-1}},\quad \forall x,y \in X^{(1)} \]
   We need to show that
   \[ z \widehat{\theta}_\mathfrak{o}(y) = \widehat{\theta}_\mathfrak{o}(y) z \quad \forall y \in X^{(1)} \]
   This amounts to showing that
   \[ \overline{\bbf{X}}(y,x) c_x = \overline{\bbf{X}}(yxy^{-1},y) c_{yxy^{-1}} \]
   for \textit{all} $x,y \in X^{(1)}$. But since
   \[ \overline{\bbf{X}}(yxy^{-1},y) = \overline{\bbf{X}}(x,y) \]
   this follows from
   \[ \overline{\bbf{X}}(x,y) = \overline{\bbf{X}}(y,x) \]
\end{proof}

\begin{lemma}\label{lem:centralizer_prop_lem1}
   Let $W = X\rtimes W_0$ be an affine extended Coxeter group (see \cref{def:affine_coxeter_group} and \cref{lem:unpacking_affine_extended_Coxeter_groups} for notation). Let $w \in W$ with $w \not\in X$. Then the set
   \[ \Xi := \{ x \in X : \ell(\tau^x w) = \ell(x) + \ell(w) \} \]
   satisfies
   \[ \sup \{ \ell(w(x)-x) : x \in \Xi \} = \infty \]
\end{lemma}
\begin{proof}
   By \cref{rmk:weak_Bruhat_order} we know that
   \begin{equation}\label{eq:characterizing_len_using_preceq} \ell(\tau^x w) = \ell(x) + \ell(w) \quad \Leftrightarrow \quad \vec{d}(C_0,\tau^x(C_0)) \preceq \vec{d}(C_0,(\tau^x w)(C_0)) \end{equation}
   where $C_0$ denotes the fundamental chamber and $\vec{d}$ is the `vector-valued distance' with values in $\Z^{\Phi^+}$ and $\preceq$ the partial order on $\Z^{\Phi^+}$ defined in \cref{rmk:vector-valued_distance}. Moreover, from the definition of $\vec{d}$ it follows immediately that
   \[ \vec{d}(C_0,\tau^x(C_0)) = -\nu(\rho(x)) \quad \text{and}\quad \vec{d}(C_0,(\tau^x w)(C_0)) = -\nu(\rho(x)) + \vec{d}(C_0,w(C_0)) \]
   where $\nu$ is the evaluation
   \[ \nu: V \longrightarrow \R^{\Phi^+},\quad v \longmapsto (\alpha \mapsto \alpha(v)) \]
   map. Note that $\nu(\rho(x)) \in \Z^{\Phi^+}$, as we verified in the proof of \cref{lem:finiteness_of_X_D}. Let $D \in \pi_0\left(V-\bigcup_{\alpha \in \Phi} H_\alpha\right)$ be the Weyl chamber containing $w(C_0)$. For $\alpha \in \Phi^+$ let
   \[ \varepsilon_{D,\alpha} := \op{sgn}(\alpha(p)) \]
   where $p \in D$ is any point. Then
   \[ D = \{ x \in V\ :\ \forall \alpha \in \Phi^+\ \ \varepsilon_{D,\alpha} \alpha(x) > 0 \} \]
   and hence the closure of $D$ is given by (cf. proof of \cref{lem:finiteness_of_X_D})
   \[ \overline{D} = \{ x \in V\ :\ \forall \alpha \in \Phi^+\ \ \varepsilon_{D,\alpha} \alpha(x) \geq 0 \} \] 
   Moreover, by choosing $p$ to lie in $w(C_0)$ it follows easily from the definition of $\vec{d}$ (remembering that $0 \in \overline{C_0}$) that
   \[ - \varepsilon_{D,\alpha} \vec{d}(C_0,w(C_0))_\alpha \geq 0 \quad \forall \alpha \in \Phi^+ \]
   From the above and the definition of $\preceq$ it follows that
   \[ -\nu(\rho(x)) \preceq -\nu(\rho(x)) + \vec{d}(C_0,w(C_0)) \]
   for all $x \in X_D$, where
   \[ X_D = \{ x \in X\ :\ \rho(x) \in \overline{D} \} \]
   From \eqref{eq:characterizing_len_using_preceq} it therefore follows that
   \[ X_D \subseteq \Xi \]
   Since
   \[ \ell(x) = |\vec{d}(C_0,\tau^x (C_0))| = |-\nu(\rho(x))| = \sum_{\alpha \in \Phi^+} |\alpha(\rho(x))| \]
   it follows from the definition of $X_D$ that
   \[ \ell(x+y) = \ell(x)+\ell(y) \quad \forall x,y \in X_D \]
   In particular we have $\ell(nx) = n\ell(x)$ for $n \in \N$, so in order to prove the claim it suffices to show that
   \[ \{ \ell(w(x)-x)\ :\ x \in X_D \} \]
   contains a nonzero element. If this was not the case, we would have
   \[ \rho(w(x)-x) = \rho_0(w)(\rho(x))-\rho(x) \in L = \bigcap_{\alpha \in \Phi^+} H_\alpha \]
   for all $x \in X_D$, where we recall that $\rho_0: W \rightarrow \op{GL}(V)$ denotes the composition of $\rho: W \rightarrow \op{Aut}_{\op{aff}} = V \rtimes \op{GL}(V)$ with the projection onto the linear part. But every $x \in X$ can be written as a difference $x = y-z$ with $y,z \in X_D$ by \cref{lem:X_D_generates}, hence we would have
   \[ \rho_0(w)(v) - v \in L \]
   for all $v \in \rho(X)$. Since the image of $\rho(X) \subseteq V$ under $V \twoheadrightarrow V/L$ generates the vector space $V/L$ by \axiom{ACV}, it would follow that $\rho_0(w)$ acts trivially on the quotient $V/L$. But by \cref{lem:unpacking_affine_extended_Coxeter_groups} the group $W_0 = \rho_0(W)$ acts faithfully on $V/L$, hence
   \[ w \in \op{ker}(\rho_0) = X \]
   contradicting the assumption.
\end{proof}

\begin{lemma}\label{lem:centralizer_prop_lem2}
   For all $x \in X^{(1)}$
   \[ \{ yxy^{-1} : y \in X^{(1)} \} = \{ yxy^{-1} : y \in X^{(1)},\ \ell(x y) = \ell(x)+\ell(y) \} \]
\end{lemma}
\begin{proof}
   From \cref{rmk:vector-valued_vs_plain_old_distance} recall equation \eqref{eq:length_of_translation}
   \[ \ell(x) = \ell(\pi(x)) = \sum_{\alpha \in \Phi^+} |\alpha(\rho(\pi(x)))| \]
   Now given any $x,y \in X^{(1)}$ we have
   \[ \widetilde{x} := yxy^{-1} = \widetilde{x}^k \widetilde{x} \widetilde{x}^{-k} = (\widetilde{x}^k y) x (\widetilde{x}^k y)^{-1} \]
   for all $k \in \Z$. It therefore suffices to show that
   \[ \ell(x\widetilde{x}^ky) = \ell(x)+\ell(\widetilde{x}^ky) \]
   for $k > 0$ sufficiently large. Since $X$ is commutative we have $\pi(\widetilde{x}) = \pi(x)$ and hence
   \[ \alpha(\rho(\pi(\widetilde{x}^k y))) = \alpha(\rho(\pi(x)^k \pi(y))) = k \alpha(\rho(\pi(x))) + \alpha(\rho(\pi(y))) \]
   If $\alpha(\rho(\pi(x))) \neq 0$, we can therefore always choose $k$ big enough such that $\alpha(\rho(\pi(\widetilde{x}^k y)))$ and $\alpha(\rho(\pi(x)))$ have the same sign and hence
   \[ |\alpha(\rho(\pi(x \widetilde{x}^k y)))| = |\alpha(\rho(\pi(x))) + \alpha(\rho(\pi(\widetilde{x}^k y)))| = |\alpha(\rho(\pi(x)))| + |\alpha(\rho(\pi(\widetilde{x}^k y)))| \]
   For those $\alpha$ for which $\alpha(\rho(\pi(x))) = 0$ the equation
   \[ |\alpha(\rho(\pi(x))) + \alpha(\rho(\pi(\widetilde{x}^k y)))| = |\alpha(\rho(\pi(x)))| + |\alpha(\rho(\pi(\widetilde{x}^k y)))| \]
   holds true for trivial reasons. Hence, for $k$ sufficiently large we have
   \[ |\alpha(\rho(\pi(x \widetilde{x}^k y)))| = |\alpha(\rho(\pi(x)))| + |\alpha(\rho(\pi(\widetilde{x}^k y)))| \]
   for every $\alpha \in \Phi^+$, and hence
   \[ \ell(x \widetilde{x}^k y) = \ell(x) + \ell(\widetilde{x}^k y) \]
\end{proof}

\subsection{The center of affine pro-\texorpdfstring{$p$}{p} Hecke algebras} 
\label{sub:The center of affine pro-$p$ Hecke algebras}
In this section, let $\mathcal{H}^{(1)}$ be an arbitrary affine pro-$p$ Hecke algebra. Our goal is to show that, for \textit{any} orientation $\mathfrak{o}$, the center of $\mathcal{H}^{(1)}$ is given by the invariants
\[ Z(\mathcal{H}^{(1)}) = \left(\mathcal{A}^{(1)}_\mathfrak{o}\right)^{W^{(1)}} \]
of the $R$-linear action of $W^{(1)}$ on $\mathcal{A}^{(1)}_\mathfrak{o}$ by permutation of the basis elements $\widehat{\theta}_\mathfrak{o}(x), x \in X^{(1)}$. Note that the action of $W^{(1)}$ is by \textit{algebra automorphisms}, since we have
\[ \overline{\bbf{X}}(w(x),w(y)) = \overline{\bbf{X}}(x,y) \quad \forall w \in W^{(1)},\ x,y \in X^{(1)} \]
which follows immediately from formula \eqref{eq:simple_formula_for_coboundary} and the $W_0$-invariance (see \cref{lem:invariance_of_L}) of $\overline{\bbf{L}}$ on elements of $X \subseteq W$. In particular, the invariants form a subalgebra.

Let us now show one inclusion.

\begin{prop}\label{prop:invariants_lie_in_center}
   Let $W^{(1)}\backslash X^{(1)}$ denote the set of orbits with respect to the natural conjugation action of $W^{(1)}$ on $X^{(1)}$ and let $(W^{(1)}\backslash X^{(1)})_{\op{fin}}$ denote the subset of finite orbits. For every $\gamma \in (W^{(1)}\backslash X^{(1)})_{\op{fin}}$, the element
   \[ z_\gamma := \sum_{x \in \gamma} \widehat{\theta}_\mathfrak{o}(x),\quad \text{$\mathfrak{o}$ spherical orientation} \]
   is well defined independent of the choice of a spherical orientation $\mathfrak{o}$ of $W^{(1)}$.
   Moreover, the element $z_\gamma$ lies in the center of $Z(\mathcal{H}^{(1)})$, and hence the subalgebra of $W^{(1)}$-invariants
   \[ \left(\mathcal{A}^{(1)}_\mathfrak{o}\right)^{W^{(1)}} \subseteq Z(\mathcal{H}^{(1)}) \]
   is contained in the center and independent of $\mathfrak{o}$, with distinguished $R$-basis $\{z_\gamma\}, \gamma \in (W^{(1)}\backslash X^{(1)})_{\op{fin}}$.
\end{prop}
\begin{proof}
   Using the specialization argument (see \cref{rmk:specialization_argument}), it suffices to prove the statement in the case when the $a_s \in R$ are invertible and admit square roots. In this case we have by \cref{def:normalized_bernstein_map} (for some fixed choice of square roots $\sqrt{a_s}$)
   \[ \widehat{\theta}_\mathfrak{o}(w) = \overline{\sqrt{\bbf{L}}}(w) \widetilde{\theta}_\mathfrak{o}(w) \]
   From the definition of $\overline{\sqrt{\bbf{L}}}: W \rightarrow R$ and \cref{lem:invariance_of_L} below, it follows that we have
   \[ \overline{\sqrt{\bbf{L}}}(w(x)) = \overline{\sqrt{\bbf{L}}}(x) \quad \forall w \in W^{(1)},\ x \in X^{(1)} \]
   Therefore, it follows that
   \[ \sum_{x \in \gamma} \widehat{\theta}_\mathfrak{o}(x) = \overline{\sqrt{\bbf{L}}}(x_0) \sum_{x \in \gamma} \widetilde{\theta}_\mathfrak{o}(x) \]
   for any $x_0 \in \gamma$. We may therefore prove the claim with $\widehat{\theta}$ replaced by $\widetilde{\theta}$, i.e. using the isomorphism of \cref{rmk:usefulness_of_normalized_bernstein} we may assume that $a_s = 1$. In this case the independence of the element
   \[ \sum_{x \in \gamma} \widehat{\theta}_\mathfrak{o}(x) \in \mathcal{H}^{(1)} \]
   from the choice of $\mathfrak{o}$ is equivalent to this element lying in the center since $W_0$ acts transitively on spherical orientations and because of the formula
   \[ \widehat{\theta}_\mathfrak{o}(w)\widehat{\theta}_{\mathfrak{o}\bullet{}w}(x)\widehat{\theta}_\mathfrak{o}(w)^{-1} = \widehat{\theta}_\mathfrak{o}(w(x)) \quad \forall w \in W^{(1)},\ x \in X^{(1)} \]
   So it suffices to show the well-definedness of $z_\gamma$. Since spherical orientations are in bijection with Weyl chambers and any two Weyl chambers are connected by a gallery, it suffices to show that
   \[ \sum_{x \in \gamma} \widehat{\theta}_\mathfrak{o}(x) = \sum_{x \in \gamma} \widehat{\theta}_{\mathfrak{o}\bullet{}s_\alpha}(x) \]
   where $\mathfrak{o}$ is any spherical orientation and $s_\alpha \in W_0$ is associated to a root $\alpha \in \Phi$ that is simple with respect to the Weyl chamber $D_\mathfrak{o}$ to which the orientation $\mathfrak{o}$ corresponds. In this situation, $\mathfrak{o}$ and $\mathfrak{o}\bullet{}s_\alpha$ are \textit{adjacent} in the sense of \cref{def:adjacency_of_or} since $s_\alpha$ permutes the positive roots with respect to $D_\mathfrak{o}$ that are not parallel to $\alpha$ among themselves.
   
   The decomposition $W = W_0 \ltimes X$ induces an identification $W^{(1)}/X^{(1)} \simeq W_0$, and therefore the $W^{(1)}$-orbit $\gamma$ decomposes into a disjoint union of $X^{(1)}$-orbits that are permuted amongst themselves by $W_0$. Considering the action of the subgroup $\{1,s_\alpha\} \leq W_0$, we can therefore write
   \[ \gamma = \coprod_{i \in I} \xi_i \cup s_\alpha(\xi_i) \]
   where $\xi_i \in X^{(1)}\backslash X^{(1)}$ and either $s_\alpha(\xi_i) = \xi_i$ or $s_\alpha(\xi_i) \cap \xi_i = \emptyset$. Accordingly, if $J \subseteq I$ denotes the indices $i$ where $s_\alpha(\xi_i) = i$ and $\sigma \in W^{(1)}$ denotes any lift of $s_\alpha$, we have that
   \[ \sum_{x \in \gamma} \widehat{\theta}_{\mathfrak{o}}(x) = \sum_{i \in J} \sum_{x \in \xi_i} \widehat{\theta}_{\mathfrak{o}}(x) + \sum_{i \in I-J} \sum_{x \in \xi_i} \widehat{\theta}_{\mathfrak{o}}(x)+\widehat{\theta}_{\mathfrak{o}}(\sigma(x)) \]
   Whence it suffices to show that for all $x \in \xi_i$ with $s_\alpha(\xi_i) = \xi_i$ we have that
   \[ \widehat{\theta}_{\mathfrak{o}}(x) = \widehat{\theta}_{\mathfrak{o}\bullet{}s_\alpha}(x) \]
   and that for \textit{any} $x \in X^{(1)}$ we have that
   \[ \widehat{\theta}_{\mathfrak{o}}(x) + \widehat{\theta}_{\mathfrak{o}}(\sigma(x)) = \widehat{\theta}_{\mathfrak{o}\bullet{}s_\alpha}(x) + \widehat{\theta}_{\mathfrak{o}\bullet{}s_\alpha}(\sigma(x)) \]
   Let us begin by proving the first statement, and assume that $s_\alpha(\xi_i) = \xi_i$. Note that because $X$ is commutative, $\pi: W^{(1)} \longrightarrow W$ maps $X^{(1)}$-orbits to singletons; in particular, $\pi(\sigma(x)) = \pi(x)$. Therefore
   \[ s_\alpha(\pi(x)) = \pi(\sigma(x)) = \pi(x) \]
   Further, recall that the abstract geometry of the extended Coxeter group $W$ and the concrete geometry of the affine hyperplane arrangement $(V,\mathfrak{H})$ are compatible via $\rho$ (see \cref{rmk:abstract_and_concrete_coxeter_geometry}). By definition of the $s_\alpha \in W_0$ (cf. \cref{lem:unpacking_affine_extended_Coxeter_groups}), we have $\rho(s_\alpha) = s_\alpha$, where $s_\alpha \in \op{GL}(V)$ is given by the formula
   \[ s_\alpha(\rho(\pi(x))) = \rho(\pi(x)) - \alpha(\rho(\pi(x))) \alpha^\vee \]
   Therefore, it follows from applying $\rho$ to the equality $s_\alpha(\pi(x)) = \pi(x)$ that $\alpha(\rho(\pi(x))) = 0$. This means that $1, x$ are not separated by any hyperplane of type $\alpha$, where we agree to call $H$ a hyperplane of type $\alpha$ if $H = H_{\alpha,k}$ for some $k \in \Z$. Since $\mathfrak{o}$ and $\mathfrak{o}\bullet{}s_\alpha$ agree except at the hyperplanes of type $\alpha$, it follows that
   \[ \widehat{\theta}_\mathfrak{o}(x) = \widehat{\theta}_{\mathfrak{o}\bullet{}s_\alpha}(x) \]
   Let us now prove the second statement and let $x \in X^{(1)}$ be arbitrary. Since $\mathfrak{o}$ and $\mathfrak{o}\bullet{}s_\alpha$ are adjacent, we may apply the Bernstein relation (\cref{thm:bernstein_relation}) to conclude that (remembering that $\widehat{\theta} = \widetilde{\theta}$ in our case)
   \[ \widehat{\theta}_\mathfrak{o}(x) - \widehat{\theta}_{\mathfrak{o}\bullet{}s_\alpha}(x) = \left(\sum_{\widetilde{H}} \mathfrak{o}(1,\widetilde{H})\Xi_{\mathfrak{o}\bullet{}s_\alpha}(\widetilde{H})\right)\widehat{\theta}_\mathfrak{o}(x) \]
   where the sum runs over all hyperplanes $\widetilde{H}$ of type $\alpha$ which separate $1$ and $x$. On the other hand applying \cref{thm:bernstein_relation} to $\sigma(x)$ instead of $x$ gives
   \begin{align*} \widehat{\theta}_\mathfrak{o}(\sigma(x)) - \widehat{\theta}_{\mathfrak{o}\bullet{}s_\alpha}(\sigma(x)) & = \left(\sum_H \mathfrak{o}(1,H)\Xi_{\mathfrak{o}\bullet{}s_\alpha}(H)\right) \widehat{\theta}_\mathfrak{o}(\sigma(x)) \\
      & = \left(\sum_H \mathfrak{o}(1,H)\Xi_{\mathfrak{o}\bullet{}s_\alpha}(H) \widehat{\theta}_\mathfrak{o}(\sigma(x)x^{-1})\right) \widehat{\theta}_\mathfrak{o}(x) \end{align*}
      where the sum runs over all hyperplanes $H$ of type $\alpha$ separating $1$ and $\sigma(x)$. By \cref{lem:property_of_bernstein} we have
      \[ \Xi_{\mathfrak{o}\bullet{}s_\alpha}(H) \widehat{\theta}_\mathfrak{o}(\sigma(x)x^{-1}) = \Xi_{\mathfrak{o}\bullet{}s_\alpha}(\pi(x)H\pi(x)^{-1}) \]
   The result follows if we can show that
   \[ H \longmapsto \widetilde{H} := \pi(x)H\pi(x)^{-1} \]
   gives a bijection between the hyperplanes $H$ of type $\alpha$ separating $1$ and $\sigma(x)$ and the hyperplanes $\widetilde{H}$ of type $\alpha$ that separate $1$ and $x$, and that
   \[ \mathfrak{o}(1,\pi(x)H\pi(x)^{-1}) = - \mathfrak{o}(1,H) \]
   since then
   \begin{align*} \widehat{\theta}_\mathfrak{o}(\sigma(x)) - \widehat{\theta}_{\mathfrak{o}\bullet{}s_\alpha}(\sigma(x)) & = \left(\sum_H \mathfrak{o}(1,H)\Xi_{\mathfrak{o}\bullet{}s_\alpha}(\pi(x)H\pi(x)^{-1}) \right) \widehat{\theta}_\mathfrak{o}(x) \\
      & = -\left(\sum_{\widetilde{H}} \mathfrak{o}(1,\widetilde{H})\Xi_{\mathfrak{o}\bullet{}s_\alpha}(\widetilde{H})\right) \widehat{\theta}_\mathfrak{o}(x) \\ 
      & = - \left(\widehat{\theta}_\mathfrak{o}(x) - \widehat{\theta}_{\mathfrak{o}\bullet{}s_\alpha}(x)\right)\end{align*}
      Let $H = H_{\alpha,k}$ be a hyperplane of type $\alpha$ and $y \in C_0$ be an arbitrary. Then $H$ separates two elements $w,w' \in W$ if and only if $w(y)$ and $w'(y)$ lie in different connected components of $V-H_{\alpha,k}$, i.e. if and only if $\alpha(w(y))+k$ and $\alpha(w'(y))+k$ have different signs. Moreover, for a hyperplane $H$ we have $H = H_{\alpha,k}$ if and only if $\rho(s_H) = s_{\alpha,k}$ where
   \[ s_{\alpha,k}(y) = y - (\alpha(y)+k)\alpha^\vee \]
   Denoting by $\tau^v \in \op{Aut}_{\op{aff}}(V)$ the translation by a vector $v \in V$, we have the formula
   \[ \tau^y s_{\alpha,k} \tau^{-y} = s_{\alpha,k-\alpha(y)} \]
   Let now $H = H_{\alpha,k}$ be a hyperplane of type $\alpha$. Since $X \subseteq W$ gets mapped into the subgroup $V \leq \op{Aut}_{\op{aff}}(V)$ of translations under $\rho$, it hence follows that
   \[ \rho(\pi(x)H\pi(x)^{-1}) = \tau^{\rho(\pi(x))} s_{\alpha,k} \tau^{-\rho(\pi(x))} = s_{\alpha,k-\alpha(\rho(\pi(x)))} \]
   Hence, $\pi(x)H\pi(x)^{-1}$ separates $1, x$ if and only if $\alpha(y)+k-\alpha(\rho(\pi(x)))$ and
   \[ \alpha(y+\rho(\pi(x))) + k - \alpha(\rho(\pi(x))) = \alpha(y) + k \]
   have different signs. On the other hand $H$ separates $1,\sigma(x)$ if and only if $\alpha(y)+k$ and
   \[ \alpha(y+\rho(s_\alpha(\pi(x)))) + k = \alpha(y) + k - \alpha(\rho(\pi(x))) \]
   have different signs. Hence, $H \mapsto \widetilde{H}$ gives a bijection as desired. Moreover
   \[ \mathfrak{o}(1,\pi(x)H\pi(x)^{-1}) = -\mathfrak{o}(1,H) \]
   By \cref{not:orientations}, $\mathfrak{o}(1,H)$ is the sign attached by $\mathfrak{o}$ to crossing $H$ at any chamber lying in the same half-space as the fundamental chamber. Letting $\varepsilon \in \{\pm \}$ be such that $\varepsilon \alpha$ is positive with respect to the Weyl chamber $D_\mathfrak{o}$ corresponding to $\mathfrak{o}$, it then follows that
   \[ \mathfrak{o}(1,H) = - \varepsilon \op{sgn}(\alpha(y)+k) \]
   and
   \[ \mathfrak{o}(1,\pi(x)H\pi(x)^{-1}) = -\varepsilon \op{sgn}(\alpha(y)+k-\alpha(\pi(x))) \]
   As we saw above, $H$ separates $1,x$ if and only if $\alpha(y)+k$ and $\alpha(y)+k-\alpha(\pi(x))$ have different signs. Hence, the claim follows.
\end{proof}

\begin{lemma}\label{lem:invariance_of_L}
   The length function of \cref{def:generalized_length}
   \[ \bbf{L}: W \longrightarrow \N[\mathfrak{H}] \]
   satisfies
   \[ \bbf{L}(w(x)) = \rho_0(w)(\bbf{L}(x)) \quad \forall w \in W,\ x \in X \]
   where $\rho_0: W \rightarrow W_0$ denotes the projection.
\end{lemma}
\begin{proof}
   Recall from \cref{rmk:vector-valued_vs_plain_old_distance} that the \textit{number} of hyperplanes of the form $H_{\alpha,k}$, $k \in \Z$ separating the fundamental chamber $C_0$ from $\rho(x)(C_0)$ is given by
   \[ |\vec{d}(C_0,\rho(x)(C_0))_\alpha| = |-\alpha(\rho(x))| \]
   With a bit more notation, we can be more precise and specify the \textit{set} of these hyperplanes. For $k \in \Z$ let
   \[ \left[0,k\right[ := \begin{cases}
      \{0,1,\ldots,k-1\} &\text{ : } k > 0 \\
      \emptyset &\text{ : } k = 0 \\
      \{-k+1,\ldots,-1,0\} & \text{ : } k < 0
   \end{cases} \]
   Using that $0 \in \overline{C_0}$ by \axiom{ACIX} and that
   \[ C_0 \subseteq \{ v \in V\ :\ \forall \alpha \in \Phi^+\ \ \alpha(v) > 0 \} \]
   by definition of $\Phi^+$, it is easy to see that the set of hyperplanes of the form $H_{\alpha,k}$ which separate $C_0$ and $\rho(x)(C_0)$ is in fact given by
   \[ \{ H_{\alpha,k} \ :\ k \in \left[0,-\alpha(\rho(x))\right[ \} \]
   Hence
   \[ \bbf{L}(x) = \prod_{\alpha \in \Phi^+} \prod_{k \in \left[0,-\alpha(\rho(x))\right[} H_{\alpha,k} \]
   Moreover
   \[ \rho(w(x)) = \rho_0(w)(\rho(x)) \]
   and hence
   \[ \alpha(\rho(w(x))) = \alpha(\rho_0(w)(\rho(x))) = (\rho_0(w)^{-1} \bullet{} \alpha)(\rho(x)) \]
   Since $\Phi$ is the disjoint union of $\Phi^+$ and $-\Phi^+$, we have
   \[ \rho_0(w)^{-1}\bullet{} \alpha = \varepsilon_\alpha \phi(\alpha) \]
   for some uniquely determined $\varepsilon_\alpha \in \{\pm\}$ and $\phi(\alpha) \in \Phi^+$. Using that
   \[ w(H_{\alpha,k}) = H_{w\bullet{} \alpha, k} \quad \forall w \in W_0,\ \alpha \in \Phi,\ k \in \Z \]
   that
   \[ H_{\alpha,k} = H_{-\alpha,-k} \quad \forall \alpha \in \Phi,\ k \in \Z \]
   and that $\phi: \Phi^+ \rightarrow \Phi^+$ is a bijection, we now simply compute
   \begin{align*} \bbf{L}(w(x)) & = \prod_{\alpha \in \Phi^+} \prod_{k \in \left[0, -\alpha(\rho(w(x)))\right[} H_{\alpha,k} \\
      & = \prod_{\alpha \in \Phi^+} \prod_{k \in \left[0, -\varepsilon_\alpha \phi(\alpha)(\rho(x))\right[} H_{\alpha,k} \\
      & = \prod_{\alpha \in \Phi^+} \prod_{k \in \left[0, -\phi(\alpha)(\rho(x))\right[} H_{\alpha, \varepsilon_\alpha k} \\
      & = \prod_{\alpha \in \Phi^+} \prod_{k \in \left[0, -\phi(\alpha)(\rho(x))\right[} H_{\varepsilon_\alpha \alpha, k} \\
      & = \prod_{\alpha \in \Phi^+} \prod_{k \in \left[0, -\phi(\alpha)(\rho(x))\right[} \rho_0(w) (H_{\phi(\alpha), k}) \\
      & = \rho_0(w) \left(\prod_{\alpha \in \Phi^+} \prod_{k \in \left[0, -\phi(\alpha)(\rho(x))\right[} H_{\phi(\alpha),k}\right) \\
      & = \rho_0(w) (\bbf{L}(x))
   \end{align*}
\end{proof}

We will now show that the center $Z(\mathcal{H}^{(1)})$ is in fact equal to $\left(\mathfrak{A}^{(1)}_\mathfrak{o}\right)^{W^{(1)}}$, via induction on the support (see \cref{def:support}) of an element.

\begin{theorem}\label{thm:center}
   The center $Z(\mathcal{H}^{(1)})$ of the affine pro-$p$ Hecke algebra $\mathcal{H}^{(1)}$ is given by
   \[ Z(\mathcal{H}^{(1)}) = \left(\mathcal{A}^{(1)}_\mathfrak{o}\right)^{W(1)} \]
   for every spherical orientation $\mathfrak{o}$ of $W^{(1)}$. It is a free $R$-module with distinguished basis $\{z_\gamma\}_{\gamma}$ indexed by the finite orbits $\gamma \in (W^{(1)}\backslash X^{(1)})_{\text{fin}}$ of $W^{(1)}$ in $X^{(1)}$, where
   \[ z_\gamma = \sum_{x \in \gamma} \widehat{\theta}_\mathfrak{o}(x) \]
   for every spherical orientation $\mathfrak{o}$.
\end{theorem}
\begin{proof}
   It only remains to prove that
   \[ Z(\mathcal{H}^{(1)}) \subseteq \left(\mathcal{A}^{(1)}_\mathfrak{o}\right)^{W^{(1)}} \]
   In view of the computation of the centralizer of $\mathcal{A}^{(1)}_\mathfrak{o}$ in \cref{prop:centralizer} we already know that
   \[ Z(\mathcal{H}^{(1)}) \subseteq \left(\mathcal{A}^{(1)}_\mathfrak{o}\right)^{X^{(1)}} \]
   Therefore, it only remains to show that for an element
   \[ z = \sum_{x \in X^{(1)}} c_x \widehat{\theta}_{\mathfrak{o}}(x) = \sum_{\xi \in (X^{(1)}\backslash X^{(1)})_{\text{fin}}} c_\xi \sum_{x \in \xi} \widehat{\theta}_\mathfrak{o}(x) \in Z(\mathcal{H}^{(1)}),\quad c_\xi \in R \]
   we have
   \[ c_\xi = c_{w(\xi)} \]
   for every $w \in W_0$. We will prove this by induction on $\op{supp}_\mathfrak{o}(z)$ using \cref{prop:invariants_lie_in_center}. If $\op{supp}_\mathfrak{o}(z) = \emptyset$, then $z = 0$ and the claim is clear. So let's assume that $\op{supp}_\mathfrak{o}(z) \neq \emptyset$. Choose $x \in \op{supp}_\mathfrak{o}(z)$ with $\ell(x)$ maximal and let $\xi = X^{(1)}\bullet{}x$ be the (finite) $X^{(1)}$-orbit associated to it. We now want to show that
   \[ c_\xi = c_{w(\xi)} \quad \forall w \in W_0 \]
   in order to apply induction. For this, recall that $W_0$ is generated by the reflections $s_\alpha$ for roots $\alpha$ that are simple with respect to the Weyl chamber $D_0$ containing $C_0$ (see \cref{lem:unpacking_affine_extended_Coxeter_groups}) and let $s = s_\alpha \in S$ for any such $\alpha$. We have $\ell(n_s x) = \ell(x) \pm 1$. Moreover, we claim that
   \[ \ell(n_s x) = \ell(x) - 1 \quad \Rightarrow\quad \ell(x n_s) = \ell(x)+1 \]
   This is obvious. To see that this is obvious, let $x_0 \in C_0$ be arbitrary. We have $\ell(n_s x) = \ell(x) + 1$ if and only if
   \[ d(C_0,n_s x C_0) = d(C_0, n_s C_0) + d(n_s C_0, n_s x C_0) \]
   that is, if and only if the set of hyperplanes separating $C_0$ and $n_s C_0$ and the set of hyperplanes separating $n_s C_0$ and $n_s x C_0 = n_s(x) n_s C_0 = n_s(\pi(x))+n_s C_0$ are disjoint. Since $C_0$ and $n_s C_0$ are separated only by $H_{\alpha} = \op{ker}(\alpha)$
   and this hyperplanes separates $n_s C_0$ and $n_s(\pi(x))+n_s C_0$ if and only if
   \[ \op{sgn}(\alpha(n_s(\pi(x))+n_s(x_0))) = -\op{sgn}(\alpha(n_s(x_0))) \]
   we see that
   \[ \ell(n_s x) = \ell(x) + 1 \quad \text{or} \quad \ell(n_s x) = \ell(x) - 1 \]
   depending on whether
   \[ \op{sgn}(\alpha(\pi(x)) + \alpha(x_0)) = \op{sgn}(\alpha(x_0)) \quad \text{or}\quad \op{sgn}(\alpha(\pi(x))+\alpha(x_0)) = -\op{sgn}(\alpha(x_0)) \]
   respectively. Using that $x n_s = n_s s(x)$ and $\ell(s(x)) = \ell(x)$ it follows from the above that
   \[ \ell(x n_s) = \ell(x) + 1\quad \text{or}\quad \ell(x n_s) = \ell(x) - 1 \]
   depending on whether
   \[ \op{sgn}(-\alpha(\pi(x))+\alpha(x_0)) = \op{sgn}(\alpha(x_0)) \quad \text{or}\quad \op{sgn}(-\alpha(\pi(x))+\alpha(x_0)) = -\op{sgn}(\alpha(x_0)) \]
   In particular it follows that
   \[ \ell(n_s x) = \ell(x) - 1 \quad \Rightarrow\quad \ell(x n_s) = \ell(x) + 1 \]
   We now distinguish two cases. First, assume that $\ell(n_s x) = \ell(x)+1$. We have
   \[ \widehat{\theta}_\mathfrak{o}(n_s) z = \sum_{y \in X^{(1)}} c_y \widehat{\theta}_\mathfrak{o}(n_s) (\widehat{\theta}_\mathfrak{o}(y)-\widehat{\theta}_{\mathfrak{o}\bullet{}s}(y)) + \sum_{y \in X^{(1)}} \overline{\bbf{X}}(s,\pi(y)) \widehat{\theta}_\mathfrak{o}(n_s y) \]
   We claim that the $n_s x$ does not appear in the support of the first big sum. In fact, by the change of basis formula (\cref{cor:hat change of basis}) we have
   \[ \op{supp}_\mathfrak{o}(\widehat{\theta}_\mathfrak{o}(y) - \widehat{\theta}_{\mathfrak{o}\bullet{}s}(y)) \subseteq \{ w \in W^{(1)} : w < y \} \]
   In particular each $w$ appearing in the support of $\widehat{\theta}_\mathfrak{o}(y)-\widehat{\theta}_{\mathfrak{o}\bullet{}s}(y)$ is of length
   \[ \ell(w) \leq \ell(y)-1 \]
   Hence, for all
   \[ w \in \op{supp}_\mathfrak{o}(\widehat{\theta}_\mathfrak{o}(n_s)(\widehat{\theta}_\mathfrak{o}(y)-\widehat{\theta}_{\mathfrak{o}\bullet{}s}(y))) \]
   we have
   \[ \ell(w) \leq \ell(n_s)+(\ell(y)-1) = \ell(y) \leq \ell(x) < \ell(x)+1 = \ell(n_s x) \]
   The coefficient of $\widehat{\theta}_{\mathfrak{o}}(n_s x)$ in $\widehat{\theta}_\mathfrak{o}(n_s)z$ is therefore given by
   \[ c_x \overline{\bbf{X}}(s,\pi(x)) = c_x = c_{\xi} \]
   Here we have used that $\overline{\bbf{X}}(s,\pi(x)) = 1$, which follows from $\ell(n_s x) = \ell(n_s) + \ell(x)$ and \cref{rmk:coboundary}. On the other hand we have
   \[ z \widehat{\theta}_\mathfrak{o}(n_s) = \sum_{y \in X^{(1)}} c_y \overline{\bbf{X}}(\pi(y),s) \widehat{\theta}_\mathfrak{o}(y n_s) \]
   and hence the coefficient of $n_s x = n_s(x) n_s$ in $\widehat{\theta}_\mathfrak{o}(n_s)z$ is given by
   \[ c_{n_s(x)} \overline{\bbf{X}}(\pi(s(x)),s) = c_{n_s(x)} \]
   Where again we have used \cref{rmk:coboundary} and the fact that
   \[ \ell(s(x) n_s) = \ell(n_s x) = \ell(x)+\ell(n_s) = \ell(n_s(x)) + \ell(n_s) \]
   Since $\widehat{\theta}_\mathfrak{o}(x) z = z \widehat{\theta}_\mathfrak{o}(x)$ it follows that
   \[ c_{x} = c_{n_s(x)} \]
   Consider now the case $\ell(n_s x) = \ell(x)-1$. As already observed, in this situation we must have $\ell(n_s s(x)) = \ell(x n_s) = \ell(x) + 1$. Replacing $x$ by $s(x)$, we are therefore reduced to the first case. Thus we have shown that for any $x \in \op{supp}_\mathfrak{o}(z)$ with $\ell(x)$ maximal we have
   \[ c_\xi = c_x = c_{n_s(x)} = c_{s(\xi)} \]
   for all simple reflections $s = s_\alpha \in W_0$, $\alpha \in \Delta$, where $\xi = X^{(1)}$ denotes the (finite) $X^{(1)}$-orbit of $x$. In particular $n_s(x)$ is again an element of $\op{supp}_\mathfrak{o}(z)$, and since $\ell(x) = \ell(n_s(x))$ it is also of maximal length. Inductively it therefore follows that
   \[ c_\xi = c_{w(\xi)} \]
   for all $w \in W_0$. Letting $\gamma = W^{(1)}\bullet{}x = \bigcup_{w \in W_0} w\bullet{}\xi$, the element
   \[ z - c_\xi z_{\gamma} \in Z(\mathcal{H}^{(1)}) \]
   therefore has support strictly contained in $\op{supp}_\mathfrak{o}(z)$, and by induction we conclude that
   \[ z \in \left(\mathcal{A}^{(1)}_\mathfrak{o}\right)^{W^{(1)}} \]
\end{proof}

\subsection{The structure of affine pro-\texorpdfstring{$p$}{p} Hecke algebras} 
\label{sub:The structure of affine pro-$p$ Hecke algebras}
In this section we will give the main theorem on the structure of affine pro-$p$ Hecke algebras.

\begin{theorem}\label{thm:structure of affine pro-$p$ Hecke algebras}
   Let $\mathcal{H}^{(1)}$ be an affine pro-$p$ Hecke algebra over a ring $R$ in the sense of \cref{def:affine_pro-$p$_Hecke_algebra} and let $\mathfrak{o}$ a spherical orientation of $W^{(1)}$ in the sense of \cref{def:spherical_orientation}. Then the following holds.
   \begin{enumerate}
      \item[(i)] There exists an $R$-subalgebra
         \[ \mathcal{A}^{(1)}_\mathfrak{o} \subseteq \mathcal{H}^{(1)} \]
         with $R$-basis $\{\widehat{\theta}_\mathfrak{o}(x)\}_{x \in X^{(1)}}$ defined in \cref{thm:existence_widehat_theta}. The product of two basis elements is given by
         \[ \widehat{\theta}_\mathfrak{o}(x)\widehat{\theta}_\mathfrak{o}(y) = \overline{\bbf{X}}(\pi(x),\pi(y)) \widehat{\theta}_\mathfrak{o}(xy) \]
         where $\overline{\bbf{X}}: W\times W \rightarrow R$ denotes the `2-coboundary' defined in \cref{not:overline_bbf_X}.
      \item[(ii)] The `conjugation action' of $W^{(1)}$ on $X^{(1)}$ induces an action on $\mathcal{A}^{(1)}_\mathfrak{o}$ by $R$-algebra automorphisms via
         \[ w(\widehat{\theta}_\mathfrak{o}(x)) = \widehat{\theta}_\mathfrak{o}(w(x)) \]
         The centralizer $C_{\mathcal{H}^{(1)}}(\mathcal{A}^{(1)}_\mathfrak{o})$ of $\mathcal{A}^{(1)}_\mathfrak{o}$ in $\mathcal{H}^{(1)}$ is given by the subalgebra of $X^{(1)}$-invariants. In particular, the centralizer is contained in $\mathcal{A}^{(1)}_\mathfrak{o}$ and hence equals the center of $\mathcal{A}^{(1)}_\mathfrak{o}$:
         \[ Z(\mathcal{A}^{(1)}_\mathfrak{o}) = C_{\mathcal{H}^{(1)}}(\mathcal{A}^{(1)}_\mathfrak{o}) = \left(\mathcal{A}^{(1)}_\mathfrak{o}\right)^{X^{(1)}} \]
   \item[(iii)] If the abelian group $T$ is finitely generated, the $R$-algebra $\mathcal{A}^{(1)}_\mathfrak{o}$ is finitely generated. More precisely, $\mathcal{A}^{(1)}_\mathfrak{o}$ is a finite sum
      \[ \mathcal{A}^{(1)}_{\mathfrak{o}} = \sum_D \iota_D(R[X^{(1)}_D]) \]
      of subalgebras, where the sum ranges over the Weyl chambers $D \in \pi_0(V-\bigcup_{\alpha \in \Phi} H_\alpha)$. Here $R[X^{(1)}_D]$ denotes the monoid algebra over the submonoid
      \[ X^{(1)}_D = \{ x \in X^{(1)} : \rho(\pi(x)) \in \overline{D} \} \]
      of $X^{(1)}$ consisting of those elements which act through $\rho: W \rightarrow \op{Aut}_{\op{aff}}(V)$ by translation by an element lying in the closure of $D \subseteq V$, and $\iota_D$ denotes the algebra embedding
      \[ \iota_D: R[X^{(1)}_D] \hookrightarrow \mathcal{A}^{(1)}_\mathfrak{o} \]
      determined by $\iota_D(x) = \widehat{\theta}_\mathfrak{o}(x)$ for all $x \in X^{(1)}_D$. Moreover, if $T$ is finitely generated, then the submonoid $X^{(1)}_D$ and hence the algebra $R[X^{(1)}_D]$ are finitely generated.
   \item[(iv)] If $(W^{(1)}\backslash X^{(1)})_{\op{fin}} = W^{(1)}\backslash X^{(1)}$ and either $T$ is finite or $T$ is finitely generated and $R$ is notherian, then $\mathcal{A}^{(1)}_\mathfrak{o}$ is a finitely generated $\left(\mathcal{A}^{(1)}_\mathfrak{o}\right)^{X^{(1)}}$-module. Here $W^{(1)}\backslash X^{(1)}$ denotes the set of orbits of the conjugation action of $W^{(1)}$ on $X^{(1)}$, and $(W^{(1)}\backslash X^{(1)})_{\op{fin}}$ denotes the subset of finite orbits.
   \item[(v)] If $(W^{(1)}\backslash X^{(1)})_{\op{fin}} = W^{(1)}\backslash X^{(1)}$ and either $T$ is finite or $T$ is finitely generated and $R$ is noetherian, then $\left(\mathcal{A}^{(1)}_\mathfrak{o}\right)^{X^{(1)}}$ is a finitely generated $R$-algebra.
   \item[(vi)] The center $Z(\mathcal{H}^{(1)})$ of $\mathcal{H}^{(1)}$ is given by the subalgebra of $W^{(1)}$-invariants
      \[ Z(\mathcal{H}^{(1)}) = \left(\mathcal{A}^{(1)}_\mathfrak{o}\right)^{W^{(1)}} \]
      It has a distinguished $R$-basis $\{z_\gamma\}, \gamma \in (W^{(1)}\backslash X^{(1)})_{\op{fin}}$ with
      \[ z_\gamma = \sum_{x \in \gamma} \widehat{\theta}_\mathfrak{o}(x) \]
      independent of the choice of the spherical orientation $\mathfrak{o}$.
   \item[(vii)] $\mathcal{H}^{(1)}$ is a finitely generated left $\mathcal{A}^{(1)}_\mathfrak{o}$-module. More precisely, a finite set of generators is given as follows. 
      For $w \in W_0$ let $X\bullet{}w^{-1}(C_0)$ denote the set of $X$-translates of the chamber $w^{-1}(C_0)$. Let $\Lambda_w$ denote the set of minimal elements of $X\bullet{}w^{-1}(C_0)$ with respect to the partial order $\preceq_{C_0}$ defined in \cref{rmk:weak_Bruhat_order}. By \cref{cor:weak Bruhat is wpo}, the set $\Lambda_w$ is finite.
      
      Choose for each $C \in \Lambda_w$ an element $\widetilde{w} \in W^{(1)}$ with
      \[ \pi(\widetilde{w}) \in X w \subseteq W\quad \text{and}\quad \widetilde{w}^{-1}(C_0) = C \]
      and let $\widetilde{\Lambda_w} \subseteq W^{(1)}$ denote the set of these elements. Then
      \[ \{ \widehat{\theta}_\mathfrak{o}(\widetilde{w}) : w \in W_0,\ \widetilde{w} \in \widetilde{\Lambda_w} \} \]
      is a set of generators of $\mathcal{H}^{(1)}$ as a left module over $\mathcal{A}^{(1)}_\mathfrak{o}$.
   \item[(viii)] If $(W^{(1)}\backslash X^{(1)})_{\op{fin}} = W^{(1)}\backslash X^{(1)}$ and either $T$ is finite or $T$ is finitely generated and $R$ is noetherian, then $\mathcal{A}^{(1)}_\mathfrak{o}$ is a finitely generated $Z(\mathcal{H}^{(1)})$-module.
   \item[(ix)] If $(W^{(1)}\backslash X^{(1)})_{\op{fin}} = W^{(1)}\backslash X^{(1)}$ and either $T$ is finite or $T$ is finitely generated and $R$ is noetherian, then $Z(\mathcal{H}^{(1)})$ is a finitely generated $R$-algebra.
   \item[(x)] If $(W^{(1)}\backslash X^{(1)})_{\op{fin}} = W^{(1)}\backslash X^{(1)}$ and either $T$ is finite or $T$ is finitely generated and $R$ is noetherian, then $\mathcal{H}^{(1)}$ is a finitely generated $Z(\mathcal{H}^{(1)})$-module.
   \end{enumerate}
\end{theorem}
\begin{proof} (i): Was proven in \cref{sub:Integral and normalized Bernstein maps}. (ii): Was shown in \cref{prop:centralizer}. (vi): Was shown in \cref{thm:center}. (vii): By \cref{cor:hat mult rule} and \cref{rmk:coboundary} we know that
\[ \ell(ww') = \ell(w)+\ell(w') \quad \Rightarrow \quad \widehat{\theta}_\mathfrak{o}(ww') = \widehat{\theta}_\mathfrak{o}(w)\widehat{\theta}_{\mathfrak{o}\bullet{}w}(w') \quad \forall\ w,w' \in W^{(1)} \]
Since $\mathfrak{o}$ is assumed to be a spherical orientation, we have $\mathfrak{o}\bullet{}x = \mathfrak{o}$ for every $x \in X^{(1)}$ and hence
\begin{equation}\label{eq:main_thm_eq1} \ell(xw) = \ell(x)+\ell(w) \quad \Rightarrow\quad \widehat{\theta}_\mathfrak{o}(xw) = \widehat{\theta}_\mathfrak{o}(x)\widehat{\theta}_\mathfrak{o}(w) \quad \forall\ x \in X^{(1)},\ w \in W^{(1)}
\end{equation}
We have the equivalences
\begin{align}\label{eq:length_equivs} \ell(xw) = \ell(x)+\ell(w) & \Leftrightarrow \ell(w^{-1}x^{-1}) = \ell(x^{-1}) + \ell(w^{-1}) \\
   & \Leftrightarrow d(C_0,(w^{-1}x^{-1})(C_0)) = d(C_0, w^{-1}(C_0)) + d(C_0,x^{-1}(C_0)) \notag \\
   & \Leftrightarrow d(C_0,(w^{-1}x^{-1})(C_0)) = d(C_0, w^{-1}(C_0)) \notag \\
   & + d(w^{-1}(C_0),(w^{-1}x^{-1})(C_0)) \notag \\
   & \Leftrightarrow w^{-1}(C_0) \preceq_{C_0} (w^{-1}x^{-1})(C_0) \notag
\end{align}
Here the first equivalence is simply the invariance $\ell(w) = \ell(w^{-1})$ of the length under inverses, the third equivalence is the $W$-invariance of the distance $d$ and the last equivalence is by definition (see \cref{rmk:weak_Bruhat_order}).

Let now $w' \in W^{(1)}$ be arbitrary. Because $W = X\rtimes W_0$, we have $\pi(w') \in X w$ for some $w \in W_0$. In particular $(w')^{-1}(C_0) \in X\bullet{}w^{-1}(C_0)$ and hence by definition of $\widetilde{\Lambda_w}$ we can find $\widetilde{w} \in \widetilde{\Lambda_w}$ with $\pi(\widetilde{w}) \in X w$ and $\widetilde{w}^{-1}(C_0) \preceq_{C_0} (w')^{-1}(C_0)$. Hence, for some $x \in X^{(1)}$ we have
\[ w' = x \widetilde{w} \]
From
\[ \widetilde{w}^{-1}(C_0) \preceq (\widetilde{w}^{-1}x^{-1})(C_0) \]
and \eqref{eq:length_equivs} above it therefore follows that
\[ \ell(x\widetilde{w}) = \ell(x) + \ell(\widetilde{w}) \]
Hence, by equation \eqref{eq:main_thm_eq1} above we have
\[ \widehat{\theta}_\mathfrak{o}(w') = \widehat{\theta}_\mathfrak{o}(x)\widehat{\theta}_\mathfrak{o}(\widetilde{w}) \]
and hence $\widehat{\theta}_\mathfrak{o}(w')$ lies in the $\mathcal{A}^{(1)}_\mathfrak{o}$-submodule generated by the set we claim to be a set of generators. Since $w' \in W^{(1)}$ was arbitrary, (vii) follows.

Next, we prove (iii). First, we need to show that $\iota_D$ is well-defined. Recall that it was shown in the proof of \cref{lem:centralizer_prop_lem1} that on the submonoid
\[ X_D = \pi(X^{(1)}_D) = \{ x \in X\ :\ \rho(x) \in \overline{D} \} \leq X \]
the length function is additive
\[ \ell(x+y) = \ell(x)+\ell(y)\quad \forall x,y \in X_D \]
Since the length function on $W^{(1)}$ arises by pullback along $\pi: W^{(1)} \rightarrow W$ of the length function on $W$, it follows that $\ell(xy) = \ell(x)+\ell(y)$ and hence $\overline{\bbf{X}}(x,y) = 1$ for all $x,y \in X^{(1)}_D$. The product formula (\cref{cor:hat mult rule}) therefore implies that $x \mapsto \widehat{\theta}_\mathfrak{o}(x)$ defines a morphism of monoids
\[ X^{(1)}_D \longrightarrow \mathcal{A}^{(1)}_\mathfrak{o} \]
inducing $\iota_D$. Moreover, since $V = \bigcup_D \overline{D}$, we have $X^{(1)} = \bigcup_D X^{(1)}_D$ and therefore $\mathcal{A}^{(1)}_\mathfrak{o}$ is the $R$-module sum of the subalgebra $\iota_D(R[X^{(1)}_D])$ as claimed. Lastly we need to show that the monoid $X^{(1)}_D$ is finitely generated, assuming that $T$ is finitely generated as an abelian group (and hence as a monoid). But since $T = \op{ker}(\pi: W^{(1)} \twoheadrightarrow W)$ is the kernel of $\pi$, it suffices to show that the image $X_D = \pi(X^{(1)}_D)$ is finitely generated as a monoid. But this was shown in \cref{lem:finiteness_of_X_D}.

It remains to show claims (iv),(v),(viii) and (ix). Since the subalgebras
\[ Z(\mathcal{H}^{(1)}) = \left(\mathcal{A}^{(1)}_\mathfrak{o}\right)^{W^{(1)}} \subseteq \left(\mathcal{A}^{(1)}_\mathfrak{o}\right)^{X^{(1)}} \subseteq \mathcal{A}^{(1)}_\mathfrak{o} \subseteq \mathcal{H}^{(1)} \]
have explicit $R$-bases, it is easy to see that they are preserved under base change. If we are in the situation where $T$ is finite, we may therefore reduce to the case $R = R^{\op{univ}}$ of the universal coefficient ring, which exists and is noetherian by \cref{rmk:base_change_and_universal_coefficients}. Therefore, it suffices to prove the claim in the case where $T$ is finitely generated and $R$ is notherian. In this case, claims (iv),(v) and (viii),(ix) each follow from the next lemma. To get (iv) and (v), we apply it with
\[ C = \left(\mathcal{A}^{(1)}_\mathfrak{o}\right)^{X^{(1)}} \subseteq \mathcal{A}^{(1)}_\mathfrak{o} = B \]
\[ \Lambda = \{ \widehat{\theta}_\mathfrak{o}(x) : x \in \Lambda_D,\ \text{$D$ Weyl chamber}\} \subseteq \mathcal{A}^{(1)}_\mathfrak{o} \]
and
\[ \Pi = \{ xyx^{-1}y^{-1} : x,y \in \Lambda \} \subseteq T \subseteq Z(\mathcal{A}^{(1)}_\mathfrak{o}) = \left(\mathcal{A}^{(1)}_\mathfrak{o}\right)^{X^{(1)}} \]
where $\Lambda_D$ denotes any finite set of generators of the monoid $X^{(1)}_D$ whose existence is guaranteed by (iii). Let us verify that assumptions (a)-(d) of the lemma are satisfied. In the discussion of claim (iii) we have seen that (a) holds. To see that (d) holds, note that an element $t \in T$ is annihilated by the monic polynomial
\[ \prod_{t' \in W\bullet{} t} (X-t') \]
with coefficients in
\[ R[T]^W \subseteq \left(\mathcal{A}^{(1)}_\mathfrak{o}\right)^{W^{(1)}} = C \]
Assumption (c) follows by a formal computation from the product formula, the fact that (cf. \cref{rmk:hat comm rules}) 
\[ \widehat{\theta}_\mathfrak{o}(tw) = t \widehat{\theta}_\mathfrak{o}(w) \quad \forall w \in W^{(1)} \]
and the fact that (cf. \cref{rmk:comm subalgebras})
\[ \forall w,w' \in W^{(1)}\quad ww' = w'w\quad \Rightarrow\quad \overline{\bbf{X}}(w,w') = \overline{\bbf{X}}(w',w) \]
Indeed, for any $x,y \in X^{(1)}$ we have
\begin{align*} \widehat{\theta}_\mathfrak{o}(x)\widehat{\theta}_\mathfrak{o}(y) & = \overline{\bbf{X}}(x,y) \widehat{\theta}_\mathfrak{o}(xy) \\
   & = \overline{\bbf{X}}(y,x) \widehat{\theta}_\mathfrak{o}(\underbrace{xyx^{-1}y^{-1}}_{=: t \in \Pi}yx) \\
   & = t \overline{\bbf{X}}(y,x) \widehat{\theta}_\mathfrak{o}(yx) \\
   & = t \widehat{\theta}_\mathfrak{o}(y)\widehat{\theta}_\mathfrak{o}(x)
\end{align*}

Finally we need to verify (b), i.e. we need to provide monic polynomials $f_x(Z) \in \left(\mathcal{A}^{(1)}_\mathfrak{o}\right)^{X^{(1)}}[Z]$ with $f_x(\widehat{\theta}_\mathfrak{o}(x)) = 0$. Even though $\mathcal{A}^{(1)}_\mathfrak{o}$ is possibly non-commutative, it still makes sense to form the polynomial ring $\mathcal{A}^{(1)}_\mathfrak{o}[Z]$ in one variable $Z$ that commutes with $\mathcal{A}^{(1)}_\mathfrak{o}$. For $x \in X^{(1)}$ arbitrary, let $\xi = X^{(1)}\bullet{}x$ and
\[ f_x(Z) := f_\xi(Z) := \prod_{y \in \xi} (Z-\widehat{\theta}_\mathfrak{o}(y)) \in \mathcal{A}^{(1)}_\mathfrak{o}[Z] \]
Note that $\xi$ is finite since $W^{(1)}$ acts on $X^{(1)}$ with finite orbits by assumption. However, a priori the above expression is still ill-defined, as it depends on the choice of an ordering of the factors. However, the elements $\widehat{\theta}_\mathfrak{o}(y)$ with $y \in \xi$ in fact commute with each other pairwise. This follows from \cref{rmk:comm subalgebras} and the fact that the elements of the orbit $\xi$ themselves commute with each other pairwise, as an easy computation shows. The expression $f_\xi(Z)$ therefore is well-defined. Moreover, the $R$-algebra action of $W^{(1)}$ on $\mathcal{A}^{(1)}_\mathfrak{o}$ extends to $\mathcal{A}^{(1)}_\mathfrak{o}[Z]$ by acting on coefficients. A formal computation shows that $f_\xi(Z)$ is invariant under $X^{(1)}$ with respect to this action, hence we have a well-defined element
\[ f_\xi(Z) \in \left(\mathcal{A}^{(1)}_\mathfrak{o}\right)^{X^{(1)}} [Z] \]
Moreover, $f_\xi(Z)$ annihilates $\widehat{\theta}_\mathfrak{o}(x)$ which can be seen as follows. Let $A$ denote the subalgebra of $\mathcal{A}^{(1)}_\mathfrak{o}$ generated by the $\widehat{\theta}_\mathfrak{o}(y)$, $y \in \xi$ over the center $\left(\mathcal{A}^{(1)}_\mathfrak{o}\right)^{X^{(1)}}$ of $\mathcal{A}^{(1)}_\mathfrak{o}$. From the previous remarks it follows that $A$ is commutative. Moreover, we have
\[ f_\xi(Z) = \prod_{y \in \xi} (Z-\widehat{\theta}_\mathfrak{o}(y)) \in A[Z] \]
as an equation in $A[Z]$. Using the evaluation homomorphism
\begin{align*} \op{ev}: A[Z] & \longrightarrow A \\
   f(Z) & \longmapsto f(\widehat{\theta}_\mathfrak{o}(x))
\end{align*}
it follows that
\[ f_\xi(\widehat{\theta}_\mathfrak{o}(x)) = \op{ev}(f_\xi) = \prod_{y \in \xi} (\widehat{\theta}_\mathfrak{o}(x)-\widehat{\theta}_\mathfrak{o}(y)) = 0 \]
Thus the assumptions of the next lemma are satisfied and (iv) and (v) follow. In order to get claims (viii) and (ix), we apply the lemma with $B$,$\Lambda$ and $\Pi$ as before but with
\[ C = \left(\mathcal{A}^{(1)}_\mathfrak{o}\right)^{W^{(1)}} \]
In order to see that assumption (b) of the lemma is satisfied, it suffices to show that
\[ g_x(Z) := \prod_{\eta \in W_0\bullet{}\xi} f_\eta(Z) \in \left(\mathcal{A}^{(1)}_\mathfrak{o}\right)^{X^{(1)}}[Z] \]
has coefficients in $\left(\mathcal{A}^{(1)}_\mathfrak{o}\right)^{W^{(1)}}$, i.e. that it is invariant under $W_0$. But, considering the expression of the coefficients of $f_\xi$ as symmetric polynomials in the $\widehat{\theta}_\mathfrak{o}(y)$, $y \in \xi$, it follows that
\[ w(f_\xi) = f_{w(\xi)} \quad \forall w \in W_0 \]
Hence, it follows that $g_x$ is invariant under $W_0$ by a formal computation.
\end{proof}

\begin{lemma}
   Let $R$ be a commutative ring, $B$ a not necessarily commutative $R$-algebra and $C \subseteq Z(B)$ an $R$-subalgebra of the center of $B$. Assume that there exist finite subsets
   \[ \Lambda = \{x_1,\ldots,x_n\} \subseteq B \]
   and
   \[ \Pi \subseteq Z(B) \]
   such that
   \begin{enumerate}
      \item[(a)] $B$ is generated as an $R$-algebra by $\Lambda$.
      \item[(b)] Every $x_i \in \Lambda$ satisfies a monic equation
         \[ f_i(x_i) = 0,\quad f_i(X) = X^{n_i} + a_{i,1} X^{n_i-1} + \ldots + a_{i,n_i} \in C[X] \]
      with coefficients in $C$.
      \item[(c)] The generators commute up to elements of $\Pi$, i.e.
         \[ \forall x,y \in \Lambda\ \exists t \in \Pi\ \ x y = y x \]
      \item[(d)] Every $t \in \Pi$ satisfies a monic equation with coefficients in $C$, i.e. the $C$-subalgebra $C[\Pi] \subseteq Z(B)$ generated by $\Pi$ is finitely generated as a $C$-module.
   \end{enumerate}
   Then
   \begin{enumerate}
      \item[(i)] $B$ is generated as a $C$-module by
         \[ \{ t x_1^{\nu_1} \ldots x_n^{\nu_n} : t \in C[\Pi],\ 0 \leq \nu_i < n_i\ \forall i \} \]
         In particular $B$ is a finitely generated $C$-module.
      \item[(ii)] If $R$ is noetherian, then $C$ is a finitely generated $R$-algebra.
   \end{enumerate}
\end{lemma}
\begin{proof}
   Claim (i) follows immediately by combining assumptions (a)-(c). Claim (ii) follows as in the classical commutative case by dévissage. Let $C'$ be the $R$-subalgebra of $C$ generated by the coefficients $a_{i,j}$ and the coefficients of the monic equations satisfied by the elements of $\Pi$, and let $C'\left<x_1,\ldots,x_n\right>$ be the $C'$-subalgebra of $B$ generated by $x_1,\ldots,x_n$. This situation is summarized in the following diagram
   \[ C' \subseteq C \subseteq C'\left<x_1,\ldots,x_n\right> \subseteq B \]
   The assumption of this lemma are still satisfied if one replaces $C$ by $C'$ and $B$ by $C'\left<x_1,\ldots,x_n\right>$. From part (i) it therefore follows that $C'\left<x_1,\ldots,x_n\right>$ is a finite $C'$-module. Since $C'$ is the homomorphic image of a polynomial ring over $R$ in a finite number of variables it is noetherian, hence it follows that the submodule $C \subseteq C'\left<x_1,\ldots ,x_n\right>$ is also finitely generated. In particular $C$ is a finitely generated $C'$-algebra. Since $C'$ is a finitely generated $R$-algebra, it follows by transitivity that $C$ is a finitely generated $R$-algebra.
\end{proof}

\begin{rmk}\label{rmk:finiteness_of_orbits}
   In some of the finiteness results proved in the main theorem we had to assume that $W^{(1)}$ acts with finite orbits on $X^{(1)}$. Let us see what this condition amounts to. Since $W_0 \simeq W^{(1)}/X^{(1)}$ is finite, the group $W^{(1)}$ acts by finite orbits if and only if the subgroup $X^{(1)}$ acts by finite orbits. But, if $x,y \in X^{(1)}$ then by definition
   \[ \pi(x)\bullet{}y = xyx^{-1} = \underbrace{xyx^{-1}y^{-1}}_{=: [x,y]} y \]
   Since $X$ is commutative by assumption, the commutator $[x,y]$ lies in $T$. Thus if $T$ is finite, the group $W^{(1)}$ always acts with finite orbits.
   
   Let us now consider the case when $T$ is contained in the center of $X^{(1)}$ (but possibly infinite). This means that $X^{(1)}$ is a central extension
   \[ \begin{xy} \xymatrix{ 1 \ar[r] & T \ar[r] & X^{(1)} \ar[r] & X \ar[r] & 0 } \end{xy} \]
   of abelian groups, and therefore the commutator $[x,y]$ only depends on $\pi(x)$ and $\pi(y)$ and gives rise to an alternating bilinear pairing
   \[ [\--,\--]: X\times X \longrightarrow T \]
   By the above computation the orbit of an element $y \in X^{(1)}$ under $X$ is given by the coset
   \[ X^{(1)}\bullet{} y = [X,\pi(y)] y \]
   under the subgroup
   \[ [X,\pi(y)] \leq T \]
   This subgroup is always finitely generated, since $X$ is finitely generated by assumption. It is therefore finite if and only if it lies in the torsion subgroup $T_{\op{tors}} \leq T$. Thus, when $T$ is contained in the center of $X^{(1)}$, the group $W^{(1)}$ (actually $W$) acts with finite orbits if and only if the pairing $[\--,\--]$ takes values in $T_{\op{tors}}$. This is for instance the case if $X^{(1)}$ is abelian or $T$ is finite.
\end{rmk}

\printbibliography

\end{document}